\documentclass{amsart}
\pdfoutput=1
\usepackage[margin=1in]{geometry}
\usepackage[utf8]{inputenc}
\usepackage{graphicx,cite}
\usepackage{amsmath,amssymb}
\usepackage{amsthm}
\usepackage{color}
\usepackage{tikz}
\usetikzlibrary{arrows.meta,positioning,calc}
\usepackage{pdflscape}
\usetikzlibrary{arrows.meta,calc}
\usepackage{adjustbox}
\usetikzlibrary{decorations.pathreplacing}
\usepackage{float}
\usepackage{calc}
\usepackage{hyperref, url}
\usepackage{caption}
\allowdisplaybreaks

\makeatletter
\def\l@section{\@tocline{1}{12pt plus2pt}{0pt}{}{\bfseries}}
\def\l@subsection{\@tocline{2}{0pt}{2pc}{2pc}{}}
\makeatother
\makeatletter
\def\subsection{\@startsection{subsection}{2}{\z@}%
	{-3.25ex\@plus -1ex \@minus -.2ex}%
	{1.5ex \@plus .2ex}%
	{\normalfont\bfseries\boldmath}}
\def\subsubsection{\@startsection{subsubsection}{3}%
	\z@{.5\linespacing\@plus.7\linespacing}{-.5em}%
	{\normalfont\bfseries\boldmath}}
\renewcommand\paragraph{\@startsection{paragraph}{4}{\z@}%
	{3.25ex \@plus1ex \@minus.2ex}%
	{-1em}%
	{\normalfont\normalsize\bfseries}}
\makeatother

\theoremstyle{plain}
\newtheorem{thm}{Theorem}[section]

\newtheorem{lem}[thm]{Lemma}
\newtheorem{prop}[thm]{Proposition}

\theoremstyle{definition}
\newtheorem{defn}[thm]{Definition}

\theoremstyle{remark}
\newtheorem{rem}[thm]{Remark}
\newtheorem{obs}[thm]{Observation}
\newtheorem{exa}[thm]{Example}
\theoremstyle{plain}

\numberwithin{equation}{section}
\newtheorem{ques}[thm]{Question}

\theoremstyle{plain} 
\newcommand{\thistheoremname}{}
\newtheorem{genericthm}[thm]{\thistheoremname}

  \newtheorem*{genericthm*}{\thistheoremname}
\newenvironment{namedthm*}[1]
  {\renewcommand{\thistheoremname}{#1}%
   \begin{genericthm*}}
  {\end{genericthm*}}

\newcommand{\del}{\delta}

\newcommand{\B}{{\mathbb B}}
\newcommand{\D}{{\mathbb D}}

\newcommand{\R}{{\mathbb R}}

\newcommand{\C}{{\mathbb C}}

\newcommand{\N}{{\mathbb N}}

\newcommand{\calD}{{\mathcal D}}

\newcommand{\calG}{{\mathcal G}}

\newcommand{\calA}{{\mathcal A}}

\newcommand{\calP}{{\mathcal P}}

\newcommand{\calS}{{\mathcal S}}
\newcommand{\calH}{{\mathcal H}}
\newcommand{\calI}{{\mathcal I}}

\makeatletter
\newcommand{\vast}{\bBigg@{4}}
\newcommand{\Vast}{\bBigg@{5}}
\makeatother

\def\udot#1{\ifmmode\oalign{$#1$\crcr\hidewidth.\hidewidth
    }\else\oalign{#1\crcr\hidewidth.\hidewidth}\fi}

\def\R{\mathbb{R}}

\def\T{\mathbb{T}}
\def\C{\mathbb{C}}

\def\beq{\begin{equation}}
\def\eeq{\end{equation}}
\makeatletter
\newcommand{\doublewidetilde}[1]{{%
  \mathpalette\double@widetilde{#1}%
}}
\newcommand{\double@widetilde}[2]{%
  \sbox\z@{$\m@th#1\widetilde{#2}$}%
  \ht\z@=.9\ht\z@
  \widetilde{\box\z@}%
}
\makeatother

\def\one{\mbox{1\hspace{-4.25pt}\fontsize{12}{14.4}\selectfont\textrm{1}}}

\usepackage{tikz}

\makeatletter
\def\@makefnmark{%
  \leavevmode
  \raise.9ex\hbox{\fontsize\sf@size\z@\normalfont\tiny\@thefnmark}}
\makeatother

\begin{document}
	
\title[]{Cancellation of complex kernels and sharp critical lines and endpoint theory for the Forelli--Rudin operators I: the purely hypersingular case}

\author{Bingyang Hu}
\address{(Bingyang Hu) Department of Mathematics and Statistics\\
       Auburn University\\
       Auburn, Alabama, U.S.A, 36849}
\email{bzh0108@auburn.edu}

\author{Zipeng Wang}
\address{(Zipeng Wang) College of Mathematics and Statistics \\
Chongqing University \\
Chongqing, China, 401331}
\email{zipengwang2012@gmail.com; zipengwang@cqu.edu.cn}

\author{Kenan Zhang}
\address{(Kenan Zhang) School of Mathematical Sciences\\
Fudan University\\
Shanghai, China, 200433}
\email{knzhang21@m.fudan.edu.cn}

\author{Xiaojing Zhou}
\address{(Xiaojing Zhou) Department of Mathematics and Statistics\\
         Auburn University\\
         Auburn, Alabama, U.S.A, 36849}
\email{xiz0003@auburn.edu}

\begin{abstract}
For $a,b,c\in\mathbb R$, we consider the Forelli--Rudin operators
$$
T_{a,b,c}f(z):=(1-|z|^2)^a\int_{\mathbb D}\frac{(1-|w|^2)^b}{(1-z\overline w)^c}f(w)\,dA(w)
$$
and their positive counterparts
$$
S_{a,b,c}f(z):=(1-|z|^2)^a\int_{\mathbb D}\frac{(1-|w|^2)^b}{|1-z\overline w|^c}f(w)\,dA(w).
$$
We obtain a complete and sharp classification of their weak- and restricted weak-type mapping properties in the hypersingular regime
$$
\Omega_{\mathcal H}:=\{(p,q):1\leq p,q\leq\infty,\ p>q\}, 
$$
thereby substantially extending the recent work of the first and fourth authors on hypersingular Bergman projections.
One of the main discoveries of this work is an intrinsic cancellation phenomenon associated with the complex Forelli--Rudin kernel: at certain critical endpoints, cancellation creates a sharp separation between the two operators, with $T_{a,b,c}$ remaining bounded while its positive counterpart $S_{a,b,c}$ fails to be bounded. Perhaps surprisingly, this cancellation is invisible in the strong $L^p$--$L^q$ theory established by Zhao and Zhou in 2022, where the two operators have the same boundedness range, and emerges only at the weak- and restricted weak-type levels.

Allowing the parameters $a,b,c$ to vary, we show that the collection of all weak-type Forelli--Rudin pairs in $\Omega_{\calH}$ is precisely
$$
\mathcal{FR}_w=\{(p,q)\in\Omega_{\mathcal H}:1<p\leq2\},
$$
whereas the collection of all restricted weak-type Forelli--Rudin pairs is
$$
\mathcal{FR}_{rw}=\{(p,q)\in\Omega_{\mathcal H}:p\neq\infty\}.
$$
These ranges, together with all corresponding endpoint estimates and failures, are sharp. Our approach combines dyadic decompositions, probabilistic constructions, and weak-type Hardy estimates.
\end{abstract}

\date{\today}
\subjclass[2020]{Primary 32A25; Secondary 30H20, 42B20, 46E30}

\keywords{Forelli--Rudin operators, Bergman-type integral operators, Bergman projections, weak-type estimates, restricted weak-type estimates, Lorentz spaces, dyadic analysis}

\maketitle
\tableofcontents

\section{Introduction} \label{Sec01} 
The purpose of this paper is to address an open question raised in the recent work \cite{HZ2025} of the first and fourth authors concerning sharp endpoint and critical-line estimates for general Bergman-type operators, also known as Forelli--Rudin operators. More precisely, for $a,b,c\in\mathbb{R}$, we consider the \emph{Forelli--Rudin type operators}
\begin{equation} \label{20260514eq02}
T_{a,b,c}f(z):=(1-|z|^2)^a\int_{\mathbb D}
\frac{(1-|w|^2)^b}{(1-z\overline w)^c}f(w)\,dA(w),
\end{equation} 
together with their positive counterparts, the Berezin-type operators
\begin{equation} \label{20260514eq03}
S_{a,b,c}f(z):=(1-|z|^2)^a\int_{\mathbb D}
\frac{(1-|w|^2)^b}{|1-z\overline w|^c}f(w)\,dA(w), 
\end{equation} 
where $\D$ is the unit disc\footnote{For simplicity and clarity of exposition, we first formulate our main results in the setting of the unit disc. See Section \ref{Sec14} for the extension of these results to the higher-dimensional setting.} in $\C$ and $dA$ is the normalized area measure on $\D$. The operators $T_{a,b,c}$ and $S_{a,b,c}$, dating back to the early work of
Forelli and Rudin \cite{FR1974}, arise naturally as far-reaching generalizations
of the classical Bergman projection and its positive counterpart,
\begin{equation} \label{20260513eq00}
Pf(z):=\int_{\mathbb D}\frac{f(w)}{(1-z\overline w)^2}\,dA(w),
\qquad
P^{+}f(z):=\int_{\mathbb D}\frac{f(w)}{|1-z\overline w|^2}\,dA(w).
\end{equation} 
These operators play a fundamental role in classical complex function theory; see, e.g.,  \cite{HKZ2000,DurenSchuster2004,Zhu2005}. A fruitful recent line of research uses modern developments in dyadic harmonic analysis to study the Bergman projection $P$ and its positive counterpart $P^+$, as well as their fractional analogues. Topics in this direction include, but are not limited to,
\begin{enumerate}
	\item weighted estimates and B\'ekoll\'e--Bonami theory; see, e.g., \cite{PottReguera2013,RTW2017,Sehba2018};
	
	\item applications of dyadic decompositions and sparse domination to Bergman-type operators on domains in several complex variables; see, e.g., \cite{GHK2022,HWW2021,WW2021};
	
	\item commutator and BMO/VMO theory; see, e.g., \cite{Zhu1992,LiLuecking1994,HHLW2024}.
\end{enumerate}

One \emph{key} reason why dyadic harmonic analysis enters the theory of the
Bergman projection is that $P$ and $P^+$ can be viewed as generalized
Calder\'on--Zygmund operators (see, e.g., \cite{McNeal94}). However, this Calder\'on--Zygmund viewpoint does \emph{not} apply in full generality:
the Forelli--Rudin operators $T_{a,b,c}$ and $S_{a,b,c}$ may exhibit substantially
more singular behavior. A model example is given by the hypersingular Bergman
projection and its positive counterpart. For $1<t<3/2$, define
\begin{equation} \label{20260514eq30}
K_{2t}f(z):=\int_{\mathbb D}\frac{f(w)}{(1-z\overline w)^{2t}}\,dA(w),
\qquad
K_{2t}^{+}f(z):=\int_{\mathbb D}\frac{f(w)}{|1-z\overline w|^{2t}}\,dA(w).
\end{equation} 
These operators are more singular than the classical Bergman projection $P$ and
its positive counterpart $P^+$. More precisely, it is known that $K_{2t}$ and
$K_{2t}^{+}$
\begin{enumerate}
    \item[$\bullet$] map $L^p(\mathbb D)$ boundedly into $L^q(\mathbb D)$ whenever
    $$
    \frac{1}{q}-\frac{1}{p}>2t-2,
    \qquad 1\le p,q\le \infty;
    $$
    \item[$\bullet$] map $L^p(\mathbb D)$ boundedly into $L^{q,\infty}(\mathbb D)$ whenever
    \begin{equation} \label{20260513eq01}
    \frac{1}{q}-\frac{1}{p}=2t-2,
    \qquad 1\le p,q\le \infty.
    \end{equation} 
    \end{enumerate}
Moreover, the above $(p,q)$-admissible ranges are sharp. In particular, on the
critical line \eqref{20260513eq01}, the corresponding strong-type estimates for $K_{2t}$ and $K_{2t}^{+}$ fail. This
is in sharp contrast with the classical Calder\'on--Zygmund theory, as well as
with the $L^p$ theory of the Bergman projection.

We briefly recall the literature on the study of $K_{2t}$ and $K_{2t}^{+}$. To the best of our knowledge, the strong-type estimates for $K_{2t}$ were first established by Cheng, Fang, the second author, and Yu \cite{CFWY2017}. These estimates were later extended by Zhao and Zhou \cite{ZZ2022} to the general Forelli--Rudin operators $T_{a,b,c}$ and $S_{a,b,c}$. It is worth noting that the methods in these works rely primarily on complex-analytic and functional-analytic techniques (due in part to the failure of strong-type estimates on the critical line; see \eqref{20260513eq01}). Only very recently, inspired by recent developments in modern dyadic harmonic analysis, the first and fourth authors \cite{HZ2025} introduced a real-variable framework, referred to as the Forelli--Rudin method, and used it to establish critical-line weak-type estimates for $K_{2t}$ and $K_{2t}^{+}$.

\vspace{0.1cm}

The first \emph{theme} of the present paper is to further the line of research
initiated in \cite{HZ2025} by studying sharp endpoint and critical-line behavior
for $T_{a,b,c}$ and $S_{a,b,c}$. We emphasize that the analysis of the general Forelli--Rudin operators is substantially \emph{more} delicate: in addition to the critical line, as in \eqref{20260513eq01}, one must also take into account two additional cut-off lines that naturally enter the picture and play an essential role in determining the admissible range. As a consequence, the present paper requires a richer toolbox. Beyond the Forelli--Rudin method, our arguments bring together techniques from probability, complex function theory, and dyadic harmonic analysis, whose interaction is
essential to the analysis. We refer the reader to Section~\ref{subsechighlights} for the highlights of the present paper.

\vspace{0.1cm}

The second \emph{theme}, which emerged along the way in our study of the endpoint
and critical-line behavior of $T_{a,b,c}$ and $S_{a,b,c}$, is a \emph{surprising
intrinsic fact} about the \emph{cancellation of the complex kernel}, which we
refer to as the \emph{Forelli--Rudin phenomenon}. We first take a brief detour
to explain the motivation behind this phenomenon.

\vspace{0.1cm}

\noindent \textbf{$\clubsuit$ \underline{A Folklore Puzzle in Classical Complex Function Theory: }} Consider
the classical Bergman projection $P$ and its positive counterpart $P^+$, defined
as in \eqref{20260513eq00}. On the one hand, it is easy to see that
\[
(P1)(z)=\int_{\mathbb D}\frac{dA(w)}{(1-z\overline w)^2}=1,
\]
whereas
\[
(P^+1)(z)=\int_{\mathbb D}\frac{dA(w)}{|1-z\overline w|^2}
\simeq 1+\log\frac{1}{1-|z|},
\qquad z\in\mathbb D.
\]
In particular, $P1\in L^\infty(\mathbb D)$, while\footnote{Here, $\mathrm{BMO}(\mathbb D)$ can either be interpreted as  the usual Euclidean BMO space on
$\mathbb D$, or the BMO spaces over Bergman metric balls in the spirit of
B\'ekoll\'e, Berger, Coburn, and Zhu \cite{BBCZ1990}.}
$P^+1\in \mathrm{BMO}(\mathbb D)$.

\vspace{0.1cm}

On the other hand, it is well known that
\[
\text{$P$ and $P^+$ have the same strong-type $L^p$ boundedness range.}
\]
Moreover, the same conclusion holds for the general Forelli--Rudin operators
$T_{a,b,c}$ and $S_{a,b,c}$ (see, e.g., \cite{ZZ2022}). Heuristically, this
suggests that, \emph{at the level of strong-type estimates, the cancellation of
the complex kernel $(1-z\overline w)^{-2}$ is invisible}. This differs sharply from the
situation for the Hilbert transform, and more generally for Calder\'on--Zygmund
operators, where cancellation assumptions on the kernel are essential for the operator to be well-defined and bounded.

\vspace{0.1cm}

This naturally leads to the following question.

\begin{ques} \label{20260520ques01}
Where does the cancellation of the complex kernel become visible, and through what mechanism?
\end{ques}

This motivates the following definition, which should be viewed as an intrinsic
complex-analytic property of the Forelli--Rudin operators.

\begin{defn}
Let $a,b,c\in\mathbb R$, and let $T_{a,b,c}$ and $S_{a,b,c}$ be defined as in \eqref{20260514eq02} and \eqref{20260514eq03}, respectively. 
For $1\le p,q\le \infty$, we say that $(p,q)$ is a \emph{Forelli--Rudin pair} for the parameters $(a,b,c)$ if $T_{a,b,c}$ and $S_{a,b,c}$ have different boundedness behavior at $(p,q)$ with respect to a type of estimate. 
We refer to this phenomenon as the \emph{Forelli--Rudin phenomenon}.
\end{defn}

\begin{rem} \label{20260521rem01}
We point out that a related form of the Forelli--Rudin phenomenon has appeared in the theory of Bergman projections induced by radial weights, especially for radial weights with rapid growth\footnote{Here, a radial weight with rapid growth refers to a radial weight whose boundary growth is faster than that of the standard Bergman weights $(1-r^2)^\alpha$, $\alpha>-1$, which are commonly used in the classical
theory of Bergman spaces (see, e.g., \eqref{20260520eq01}).} near the unit circle $\mathbb T$. Here is a model case. Let
\begin{equation} \label{20260520eq01}
\omega(r):=\frac{1}{(1-r^2)\left(\log \frac{e}{1-r}\right)^2},
\qquad 0\le r<1,
\end{equation} 
and define the radial weight $\omega(z):=\omega(|z|)$ for $z\in\D$. Let the weighted Bergman projection $P_\omega$ be given by
$$
P_{\omega}f(z):=\int_{\D} f(w)B_z^{\omega}(w)\omega(w)\,dA(w),
$$
and its associated positive operator be
$$
P_{\omega}^+f(z):=\int_{\D} f(w)|B_z^{\omega}(w)|\omega(w)\,dA(w),
$$
where $B_z^{\omega}$ denotes the reproducing kernel of the weighted Bergman space $A_\omega^2(\D)$. It is known that, for every $1<p<\infty$, $P_{\omega}$ is bounded on $L_\omega^p(\D)$, while
$P_{\omega}^+$ fails to be bounded on $L_\omega^p(\D)$. We refer the reader to the excellent work \cite{PR2021} of Pel\'aez and R\"atty\"a, where the above model case is treated as part of a more general theory. For further background and recent developments on Bergman projections induced by radial weights, see also \cite{Dostanic2004,PR2014,PR2016,PRW2019,MOPR2023,PdRR2024} and the references therein.
\end{rem}

The example above suggests one possible approach to Question
\ref{20260520ques01}: one may study the cancellation of complex kernels arising from Bergman projections induced by rapidly increasing radial weights, such as the kernel \(B_z^\omega\) in the preceding example.  However, this is
\emph{not} the direction pursued in the present paper.  Question \ref{20260520ques01} is concerned with the cancellation of the Bergman-type kernel as a \emph{complex-analytic object} in its own right.  In the radially weighted setting, the kernel is itself induced by the underlying weight, and the observed
cancellation reflects the combined effect of the analytic structure of the kernel and the special boundary behavior of the weight.

Another reason for not taking this route is that the associated weighted Bergman spaces $A_\omega^p(\D)$ are structurally quite different from the classical weighted Bergman spaces $A_\alpha^p(\D)$, $\alpha>-1$.  For such weights, the stronger boundary growth imposes a much more restrictive integrability condition near $\mathbb T$.  In particular, the usual
Littlewood--Paley formula fails in this setting (see, e.g., 
\cite[Theorems 5 and 6]{PR2021}).

For these reasons, the present paper takes a different point of view.  Instead of working in the above framework of Bergman projections induced by radial weights, we take a more intrinsic point of view.  Namely, we focus on the complex-analytic content of Question \ref{20260520ques01} and study the
cancellation of the Bergman-type kernel itself, without introducing an auxiliary radial weight that changes the kernel.
Our main result is,

\begin{thm}[Informal] \label{20260514thm01}
Let
\begin{equation} \label{20260515eq01}
\Omega_{\mathcal H}
:=
\left\{
(p,q): 1\le p,q\le \infty,\ p>q
\right\}
\end{equation} 
denote the hypersingular regime. Then the following hold.

\begin{enumerate}
    \item At the level of weak-type estimates, the Forelli--Rudin phenomenon occurs.
    More precisely, the collection of all Forelli--Rudin pairs is
    $$
    \mathcal{FR}_w:=\left\{
    (p,q)\in \Omega_{\mathcal H}: 1<p\le 2
    \right\}.
    $$
    In other words, if $(p,q)\in \mathcal{FR}_w$, then there exist parameters
    $a,b,c\in\mathbb R$ such that
    $$
    T_{a,b,c}: L^p(\D)\to L^{q,\infty}(\D)
    \quad \textnormal{holds},
    \qquad
    \textnormal{whereas}
    \quad
    S_{a,b,c}: L^p(\D)\to L^{q,\infty}(\D)
    \quad \textnormal{fails}.
    $$

    \item At the level of restricted weak-type estimates, the Forelli--Rudin
    phenomenon occurs again. In this case, the collection of all Forelli--Rudin
    pairs is
    $$
    \mathcal{FR}_{rw}
    :=
    \left\{
    (p,q)\in \Omega_{\mathcal H}: p\neq \infty
    \right\}.
    $$
    That is, if $(p,q)\in \mathcal{FR}_{rw}$, then there exist parameters
    $a,b,c\in\mathbb R$ such that
    $$
    T_{a,b,c}: L^{p,1}(\D)\to L^{q,\infty}(\D)
    \quad \textnormal{holds},
    \qquad
    \textnormal{whereas}
    \quad
    S_{a,b,c}: L^{p,1}(\D)\to L^{q,\infty}(\D)
    \quad \textnormal{fails}.
    $$
\end{enumerate}
We refer the reader to Figure~\ref{Fig1} below for the corresponding regions of
Forelli--Rudin pairs at the levels of weak-type and restricted weak-type estimates, respectively.
\end{thm}

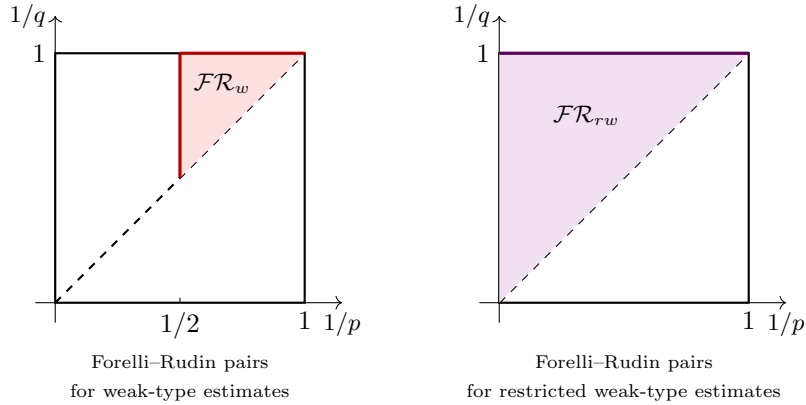
\begin{figure}[H]
\centering
\begin{tikzpicture}[scale=3.3]

\begin{scope}
\draw[->] (-0.08,0) -- (1.15,0) node[below] {\small{$1/p$}};
\draw[->] (0,-0.08) -- (0,1.15) node[left] {\small{$1/q$}};
\draw[thick] (0,0) rectangle (1,1);

\draw (0.5,0.02) -- (0.5,0) node[below] {$1/2$};
\draw (1,0.02) -- (1,0) node[below] {$1$};
\draw (0.02,1) -- (0,1) node[left] {$1$};

\draw[dashed, thick] (0,0) -- (1,1);

\fill[red!12] (0.5,0.5) -- (0.5,1) -- (1,1) -- cycle;
\draw[red!70!black, very thick] (0.5,0.5) -- (0.5,1);
\draw[red!70!black, very thick] (0.5,1) -- (1,1);
\node at (0.67,0.88) {\small{$\mathcal{FR}_w$}};
\node[align=center] at (0.5,-0.3)
{\scriptsize Forelli--Rudin pairs\\[-1pt]
 \scriptsize for weak-type estimates};
\end{scope}

\begin{scope}[shift={(1.78,0)}]
\draw[->] (-0.08,0) -- (1.15,0) node[below] {\small{$1/p$}};
\draw[->] (0,-0.08) -- (0,1.15) node[left] {\small{$1/q$}};
\draw[thick] (0,0) rectangle (1,1);
\draw (1,0.02) -- (1,0) node[below] {$1$};
\draw (0.02,1) -- (0,1) node[left] {$1$};
\draw[dashed, thick] (0,0) -- (1,1);
\fill[violet!12] (0,1) -- (1,1) -- (0,0)-- cycle;
\draw[violet!70!black, very thick] (0,1) -- (1,1);
\node at (0.35,0.75) {\small{$\mathcal{FR}_{rw}$}};
\node[align=center] at (0.5,-0.3)
{\scriptsize Forelli--Rudin pairs\\[-1pt]
 \scriptsize for restricted weak-type estimates};
\end{scope}
\end{tikzpicture}
\caption{\small The Forelli--Rudin Phenomenon.}
\label{Fig1}
\end{figure}

The preceding theorem describes all possible Forelli--Rudin pairs as the parameters vary. 
We next give the corresponding parameterwise characterization.
\begin{thm}\label{thm:Forelli--Rudin}
Let $a,b,c\in\R$.
\begin{itemize}
	\item[(i)] The weak-type Forelli--Rudin phenomenon occurs in $\Omega_{\mathcal H}$ if and only if
		\begin{equation}\label{eq:weak}
			c>1,
			\qquad
			\frac{1}{2}\le 2-c+b<-a\le1.
		\end{equation}
		If condition \eqref{eq:weak} holds, the weak-type Forelli--Rudin pair is unique and is given by
		\[
			\left(\frac{1}{p},\frac{1}{q}\right)=(2-c+b,-a).
		\]
	\item[(ii)] The restricted weak-type Forelli--Rudin phenomenon occurs in $\Omega_{\mathcal H}$ if and only if
		\begin{equation}\label{eq:restricted-weak}
			c=1,
			\qquad
			0<b+1<-a\le1.
		\end{equation}
		If condition \eqref{eq:restricted-weak} holds, the restricted weak-type Forelli--Rudin pair is unique and is given by
		\[
			\left(\frac{1}{p},\frac{1}{q}\right)=(b+1,-a).
		\]
\end{itemize}
\end{thm}

Theorem~\ref{thm:Forelli--Rudin} follows from the preliminary reductions in Section~\ref{Sec03} and the complete classifications obtained in Sections~\ref{Sec10} and~\ref{Sec12}. We make several remarks on the Forelli--Rudin phenomenon for $T_{a, b, c}$ and $S_{a, b, c}$ in the hypersingular regime.

\begin{rem}
\begin{enumerate}
\item Observe that, in Theorem~\ref{20260514thm01}, a natural threshold $p=2$ appears in the definition of $\mathcal{FR}_w$ at the level of weak-type estimates. This threshold is an intrinsic complex-analytic feature of the general Forelli--Rudin operators. The analysis at such threshold is the crux of our endpoint and critical-line analysis, where methods from probability, complex function theory, and dyadic harmonic analysis meet in perfect harmony. More precisely, $p=2$ is
\begin{enumerate}
    \item the threshold arising from the exponent $2$ in Khintchine's inequality,
    where the fourth-moment probabilistic construction begins to take effect; see Theorem \ref{Forelli--Rudin threshold}, (ii);  
    \item the threshold at which the complex Hardy space method reaches its natural limit, through the embedding of a Hardy space into a Bergman space, in the spirit of \cite[Theorem~4.41]{Zhu2005}; see Theorem \ref{Forelli--Rudin threshold}, (iii). 
\end{enumerate}

Thus, the complex-analytic mechanism stops exactly where the probabilistic mechanism begins. This subtle behavior becomes invisible once the complex kernel is replaced by its absolute value. In particular, it cannot be detected by dyadic Carleson embedding arguments for the corresponding positive operator at
any Forelli--Rudin pair $(p,q)$ for weak-type estimates; see Theorem \ref{Forelli--Rudin threshold}, (i). 

Taken together, these phenomena show that, at the level of weak-type estimates,
the Forelli--Rudin phenomenon persists precisely in the range $1<p\le 2$,
whereas it disappears once $p>2$. Therefore, we refer to $p=2$ as the \emph{Forelli--Rudin threshold for weak-type estimates}.

\vspace{0.1cm}

\item The phenomena established in Theorems \ref{20260514thm01} and \ref{thm:Forelli--Rudin} are in sharp contrast with the example of rapidly increasing radial weights discussed in Remark~\ref{20260521rem01}, where the cancellation effect is visible along the entire diagonal line. For $T_{a,b,c}$ and $S_{a,b,c}$ with\footnote{This is the main case in the purely hypersingular regime; see
\eqref{maincase}.} 
$$
0<c-a-b-2<1,
$$
the corresponding Forelli--Rudin phenomenon can occur, if at all, \emph{only} at a single endpoint. To visualize Forelli--Rudin pairs, we recall the picture behind the
strong-type theory. For $T_{a,b,c}$ and $S_{a,b,c}$, this theory gives a natural critical line \eqref{20260522eq01} together with the vertical and horizontal
cut-off lines \eqref{20260515eq11} and \eqref{20260515eq12}, respectively; these three lines describe the boundary of the strong-type regime in \cite[Theorem~1.1]{ZZ2022}. We refer the reader to Section~\ref{20260522sec01} for a brief overview.

With this picture in mind, a Forelli--Rudin pair $(p,q)\in \mathcal{FR}_w$ at the level of weak-type estimates occurs precisely
at the intersection of the critical line and the horizontal cut-off line, provided
that this intersection point lies strictly on the left-hand side of the vertical cut-off line (see Figure \ref{Fig5}).

\begin{figure}[ht]
\centering
\begin{tikzpicture}[scale=4]
\def\yh{0.78}     
\def\xP{0.58}     
\def\xaux{0.72}   
\coordinate (P) at (\xP,\yh);

\coordinate (Q) at (\xaux,0.92);

\draw[->] (-0.08,0) -- (1.15,0) node[below] {\small{$1/p$}};
\draw[->] (0,-0.08) -- (0,1.15) node[left] {\small{$1/q$}};
\fill[red!12, fill opacity=0.7] (0.5,0.5) -- (0.5,1) -- (1,1) -- cycle;
\fill[gray!30, fill opacity=0.5] (0,\yh) -- (P) -- (Q) -- (\xaux,1) -- (0,1) -- cycle;
\draw[thick] (0,0) rectangle (1,1);
\draw (0.5,0.02) -- (0.5,0);
\draw (0.47, 0) node[below] {$1/2$};
\draw (1,0.02) -- (1,0) node[below] {$1$};
\draw (0.02,1) -- (0,1) node[left] {$1$};
\draw[dashed, thick] (0,0) -- (1,1);
\draw[red!70!black, very thick] (0.5,0.5) -- (0.5,1);
\draw[dashed, thick] (0.5,0) -- (0.5,0.5);
\draw[red!70!black, very thick] (0.5,1) -- (1,1);
\draw[green!55!black, very thick] (0,\yh) -- (1,\yh);
\draw[orange!85!black, very thick] (\xaux,0) -- (\xaux,1);
\draw[blue, very thick] (0, 0.2) -- (0.8, 1);
\filldraw[black] (P) circle (.4pt);
\draw[->] (1.29, 0.45) -- (0.6, 0.76);
\node[align=center] at (1.62, 0.42) {\scriptsize Forelli--Rudin pair \\ \scriptsize for weak-type estimates};
\node at (0.6, 0.95) {\tiny{$\mathcal{FR}_w$}};
\node[left] at (0,\yh) {\scriptsize horizontal cut-off line};
\node[left] at (0, 0.2) {\scriptsize critical line};
\node[align=center] at (\xaux,-0.1) {\scriptsize vertical \\
[-1pt] \scriptsize cut-off line};
\end{tikzpicture}
\caption{\small Forelli--Rudin pairs for weak-type estimates: the blue line denotes the critical line, the green line the horizontal cut-off line, and the orange line the vertical cut-off line. The red shaded region represents $\mathcal{FR}_w$, while the gray shaded region represents the strong-type region described in
\cite[Theorem~1.1]{ZZ2022}.}
\label{Fig5}
\end{figure}
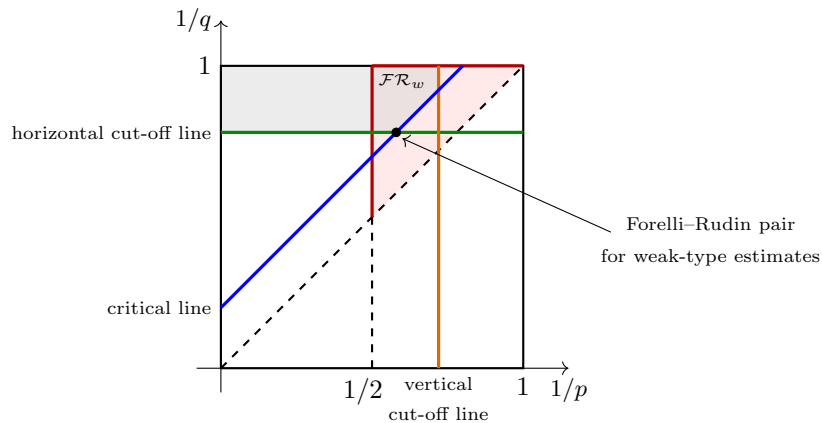

\noindent At the level of restricted weak-type estimates, a Forelli--Rudin pair $(p,q)\in \mathcal{FR}_{rw}$ occurs when the critical line and the two cut-off
lines meet at a common point (see Figure \ref{Fig6}).

\begin{figure}[ht]
\centering
\begin{tikzpicture}[scale=4]

\def\xP{0.38}     
\def\yh{0.58}     
\coordinate (P) at (\xP,\yh);

\draw[->] (-0.08,0) -- (1.15,0) node[below] {\small{$1/p$}};
\draw[->] (0,-0.08) -- (0,1.15) node[left] {\small{$1/q$}};

\fill[violet!20, fill opacity=0.45] (0,0) -- (0,1) -- (1,1) -- cycle;
\fill[gray!30, fill opacity=0.5] (0,\yh) -- (P) -- (\xP, 1) -- (0,1) -- cycle;
\draw[thick] (0,0) rectangle (1,1);
\draw (0.5,0.02) -- (0.5,0);
\draw (0.5, 0) node[below] {$1/2$};
\draw[violet!70!black, very thick] (0,1) -- (1,1);
\draw (1,0.02) -- (1,0) node[below] {$1$};
\draw (0.02,1) -- (0,1) node[left] {$1$};
\draw[dashed, thick] (0,0) -- (1,1);
\draw[blue, very thick] (0,0.2) -- (0.8,1);    
\draw[green!55!black, very thick] (0,\yh) -- (1,\yh);        
\draw[orange!85!black, very thick] (\xP,0) -- (\xP,1);      
\filldraw[black] (P) circle (.4pt);
\draw[->] (1.2,0.43) -- (0.4, 0.55);
\node[align=center] at (1.61, 0.4)
{\scriptsize Forelli--Rudin pair \\
 \scriptsize for restricted weak-type estimates};
\node at (0.12, 0.46) {\tiny{$\mathcal{FR}_{rw}$}};
\node[left] at (0,\yh) {\scriptsize horizontal cut-off line};
\node[left] at (0,0.2) {\scriptsize critical line};
\node[align=center] at (\xP, 1.11)
{\scriptsize vertical \\
[-1pt] \scriptsize cut-off line};

\end{tikzpicture}
\caption{\small Forelli--Rudin pairs for restricted weak-type estimates: the blue line represents the critical line, the green line the horizontal cut-off line, and the orange line the vertical cut-off line. The violet shaded region corresponds to $\mathcal{FR}_{rw}$, while the gray shaded region indicates the strong-type region described in \cite[Theorem~1.1]{ZZ2022}.}
\label{Fig6}
\end{figure}

\vspace{-0.1cm}

\end{enumerate}
\end{rem}

The main results of the present paper are summarized below, providing a roadmap for the analysis that follows; see Figure \ref{Figroadmap}.

\begin{figure}[p]
\centering

\begin{adjustbox}{
    max totalsize={\textwidth}{0.91\textheight},
    center
}
\begin{tikzpicture}[
    x=1cm,
    y=1cm,
    >={Latex[length=2mm]},
    flow/.style={
        ->,
        line width=.7pt,
        draw=black!55,
        rounded corners=2pt
    },
    box/.style={
        draw,
        rounded corners=3pt,
        line width=.7pt,
        inner sep=4.5pt,
        align=left,
        font=\scriptsize,
        text=black!92
    },
    framework/.style={
        box,
        draw=black!65,
        fill=black!3,
        text width=12cm,
        align=center
    },
    mainhead/.style={
        box,
        draw=blue!60!black,
        fill=blue!7,
        text width=12cm,
        align=center
    },
    minorhead/.style={
        box,
        draw=orange!75!black,
        fill=orange!9,
        text width=12cm,
        align=center
    },
    maincase/.style={
        box,
        draw=blue!55!black,
        fill=blue!3,
        text width=7.15cm,
        minimum height=2.35cm
    },
    mainfull/.style={
        box,
        draw=blue!55!black,
        fill=blue!3,
        text width=15.5cm,
        minimum height=2.1cm
    },
    minorcase/.style={
        box,
        draw=orange!75!black,
        fill=orange!4,
        text width=7.15cm,
        minimum height=2.65cm
    },
    global/.style={
        box,
        draw=violet!70!black,
        fill=violet!6,
        text width=12cm,
        align=center
    }
]


\node[framework] (setup) at (0,0) {%
    {\normalsize\bfseries
    Forelli--Rudin operators in the hypersingular regime}\\[-1pt]
    \[
    T_{a,b,c}f(z)
    =
    (1-|z|^2)^a
    \int_{\D}
    \frac{(1-|w|^2)^b f(w)}
         {(1-z\overline w)^c}\,dA(w),\]
    \[
    S_{a,b,c}f(z)
    =
    (1-|z|^2)^a
    \int_{\D}
    \frac{(1-|w|^2)^b f(w)}
         {|1-z\overline w|^c}\,dA(w).\]
    \[n=1,\qquad
    \alpha=\beta=0,\qquad
    a\geq-1,\qquad
    b>-1,\qquad
    0<c-a-b-2\leq1,\]
    \[\Omega_{\mathcal H}
    =
    \{(p,q):1\leq p,q\leq\infty,\ p>q\}.\]
    \[\mathrm{(CL)}:\ 
    1/q=1/p+c-a-b-2,
    \qquad
    \mathrm{(VL)}:\ 
    1/p=b+1,
    \qquad
    \mathrm{(HL)}:\ 
    1/q=-a.\]
    {\itshape
    The case $c-a-b-2=1$ is settled in Reduction~II;
    the remaining range splits as follows.}
};


\node[mainhead] (main) at (0,-3.30) {%
    \vspace{-0.1cm}
    \[\mbox{{\bfseries Main Case:}}
    \qquad
    a,b>-1,\qquad 0<c-a-b-2<1\]
    \vspace{-0.3cm}
};

\draw[flow]
    (setup.south) -- (main.north);


\node[maincase] (MI) at (-4.05,-5.75) {%
    {\bfseries Main I}
    \qquad
    $c\geq\max\{a+2,b+2\}$\\[2pt]
    The operators $T_{a,b,c}$ and $S_{a,b,c}$ have the same
    weak- and restricted weak-type behavior. Hence, no
    Forelli--Rudin phenomenon occurs.\\[2pt]
    {\tiny\itshape
    See Theorem~\ref{mainsubcaseI}.}
};

\node[maincase] (MII) at (4.05,-5.75) {%
    {\bfseries Main II}
    \qquad
    $b+2\leq c<a+2$\\[2pt]
    The critical and vertical boundaries admit a complete
    weak- and restricted weak-type description, identical for
    $T_{a,b,c}$ and $S_{a,b,c}$. Thus, no Forelli--Rudin
    phenomenon occurs.\\[2pt]
    {\tiny\itshape
    See Theorem~\ref{mainsubcaseII}.}
};

\coordinate (rowone) at (0,-4.15);

\draw[flow]
    (main.south) -- (rowone);

\draw[flow]
    (rowone) -| (MI.north);

\draw[flow]
    (rowone) -| (MII.north);


\node[maincase] (MIII) at (-4.05,-8.80) {%
    {\bfseries Main III}
    \qquad
    $a+2\leq c<b+2$\\[2pt]
    At the critical--horizontal endpoint, both weak-type bounds
    fail if
    $0<2-c+b<\frac12$.
    If
    $\frac12\leq2-c+b<1$,
    then $T_{a,b,c}$ is of weak type, whereas $S_{a,b,c}$ is not.\\[1pt]
    {\color{red!70!black}\bfseries
    Weak-type F--R threshold: $\frac12$.}\\[2pt]
    {\tiny\itshape
    See Theorems~\ref{mainsubcaseIII}
    and~\ref{Forelli--Rudin threshold}.}
};

\node[maincase] (MIV) at (4.05,-8.80) {%
    {\bfseries Main IV}
    \qquad
    $1<c<\min\{a+2,b+2\}$\\[2pt]
    The same critical--horizontal dichotomy holds. A weak-type
    Forelli--Rudin phenomenon occurs precisely when
    $\frac12\leq2-c+b<1$.
    The remaining boundary behavior is shared by the two
    operators.\\[1pt]
    {\color{red!70!black}\bfseries
    Weak-type F--R threshold: $\frac12$.}\\[2pt]
    {\tiny\itshape
    See Theorems~\ref{mainsubcaseIV}
    and~\ref{Forelli--Rudin threshold}.}
};

\coordinate (rowtwo) at (0,-7.20);

\draw[flow]
    (rowone) -- (rowtwo);

\draw[flow]
    (rowtwo) -| (MIII.north);

\draw[flow]
    (rowtwo) -| (MIV.north);


\node[mainfull] (MV) at (0,-11.45) {%
    {\bfseries Main V}
    \qquad
    $c\leq1$\\[2pt]
    If $c<1$, the two operators have the same endpoint behavior.
    If $c=1$, then, at the critical--vertical--horizontal point,
    $T_{a,b,1}$ is of restricted weak type, whereas $S_{a,b,1}$
    is not.\\[1pt]
    {\color{red!70!black}\bfseries
    Restricted weak-type F--R phenomenon only at $c=1$.}\\[2pt]
    {\tiny\itshape
    See Theorems~\ref{mainsubcaseV}
    and~\ref{Forelli--Rudin Threshold restricted weak-type}.}
};

\draw[flow]
    (rowtwo) -- (MV.north);


\node[minorhead] (minor) at (0,-13.65) {%
    \vspace{-0.1cm}
    \[\mbox{{\bfseries Minor Case: }}
    \qquad
    a=-1,\qquad b>-1,\qquad 0\leq2-c+b<1\]
    \vspace{-0.3cm}
};

\draw[flow]
    (MV.south) -- (minor.north);


\node[minorcase] (mI) at (-4.05,-16.35) {%
    {\bfseries Minor I}
    \qquad
    $c>1$\\[2pt]
    At
    $(1/p,1/q)=(2-c+b,1)$,
    restricted weak-type fails for both operators when $2-c+b=0$, holds for both operators when $0<2-c+b<1$. 
    Weak type fails for
    both operators if
    $0\leq2-c+b<\frac12$.
    If
    $\frac12\leq2-c+b<1$,
    it holds for $T_{-1,b,c}$ but fails for $S_{-1,b,c}$.\\[1pt]
    {\color{red!70!black}\bfseries
    Weak-type F--R threshold: $\frac12$.}\\[2pt]
    {\tiny\itshape
    See Theorems~\ref{minorsubcaseI}
    and~\ref{Minor: weak Forelli--Rudin}.}
};

\node[minorcase] (mII) at (4.05,-16.35) {%
    {\bfseries Minor II}
    \qquad
    $c\leq1$\\[2pt]
    At
    $(1/p,1/q)=(b+1,1)$,
    weak type fails for both operators. If $c<1$, both are of
    restricted weak type. If $c=1$, only $T_{-1,b,1}$ is of
    restricted weak type.\\[1pt]
    {\color{red!70!black}\bfseries
    Restricted weak-type F--R phenomenon only at $c=1$.}\\[2pt]
    {\tiny\itshape
    See Theorems~\ref{minorsubcaseII}
    and~\ref{Minor: restricted weak Forelli--Rudin}.}
};

\coordinate (minorrow) at (0,-14.55);

\draw[flow]
    (minor.south) -- (minorrow);

\draw[flow]
    (minorrow) -| (mI.north);

\draw[flow]
    (minorrow) -| (mII.north);


\node[global] (conclusion) at (0,-19.20) {%
    {\bfseries
    Global Forelli--Rudin picture in $\Omega_{\mathcal H}$}
    \[
    \mathcal{FR}_{w}
    =
    \bigl\{
       (p,q)\in\Omega_{\mathcal H}:1<p\leq2
    \bigr\},
    \qquad
    \mathcal{FR}_{rw}
    =
    \bigl\{
       (p,q)\in\Omega_{\mathcal H}:p\neq\infty
    \bigr\}.\]
    {\tiny\itshape
    See Theorem~\ref{20260514thm01}.}
};

\draw[flow]
    (mI.south)
    -- ++(0,-.35)
    -| (conclusion.north);

\draw[flow]
    (mII.south)
    -- ++(0,-.35)
    -| (conclusion.north);

\end{tikzpicture}
\end{adjustbox}

\caption{Roadmap of the main and minor cases in the hypersingular regime.}
\label{Figroadmap}
\end{figure}
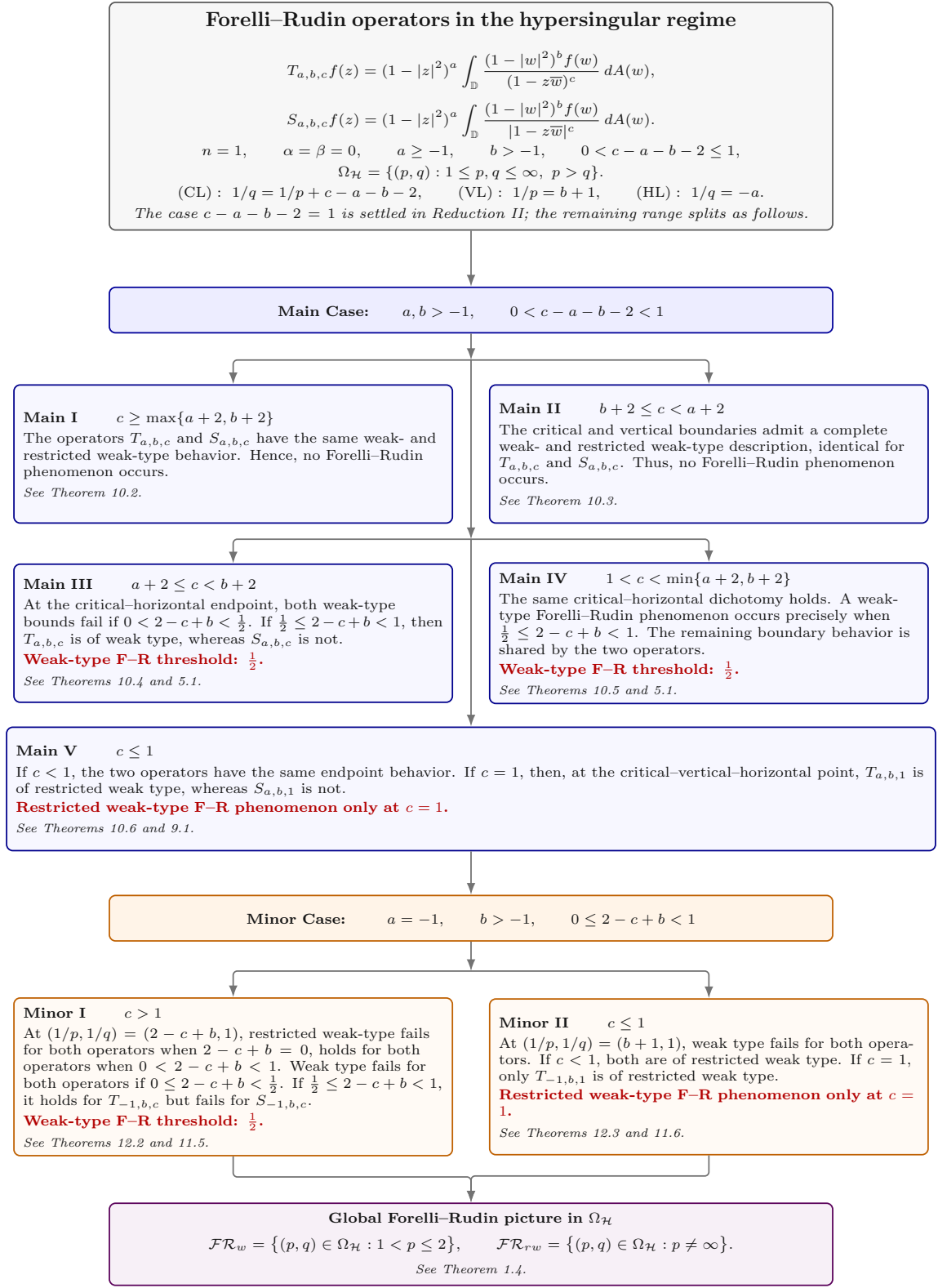

\subsection{Overview of the series} \label{20260522sec01}
One way to view this paper is as the first part of a broader program developing
the sharp critical-line and endpoint theory for Forelli--Rudin type operators.
The present paper treats the purely hypersingular case. 
The singular and sub-singular cases are treated in the sequel \cite{HWZZ2026a}, while the mixed case is investigated in \cite{HWZZ2026b}. Together, these works complete the critical-line and endpoint theory and resolve an open question raised in \cite{HZ2025}. 
In this section, we briefly explain how these three cases are formulated and highlight some of their main features.

We begin by recalling the strong-type theory for the Forelli--Rudin operators studied in \cite{ZZ2022}. 
For the purposes of this overview, we restrict our attention to the unit disc $\D$; the corresponding statements on the unit ball follow from standard modifications. 
See Section~\ref{Sec14} for some details.

Let $a,b,c\in\mathbb R$, and let $T_{a,b,c}$ and $S_{a,b,c}$ be defined as in
\eqref{20260514eq02} and \eqref{20260514eq03}, respectively. For
$\alpha>-1$ and $1\le p\le \infty$, let $L_\alpha^p(\D)$ denote the weighted
Lebesgue space on $\D$ associated with the radial weight
$(1-|z|^2)^\alpha$. The strong-type bounds
$$
T_{a,b,c},\; S_{a,b,c}: L_\alpha^p(\D)\to L_\beta^q(\D)
$$ 
are characterized by the following critical lines and cut-off lines.

\begin{enumerate}
    \item[$\bullet$] \emph{The critical line.}
    \begin{enumerate}
        \item[(a)] If $p\le q$, the critical line is
        $$
        \frac{1}{q}=\frac{2+\alpha}{2+\beta}\cdot \frac{1}{p}+\frac{c-a-b-2}{2+\beta}.
        $$
        \item[(b)] If $p>q$, the critical line is
        $$
        \frac{1}{q}=\frac{1+\alpha}{1+\beta}\cdot \frac{1}{p}+\frac{c-a-b-2}{1+\beta}.
        $$ 
    \end{enumerate}

    \item[$\bullet$] \emph{Two cut-off lines.}
    \begin{enumerate}
        \item[(a)] The vertical cut-off line is
        \begin{equation} \label{20260515eq11}
        \frac{1}{p}
        =
        \frac{b+1}{\alpha+1}.
        \end{equation} 
        \item[(b)] The horizontal cut-off line is
        \begin{equation} \label{20260515eq12}
        \frac{1}{q}
        =
        \frac{-a}{\beta+1}.
        \end{equation} 
    \end{enumerate}
\end{enumerate}

The underlying picture is that the strong-type boundedness region is described inside the $(1/p,1/q)$-square by intersecting the regions lying above the critical and horizontal cut-off lines and to the left of the vertical cut-off line. 
Whether points on these boundary lines are included depends subtly on whether $p\le q$ or $p>q$. 
Moreover, both $T_{a, b, c}$ and $S_{a, b, c}$ share the same strong-type bounds. 
We refer the interested reader to \cite{ZZ2022} for more details.

\begin{exa} \label{20260516exa01}
Recall the hypersingular Bergman projection $K_{2t}$ and its associated positive
operator $K_{2t}^+$, defined as in \eqref{20260514eq30}. In this case,
$a=b=0$, $c=2t$, and $\alpha=\beta=0$. Hence the critical line is
$$
\frac{1}{q}=\frac{1}{p}+2t-2,
$$
while the vertical and horizontal cut-off lines are given by
$$
\frac{1}{p}=1 \qquad \text{and} \qquad
\frac{1}{q}=0,
$$
respectively. Therefore, as a consequence of \cite[Theorem~1.1]{ZZ2022}, both $K_{2t}$ and
$K_{2t}^+$ map $L^p(\D)$ boundedly into $L^q(\D)$ whenever
$$
\frac{1}{q}>\frac{1}{p}+2t-2, \quad \textrm{with} \ 1 \le p, q \le +\infty, 
$$
and this range is sharp.
\end{exa}

Guided by the strong-type behavior, and taking into account the greater subtlety
of the critical-line and endpoint theory, we divide our study into a series of
three works:
\begin{enumerate}
    \item[(1)] In the first part of the series, namely the present paper, we
    consider the case $\alpha=\beta=0$. In this case, the critical line is given by 
    \begin{equation} \label{20260522eq01}
    \frac{1}{q}=\frac{1}{p}+c-a-b-2, \qquad 1 \le p, q \le +\infty, 
    \end{equation} 
    and lies entirely
    in the hypersingular regime $\Omega_{\mathcal H}$, as defined in
    \eqref{20260515eq01} (see Figure \ref{Fig3}).

\begin{figure}[ht]
\centering
\begin{tikzpicture}[scale=3.5]
    \def\xA{0.8}
    \def\yB{0.2}
    \fill[gray!10] (0,0) -- (0,1) -- (1,1) -- cycle;
    \draw[->, thick] (-0.08,0) -- (1.15,0) node[right] {$1/p$};
    \draw[->, thick] (0,-0.08) -- (0,1.15) node[above] {$1/q$};
    \draw[thick] (0,0) rectangle (1,1);
    \draw (1,0) -- (1,-0.025) node[below] {$1$};
    \draw (0,1) -- (-0.025,1) node[left] {$1$};
    \draw[dashed] (0,0) -- (1,1);
    \coordinate (B) at (0,\yB);
    \coordinate (A) at (\xA,1);
    \draw[blue, very thick] 
        (B) -- (A)
        node[midway, above, sloped, yshift=0pt]
        {\tiny{$\frac{1}{q}=\frac{1}{p}+c-a-b-2$}};
    \node[left=4pt] at (B) {\small{$\bigl(0,c-a-b-2\bigr)$}};
    \node[above=2pt] at (A) {\small{$\bigl(3+a+b-c,1\bigr)$}};
    \node at (0.15,0.85) {\small{$\Omega_{\mathcal H}$}};
\end{tikzpicture}
\caption{ \small Critical line for $T_{a,b,c}$ and $S_{a,b,c}$ in the purely hypersingular case: the blue line represents the critical line.}
\label{Fig3}
\end{figure}
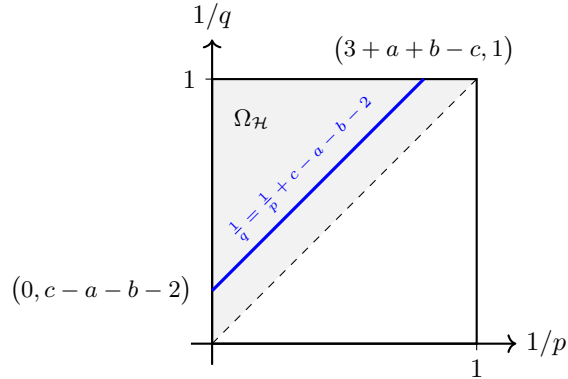

    \item[(2)] In the second part of the series \cite{HWZZ2026a}, we consider
    the case $\alpha=\beta=0$ in which the critical line is given by 
    \begin{equation} \label{20260519eq31}
    \frac{1}{q}=\frac{1}{p}+\frac{c-a-b-2}{2}, \qquad 1 \le p, q \le +\infty,
    \end{equation} 
    and lies entirely in the
    sub--singular regime 
    $$    
    \Omega_{\mathcal S}:=\left\{
    (p,q): 1\le p,q\le \infty,\ p\le q
    \right\}
    $$
    (see Figure \ref{Fig4}).

\begin{figure}[ht]
\centering
\begin{tikzpicture}[scale=3.5]
\def\yL{-0.32}         
\def\xC{0.32}          
\def\yR{0.68}         
\fill[gray!10] (0,0) -- (1,0) -- (1,1) -- cycle;
\draw[->, thick] (-0.10,0) -- (1.18,0) node[right] {$1/p$};
\draw[->, thick] (0,-0.45) -- (0,1.12) node[above] {$1/q$};
\draw[thick] (0,0) rectangle (1,1);
\draw (1,0) -- (1,-0.02) node[below] {$1$};
\draw (0,1) -- (-0.02,1) node[left] {$1$};
\draw[thick] (0,0) -- (1,1);
\draw[blue, dashed, thick] (0,\yL) -- (\xC,0);
\draw[blue, thick] (\xC,0) -- (1,\yR) node[midway, above, sloped, yshift=2pt]
        {\tiny{$\frac{1}{q}=\frac{1}{p}+\frac{c-a-b-2}{2}$}};
\node[left=4pt] at (0,\yL)
{\small{$\left(0, \; \dfrac{c-a-b-2}{2}\right)$}};
\node[right=4pt] at (1,\yR)
{\small{$\left(1, \; 2-\dfrac{c-a-b}{2}\right)$}};
\node at (0.85,0.15) {\small{$\Omega_{\mathcal S}$}};
\end{tikzpicture}
\caption{\small  Critical line for $T_{a,b,c}$ and $S_{a,b,c}$ in the purely sub-singular regime: the blue line represents the critical line.}
\label{Fig4}
\end{figure}
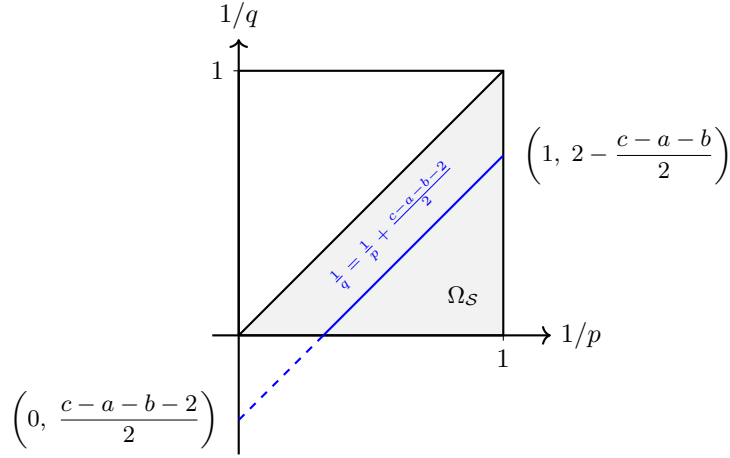

    \item[(3)] In the final part of this series \cite{HWZZ2026b}, we consider the mixed case with general weights $\alpha,\beta>-1$, where the critical line may pass through both $\Omega_{\calH}$ and $\Omega_{\calS}$; see, e.g., \eqref{20260519eq02} below.
\end{enumerate}

\vspace{0.1cm}

Here are some highlights of the forthcoming works \cite{HWZZ2026a} and
\cite{HWZZ2026b}.

\begin{enumerate}
    \item[(a)] In the sub--singular case, the critical line \eqref{20260519eq31} suggests that the operators $T_{a,b,c}$ and
$S_{a,b,c}$ are closer in nature to the Bergman projection and its
fractional counterparts. Nevertheless, the interaction between this
critical line and the two cut-off lines \eqref{20260515eq11} and
\eqref{20260515eq12} remains subtle. This interaction is one of the main
issues addressed in the sub--singular theory developed in
\cite{HWZZ2026a}, where the central question is whether the Forelli--Rudin
phenomenon can still be detected in the sub--singular regime.

This perspective is different in nature
from many recent dyadic approaches to Bergman-type operators, which make use of size estimates for the Bergman kernel and sparse domination; see, e.g., \cite{HWW2021, WW2021, GHK2022, RTW2017}. The point here is more delicate: we ask whether the cancellation inherent in the complex analytic kernel itself can still change the endpoint mapping behavior, even in the presence of Calder\'on--Zygmund theory.  

    \item[(b)] In the mixed case, especially when $\alpha\neq\beta$, the
    critical line may exhibit a ``refraction'' phenomenon. For example, consider the Forelli--Rudin operators
    \begin{equation} \label{20260519eq02}
    T_{0,0,5/2}f(z)=
    \int_{\D}\frac{f(w)}{(1-z\overline{w})^{5/2}}\,dA(w), \quad \textrm{and} \quad 
    S_{0,0,5/2}f(z)=
    \int_{\D}\frac{f(w)}{|1-z\overline{w}|^{5/2}}\,dA(w)
    \end{equation} 
    acting from $L^p(\D)$ to $L^q_1(\D)$. The critical line in this case is the following polyline:
    $$
    \frac{1}{q}=
    \begin{cases}
   \frac{1}{2p}+\frac14, & p>q; \\
   \\
    \frac{2}{3p}+\frac16, & p\le q,
    \end{cases}
    $$
    with the refraction point located at $(1/2,1/2)$ (see Figure \ref{Fig2}).

\begin{figure}[ht]
\centering
\begin{tikzpicture}[scale=4]
\coordinate (O) at (0,0);
\coordinate (A) at (1,0);
\coordinate (B) at (1,1);
\coordinate (C) at (0,1);
\coordinate (L) at (0,0.25);
\coordinate (P) at (0.5,0.5);
\coordinate (R) at (1,0.8333);
\draw[->, thick] (-0.06,0) -- (1.13,0) node[right] {$1/p$};
\draw[->, thick] (0,-0.06) -- (0,1.13) node[above] {$1/q$};
\draw[thick] (0,0) rectangle (1,1);
\draw[dashed, thick] (0,0) -- (1,1);
\draw[blue, very thick] (L) -- (P);
\draw[blue, very thick] (P) -- (R);
\fill[red] (P) circle (0.02);
\draw (0.5,0) -- (0.5,-0.02) node[below=2pt] {$1/2$};
\draw (1,0) -- (1,-0.02) node[below=2pt] {$1$};
\draw (0,1) -- (-0.02,1) node[left=2pt] {$1$};
\node[below right=2pt] at (0.4, 0.46) {\footnotesize{$\left(1/2, 1/2\right)$}};
\node[left] at (0, 0.25) {\footnotesize{$(0, 1/4)$}};
\node[right] at (1, 0.8333) {\footnotesize{$(1, 5/6)$}};
\node[blue!85!black, rotate=26] at (0.25,0.45)
    {\scriptsize $\frac1q=\frac{1}{2p}+\frac14$};
\node[blue!85!black, rotate=33] at (0.78,0.58)
    {\scriptsize $\frac1q=\frac{2}{3p}+\frac16$};
\end{tikzpicture}
\caption{ \small The ``refracted'' critical line for 
$T_{0,0,5/2}$ and $S_{0,0,5/2}$ as mappings 
from $L^p(\D)$ to $L_1^q(\D)$: the blue segment denotes the refracted portion, and the red dot marks the refraction point $(1/2,1/2)$.}
\label{Fig2}
\end{figure}
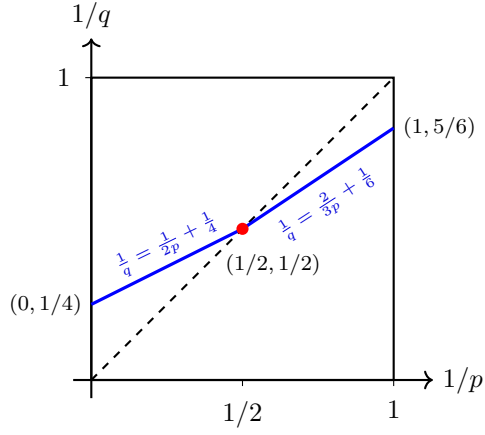

In the example above, the critical line is refracted from the hypersingular regime into the sub-singular regime; in general, the refraction may also occur in the opposite direction. Away from the refraction point, the expected picture should follow by combining the hypersingular theory developed here with the sub-singular theory in \cite{HWZZ2026a}. The main subtlety in the mixed case is the endpoint behavior at the refraction point, where the two adjacent branches are governed by different mechanisms: one is closer to the theory of Bergman projections and their fractional counterparts, while the other is genuinely more singular.

This becomes particularly delicate when the refraction point meets one or both of the horizontal and vertical cut-off lines, so that the change of regime and the cut-off phenomenon occur at the same endpoint. One must then determine whether any Forelli--Rudin phenomenon remains visible there. These mixed cases are treated in \cite{HWZZ2026b}, which, together with the present paper and \cite{HWZZ2026a}, completes the critical-line and endpoint theory for Forelli--Rudin type operators.
\end{enumerate}

\subsection{Highlights of the present paper} \label{subsechighlights}

The contribution of the current paper can be summarized as follows.
\begin{enumerate}
    \item[(a)] From the viewpoint of harmonic analysis, we initiate a series of
    works aimed at giving a complete answer to the open problem raised in
    \cite{HZ2025} concerning the endpoint and critical-line behavior of general
    Forelli--Rudin operators. The present paper is the first part of this series
    and treats the hypersingular regime, where the relevant operators lie beyond
    the scope of Calder\'on--Zygmund theory.

    The main difficulty is that, for the general operators $T_{a,b,c}$ and
    $S_{a,b,c}$, the critical line may intersect one or both of the two
    cut-off lines \eqref{20260515eq11} and \eqref{20260515eq12}. This is in sharp contrast with the hypersingular Bergman
    projection and its positive counterpart studied in \cite{HZ2025}, where the
    critical line stays away from the cut-off lines (see Example \ref{20260516exa01}). Consequently, the endpoint
    behavior is not governed by the critical-line mechanism alone; one must also
    understand how this mechanism interacts with the cut-off phenomena.

    This interaction leads to a substantially more delicate endpoint and
    critical-line theory. While the Forelli--Rudin method developed in
    \cite{HZ2025} remains effective for estimates on the critical line, it is no
    longer sufficient by itself. We therefore develop a broader framework that
    combines this method with
    \begin{enumerate}
        \item[$\bullet$] constructions based on a fourth-moment probabilistic
        argument;

        \item[$\bullet$] dyadic Carleson embeddings, in the spirit of Nazarov,
        Treil, and Volberg \cite{NTV2003}, adapted to the geometry of Carleson
        boxes;

        \item[$\bullet$] a refined sparse domination argument, designed to handle
        the extra factor $(1-|z|^2)^a$ in the definitions of $T_{a,b,c}$ and
        $S_{a,b,c}$. 

    \end{enumerate}

 \item[(b)] From the viewpoint of complex analysis, the main finding of the
    present paper is the discovery of a genuinely new cancellation phenomenon for
    Forelli--Rudin operators. To the best of our knowledge, this is the first
    work showing that the cancellation of the complex analytic kernel can affect
    the mapping behavior of general Forelli--Rudin operators; we call this the
    \emph{Forelli--Rudin phenomenon}.

    The phenomenon is invisible in the strong-type theory, where $T_{a,b,c}$ and
    $S_{a,b,c}$ have the same boundedness range, but becomes visible at the
    weak-type and restricted weak-type levels. In particular, the weak-type
    theory reveals an intrinsic transition at $p=2$, the Forelli--Rudin
    threshold. 

    This is somewhat unexpected from the standpoint of complex analysis: level-set estimates usually do not preserve the analytic features of the kernel. Here, however, the cancellation of the analytic kernel remains detectable at the
    level of distribution estimates. The key mechanism behind this phenomenon is a novel weak-type Hardy estimate that bridges
    Lorentz-space estimates with complex function theory.
\end{enumerate}

The rest of the paper is organized as follows. 
Section~\ref{Sec02} fixes the notation and recalls the basic facts concerning Carleson tents, dyadic arcs, Lorentz spaces, and weak- and restricted weak-type estimates. 
Section~\ref{Sec03} establishes the preliminary reductions on the parameters. 
In Section~\ref{Sec04}, we divide the Main Case into five subcases and establish several further reductions used in the subsequent analysis. 
Section~\ref{Sec05} studies the weak-type bounds at the critical--horizontal intersection $\calP_{\textnormal{C-H}}$ and proves the weak-type Forelli--Rudin threshold. 
Sections~\ref{Sec06}, \ref{Sec07}, and~\ref{Sec08} establish the restricted weak-type bounds at the critical--horizontal, critical--vertical, and vertical--horizontal intersections, respectively. 
Section~\ref{Sec09} treats the common intersection of the critical, vertical, and horizontal lines and establishes the Forelli--Rudin phenomenon at the restricted weak-type level. 
In Section~\ref{Sec10}, we combine these results to obtain the complete weak- and restricted weak-type classification for Main Cases~I--V. 
Section~\ref{Sec11} develops the corresponding reductions and endpoint estimates for the Minor Case, and Section~\ref{Sec12} assembles these results to obtain the complete classification for Minor Cases~I and~II. 
Finally, Section~\ref{Sec14} explains the extension of our results to the unit ball in higher dimensions and records the necessary modifications to the dyadic and Hardy-space arguments.

Throughout this paper, for $a ,b \in  \mathbb{R}$, $a\lesssim b$ means there exists a positive number $C$, which is independent of $a$ and $b$, such that $a\leq C\,b$. Moreover, if both $a \lesssim b$ and $b\lesssim a$ hold, then we say $a \simeq b$. \\

\noindent{\textbf{Acknowledgments.}}
The first author was supported by NSF grant DMS-2555999 and by the Simons Foundation Travel Support for Mathematicians program under award MPS-TSM-00007213.
The second author was supported by the National Natural Science Foundation of China (No.12471116) and 2025CDJ-IAIS YB-004 (Chongqing University).

\section{Notations}\label{Sec02}

In this section we fix the basic notation used throughout the paper. 

\subsection{Carleson tents and dyadic arcs}

Let $I\subseteq\T$ be an arc, and let $|I|$ denote its normalized length.  The
\emph{Carleson tent associated with $I$} is
\begin{equation} \label{Carlesontent}
Q_I:=\left\{z\in\D:\ \frac{z}{|z|}\in I,\quad 1-|I|\le |z|<1
\right\}.
\end{equation} 
We shall also use the \emph{upper Carleson tent of $Q_I$}, denoted by
$Q_I^{\mathrm{up}}$, and defined by
\begin{equation} \label{UpperCarlesontent}
Q_I^{\mathrm{up}}:=\left\{z\in\D:\ \frac{z}{|z|}\in I,\quad 1-|I|\le |z|<1-\frac{|I|}{2}
\right\}.
\end{equation} 

A \emph{dyadic system}, or \emph{dyadic grid}, on $\T$ is a collection
$$
\mathcal D=\bigcup_{k\ge 0}\mathcal D_k
$$
such that each generation $\mathcal D_k$ consists of $2^k$ pairwise disjoint
arcs of length $2^{-k}$ whose union is $\T$, and each arc
$I\in\mathcal D_k$ is the union of two arcs in $\mathcal D_{k+1}$, called its
\emph{dyadic children}.  Equivalently, after identifying $\T$ with $[0,1)$, one may
take
$$
\mathcal D_k=\left\{\left[\frac{m}{2^k},\frac{m+1}{2^k}\right): m=0,1,\dots,2^k-1
\right\},
$$
and interpret these intervals as arcs on $\T$.

Finally, for $J\in \calD$, we denote by $\calD_k(J)$ the collection of all dyadic arcs \(I\subseteq J\) satisfying $|I|=2^{-k}|J|$. 

\subsection{Weak and Lorentz spaces}

For $1\le p<\infty$, the \emph{weak Lebesgue space} $L^{p,\infty}(\D)$ consists of all
measurable functions $f$ on $\D$ such that
\[
\|f\|_{L^{p,\infty}(\D)}
:=
\sup_{\lambda>0}
\lambda\,\big|\{z\in\D:\ |f(z)|>\lambda\}\big|^{1/p}
<\infty .
\]
More generally, for $1\le p<\infty$ and $0<r<\infty$, the \emph{Lorentz space}
$L^{p,r}(\D)$ is equipped with the quasi-norm
\[
\|f\|_{L^{p,r}(\D)}
:=
\left(
p\int_0^\infty
\lambda^r
\big|\{z\in\D:\ |f(z)|>\lambda\}\big|^{r/p}
\frac{d\lambda}{\lambda}
\right)^{1/r}.
\]
The endpoint cases are understood with the standard modifications.

\subsection{Weak and restricted weak-type estimates}

Let $1\le p,q\le\infty$. We say that $\mathcal T$ is of \emph{weak type $(p,q)$} if 
$$
\mathcal T:L^p(\D)\to L^{q,\infty}(\D)
$$ 
is bounded. Equivalently, when $1\le q<+\infty$, there exists a constant $C>0$ such that for every $f\in L^p(\D)$ and every $\lambda>0$,
$$
\big|\{z\in\D:\ |\mathcal T f(z)|>\lambda\}\big| \le
C\,\lambda^{-q}\|f\|_{L^p(\D)}^q.
$$
When $q=\infty$, this is understood as the estimate $\|\mathcal T f\|_{L^\infty(\D)}\le C\|f\|_{L^p(\D)}$. 

\vspace{0.1cm}

To this end, for $1\leq p<\infty$, we say that $\mathcal T$ is of \emph{restricted weak type $(p,q)$} if $$ \mathcal T:L^{p,1}(\D)\to L^{q,\infty}(\D) $$ is bounded. In particular, this implies the following estimate for characteristic functions: there exists a constant $C>0$ such that $$ \big|\{z\in\D:\ |\mathcal T\mathbf \one_E(z)|>\lambda\}\big| \le C\,\lambda^{-q}|E|^{q/p} $$ for every measurable set $E\subseteq\D$ with $|E|<\infty$ and every $\lambda>0$, with the usual modification when $q=\infty$.
When $p=\infty$, the restricted weak type $(\infty,q)$ is understood directly through characteristic functions:
$$
\big|\{z\in\D:\ |\mathcal T\one_E(z)|>\lambda\}\big| \le C\,\lambda^{-q},
$$
for every measurable set $E\subseteq\D$ and every $\lambda>0$, with the usual modification when $q=\infty$.

\bigskip 
\section{Preliminaries}\label{Sec03}

In this section, we set up the framework for our analysis and make several reductions. Recall that throughout this paper, we restrict ourselves to the case $\alpha=\beta=0$ and $n=1$ in the purely hypersingular regime \eqref{20260515eq01}.

\subsection{The strong-type region and its boundary}
In this section, we recall Zhao and Zhou's strong boundedness results \cite{ZZ2022} in the purely hypersingular regime $\Omega_{\mathcal H}$.

\begin{thm}
Let $a,b,c\in\R$ and $(p,q)\in\Omega_{\mathcal H}$.
Then the following statements are equivalent:
\begin{enumerate}
    \item[(i)] The operator $S_{a,b,c}$ is bounded from $L^p(\D)$ to $L^q(\D)$.
    \item[(ii)] The operator $T_{a,b,c}$ is bounded from $L^p(\D)$ to $L^q(\D)$.
    \item[(iii)] The pair $(1/p,1/q)\in\Omega(a,b,c)$, where
		\begin{equation}\label{eq:strong}
			\Omega(a,b,c)=\left\{\left(\frac{1}{p},\frac{1}{q}\right)\in[0,1]\times[0,1]: \frac{1}{q}>-a,\ \frac{1}{p}<b+1,\ \frac{1}{q}-\frac{1}{p}>c-2-a-b.\right\}
		\end{equation}
\end{enumerate}
\end{thm}

Here the critical line is given by
\begin{equation}\tag{CL}\label{criticalline}
\frac{1}{q}=\frac{1}{p}+c-a-b-2,
\end{equation}
with $0<c-a-b-2 \le 1$. On the other hand, by \eqref{20260515eq11} and \eqref{20260515eq12}, the two cut-off lines are given by 
\begin{equation} \tag{VL} \label{verticalcut}
\frac{1}{p}=b+1
\end{equation} 
and
\begin{equation} \tag{HL} \label{horizontalcut}
\frac{1}{q}=-a. 
\end{equation}

\subsection{Reduction I: Reducing the ranges of $a$ and $b$} \label{20260830subsec01}

We observe first that it suffices to consider the case when
\begin{equation} \label{20260518eq01}
a \ge -1 \qquad \textrm{and} \qquad b>-1.
\end{equation} 
Our goal is to show that, if \eqref{20260518eq01} is violated, then $T_{a,b,c}$ and $S_{a,b,c}$ fail to satisfy the restricted weak--type bounds at $(\infty,1)$.

We first argue that $b>-1$ by contradiction. Assume $b \le -1$. Then for any $0<R<1$, let 
$$
f_R(w):=\one_{|w| \le R}(w), \qquad w \in \D. 
$$
Then $\|f_R\|_{L^\infty(\D)} \le 1$.
For each $z\in\D$ and $0\le r<1$, the mean-value property gives
\[
    \frac{1}{2\pi}\int_0^{2\pi}
    \frac{d\theta}{(1-zr e^{-i\theta})^c}=1.
\]
Hence, 
\begin{align*}
    T_{a,b,c}f_R(z)
    &=(1-|z|^2)^a
      \int_{|w|<R}
      \frac{(1-|w|^2)^b}{(1-z\overline w)^c}\,dA(w)\\
    &=(1-|z|^2)^a
      \int_{|w|<R}(1-|w|^2)^b\,dA(w)\\
    &\simeq A_R(1-|z|^2)^a,
    \qquad R\to1^-,
\end{align*}
where
$$
A_R:=
\begin{cases}
\log \frac{1}{1-R}, \qquad \hfill b=-1; \\
\\
(1-R)^{b+1}, \qquad \hfill b<-1. 
\end{cases}
$$
Thus,
$$
\frac{\left\|T_{a, b, c}f_R \right\|_{L^{1, \infty}(\D)}}{\|f_R\|_{L^\infty(\D)}} \gtrsim A_R \to \infty, \qquad \textrm{as} \quad R \to 1^{-}.
$$
This shows that the restricted weak-type bounds for $T_{a, b, c}$ fail at the endpoint $(\infty, 1)$ and the same conclusion holds for $S_{a, b, c}$ due to the pointwise bound $|T_{a, b, c}f| \le S_{a, b, c}(|f|)$.

\vspace{0.1cm}

Next, we show that $a \ge -1$. Assume $a<-1$. We prove that in this case the restricted weak-type estimate at $(\infty, 1)$ fails again for $T_{a, b, c}$, and therefore so does $S_{a, b, c}$. Indeed, since $b>-1$, by the same radial averaging argument as above, we have
$$
T_{a, b, c}1(z)=(1-|z|^2)^a \int_{\D} \frac{(1-|w|^2)^b}{(1-z\overline{w})^c} dA(w)=(1-|z|^2)^a \int_{\D} (1-|w|^2)^bdA(w) \simeq (1-|z|^2)^a,
$$
which gives the desired claim since $(1-|z|^2)^a \notin L^{1, \infty}(\D)$ if $a<-1$.

\vspace{0.1cm}

The proof of \eqref{20260518eq01} is complete. 

\subsection{Reduction II: The case when $c-a-b-2=1$}

We claim that, in this case, the situation of interest is when $a=-1$; see Theorem \ref{minorsubcaseI} (a). Indeed, suppose that $a>-1$. Then the critical line in this case is given by 
$$
\frac{1}{q}=\frac{1}{p}+1. 
$$
By \cite[Theorem 1.1]{ZZ2022}, strong-type bounds fail at any admissible pair $(p, q)$ where $1 \le p, q \le +\infty$. For weak-type and restricted weak-type estimates, we have the following:

\begin{prop}\label{reductionII}
Let $a\geq-1$, $b>-1$ and $c=a+b+3$. Then 
\begin{enumerate}
    \item For any $1 \le p, q \le +\infty$ with $(p, q) \neq (\infty, 1)$, neither $T_{a,b,c}$ nor $S_{a,b,c}$ is of restricted weak type $(p,q)$.
    \item If $a, b>-1$, we have $T_{a, b, c}, \; S_{a, b, c}$ map $L^\infty(\D)$ boundedly into $L^{1, \infty}(\D)$. 
\end{enumerate}
    In particular, the statement $(1)$ shows the weak-type bounds in $(2)$ are sharp.
\end{prop}

\begin{proof}
(1). We divide the proof into several steps.

\vspace{0.1cm}

\noindent \textit{Step I: We first show that $T_{a, b, c}, \; S_{a, b, c}: L^p(\D) \not\to L^{q, \infty}(\D)$ for any $1 \le p, q \le +\infty$ with $(p, q) \neq (\infty, 1)$.} Note first that it suffices to consider the case when $q<+\infty$, since the case $q=+\infty$ follows from \cite[Theorem~1.2\&1.3]{ZZ2022}. We consider two different cases.

If $1<q<+\infty$, we argue by contradiction. Assume $S_{a, b, c}$ maps $L^p(\D)$ boundedly into $L^{q, \infty}(\D)$ for $q>1$. Then take $\varepsilon>0$ sufficiently small so that $q-\varepsilon>1$. Then by \cite[Exercise 1.1.11]{Grafakos2014}, $L^{q, \infty}(\D) \subseteq L^{q-\varepsilon}(\D)$, which gives
\begin{equation} \label{20260523eq30}
\left\|S_{a, b, c} f\right\|_{L^{q-\varepsilon}(\D)} \lesssim \left\|S_{a, b, c}f \right\|_{L^{q, \infty}(\D)} \lesssim \left\|f \right\|_{L^p(\D)}.
\end{equation} 
However, this contradicts \cite[Theorem 1.1]{ZZ2022}. The same contradiction argument applies verbatim to $T_{a,b,c}$.

If $q=1$, it suffices to prove $T_{a, b, c}: L^{p}(\D) \not\to L^{1, \infty}(\D)$. First, by our assumption, note that $1 \le p<+\infty$ and $c=a+b+3>1$. For sufficiently large $N \ge 1$, denote $A_N:=\left\{w \in \D: 1/N<1-|w|^2<2/N\right\}$ and 
$$
f_N(w):=N^{\frac{1}{p}} w^N \one_{A_N}(w), \qquad w \in \D. 
$$
On one hand,
$$
\left\|f_N \right\|_{L^p(\D)}^p \simeq N \cdot \int_{A_N} dA(w) \simeq 1. 
$$
On the other hand, by a change of variables to polar coordinates, it is easy to see that
$$
\int_{A_N} (1-|w|^2)^b w^N\overline{w}^{\,m}\,dA(w)=0
\quad \text{unless } \; m=N.
$$
Hence, for any $z \in A_N$, 
\begin{align*}
\left| T_{a, b, c}f_N(z) \right|
&=(1-|z|^2)^a N^{\frac{1}{p}} \left| \int_\D (1-|w|^2)^b w^N \one_{A_N}(w)\cdot \left(\sum_{m=0}^\infty \frac{\Gamma(m+c)}{\Gamma(c)m!} z^m \overline{w}^m \right) dA(w) \right| \\
&=(1-|z|^2)^a N^{\frac{1}{p}} \cdot \frac{\Gamma(N+c)}{\Gamma(c)N!} |z|^N \int_{A_N} (1-|w|^2)^b|w|^{2N} dA(w)  \\
& \simeq N^{-a} N^{\frac{1}{p}} N^{c-1} N^{-b-1} \\
& = N^{\frac{1}{p}+1}, 
\end{align*}
where in the second last equation above, we used Stirling's formula. Therefore,
$$
\left\|T_{a, b, c}f_N \right\|_{L^{1, \infty}} \gtrsim N^{\frac{1}{p}+1} |A_N| \simeq N^{\frac{1}{p}}, 
$$
which implies
$$
\frac{\left\|T_{a, b, c}f_N \right\|_{L^{1, \infty}(\D)}}{\left\|f_N \right\|_{L^p(\D)}} \gtrsim N^{\frac{1}{p}}. 
$$
This clearly gives the desired result. 

\vspace{0.1in}

\noindent \textit{Step II: Next, we show that neither $T_{a,b,c}$ nor $S_{a,b,c}$ is of restricted weak type $(p,q)$ for any $1 \le p, q \le +\infty$ with $(p, q) \neq (\infty, 1)$.} We first consider the case $p=\infty$. Since $(p,q)\neq(\infty,1)$, this means that $q>1$. We claim that if $S_{a, b, c}$ satisfies a restricted weak-type estimate at $(\infty, q)$, then $S_{a, b, c}$ maps $L^\infty(\D)$ boundedly into $L^{q, \infty}(\D)$, thereby contradicting \textit{Step I}.

This is indeed a general principle for linear operators; for the reader's convenience, we include a proof here. Let us first prove the above claim under the assumption $1<q<+\infty$. Recall the following standard characterization: for every $1<q<+\infty$,
\begin{equation} \label{20260528eq01}
\|f\|_{L^{q,\infty}(\D)}
\simeq
\|f\|_{\widetilde L^{q,\infty}(\D)}
:=
\sup_{E\subseteq \D,\ |E|>0}
|E|^{\frac1q-1} \left| \int_E f(z) \,dA(z) \right|,
\end{equation}
where the quantity on the right-hand side of \eqref{20260528eq01} defines a genuine norm on \(L^{q,\infty}(\D)\); see, e.g., \cite[Exercise 17]{Tao2009notes}.

Let \(f\in L^\infty(\D)\). By normalization and decomposition into the positive and negative parts of the real and imaginary components, it suffices to consider the case \(0\le f\le 1\). By the layer-cake representation and the linearity of \(S_{a,b,c}\), we have
$$
S_{a,b,c}f(z)=\int_0^1 S_{a,b,c}\one_{\{f>t\}}(z)\,dt.
$$
Therefore,
\begin{align*}
\|S_{a,b,c}f\|_{L^{q,\infty}(\D)}
&\simeq
\|S_{a,b,c}f\|_{\widetilde L^{q,\infty}(\D)}  \\
&\le
\int_0^1
\|S_{a,b,c}\one_{\{f>t\}}\|_{\widetilde L^{q,\infty}(\D)}\,dt \\
&\simeq
\int_0^1
\|S_{a,b,c}\one_{\{f>t\}}\|_{L^{q,\infty}(\D)}\,dt \\
&\lesssim 1,
\end{align*}
where in the last step we used the assumption that \(S_{a,b,c}\) satisfies a restricted weak-type estimate at \((\infty,q)\). Hence
$$
S_{a,b,c}:L^\infty(\D)\to L^{q,\infty}(\D)
$$
is bounded. If $q=\infty$, the same argument applies with the $L^\infty(\D)$ norm in place of the $L^{q,\infty}(\D)$-norm. In this case, a restricted weak-type estimate at $(\infty,\infty)$ yields
$$
\|S_{a,b,c}f\|_{L^\infty(\D)} \lesssim
\|f\|_{L^\infty(\D)}.
$$
This proves the claim and therefore completes the proof of the case \(p=+\infty\).

\medskip 

Next, consider the case when $p<+\infty$. We argue by contradiction.
Suppose that \(S_{a,b,c}\) maps \(L^{p,1}(\D)\) boundedly into
\(L^{q,\infty}(\D)\) for some \(1\le p<\infty\) and \(1\le q\le \infty\).
Since \(|\D|<\infty\), one can show that
$$
L^{p+\varepsilon}(\D)\subseteq L^{p,1}(\D)
$$
for every $\varepsilon>0$.  It follows that
$$
S_{a,b,c}:L^{p+\varepsilon}(\D)\to L^{q,\infty}(\D)
$$
is bounded, which again contradicts \textit{Step I}. 

\vspace{0.1cm}

The same argument above also applies to \(T_{a,b,c}\), and the proof is complete.

\vspace{0.1cm}

\noindent(2). It suffices to prove the desired estimate for \(S_{a,b,c}\), since $|T_{a,b,c}f|\le S_{a,b,c}(|f|)$. Let \(\|f\|_{L^\infty(\D)}=1\). Then, a simple calculation yields 
$$
|S_{a,b,c}f(z)| \lesssim
(1-|z|^2)^a \int_{\D}
\frac{(1-|w|^2)^b}{|1-z\overline w|^{a+b+3}}\,dA(w) \lesssim
(1-|z|^2)^a (1-|z|^2)^{-a-1}=
(1-|z|^2)^{-1}.
$$
Since \((1-|z|^2)^{-1}\in L^{1,\infty}(\D)\), it follows that $
S_{a,b,c}:L^\infty(\D)\to L^{1,\infty}(\D)$ is bounded. 
\end{proof}

\medskip  

\noindent Consequently, Reductions I and II lead us to divide the analysis into the following two categories:
\begin{enumerate}
    \item[$\bullet$] \emph{Main Case}:
    \begin{equation} \tag{{\bf Main}}\label{maincase}
    a, b>-1, \quad \textrm{and} \quad 0<c-a-b-2<1. 
    \end{equation}
    This is the central part of our endpoint and critical-line analysis (see Figure \ref{Fig3}).

    \vspace{0.1cm}

    \item[$\bullet$] \emph{Minor Case}:
    \begin{equation} \tag{{\bf Minor}} \label{minorcase}
    	a=-1,\quad b>-1, \quad \textrm{and} \quad 0\leq2-c+b<1.
    \end{equation} 
    This corresponds to the case in which the horizontal cut-off line coincides with the boundary line \(1/q=1\).
\end{enumerate}

The distinction between the Main and Minor Cases has a natural interpretation. 
After Reductions~I and~II, the remaining parameter range is divided according to whether there exists a pair $(p,q)\in \Omega_{\mathcal H}$ at which both $T_{a,b,c}$ and $S_{a,b,c}$ are of strong type $(p,q)$. 
In the Main Case, $(p,q)=(\infty,1)$ satisfies all three
strong-type conditions in \eqref{eq:strong}.
In the Minor Case, $a=-1$, so the condition $1/q>-a$ would require $1/q>1$, which is impossible.

\bigskip 
\section{Analysis of the Main Case: Five sub-main cases and further reductions} \label{Sec04}

The heart of our endpoint and critical-line analysis is to understand the behavior at the points where the critical line meets the two cut-off lines. In this section, we further organize the \textit{Main Case} \eqref{maincase} into five sub-cases, depending on the position of
$$
\calP_{\textnormal{V-H}}:=\left(b+1,-a\right),
$$
the intersection point of the vertical and horizontal cut-off lines, with respect to the critical line. More precisely, we consider the following cases:
\begin{enumerate}
    \item[$\bullet$] \textit{Main Case I:} 
    $$
    c\ge \max\{a+2,b+2\};
    $$
    \item[$\bullet$] \textit{Main Case II:} 
    $$
    b+2\le c<a+2;
    $$
    \item[$\bullet$] \textit{Main Case III:} 
    $$
    a+2 \le c<b+2;
    $$
    \item[$\bullet$] \textit{Main Case IV:} 
    $$
    1<c<\min\{a+2, b+2\};
    $$
    \item[$\bullet$] \textit{Main Case V:} 
    $$
    c \le 1.
    $$
\end{enumerate}

\vspace{0.1cm}

We refer the reader to Figure~\ref{Fig7} below for a visualization of these cases.

\vspace{-0.1cm}

\begin{figure}[ht]
\centering
\begin{tikzpicture}[scale=4]
\def\yh{0.28}       
\def\xc{0.72}        
\def\xright{1.65}
\def\ybottom{-0.35}
\node[left] at (0, 0.28) {$(0, c-a-b-2)$};
\node[above] at (0.72, 1) {$(3+a+b-c, 1)$};
\fill[green!20, opacity=0.3] (0,\ybottom) rectangle (\xc,\yh);      
\fill[black] (0, 0.28) circle (.3pt);
\fill[black] (0.72, 1) circle (.3pt);
\fill[magenta!20, opacity=0.3]   (\xc,\ybottom) rectangle (\xright,\yh);     
\fill[violet!20, opacity=0.3] (\xc,\yh) rectangle (1,1);   
\fill[violet!20, opacity=0.3] (1,\yh) rectangle (\xright,1);
\fill[orange!20, opacity=0.3] (0,\yh) -- (\xc,\yh) -- (\xc,1) -- cycle; 
\fill[blue!20, opacity=0.3]   (0,\yh) -- (0,1) -- (\xc,1) -- cycle;     
\draw[->, thick] (-0.08,0) -- (\xright+0.08,0) node[right] {$1/p$};
\draw[->, thick] (0,\ybottom-0.03) -- (0,1.12) node[above] {$1/q$};
\draw[thick] (0,0) rectangle (1,1);
\draw[thick] (1,0) -- (1,-0.025) node[below] {$1$};
\draw[thick] (0,1) -- (-0.025,1) node[left] {$1$};
\draw[blue, very thick] 
(0,\yh) -- (\xc,1)
node[midway, above, sloped, yshift=1pt] {\small critical line};
\draw[violet, very thick] (\xc,\yh) -- (\xc,1);       
\draw[magenta, very thick]    (\xc,\ybottom) -- (\xc,\yh);   
\draw[magenta, very thick]    (\xc,\yh) -- (\xright,\yh);   
\draw[green!60!black, very thick] (0,\yh) -- (\xc,\yh);  
\node[font=\bfseries\large, magenta!85!black] at (0.85,0.13) {I};
\node[font=\bfseries\large, green!45!black] at (0.42,0.12) {II};
\node[font=\bfseries\large, violet!85!black] at (0.85,0.63) {III};
\node[font=\bfseries\large, orange!90!black] at (0.42,0.5) {IV};
\node[font=\bfseries\large, blue!80!black] at (0.15,0.85) {V};
\end{tikzpicture}
\caption{\small The five sub-cases of the \textit{Main Case}: the magenta, green, violet, orange, and blue regions correspond to \textit{Main Cases I, II, III, IV}, and \textit{V}, respectively.}
\label{Fig7}
\end{figure}
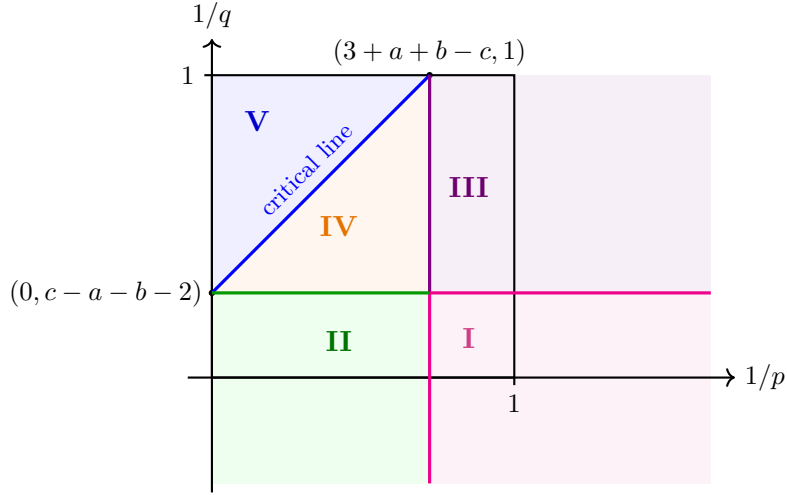

\vspace{-0.1cm}

\begin{rem}
The Forelli--Rudin phenomenon for weak-type estimates occurs in \textit{Main Cases III} and \textit{IV}, while the corresponding phenomenon for restricted weak-type estimates occurs in \textit{Main Case V}.
\end{rem}

\subsection{Reduction III: Failure of the weak-type estimates on the vertical cut-off line} \label{Subsec4.1}

We first show that in the \textit{Main Case} \eqref{maincase}, the weak-type
estimates for both $T_{a,b,c}$ and $S_{a,b,c}$ fail along the vertical cut-off line \eqref{verticalcut}. We have the following result.

\begin{prop} \label{20260710prop01}
Let $a,b,c\in \mathbb R$ satisfy the condition \eqref{maincase}. For every $q \in [1, \infty]$ such that the point $ \left(1+b,1/q \right)$ lies on the vertical cut-off line \eqref{verticalcut}, one has
\begin{equation} \label{20260523eq20}
T_{a,b,c},\; S_{a,b,c}:
L^{\frac{1}{b+1}}(\D)\not\to L^{q,\infty}(\D).
\end{equation}
\end{prop}

\begin{proof}
Without loss of generality, we may restrict ourselves to the case $b\le 0$. Otherwise, \eqref{verticalcut} lies outside the hypersingular regime $\Omega_{\mathcal H}$, and has no intersection with the critical line \eqref{criticalline}.  

We consider two different cases. If $b=0$, then  \eqref{verticalcut} is given by $1/p=1$ and has no intersection with \eqref{criticalline}. We prove \eqref{20260523eq20} by contradiction. Assume there exists some $q \in [1, \infty]$, such that
\begin{equation} \label{20260523eq21}
    T_{a,0,c}:
L^1(\D)\to L^{q,\infty}(\D). 
\end{equation}
Since $0<c-a-b-2<1$ and $a>-1$, by \cite[Theorem 1.1]{ZZ2022}, we can find $1<p_0, q_0<+\infty$ with $q_0 \neq q$ such that $T_{a, b, c}$ maps $L^{p_0}(\D)$ boundedly into $L^{q_0}(\D)$. Interpolating this estimate with \eqref{20260523eq21} by the off-diagonal 
Marcinkiewicz interpolation theorem \cite[Theorem~1.4.19]{Grafakos2014}, there exists some $1<p_1, q_1<+\infty$, such that
\begin{enumerate}
    \item $1/q_1<1/p_1+c-a-b-2$;
    \item $T_{a, b, c}: L^{p_1}(\D) \to L^{q_1, \infty}(\D)$.
\end{enumerate}
Take $\varepsilon>0$ sufficiently small, such that $q_1-\varepsilon>1$ and $1/(q_1-\varepsilon)<1/p_1+c-a-b-2$. Then arguing as in \eqref{20260523eq30}, we conclude that $T_{a, b, c}$ maps $L^{p_1}(\D)$ boundedly into $L^{q_1-\varepsilon}(\D)$. This is a contradiction to \cite[Theorem 1.1]{ZZ2022}. The same argument\footnote{The argument for $b=0$ indeed also works for the case when $b>0$: indeed, the same argument yields that the quasi-Banach weak-type estimates $T_{a, b, c}, S_{a, b, c}: L^{1/(1+b)}(\D) \to L^{q, \infty}(\D)$ fail. } applies to $S_{a, b, c}$.

\vspace{0.1cm}

If $b<0$, it suffices to prove the desired statement \eqref{20260523eq20} for $T_{a, b, c}$. Since $b<0$, take $\del>0$ such that $(b+1)(1+\del)<1$, and put
$$
f(w):=(1-|w|^2)^{-(b+1)}\left(1+\log\frac{1}{1-|w|^2}\right)^{-(b+1)(1+\del)}.
$$
Then
\begin{align}\label{Reduction I: b<0,f}
	\|f\|_{L^{\frac{1}{b+1}}(\D)}^{\frac{1}{b+1}}
	&=\int_{\D}(1-|w|^2)^{-1}\left(1+\log\frac{1}{1-|w|^2}\right)^{-(1+\del)}dA(w) \nonumber\\
	&\simeq  \int_0^1 \frac{r}{(1-r^2)\left(1+\log\frac{1}{1-r^2}\right)^{1+\del}}dr=\frac{1}{2\del} \nonumber\\
    &<\infty.
\end{align}
On the other hand, let $r>0$ be sufficiently small (only depending on $c$) such that for any $z \in \D_r:=\left\{z \in \D: |z| \le r\right\}$, one has
$$
\left|\arg(1-z\overline{w}) \right| \le \frac{\pi}{10(|c|+1)}, \qquad \textrm{for any} \quad w \in \D. 
$$
Then, for any $z \in \D_r$, one has 
\begin{align*}
	|T_{a,b,c} f (z)|
	&\gtrsim (1-|z|^2)^a\int_{\D}\frac{(1-|w|^2)^{-1}\left(1+\log\frac{1}{1-|w|^2}\right)^{-(b+1)(1+\del)}}{|1 - z\overline{w}| ^{c}} dA(w)\\
	&\gtrsim_r \int_{\D}(1-|w|^2)^{-1}\left(1+\log\frac{1}{1-|w|^2}\right)^{-(b+1)(1+\del)}dA(w)\\
    &=+\infty,
\end{align*}
where the last estimate follows by passing to polar coordinates and using the fact that $(b+1)(1+\del)<1$. In particular, the above computation shows that  $T_{a,b,c} f \notin L^{q,\infty}({\D})$, which completes the proof. 
\end{proof}

\subsection{Reduction IV: Weak-type bounds when $p=\infty$} \label{20260526subsec01}

Our next reduction is to consider estimates along the vertical line $1/p=0$. 
Recall that, since $b>-1$, the vertical cut-off line \eqref{verticalcut} stays away from the line $1/p=0$. Thus there remain two cases:
\begin{enumerate}
    \item $c-b-2=0$, which corresponds to the case where the critical line \eqref{criticalline} and the horizontal cut-off line \eqref{horizontalcut} intersect on the line $1/p=0$;
    \item $c-b-2\neq 0$, which corresponds to the case where \eqref{criticalline} and \eqref{horizontalcut} do not meet on the line $1/p=0$.
\end{enumerate}
We refer the reader to Figure \ref{Fig8} for an illustration of these two cases.

\vspace{-0.2cm}

\begin{figure}[ht]
\centering
\begin{tikzpicture}[scale=2.75]

    \tikzset{
        axis/.style={->, thick},
        box/.style={thick},
        crit/.style={blue, very thick},
        hcut/.style={green!55!black, very thick},
    }

    \begin{scope}[shift={(0,0)}]
        \def\xA{0.72}
        \def\yB{0.28}
        \def\yh{0.58}

        \draw[axis] (-0.06,0) -- (1.18,0) node[right] {$1/p$};
        \draw[axis] (0,-0.06) -- (0,1.18) node[above] {$1/q$};
        \draw[box] (0,0) rectangle (1,1);
        \draw (1,0) -- (1,-0.022) node[below] {$1$};
        \draw (0,1) -- (-0.022,1) node[left] {$1$};

        \coordinate (B) at (0,\yB);
        \coordinate (A) at (\xA,1);
        \draw[crit] (B) -- (A);
        \draw[hcut] (0,\yh) -- (1,\yh);
        \fill (0, \yh) node[left] {\small $1/q_*$};

        \node[font=\scriptsize] at (0.74, 1.06) {\eqref{criticalline}};
        \node[font=\scriptsize] at (0.74,0.65){\eqref{horizontalcut}};
    \end{scope}

    \begin{scope}[shift={(1.95,0)}]
        \def\xA{0.72}
        \def\yB{0.28}
        \def\yh{0.15}

        \draw[axis] (-0.06,0) -- (1.18,0) node[right] {$1/p$};
        \draw[axis] (0,-0.06) -- (0,1.18) node[above] {$1/q$};
        \draw[box] (0,0) rectangle (1,1);
        \draw (1,0) -- (1,-0.022) node[below] {$1$};
        \draw (0,1) -- (-0.022,1) node[left] {$1$};

        \coordinate (B) at (0,\yB);
        \coordinate (A) at (\xA,1);
        \draw[crit] (B) -- (A);
        \draw[hcut] (0,\yh) -- (1,\yh);
        \fill (0, \yB) node[left] {\small $1/q_*$}; 

        \node[font=\scriptsize] at (0.74, 1.06) {\eqref{criticalline}};
        \node[font=\scriptsize] at (0.74,0.24) {\eqref{horizontalcut}};
    \end{scope}

    \begin{scope}[shift={(4,0)}]
        \def\xA{0.72}
        \def\yB{0.28}
        \def\yh{0.28}

        \draw[axis] (-0.06,0) -- (1.18,0) node[right] {$1/p$};
        \draw[axis] (0,-0.06) -- (0,1.18) node[above] {$1/q$};
        \draw[box] (0,0) rectangle (1,1);
        \draw (1,0) -- (1,-0.022) node[below] {$1$};
        \draw (0,1) -- (-0.022,1) node[left] {$1$};

        \coordinate (B) at (0,\yB);
        \coordinate (A) at (\xA,1);
        \draw[crit] (B) -- (A);
        \draw[hcut] (0,\yh) -- (1,\yh);
        \fill (0, \yh) node[left] {\small $1/q_*$};

        \node[font=\scriptsize] at (0.7, 1.06) {\eqref{criticalline}};
        \node[font=\scriptsize] at (0.7,0.36) {\eqref{horizontalcut}};

        \node[align=center] at (0.5,-0.32)
        {\scriptsize The case when $c-b-2=0$};
    \end{scope}

    \node[align=center] at (1.48,-0.32)
    {\scriptsize The case when $c-b-2\neq 0$};

\end{tikzpicture}
\caption{\small Relative positions of the critical line and the horizontal cut-off line along the vertical line $1/p=0$: the blue line represents the critical line, the green line represents the horizontal cut-off line, and $1/q_*$ denotes the endpoint of the strong-type region on $1/p=0$. }
\label{Fig8}
\end{figure}
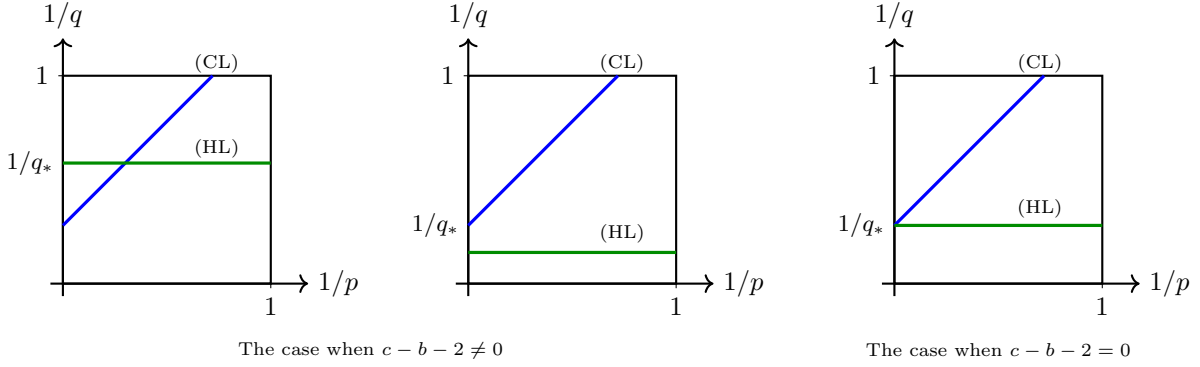

Denote
\begin{equation} \label{20260526eq01}
\frac{1}{q_*}:=\max\{-a,\;c-a-b-2\}.
\end{equation} 
By \cite[Theorem~1.1]{ZZ2022}, the operators $T_{a,b,c}$ and $S_{a,b,c}$ map $L^\infty(\D)$ boundedly into $L^q(\D)$ for every $1\le q<q_*$. We are primarily interested in the endpoint case $q=q_*$. In this subsection, we first treat the case when $c-b-2 \neq 0$. We have the following result. 

\begin{prop} \label{20260527prop01}
Let $a, b, c \in \R$ satisfy the condition \eqref{maincase} with $c-b-2 \neq 0$ and define $q_*$ as in \eqref{20260526eq01}. Then
$$
T_{a, b, c}, \; S_{a, b, c}: L^\infty(\D) \to L^{q_*, \infty}(\D).
$$
\end{prop}

\begin{proof}
Observe that it suffices to prove that $S_{a,b,c}$ is bounded from $L^\infty({\D})$ to $L^{q_*,\infty}({\D})$ when $c-2-b\neq0$.
For $f\in L^\infty({\D})$ and $\|f\|_{L^\infty({\D})}=1$, we have for each $z\in{\D}$,
\begin{align*}
	|S_{a,b,c} f (z)| 
	&= (1 - |z|^2) ^{a} \left|\int_{\D} \frac {(1-|w|^{2}) ^{b}}{|1 - z\overline{w}| ^{c}} f (w) dA (w)\right|\\
	&\leq (1 - |z|^2) ^{a} \int_{\D} \frac {(1-|w|^{2}) ^{b}}{\left|1 - z\overline{w}\right| ^{c}} dA (w) \\
    & \lesssim (1-|z|^2)^{-\frac{1}{q_*}},
\end{align*}
where in the last estimate above, we used  \cite[Theorem 1.12]{Zhu2005} with the assumption $c-b-2 \neq 0$. Hence,
$$
	|\{z\in{\D}: |S_{a,b,c} f (z)|>\lambda\}|
	\lesssim \left|\left\{z\in{\D}: (1-|z|^2)^{-\frac{1}{q_*}}>\lambda \right\}\right| \lesssim \min\{1, \lambda^{-q_*} \}. 
$$
Therefore, 
$$
\|S_{a,b,c} f\|_{L^{q_*,\infty}(\D)}
\leq \sup_{\lambda>0}\lambda |\{z\in\D: |S_{a,b,c} f (z)|>\lambda\}|^{\frac{1}{q_*}}<\infty.
$$
This implies $S_{a,b,c}$ is bounded from $L^\infty({\D})$ to $L^{q_*,\infty}({\D})$.
\end{proof}

\subsection{Reduction V: Weak-type bounds on critical line away from the cut-off intersections}

In the final part of this section, we establish weak-type bounds at points that lie on both the critical line and on the boundary of the strong-type region, while remaining away from the two cut-off lines. This corresponds to the situation when $\calP_{\textrm{V-H}}=(b+1, -a)$ lies in the region I, II, III, and IV (see Figure \ref{Fig7}). We refer the reader to Figure~\ref{Fig9} below for an example of this situation, represented there by the open line segment $\calP_{\textnormal{C-H}}\calP_{\textnormal{C-V}}$.

\vspace{-0.1cm}

\begin{figure}[ht]
\centering
\begin{tikzpicture}[scale=4]

    \tikzset{
        axis/.style={->, thick},
        box/.style={thick},
        crit/.style={blue, very thick},
        vcut/.style={orange!85!black, very thick},
        hcut/.style={green!55!black, very thick}
    }

    \def\xTop{0.72}   
    \def\yLeft{0.28}  
    \def\xv{0.56}     
    \def\yh{0.58}     

    \pgfmathsetmacro{\xHL}{\xTop*(\yh-\yLeft)/(1-\yLeft)}
    \pgfmathsetmacro{\yVL}{\yLeft+(1-\yLeft)*\xv/\xTop}

    \draw[axis] (-0.06,0) -- (1.18,0) node[right] {$1/p$};
    \draw[axis] (0,-0.06) -- (0,1.18) node[above] {$1/q$};
    \draw[box] (0,0) rectangle (1,1);

    \fill[gray!20, opacity=0.55]
        (0,1) -- (\xv,1) -- (\xv,\yVL) -- (\xHL,\yh) -- (0,\yh) -- cycle;

    \draw (1,0) -- (1,-0.022) node[below] {$1$};
    \draw (0,1) -- (-0.022,1) node[left] {$1$};
    \draw (0, \yh) -- (-0.022, \yh) node[left] {$-a$};
    \draw (\xv, 0) -- (\xv, -0.022) node[below]{$b+1$}; 

    \draw[hcut] (0,\yh) -- (1,\yh);
    \draw[vcut] (\xv,0) -- (\xv,1);

    \draw[crit] (0,\yLeft) -- (\xTop,1);

    \fill[black] (\xHL,\yh) circle (.3pt);
    \fill[black] (\xv,\yVL) circle (.3pt);

    \node[below] at (\xHL+0.1,\yh) {\small $\calP_{\textnormal{C-H}}$};
    \node[right] at (\xv+0.02,\yVL) {\small $\calP_{\textnormal{C-V}}$};

    \node[font=\scriptsize, above right] at (\xTop-0.05,1) {\eqref{criticalline}};
    \node[font=\scriptsize, right] at (0.55,0.2) {\eqref{verticalcut}};
    \node[font=\scriptsize] at (0.9,\yh+0.05) {\eqref{horizontalcut}};

    \fill[black] (\xv,\yh) circle (.3pt);
    \node[below right] at (\xv,\yh) {\small $\calP_{\textnormal{V-H}}$};

    \draw[blue!50, line width=5pt, opacity=0.18, line cap=round,
          shorten <=2.2pt, shorten >=2.2pt]
        (\xHL,\yh) -- (\xv,\yVL);

    \draw[blue!60, line width=3pt, opacity=0.28, line cap=round,
          shorten <=2.2pt, shorten >=2.2pt]
        (\xHL,\yh) -- (\xv,\yVL);

\end{tikzpicture}
\caption{\small Weak-type estimates on the critical line away from the two cut-off lines: the gray shaded region denotes the strong-type region guaranteed by \cite[Theorem~1.1]{ZZ2022}, while the highlighted blue open segment $\calP_{\textnormal{C-H}}\calP_{\textnormal{C-V}}$ indicates the portion of the critical line that lies on the boundary of this strong-type region and away from both cut-off lines.}
\label{Fig9}
\end{figure}
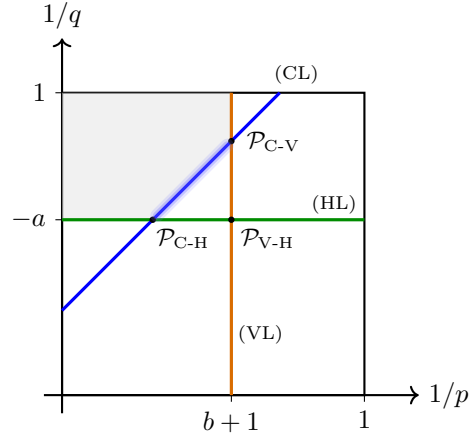

Note first that for any such $(p, q)$ lying on the critical line, we may assume $p<+\infty$. This is because the case when $p=+\infty$ is treated in Section \ref{20260526subsec01}. Our main result is the following.

\begin{prop} \label{20260527prop02}
Let $a, b, c \in \R$ satisfy the condition \eqref{maincase}. Then for any 
$$
(p, q) \in \left\{\substack{1 \le p<+\infty \\ \\ 1 \le q \le \infty}: \substack{\frac{1}{q}=\frac{1}{p}+c-a-b-2 \\ \\ \frac{1}{q}>-a, \; \;  \frac{1}{p}<b+1} \right\},
$$
one has 
$$
T_{a, b, c}, \; S_{a, b, c}: L^p(\D) \to L^{q, \infty}(\D) 
$$
\end{prop}

\begin{proof}
Again, it suffices to show the desired claim for $S_{a, b, c}$. Let $f \in L^p(\D)$. Write 
$$
S_{a, b, c}f(z)=W_\eta(z)S_{a+\eta, b, c}f(z), 
$$
where $W_{\eta}(z):=(1-|z|^2)^{-\eta}, \; z \in \D$ with $\eta:=c-2-a-b$, and $S_{a+\eta, b, c}$ is the Forelli--Rudin type operator given by
$$
S_{a+\eta, b, c}f(z)=S_{c-2-b, b, c}f(z)=(1-|z|^2)^{c-2-b} \int_{\D} \frac{(1-|w|^2)^b f(w)}{|1-z \overline{w}|^c} dA(w). 
$$
We make the following claims.
\begin{enumerate}
    \item $S_{a+\eta, b, c}=S_{c-2-b, b, c}: L^p(\D) \to L^p(\D)$ is bounded. 
    \item $W_\eta \in L^{1/\eta, \infty}(\D)$. 
\end{enumerate}
Assuming (1) and (2), and using Lorentz--H\"older inequality with the fact that $1/q=1/p+\eta$, we have
\begin{align*}
    \left\|S_{a, b, c} f \right\|_{L^{q, \infty}(\D)} 
    &=\left\|W_{\eta} \cdot S_{a+\eta, b, c}f \right\|_{L^{q, \infty}(\D)} \\
    & \lesssim \left\|W_\eta \right\|_{L^{1/\eta, \infty}(\D)} \left\|S_{c-2-b, b, c}f \right\|_{L^p(\D)} \\
    & \lesssim \left\|W_\eta \right\|_{L^{1/\eta, \infty}(\D)} \left\|f \right\|_{L^p(\D)} \\
    &\lesssim \left\|f \right\|_{L^p(\D)},
\end{align*}
which gives the desired weak-type estimate. 

Therefore, it remains to verify claims (1) and (2). Claim (1) follows from \cite[Theorem 4]{Zhao2015} and the fact that $1/p=1/q+a+b+2-c>b+2-c$. Claim (2) follows from a direct computation. 
\end{proof}

\begin{rem}
\begin{enumerate}
    \item Propositions~\ref{20260527prop01} and~\ref{20260527prop02} extend the critical-line and endpoint results in \cite{HZ2025} for $K_{2t}$ and $K_{2t}^+$ defined as in \eqref{20260514eq30} to a broader class of Forelli--Rudin type operators.
    \item The approach used in the proof of Proposition~\ref{20260527prop02} is referred to as the \emph{Forelli--Rudin method} in \cite{HZ2025}. The main idea of this method is \emph{to reduce the desired weak-type estimate for a more singular operator to the product of a weight and a less singular operator}. 
    
    We point out that this method provides a useful tool for studying weak-type bounds, especially for hypersingular operators. For instance, if one tries to prove Proposition~\ref{20260527prop02} by the direct pointwise argument used in Proposition~\ref{20260527prop01} (via H\"older), then in the case $1<p<+\infty$ and $cp'>bp'+2$, where $p'$ is the conjugate exponent of $p$, one has
\begin{align*}
    |S_{a,b,c}f(z)|
    &\le (1-|z|^2)^a \left(\int_{\D} \frac{(1-|w|^2)^{bp'}}{|1-z\overline{w}|^{cp'}}\,dA(w)\right)^{1/p'} \|f\|_{L^p(\D)} \\
    &\lesssim (1-|z|^2)^{a+b-c+\frac{2}{p'}}\|f\|_{L^p(\D)} \\
    &=(1-|z|^2)^{-\frac1q-\frac1p}\|f\|_{L^p(\D)},
\end{align*}
where in the last equation above, we used the fact that $1/q=1/p+c-a-b-2$. This estimate is \emph{not} sufficient to conclude that $S_{a,b,c}f\in L^{q,\infty}(\D)$ by a direct level-set argument, since such an argument would require the stronger pointwise bound $|S_{a,b,c}f(z)| \lesssim (1-|z|^2)^{-1/q}$.

To this end, we also comment that the Forelli--Rudin method is not limited to the complex-analytic setting. For instance, we refer the reader to \cite[Section~5]{HZ2025} for an extension of this method to a dyadic harmonic analysis setting, where the first and fourth authors used it to study hypersingular sparse operators.
\end{enumerate}
\end{rem}

We now turn to the central part of our critical-line and endpoint analysis in the \textit{Main Case} \eqref{maincase}. The \emph{key} point is to understand the weak- and restricted weak-type behavior of $T_{a,b,c}$ and $S_{a,b,c}$ at the following three intersection points:
\begin{enumerate}
\item the critical--horizontal intersection
$$
\calP_{\textnormal{C-H}}:=(2-c+b,\; -a);
$$
\item the critical--vertical intersection
$$
\calP_{\textnormal{C-V}}:=(b+1, \; c-a-1);
$$
\item and finally, the vertical--horizontal intersection
$$
\calP_{\textnormal{V-H}}:=(b+1, \; -a).
$$
\end{enumerate}
We refer the reader to Figure~\ref{Fig9} for a representative configuration. Of course, for these intersection points to play a meaningful role in the endpoint analysis, they must lie on the boundary of the strong-type region. In the particular configuration shown in Figure~\ref{Fig9}, the two relevant endpoints are $\calP_{\textnormal{C-H}}$ and $\calP_{\textnormal{C-V}}$, while $\calP_{\textnormal{V-H}}$ plays a less essential role. On the other hand, when $\calP_{\textnormal{V-H}}$ lies on or above the critical line \eqref{criticalline}, the endpoint behavior at $\calP_{\textnormal{V-H}}$ becomes relevant and enters into the overall picture.

\medskip 
The endpoint behavior described above will appear repeatedly in the analysis of the \emph{Main Cases I--V}. Therefore, before turning to the individual cases shown in Figure~\ref{Fig7}, we first treat these endpoints separately; see Sections \ref{Sec05}--\ref{Sec09}. Then, the resulting estimates will serve as building blocks for the subsequent case-by-case analysis; see Section \ref{Sec10}.  

\bigskip 

\section{Analysis of the Main Case: Weak-type Bounds at the Critical--Horizontal Intersection $\calP_{\textnormal{C-H}}$ away from the vertical cut-off line and the Forelli--Rudin Phenomenon at the level of weak-type estimates} \label{Sec05}

The goal of this section is to study the weak-type bounds at the critical--horizontal intersection $$ \calP_{\textnormal{C-H}}=(2-c+b,-a), $$ where the critical line \eqref{criticalline} meets the horizontal cut-off line \eqref{horizontalcut}, and which lies strictly to the left of the vertical-horizontal intersection $\calP_{\textnormal{V-H}}$. Equivalently, we have the following condition
\begin{equation} \tag{${\bf P_1}$} \label{conditionP1}
\begin{cases}
a,b,c \in \R \quad \textrm{satisfy \eqref{maincase}}; \\
0 \le 2-c+b <1, \quad a<0; \\
c>1.
\end{cases}
\end{equation}
Indeed, the second condition, together with \eqref{maincase}, ensures that $\calP_{\textnormal{C-H}}$ lies in the region $0\le 1/p<1$ and $0<1/q<1$,
 while the third condition guarantees that $\calP_{\textnormal{C-H}}$ lies strictly to the left of $\calP_{\textnormal{V-H}}$. This is the typical situation\footnote{See Figure~\ref{Fig7} for the definitions of Regions I--IV.} when the vertical--horizontal intersection $\calP_{\textnormal{V-H}}$ lies in Regions III and IV, where $2-c+b>0$. It may also occur when $\calP_{\textnormal{V-H}}$ lies in Regions I and II, in the borderline case $2-c+b=0$ (see Figure~\ref{Fig10} below).

\begin{figure}[ht]
\centering


\begin{minipage}{0.48\textwidth}
\centering
\begin{tikzpicture}[scale=3.3]
\def\yh{0.28}       
\def\xc{0.72}        
\def\xright{1.35}
\def\ybottom{-0.05}          

\fill[violet!20, opacity=0.3] (\xc,\yh) rectangle (1,1);   
\fill[violet!20, opacity=0.3] (1,\yh) rectangle (\xright,1);

\draw[->, thick] (-0.08,0) -- (\xright+0.08,0) node[right] {$1/p$};
\draw[->, thick] (0,\ybottom-0.03) -- (0,1.12) node[above] {$1/q$};

\draw[thick] (0,0) rectangle (1,1);
\draw[thick] (1,0) -- (1,-0.025) node[below] {$1$};
\draw[thick] (0,1) -- (-0.025,1) node[left] {$1$};

\draw[blue, very thick] (0,\yh) -- (\xc,1);
\node[above] at (\xc/2-0.05, \yh/2+0.5) {\scriptsize \eqref{criticalline}};

\draw[violet, very thick] (\xc,\yh) -- (\xc,1);        
\draw[violet, dashed]    (\xc,\yh) -- (\xright,\yh);   
\draw[black, dashed] (0,\yh) -- (\xc,\yh);  

\node[font=\bfseries\large, violet!85!black] at (0.85,0.63) {III};

\draw[green!55!black, line width=3pt, opacity=0.18, line cap=round,
      shorten <=2.2pt, shorten >=2.2pt]
    (0,0.4) -- (1,0.4);

\fill[black] (0.85, 0.4) circle (.4pt);
\node[above] at (0.85, 0.4) {\scriptsize $\calP_{\textnormal{V-H}}$};

\fill[black] (0.12, .4) circle (.4pt);
\node[below] at (0.21, .4) {\scriptsize $\calP_{\textnormal{C-H}}$};

\node[above] at (0.5, 0.4) {\scriptsize \eqref{horizontalcut}};
\end{tikzpicture}
\end{minipage}
\hfill
\begin{minipage}{0.48\textwidth}
\centering
\begin{tikzpicture}[scale=3.3]
\def\yh{0.28}       
\def\xc{0.72}        
\def\xright{1.35}
\def\ybottom{-0.05}

\fill[orange!20, opacity=0.3] (0,\yh) -- (\xc,\yh) -- (\xc,1) -- cycle;    

\draw[->, thick] (-0.08,0) -- (\xright+0.08,0) node[right] {$1/p$};
\draw[->, thick] (0,\ybottom-0.03) -- (0,1.12) node[above] {$1/q$};

\draw[thick] (0,0) rectangle (1,1);
\draw[thick] (1,0) -- (1,-0.025) node[below] {$1$};
\draw[thick] (0,1) -- (-0.025,1) node[left] {$1$};

\draw[blue, very thick] (0,\yh) -- (\xc,1);
\node[above] at (\xc/2-0.05, \yh/2+0.5) {\scriptsize \eqref{criticalline}};

\draw[orange, dashed] (\xc,\yh) -- (\xc,1);    
\draw[orange, dashed] (0,\yh) -- (\xc,\yh);  
\draw[dashed] (\xc, \yh) -- (1, \yh);
\draw[dashed] (\xc, \yh) -- (\xc, 0);

\node[font=\bfseries\large, orange!90!black] at (0.56,0.67) {IV};

\draw[green!55!black, line width=3pt, opacity=0.18, line cap=round,
      shorten <=2.2pt, shorten >=2.2pt]
    (0,0.4) -- (1,0.4);

\fill[black] (0.5, 0.4) circle (.4pt);
\node[above] at (0.5, 0.42) {\scriptsize $\calP_{\textnormal{V-H}}$};

\fill[black] (0.12, .4) circle (.4pt);
\node[below] at (0.21, .4) {\scriptsize $\calP_{\textnormal{C-H}}$};

\node[above] at (0.9, 0.4) {\scriptsize \eqref{horizontalcut}};
\end{tikzpicture}
\end{minipage}

\vspace{0.3em}

{\small $\calP_{\textnormal{V-H}}$ belongs to Regions III or IV with $2-c+b>0$}

\vspace{.3em}


\begin{minipage}{0.48\textwidth}
\centering
\begin{tikzpicture}[scale=3.3]
\def\yh{0.28}       
\def\xc{0.72}        
\def\xright{1.35}
\def\ybottom{-0.15}    

\fill[magenta!20, opacity=0.3] (\xc,\ybottom) rectangle (\xright,\yh);        

\draw[->, thick] (-0.08,0) -- (\xright+0.08,0) node[right] {$1/p$};
\draw[->, thick] (0,\ybottom-0.03) -- (0,1.12) node[above] {$1/q$};

\draw[thick] (0,0) rectangle (1,1);
\draw[thick] (1,0) -- (1,-0.025) node[below] {$1$};
\draw[thick] (0,1) -- (-0.025,1) node[left] {$1$};

\draw[blue, very thick] (0,\yh) -- (\xc,1);
\node[above] at (\xc/2-0.05, \yh/2+0.5) {\scriptsize \eqref{criticalline}};

\draw[black, dashed] (\xc,\yh) -- (\xc,1);     
\draw[magenta, very thick] (\xc,\ybottom) -- (\xc,\yh);   
\draw[magenta, very thick] (\xc,\yh) -- (\xright,\yh);     

\node[font=\bfseries\large, magenta!85!black] at (0.85,0.13) {I};

\fill[black] (0.85,\yh) circle (.4pt);
\node[above] at (0.85,\yh+0.02) {\scriptsize $\calP_{\textnormal{V-H}}$};

\fill[black] (0, \yh) circle (.4pt);
\node[left] at (-0.01,\yh) {\scriptsize  $\calP_{\textnormal{C-H}}$};

\draw[green!55!black, line width=3pt, opacity=0.18, line cap=round,
      shorten <=2.2pt, shorten >=2.2pt]
    (0,\yh) -- (1,\yh);

\node[below] at (0.4, \yh-0.01) {\scriptsize  \eqref{horizontalcut}};
\end{tikzpicture}
\end{minipage}
\hfill
\begin{minipage}{0.48\textwidth}
\centering
\begin{tikzpicture}[scale=3.3]
\def\yh{0.28}       
\def\xc{0.72}        
\def\xright{1.35}
\def\ybottom{-0.15}

\fill[green!20, opacity=0.3] (0,\ybottom) rectangle (\xc,\yh);           

\draw[->, thick] (-0.08,0) -- (\xright+0.08,0) node[right] {$1/p$};
\draw[->, thick] (0,\ybottom-0.03) -- (0,1.12) node[above] {$1/q$};

\draw[thick] (0,0) rectangle (1,1);
\draw[thick] (1,0) -- (1,-0.025) node[below] {$1$};
\draw[thick] (0,1) -- (-0.025,1) node[left] {$1$};

\draw[blue, very thick] (0,\yh) -- (\xc,1);
\node[above] at (\xc/2-0.05, \yh/2+0.5) {\scriptsize \eqref{criticalline}};

\draw[green!60!black, dashed] (\xc,\ybottom) -- (\xc,\yh);   
\draw[black, dashed] (\xc,\yh) -- (\xc,1);     
\draw[green!60!black, very thick] (0,\yh) -- (\xc,\yh);  

\node[font=\bfseries\large, green!45!black] at (0.42,0.12) {II};

\fill[black] (0.5,\yh) circle (.4pt);
\node[above] at (0.5,\yh+0.02) {\scriptsize  $\calP_{\textnormal{V-H}}$};

\fill[black] (0, \yh) circle (.4pt);
\node[left] at (-0.01,\yh) {\scriptsize  $\calP_{\textnormal{C-H}}$};

\draw[green!55!black, line width=3pt, opacity=0.18, line cap=round,
      shorten <=2.2pt, shorten >=2.2pt]
    (0,\yh) -- (1,\yh);

\node[above] at (0.9, \yh) {\scriptsize  \eqref{horizontalcut}};
\end{tikzpicture}
\end{minipage}

\vspace{0.3em}

{\small $\calP_{\textnormal{V-H}}$ belongs to Regions I or II with $2-c+b=0$}

\caption{\small Four cases when $\calP_{\textnormal{C-H}}$ lies strictly on the left-hand side of $\calP_{\textnormal{V-H}}$.}
\label{Fig10}
\end{figure}
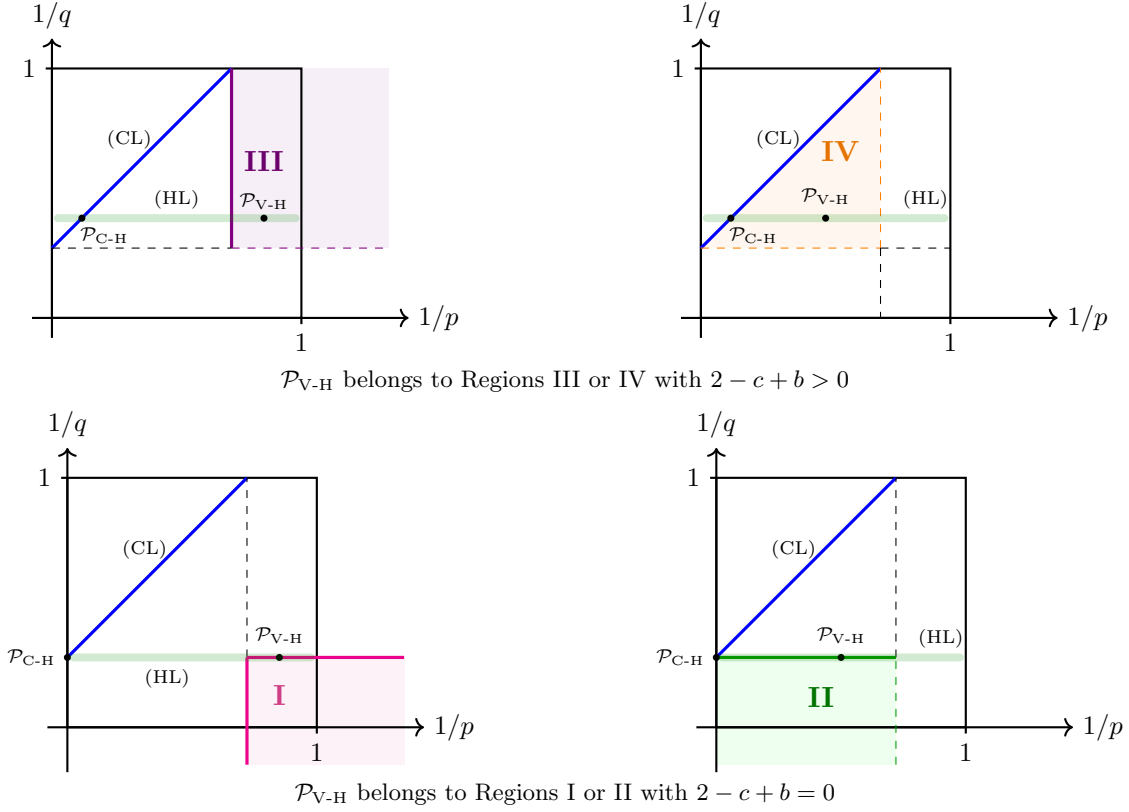

Our main result in this section is the following. 

\begin{thm}\label{Forelli--Rudin threshold}
Let $a, b, c \in \R$ satisfy \eqref{conditionP1}. Then the following statements hold.
\begin{itemize}
	\item[(i)] The Forelli--Rudin type operator $S_{a,b,c}$ is unbounded from $L^{\frac{1}{2-c+b}}(\D)$ to $L^{-\frac{1}{a},\infty}(\D)$.
	\item[(ii)] If $0 \le 2-c+b<1/2$, the Forelli--Rudin type operator $T_{a,b,c}$ is unbounded from $L^{\frac{1}{2-c+b}}(\D)$ to $L^{-\frac{1}{a},\infty}(\D)$.
	\item[(iii)] If $1/2 \le 2-c+b<1$, the Forelli--Rudin type operator $T_{a,b,c}$ is bounded from $L^{\frac{1}{2-c+b}}(\D)$ to $L^{-\frac{1}{a},\infty}(\D)$.
\end{itemize}
Here, in the borderline case $2-c+b=0$, the space $L^{\frac{1}{2-c+b}}(\D)$ is understood as $L^\infty(\D)$.
\end{thm}

Observe that Theorem~\ref{Forelli--Rudin threshold} shows that the \emph{Forelli--Rudin phenomenon occurs at the level of weak-type estimates at the critical--horizontal intersection $\calP_{\textnormal{C-H}}$} (namely, $\calP_{\textnormal{C-H}}$ is a Forelli--Rudin pair at the level of weak-type estimates when $1/2 \le 2-c+b<1$), with a natural threshold at $p=2$. 

\vspace{0.1cm}

We now turn to the proof of Theorem~\ref{Forelli--Rudin threshold}.

\subsection{Proof of Theorem \ref{Forelli--Rudin threshold}, (i): Failure of the weak-type estimates for $S_{a, b, c}$ at the intersection $\calP_{\textnormal{C-H}}$}

We first consider the case $0<2-c+b<1$. For sufficiently large integer \(N\ge 1\), denote 
$$
f_N(w):=\frac{\one_{\{1/2<|w|^2<1-2^{-N}\}}(w)}{(1-|w|^2)^{2-c+b}}, \qquad w \in \D.
$$
A straightforward calculation then yields
\begin{equation} \label{20260529eq00}
\|f_N\|_{L^{\frac{1}{2-c+b}}(\D)}^{\frac{1}{2-c+b}}=\int_{1/2<|w|^2<1-2^{-N}} \frac{1}{1-|w|^2} dA(w) \simeq N.
\end{equation} 
On the other hand, set
\[
A_N:=\{z\in\D:\ 1-2^{-N}<|z|^2<1-2^{-N-1}\}, 
\]
and for each fixed \(z\in A_N\), define
\[
E_z:=\left\{w\in\D:\ 2(1-|z|^2)<1-|w|^2<2^{-1},\ 
|\arg z-\arg w|<1-|w|^2\right\}.
\]
We refer the reader to Figure~\ref{Fig11} for an example. 

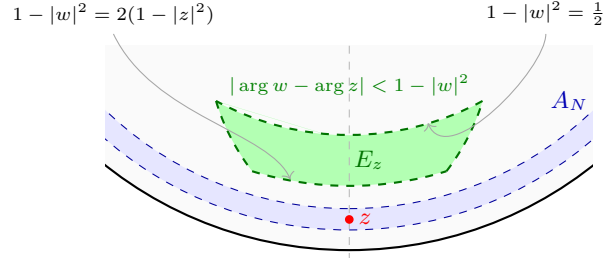
\begin{figure}[ht]
\begin{center}
\begin{tikzpicture}[scale=5.2]

    \pgfmathsetmacro{\tz}{0.15}              
    \pgfmathsetmacro{\rz}{sqrt(1-\tz)}

    \pgfmathsetmacro{\tAmin}{0.1}
    \pgfmathsetmacro{\tAmax}{0.2}
    \pgfmathsetmacro{\rAout}{sqrt(1-\tAmin)}
    \pgfmathsetmacro{\rAin}{sqrt(1-\tAmax)}

    \pgfmathsetmacro{\rEin}{sqrt(1/2)}
    \pgfmathsetmacro{\rEout}{sqrt(1-2*\tz)}

    \pgfmathsetmacro{\ain}{deg(1-\rEin*\rEin)}
    \pgfmathsetmacro{\aout}{deg(1-\rEout*\rEout)}

    \begin{scope}[rotate=-90]

        \clip (0.48,-0.62) rectangle (1.04,0.62);

        \fill[gray!4] (0,0) circle (1);
        \draw[black, thick] (0,0) circle (1);

        \path[fill=blue!10, even odd rule]
            (0,0) circle (\rAout)
            (0,0) circle (\rAin);

        \draw[blue!60!black, dashed] (0,0) circle (\rAin);
        \draw[blue!60!black, dashed] (0,0) circle (\rAout);
\path[fill=green!30, draw=none]
    ({\rEout*cos(\aout)},{\rEout*sin(\aout)})
    plot[domain=\rEout:\rEin, variable=\r, samples=120]
        ({\r*cos(deg(1-\r*\r))},{\r*sin(deg(1-\r*\r))})
    arc[start angle=\ain, end angle=-\ain, radius=\rEin]
    plot[domain=\rEin:\rEout, variable=\r, samples=120]
        ({\r*cos(deg(-(1-\r*\r)))},{\r*sin(deg(-(1-\r*\r)))})
    arc[start angle=-\aout, end angle=\aout, radius=\rEout]
    -- cycle;


\draw[green!45!black, dashed, line width=0.8pt]
    plot[domain=\rEout:\rEin, variable=\r, samples=120]
        ({\r*cos(deg(1-\r*\r))},{\r*sin(deg(1-\r*\r))});

\draw[green!45!black, dashed, line width=0.8pt]
    ({\rEin*cos(\ain)},{\rEin*sin(\ain)})
    arc[start angle=\ain, end angle=-\ain, radius=\rEin];

\draw[green!45!black, dashed, line width=0.8pt]
    plot[domain=\rEin:\rEout, variable=\r, samples=120]
        ({\r*cos(deg(-(1-\r*\r)))},{\r*sin(deg(-(1-\r*\r)))});

\draw[green!45!black, dashed, line width=0.8pt]
    ({\rEout*cos(-\aout)},{\rEout*sin(-\aout)})
    arc[start angle=-\aout, end angle=\aout, radius=\rEout];

        \draw[gray!55, dashed] (0.48,0) -- (1.02,0);

\fill[red] (\rz,0) circle (0.011);
        \coordinate (EzLabelPos)    at (0.75,0.02);
        \coordinate (ANLabelPos)    at (0.93,0.30);
        \coordinate (ZLabelPos)     at (0.955,0.035);

        \coordinate (InnerTarget)   at ({\rEin*cos(\ain)},{\rEin*sin(\ain)});
        \coordinate (OuterTarget)   at ({\rEout*cos(-\aout)},{\rEout*sin(-\aout)});

        \coordinate (InnerLabelPos) at (0.47,0.48);
        \coordinate (OuterLabelPos) at (0.47,-0.42);
        \coordinate (AngleLabelPos) at (0.76,-0.70);

    \end{scope}

  \node[green!45!black] at (0.05, -.77) {\small $E_z$}; \node[blue!70!black] at (.56,-0.62) {\small $A_N$}; 
  \node[red, anchor=west] at (0, -.92) {\small $z$};
    \node[above] at (0.5, -0.45)
        {\scriptsize $1-|w|^2=\frac12$};
  \draw[->, thin, gray!70]
        (0.5, -0.45)
        .. controls (0.5, -0.54) and (0.25, -0.8) ..
        (0.2, -0.677);

    \node[above] at (-0.6, -0.45)
        {\scriptsize $1-|w|^2=2(1-|z|^2)$};
  \draw[->, thin, gray!70]
    (-0.6, -0.45) .. controls (-0.5, -0.66) and (-0.2, -0.7) .. (-0.15, -0.82);

    \node[font=\scriptsize, green!50!black, align=center] at (0, -0.58)
        {$|\arg w-\arg z|<1-|w|^2$};

\end{tikzpicture}
\caption{An example of the set \(E_z\) associated with a fixed point \(z\in A_N\).}
\label{Fig11}
\end{center} 
\end{figure}

We make the following observations.
\begin{enumerate}
    \item For any $z \in A_N$,
    $$
    E_z \subseteq \{w \in \D: 1/2<|w|^2<1-2^{-N}\}.
    $$
    Indeed, for any $w \in E_z$, one has
    $$
    |w|^2<1-2(1-|z|^2)<1-2\cdot 2^{-N-1}=1-2^{-N}.
    $$

    \item For any $z \in A_N$ and $w \in E_z$, one has
    $$
    |1-z\overline{w}| \lesssim 1-|w|^2.
    $$
    This follows from the following direct computation: 
    \begin{align} \label{20260530eq01}
    |1-z\bar{w}|
    &\leq (1-|z||w|)+|z||w|\left|1-\frac{z}{|z|}\frac{\bar{w}}{|w|}\right| \nonumber \\
    &\leq(1-|z|)+|z|(1-|w|)+\left|1-\frac{z}{|z|}\frac{\bar{w}}{|w|}\right| \nonumber \\
    &\lesssim  (1-|z|^2)+(1-|w|^2)+|1-e^{i(\arg z-\arg w)}| \nonumber \\
    &\lesssim (1-|w|^2)+\sin\left|\frac{\arg z-\arg w}{2}\right| \nonumber \\
    &\lesssim 1-|w|^2.
    \end{align}
\end{enumerate}

These observations, together with the assumption $c>1$, imply that for any $z \in A_N$,
\begin{align} \label{20260529eq01}
    S_{a,b,c} f_N(z)
    &=(1 - |z|^2) ^{a} \int_{\D} \frac {(1-|w|^{2}) ^{b}}{|1 - z\bar{w}|^{c}} f_N(w) dA(w) \nonumber \\
    & \ge (1 - |z|^2) ^{a} \int_{E_z} \frac {(1-|w|^{2}) ^{b}}{|1 - z\bar{w}|^{c}} \cdot \frac{1}{(1-|w|^2)^{2-c+b}} dA(w) \nonumber \\
    &\gtrsim (1 - |z|^2) ^{a} \int_{E_z} \frac{1}{(1-|w|^2)^2} dA(w).
\end{align}
Moreover,
$$
    \int_{E_z} \frac{1}{(1-|w|^2)^2} dA(w)
    \simeq \int_{2^{-1/2}}^{\sqrt{1-2(1-|z|^2)}} \frac{r}{(1-r^2)^2}
    \left(\int_{\arg z-(1-r^2)}^{\arg z+(1-r^2)} d\theta  \right) dr \gtrsim \log\left(\frac{1}{1-|z|^2}\right).
$$
Combining this with \eqref{20260529eq01}, we obtain
$$
S_{a,b,c} f_N(z) \gtrsim  (1 - |z|^2) ^{a} \log\left(\frac{1}{1-|z|^2}\right) \simeq  N2^{-Na}, \qquad z \in A_N.
$$
Hence,
\begin{align*}
    \|S_{a,b,c} f_N\|_{L^{-\frac{1}{a},\infty}(\D)}
    & \gtrsim  N2^{-Na} \left| \left\{z\in\D: |S_{a,b,c} f_N(z)| \gtrsim N2^{-Na}\right\} \right|^{-a} \\
    &\gtrsim  N2^{-Na}|A_N|^{-a}  \simeq N2^{-Na} \cdot 2^{Na} =N.
\end{align*}
Therefore,
\[
    \frac{\|S_{a,b,c} f_N\|_{L^{-\frac{1}{a},\infty}(\D)}}{\|f_N\|_{L^{\frac{1}{2-c+b}}(\D)}}
    \gtrsim \frac{N}{N^{2-c+b}}=N^{c-b-1}  \to \infty
\]
as $N\to\infty$. Here, in the last estimate, we used the assumption $0<2-c+b<1$. This completes the proof of Theorem~\ref{Forelli--Rudin threshold}, (i), under the assumption $0<2-c+b<1$.

\medskip

Finally, we consider the case $2-c+b=0$. The same argument above applies, with the only modification that we take
\[
f_N(w):=\one_{\{1/2<|w|^2<1-2^{-N}\}}(w).
\]
Then $\|f_N\|_{L^\infty(\D)}=1$, while the preceding lower bound becomes 
\[
\|S_{a,b,c}f_N\|_{-\frac1a,\infty}\gtrsim N.
\]
Therefore, $S_{a,b,c}:L^\infty(\D)\to L^{-\frac1a,\infty}(\D)$ fails to be bounded.

\medskip 

The proof of Theorem \ref{Forelli--Rudin threshold}, (i) is complete. 

\subsection{Proof of Theorem \ref{Forelli--Rudin threshold}, (ii): Preliminaries} \label{20260831subsec01}

Next, we turn to the proof of Theorem~\ref{Forelli--Rudin threshold}, (ii). This part is more delicate than the proof of (i), because the complex kernel
\begin{equation} \label{20260531eq01}
\frac{1}{(1-z\overline{w})^c}
\end{equation} 
in $T_{a,b,c}$ exhibits cancellation that is absent from the positive kernel in $S_{a,b,c}$. In particular, the construction used in the proof of (i) is too crude to detect the cancellation structure of the complex kernel. We therefore prove Theorem~\ref{Forelli--Rudin threshold}, (ii), by a more refined construction, combining a fourth-moment probabilistic argument with dyadic analysis on Carleson tents, which is tailored to the hidden cancellation of the complex kernel. We begin with some preliminary results.

For any arc \(I\subseteq \T\), recall that the Carleson tent and the upper
Carleson tent associated with \(I\) are defined in \eqref{Carlesontent} and
\eqref{UpperCarlesontent}, respectively. We first introduce a truncated upper
Carleson tent. Let \(0<\varepsilon<1\) be sufficiently small, and set
\begin{equation} \label{20260601eq02}
Q_{I,\varepsilon}:=
\left\{
z\in\D:
\frac{z}{|z|}\in I_{\varepsilon},\ 
1-|I|<|z|<1-(1-\varepsilon)|I|
\right\},
\end{equation} 
where \(I_{\varepsilon}\subseteq I\) denotes the subarc concentric with \(I\)
and satisfying $|I_{\varepsilon}|=\varepsilon |I|$ (see Figure \ref{Fig12}). 

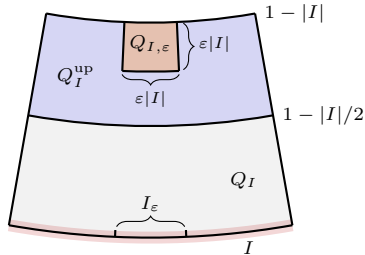
\begin{figure}[ht]
\centering
\begin{tikzpicture}[scale=1.8, every node/.style={font=\small}]

\begin{scope}[shift={(0,6)}]

\def\angL{260}
\def\angR{280}
\def\angLe{267.5}
\def\angRe{272.5}

\def\Rzero{6.00}
\def\Rmid{5.18}
\def\Reps{4.78}
\def\Rtop{4.42}


\fill[gray!35, opacity=0.28]
  (\angL:\Rtop)
  arc[start angle=\angL, end angle=\angR, radius=\Rtop]
  -- (\angR:\Rzero)
  arc[start angle=\angR, end angle=\angL, radius=\Rzero]
  -- cycle;

\fill[blue!55, opacity=0.22]
  (\angL:\Rtop)
  arc[start angle=\angL, end angle=\angR, radius=\Rtop]
  -- (\angR:\Rmid)
  arc[start angle=\angR, end angle=\angL, radius=\Rmid]
  -- cycle;

\fill[orange!75, opacity=0.38]
  (\angLe:\Rtop)
  arc[start angle=\angLe, end angle=\angRe, radius=\Rtop]
  -- (\angRe:\Reps)
  arc[start angle=\angRe, end angle=\angLe, radius=\Reps]
  -- cycle;

\draw[red!70!black, line width=4.5pt, opacity=0.16]
  (\angL:\Rzero) arc[start angle=\angL, end angle=\angR, radius=\Rzero];

\draw[thick]
  (\angL:\Rzero) arc[start angle=\angL, end angle=\angR, radius=\Rzero];

\draw[thick]
  (\angL:\Rmid) arc[start angle=\angL, end angle=\angR, radius=\Rmid];

\draw[thick]
  (\angLe:\Reps) arc[start angle=\angLe, end angle=\angRe, radius=\Reps];

\draw[thick]
  (\angL:\Rtop) arc[start angle=\angL, end angle=\angR, radius=\Rtop];

\draw[thick] (\angL:\Rzero) -- (\angL:\Rtop);
\draw[thick] (\angR:\Rzero) -- (\angR:\Rtop);

\draw[thick] (\angLe:\Reps) -- (\angLe:\Rtop);
\draw[thick] (\angRe:\Reps) -- (\angRe:\Rtop);

\node at (277:5.62) {\scriptsize $Q_I$};
\node at (263.5:4.85) {\scriptsize $Q_I^{\mathrm{up}}$};
\node at (270:4.58) {\scriptsize $Q_{I,\varepsilon}$};

\node[right] at (280:\Rtop) {\scriptsize $1-|I|$};
\node[right] at (280:\Rmid) {\scriptsize $1-|I|/2$};

\draw[thick] (\angLe:\Rzero) -- ++(\angLe:-0.05);
\draw[thick] (\angRe:\Rzero) -- ++(\angRe:-0.05);

\draw[
    decorate,
    decoration={brace, amplitude=4pt},
    thin
]
  (-0.25,-5.92) -- (0.25,-5.92);

\node at (0,-5.75) {\scriptsize $I_{\varepsilon}$};

\node at (0.72,-6.07) {\scriptsize $I$};
\node at (0.47,-4.58) {\tiny $\varepsilon |I|$};
\node at (0,-4.98) {\tiny $\varepsilon |I|$};

\draw[
    decorate,
    decoration={brace, amplitude=4pt},
    thin
]
  (0.24,-4.425) -- (0.24,-4.765);

  \draw[
    decorate,
    decoration={brace, amplitude=4pt},
    thin
]
  (0.21,-4.8) -- (-0.21,-4.8);

\end{scope}

\end{tikzpicture}
\caption{$I \subseteq \T$ and its associated Carleson tent $Q_I$, upper Carleson tent $Q_I^{\textrm{up}}$, and ``truncated" upper Carleson tent $Q_{I,\varepsilon}$.}
\label{Fig12}
\end{figure}

To this end, for \(N\geq 1\) sufficiently large, we denote by \(A_N\) the annulus
\begin{equation} \label{20260531eq02}
A_N:=\left\{z\in\D:\ 1-2^{-N}\leq |z|^2<1-2^{-N-1}\right\}
\end{equation} 
as usual. 

We have the following geometric lemma. 

\begin{lem}\label{20260530lem01}
Let $N\geq 1$ be sufficiently large. Let further, $0\leq k\leq N$ and $I\in\calD_k$.
Then, for every $\delta>0$, there exists $\varepsilon=\varepsilon(\delta)>0$, independent of the choices of $I$, $N$, and $k$, such that
$$
	\sup_{z\in A_N\cap Q_I, \; w,w'\in Q_{I,\varepsilon}}
	\left|
	\arg(1-z\overline{w})-\arg(1-z\overline{w'})
	\right|
	\leq \delta .
$$
\end{lem}

\begin{proof}
Using an argument similar to the one in \eqref{20260530eq01}, it is straightforward to verify that for any $z\in A_N\cap Q_I$ and $w,w'\in Q_{I,\varepsilon}$, the following estimates hold:
\begin{itemize}
	\item[(i)] $2^{-N-1}\leq 1-|z|^2\leq 2^{-N}$;
	\item[(ii)] $(1-\varepsilon)2^{-k}\leq 1-|w|^2\leq 2^{-k+1}$;
	\item[(iii)] $(1-\varepsilon)2^{-k}\leq |1-z\overline{w}|\leq C_1 2^{-k}$ for some constant $C_1>0$ independent of $k,N$, and $I$;
	\item[(iv)] $|w-w'|\leq C_2\varepsilon 2^{-k}$ for some constant $C_2>0$ independent of $k,N$, and $I$.
\end{itemize}
Hence,
$$
\left|\frac{1-z\overline{w}}{1-z\overline{w'}}-1\right|
=
\left|\frac{z\overline{(w'-w)}}{1-z\overline{w'}}\right|
\leq \frac{C_2\varepsilon}{1-\varepsilon}.
$$
Thus, by choosing $\varepsilon>0$ sufficiently small, we obtain the desired conclusion.
\end{proof}

Heuristically, Lemma~\ref{20260530lem01} says that, once $\varepsilon>0$ is chosen sufficiently small, the complex kernel \eqref{20260531eq01} is essentially constant as $w$ varies over $Q_{I,\varepsilon}$, for each fixed $z\in A_N\cap Q_I$. This suggests that the refined boxes $Q_{I,\varepsilon}$ should serve as the basic building blocks in the construction of our test function.

More precisely, let $\varepsilon>0$ be sufficiently small and $\eta>-1/2$, and assume that $a,b,c\in\R$ satisfy \eqref{conditionP1}. For any $I\in\calD$, define
\begin{equation}\label{20260531eq10}
	R_{I,\varepsilon,\eta}(z)
	:=
	(1-|z|^2)^a
	\int_{Q_{I,\varepsilon}}
	\frac{(1-|w|^2)^{c-2}}{(1-z\overline{w})^{c}}
	\left(1+\log \frac{1}{1-|w|^2}\right)^{\eta}
	dA(w), \qquad z \in \D. 
\end{equation}
Moreover, for sufficiently large $N\ge 1$, define the associated dyadic quadratic energy at level $N$ by
\begin{equation}\label{20260531eq11}
	H_{N,\varepsilon,\eta}(z)
	:=
	\sum_{1\leq k\leq N}
	\sum_{I\in\calD_k}
	|R_{I,\varepsilon,\eta}(z)|^2, \qquad z \in \D.
\end{equation}

\begin{lem} \label{20260531eq30}
Let $N \ge 1$ be sufficiently large, and let $A_N$ be defined as in \eqref{20260531eq02}. Then there exists $\varepsilon=\varepsilon(c)>0$, depending only on $c$, such that for any $I \in \calD_k$ with $1 \le k \le N$ and any $z \in A_N \cap Q_I$, one has
\begin{equation}\label{20260531eq03}
    |R_{I,\varepsilon,\eta}(z)|\gtrsim 2^{-Na} \varepsilon^2 k^\eta .
\end{equation}
Furthermore, for any $z \in A_N$,
\begin{equation}\label{20260531eq04}
    H_{N,\varepsilon,\eta}(z)\gtrsim 2^{-2Na} \varepsilon^4 N^{2\eta+1}.
\end{equation}
The implicit constants in \eqref{20260531eq03} and \eqref{20260531eq04} depend only on $a,c$, and $\eta$.
\end{lem}

\begin{proof}
We first prove \eqref{20260531eq03}. By Lemma \ref{20260530lem01}, there exists some $\varepsilon>0$ such that\footnote{Recall that by our assumption \eqref{conditionP1}, $c>1$.} 
$$
\sup_{z \in A_N \cap Q_I, \; w,w'\in Q_{I,\varepsilon}}|\arg(1-z\overline{w})-\arg(1-z\overline{w'})|\leq \frac{\pi}{4c}
$$
which implies for any $z \in A_N \cap Q_I$, 
\begin{align*}
	|R_{I,\varepsilon,\eta}(z)|
	&=(1-|z|^2)^a\left|\int_{Q_{I,\varepsilon}}\frac{(1-|w|^2)^{c-2}}{(1-z\overline{w})^{c}} \left(1+\log \frac{1}{1-|w|^2}\right)^{\eta} dA(w)\right|\\
	&\gtrsim (1-|z|^2)^a\int_{Q_{I,\varepsilon}}\frac{(1-|w|^2)^{c-2}}{|1-z\overline{w}|^c} \left(1+\log \frac{1}{1-|w|^2}\right)^{\eta} dA(w)\\
	&\gtrsim 2^{-Na}|Q_{I,\varepsilon}| |I|^{-2}k^{\eta}  \simeq 2^{-Na} \varepsilon^2 |I|^2 \cdot |I|^{-2}k^{\eta}\\
	&\gtrsim 2^{-Na}\varepsilon^2 k^{\eta}, 
\end{align*}
which gives \eqref{20260531eq03}. 

\vspace{0.1cm}

Next, we prove \eqref{20260531eq04}. Note that for each $z\in A_N$, there exists a unique dyadic arc $I_k(z)\in \calD_k$ such that $z\in Q_{I_k(z)}$ for $0 \le k \le N$. Therefore, by \eqref{20260531eq03}, we have
\begin{align*}
	H_{N, \varepsilon, \eta}(z)
	&=\sum_{I\in \calD_k, \; 1\leq k \leq N} |R_{I,\varepsilon,\eta}(z)|^2\\
	& \geq\sum_{1\leq k \leq N} |R_{I_k(z),\varepsilon,\eta}(z)|^2\\
    & \gtrsim 2^{-2Na} \varepsilon^4 \sum_{1\leq k \leq N}k^{2\eta}\\
	& \simeq 2^{-2Na} \varepsilon^4 N^{2\eta+1},
\end{align*}
where the last inequality follows from the assumption $\eta>-1/2$. This completes the proof of \eqref{20260531eq04}.
\end{proof}

\begin{rem}
We point out that the assumption $\eta>-1/2$ in Lemma~\ref{20260531eq30} is essential for the construction of our test functions. Indeed, if $\eta<-1/2$, then the sum $\sum_{k=1}^N k^{2\eta}$ remains uniformly bounded in $N$, and hence the argument above yields at most the lower bound
$$
H_{N,\varepsilon,\eta}(z)\gtrsim_{\varepsilon} 2^{-2Na}, \qquad z\in A_N.
$$
This estimate is not strong enough to produce a family of test functions leading to the failure of the weak-type estimate
$$
T_{a,b,c}: L^{\frac{1}{2-c+b}}(\D)\to L^{-\frac{1}{a},\infty}(\D).
$$
In fact, in this range the corresponding weak-type estimate does hold; see Theorem~\ref{Forelli--Rudin threshold}, (iii).
\end{rem}

\subsection{Proof of Theorem \ref{Forelli--Rudin threshold}, (ii): Failure of the weak-type estimates for $T_{a, b, c}$ at the intersection $\calP_{\textnormal{C-H}}$ when $0 \le 2-c+b<1/2$} \label{20260603sec01}

We first consider the case when $2-c+b>0$. It suffices to find a sequence of functions $\{f_N\}_{N\in \N}$ such that 
\begin{equation} \label{20260601eq20}
\sup_{N}\|f_N\|_{L^{\frac{1}{2-c+b}}(\D)}<\infty \qquad \textrm{while} \qquad  
\lim_{N\to\infty}\|T_{a,b,c}f_N\|_{L^{-\frac{1}{a},\infty}(\D)}=\infty.
\end{equation} 
Since $0<2-c+b<1/2$, we can choose $\tau>0$ sufficiently small so that
\begin{equation} \label{20260601eq01}
\eta:=(c-2-b)(1+\tau)>-\frac{1}{2}.
\end{equation} 
For each $N \ge 1$, let
$$
	\calI_N:=\bigcup\limits_{1 \le k \leq N} \calD_k
$$
be the collection of all dyadic arcs of generations $1,\ldots,N$, and 
$$
\Delta_N:=\{-1,1\}^{\calI_N}
$$
be the collection of all binary sign sequences of length\footnote{Here and in what follows, $\#E$ denotes the cardinality of a finite set $E$.} $\# \calI_N$, equipped with the uniform probability measure.
For each $\delta\in\Delta_N$, write
$$
\delta=(\delta(I))_{I\in\calI_N},
$$
where the coordinates $\delta(I)\in\{-1,1\}$ are i.i.d. Rademacher random variables with respect to the uniform probability measure on $\Delta_N$.

\vspace{0.1cm}

We are now ready to define the test functions. In what follows, we may assume that $N\geq 1$ is sufficiently large. Fix $\varepsilon>0$ as in Lemma~\ref{20260531eq30}, and choose $\eta>-1/2$ as in \eqref{20260601eq01}. For each $
\delta \in\Delta_N$, define
\begin{equation} \label{20260603eq20}
f_{N,\delta,\varepsilon,\eta}(z)
:=\sum_{I\in\calI_N}
\delta(I)\,\one_{Q_{I,\varepsilon}}(z)
(1-|z|^2)^{c-2-b}
\left(1+\log \frac{1}{1-|z|^2}\right)^{\eta},
\end{equation} 
where $Q_{I,\varepsilon}$ is defined in \eqref{20260601eq02}. Our goal is to show that, for every sufficiently large $N\geq 1$, one can choose
$\del_N\in\Delta_N$ such that the resulting sequence
\begin{equation} \label{20260601goal01}
\{f_{N,\del_N,\varepsilon,\eta}\}
\end{equation} 
satisfies \eqref{20260601eq20}. Let us now turn to the details. 

\vspace{0.1cm}

On one hand, by the mutual disjointness of the collection $\{Q_{I, \varepsilon}\}_{I \in \calI_N}$, one has
\begin{align*}
	\|f_{N,\delta, \varepsilon, \eta}\|_{L^{\frac{1}{2-c+b}}(\D)}^{\frac{1}{2-c+b}} 
	& \leq\int_{\D} (1-|z|^2)^{-1} \left(1+\log \frac{1}{1-|z|^2}\right)^{-(1+\tau)} dA(z)\\
	& \lesssim \int_0^1 r (1-r^2)^{-1} \left(1+\log \frac{1}{1-r^2}\right)^{-(1+\tau)} dr\\
	& <\infty.
\end{align*}
Hence, 
\begin{equation} \label{20260601eq98}
\sup_{N \ge 1} \sup_{\delta \in \Delta_N}\|f_{N,\delta, \varepsilon, \eta}\|_{L^{\frac{1}{2-c+b}}(\D)}<\infty, 
\end{equation} 
which gives the first assertion in \eqref{20260601eq20}. 

On the other hand, a direct computation gives
\begin{align*}
(T_{a,b,c}f_{N,\delta, \varepsilon, \eta})(z)
&=(1 - |z|^2) ^{a}\int_{\D} \frac {(1-|w|^{2}) ^{b}}{(1 - z\overline{w}) ^{c}} f_{N,\delta, \varepsilon, \eta} (w) dA(w)\\
&=\sum_{1\leq k \leq N}\sum_{I\in \calD_k}\delta(I) (1-|z|^2)^a \int_{Q_{I,\varepsilon}}
	\frac{(1-|w|^2)^{c-2}}{(1-z\overline{w})^{c}}
	\left(1+\log \frac{1}{1-|w|^2}\right)^{\eta}
	dA(w) \\
&= \sum_{I \in \calI_N} \delta(I) R_{I, \varepsilon, \eta}(z),   
\end{align*}
where $R_{I,\varepsilon,\eta}$ is defined in \eqref{20260531eq10}. Since $\{\delta(I)\}_{I\in\calI_N}$ are independent Rademacher random variables, orthogonality gives
$$
\mathbb{E}\left[\left| \left(T_{a,b,c}f_{N,\delta,\varepsilon,\eta}\right)(z)\right|^2\right]
=\sum_{I\in\calI_N}\left|R_{I,\varepsilon,\eta}(z)\right|^2
=H_{N,\varepsilon,\eta}(z), \qquad z \in \D, 
$$
where $H_{N,\varepsilon,\eta}$ denotes the dyadic quadratic energy at level $N$ defined in \eqref{20260531eq11}. Moreover, by Khintchine's inequality, we have 
\begin{equation} \label{20260601eq54}
\mathbb{E}\left[ \left| \left(T_{a,b,c}f_{N,\delta, \varepsilon, \eta}\right)(z) \right|^4\right]\ \simeq  (H_{N, \varepsilon, \eta}(z))^2, \qquad z \in \D. 
\end{equation} 
From now on, for each sufficiently large $N\geq 1$, we restrict our attention to $z\in A_N$ and define
$$
G_{N, \varepsilon, \eta}(z):=\left\{
\delta\in\Delta_N:
\left|(T_{a,b,c}f_{N,\delta,\varepsilon,\eta})(z)\right|^2
\geq
\frac{1}{10}H_{N,\varepsilon,\eta}(z)
\right\}.
$$
Observe that
\begin{align*} 
H_{N, \varepsilon, \eta}(z)
&= \mathbb{E}\left[|(T_{a,b,c}f_{N,\delta, \varepsilon, \eta})(z)|^2\right] \nonumber \\
&=\frac{1}{\#(\Delta_N)}\sum_{\delta\in\Delta_N}\left|\sum_{I\in \calI_N}\delta(I) R_{I,\varepsilon,\eta}(z)\right|^2 \nonumber \\
&=\frac{1}{\#(\Delta_N)}\left(\sum_{\delta\in G_{N, \varepsilon, \eta}(z)}+\sum_{\delta\in\Delta_N\setminus G_{N, \varepsilon, \eta}(z)}\right)\left|\sum_{I\in\calI_N}\delta(I) R_{I,\varepsilon,\eta}(z)\right|^2 \nonumber \\
& \le\frac{1}{\#(\Delta_N)}\sum_{\delta\in G_{N, \varepsilon, \eta}(z)}\left|\sum_{I\in\calI_N}\delta(I) R_{I,\varepsilon,\eta}(z)\right|^2+\frac{1}{10\#(\Delta_N)}\sum_{\delta\in\Delta_N\setminus G_{N, \varepsilon, \eta}(z)}  H_{N, \varepsilon, \eta}(z) \nonumber\\
& \le\frac{1}{\#(\Delta_N)}\sum_{\delta\in G_{N, \varepsilon, \eta}(z)}\left|\sum_{I\in\calI_N}\delta(I) R_{I,\varepsilon,\eta}(z)\right|^2+ \frac{ H_{N, \varepsilon, \eta}(z)}{10},
\end{align*}
which implies
\begin{equation} \label{20260601eq21}
H_{N, \varepsilon, \eta}(z)
\lesssim \frac{1}{\#(\Delta_N)}\sum_{\delta\in G_{N, \varepsilon, \eta}(z)}\left|\sum_{I\in\calI_N}\delta(I) R_{I,\varepsilon,\eta}(z)\right|^2.
\end{equation} 
Here, we have implicitly used the fact that $H_{N,\varepsilon,\eta}(z)<+\infty$ for every $z\in A_N$, which follows from \eqref{20260531eq11}, where $H_{N,\varepsilon,\eta}(z)$ is defined as a finite sum. As a consequence of \eqref{20260601eq21}, we have the following: 

\medskip 

\noindent \textit{\underline{Claim I}: For each $z \in A_N$, $G_{N, \varepsilon, \eta}(z) \neq \emptyset$.}

\begin{proof} [Proof of Claim I]
Assume $G_{N, \varepsilon, \eta}(z)=\emptyset$ for some $z \in A_N$. Then by \eqref{20260601eq21}, $H_{N, \varepsilon, \eta}(z)=0$, but this is a contradiction to \eqref{20260531eq04}. 
\end{proof}

\vspace{0.1cm}

\noindent \textit{\underline{Claim II}: For any $z\in A_N$, one has
\begin{equation}\label{lower bounded of G_N}
    \#G_{N, \varepsilon, \eta}(z)\gtrsim \#\Delta_N.
\end{equation}
}

\begin{proof}[Proof of Claim II]
By \eqref{20260601eq21}, Cauchy--Schwarz, and \eqref{20260601eq54}, we have
\begin{align*}
H_{N,\varepsilon,\eta}(z)
&\lesssim
\left(\frac{\#G_{N, \varepsilon, \eta}(z)}{\#\Delta_N}\right)^{1/2}
\mathbb{E}\left[
\left|
\left(T_{a,b,c}f_{N,\delta,\varepsilon,\eta}\right)(z)
\right|^4
\right]^{1/2} \\
&\lesssim
\left(\frac{\#G_{N, \varepsilon, \eta}(z)}{\#\Delta_N}\right)^{1/2}
H_{N,\varepsilon,\eta}(z).
\end{align*}
Since $0<H_{N,\varepsilon,\eta}(z)<+\infty$ for $z\in A_N$, the desired estimate
\eqref{lower bounded of G_N} follows immediately.
\end{proof}
Note that, heuristically, \eqref{lower bounded of G_N} asserts that for each $z\in A_N$, there are ``lots of'' choices of $\delta\in\Delta_N$ such that $
\left|(T_{a,b,c}f_{N,\delta,\varepsilon,\eta})(z)\right|^2$ is comparable to the dyadic quadratic energy $H_{N,\varepsilon,\eta}(z)$ at $z$. However, \eqref{lower bounded of G_N} only guarantees the existence of such choices of $\delta$ after $z$ has been fixed; in other words, the selected $\delta$ may depend on $z$. In contrast, for our purpose we need to choose $\delta$ depending only on $N$; see \eqref{20260601goal01}. This naturally motivates the next step: we shall show that there exists a large portion of $A_N$ on which a common choice of $\delta\in\Delta_N$ works simultaneously. We achieve this by the following Fubini-type argument. 

\vspace{0.1cm}

For each $\delta\in\Delta_N$, define
\begin{equation}\label{E_N,delta}
	E_{N, \delta, \varepsilon, \eta}
	=\left\{z\in A_N: \delta\in G_{N, \varepsilon, \eta}(z) \right\}
	=\left\{z\in A_N: |(T_{a,b,c}f_{N,\delta, \varepsilon, \eta})(z)|^2\geq \frac{H_{N, \varepsilon, \eta}(z)}{10} \right\}.
\end{equation}
By \eqref{lower bounded of G_N}, we have
\begin{align*}
	\mathbb{E}[|E_{N, \delta, \varepsilon, \eta}|]
	&=\frac{1}{\#(\Delta_N)}\sum_{\delta\in\Delta_N}\int_{E_{N, \delta, \varepsilon, \eta}} dA(z)\\
	&=\frac{1}{\#(\Delta_N)}\sum_{\delta\in\Delta_N}\int_{A_N} \one_{E_{N, \delta, \varepsilon, \eta}}(z) dA(z)\\
    &=\frac{1}{\#(\Delta_N)}\int_{A_N} \left( \sum_{\delta\in\Delta_N} \one_{E_{N, \delta, \varepsilon, \eta}}(z) \right) dA(z)\\
	&=\frac{1}{\#(\Delta_N)}\int_{A_N} \#\{\delta\in\Delta_N: z\in E_{N, \delta, \varepsilon, \eta}\} dA(z)\\
	&=\frac{1}{\#(\Delta_N)}\int_{A_N} \#\{\delta\in\Delta_N: \delta\in G_{N, \varepsilon, \eta}(z)\} dA(z)\\
	&=\frac{1}{\#(\Delta_N)}\int_{A_N} \# (G_{N, \varepsilon, \eta}(z)) dA(z) \\
    &\simeq 2^{-N}.
\end{align*}
Thus, for every sufficiently large $N\geq 1$, there exists $\del_N\in\Delta_N$ such that
$$
\left|E_{N,\del_N,\varepsilon,\eta}\right|\gtrsim 2^{-N}.
$$
Moreover, for every $z\in E_{N,\del_N,\varepsilon,\eta}$, it follows from
\eqref{E_N,delta} and \eqref{20260531eq04} that
$$
\left|(T_{a,b,c}f_{N,\del_N,\varepsilon,\eta})(z)\right|
\gtrsim
H_{N,\varepsilon,\eta}(z)^{1/2}
\gtrsim
2^{-Na}N^{\eta+\frac{1}{2}}.
$$
Combining the above estimates, we obtain
\begin{align} \label{20260601eq13}
\left\|T_{a,b,c}f_{N,\del_N,\varepsilon,\eta}\right\|_{L^{-\frac{1}{a},\infty}(\D)}
&\gtrsim 2^{-Na}N^{\eta+\frac{1}{2}} \left|\left\{z\in\D:\left|\left(T_{a,b,c}f_{N,\del_N,\varepsilon,\eta}\right)(z)\right|
\gtrsim 2^{-Na}N^{\eta+\frac{1}{2}}\right\}\right|^{-a} \nonumber \\
&\gtrsim 2^{-Na}N^{\eta+\frac{1}{2}}
\left|E_{N,\del_N,\varepsilon,\eta}\right|^{-a} \nonumber \\
&\gtrsim N^{\eta+\frac{1}{2}}.
\end{align}
Since $\eta>-1/2$ by \eqref{20260601eq01}, this implies the second assertion in \eqref{20260601eq20}, namely,
$$
\lim_{N\to\infty}
\left\|T_{a,b,c}f_{N,\del_N,\varepsilon,\eta}\right\|_{L^{-\frac{1}{a},\infty}(\D)}
=+\infty.
$$
This proves the desired unboundedness of
$T_{a,b,c}:L^{\frac{1}{2-c+b}}(\D)\to L^{-\frac{1}{a},\infty}(\D)$, under the assumption $0<2-c+b<1/2$.

\vspace{0.1cm}

To this end, we treat the case $2-c+b=0$ as a simple modification of the preceding argument. More precisely, one only needs to replace the space $L^{\frac{1}{2-c+b}}(\D)$ by $L^\infty(\D)$ and take $\eta=0$ in the above construction. The remaining details are left to the interested reader.

\vspace{0.1cm}

The proof of Theorem \ref{Forelli--Rudin threshold}, (ii) is complete. 

\medskip 

\subsection{Proof of Theorem \ref{Forelli--Rudin threshold}, (iii): The weak-type estimates for $T_{a, b, c}$ at the intersection $\calP_{\textnormal{C-H}}$ when $1/2 \le 2-c+b <1$}
\label{20260831subsec10}

Recall that, in the proof of Theorem~\ref{Forelli--Rudin threshold}, (ii), our probabilistic construction fails to work at the threshold $2-c+b=1/2$; see \eqref{20260601eq13}. It turns out that this is not a technical obstacle, but rather reflects an intrinsic complex-analytic feature of $T_{a,b,c}$. Indeed, when $2-c+b\geq 1/2$, tools from complex function theory begin to play a decisive role, and they eventually lead to the weak-type estimates for $T_{a,b,c}$ in this regime. For this purpose, we first recall a few standard definitions and facts from complex function theory.

For $0<p<\infty$, the analytic Hardy space $H^p(\D)$ consists of all holomorphic functions $h$ on $\D$ such that
$$
\|h\|_{H^p(\D)}:=\sup_{0<r<1}
\left(\int_0^{2\pi}|h(re^{i\theta})|^p\,\frac{d\theta}{2\pi}\right)^{1/p}
<+\infty.
$$
When $1<p<\infty$, we shall use the standard duality
\begin{equation} \label{Hardyduality}
\|h\|_{H^p(\D)}\simeq \sup_{\|g\|_{H^{p'}(\D)}=1} \left|\langle h,g\rangle_{\T}\right|,
\qquad \textnormal{where} \quad \frac1p+\frac1{p'}=1.
\end{equation} 

Moreover, for $\alpha>-1$ and $0<p<\infty$, the weighted Bergman space $A^p_\alpha(\D)$ consists of all holomorphic functions $h$ on $\D$ such that
$$
\|h\|_{A^p_\alpha(\D)}
:=
\left(
\int_{\D}|h(z)|^p(1-|z|^2)^\alpha\,dA(z)
\right)^{1/p}
<+\infty.
$$

\vspace{0.1cm}

Returning to the proof of Theorem~\ref{Forelli--Rudin threshold}, (iii), we note that the assumption $c-2-a-b>0$, together with $2-c+b\geq 1/2$, gives
$$
\frac{1}{2}\leq 2-c+b<-a<1.
$$
Hence
$$
1<\frac{1}{2-c+b}\leq 2,
\qquad \text{and} \qquad
2\leq \frac{1}{c-1-b}<+\infty.
$$
To this end, write
$$
T_{a,b,c}f(z)=(1-|z|^2)^aT_{0,b,c}f(z), \qquad z \in \D. 
$$
We divide the rest of the proof into two steps.

\vspace{0.1cm}

\noindent\textbf{$\diamond$    \textsl{Step 1.}}
We first show that 
\begin{equation} \label{20260602eq01}
\|T_{0,b,c}f\|_{H^{\frac{1}{2-c+b}}(\D)}
\lesssim \|f\|_{L^{\frac{1}{2-c+b}}(\D)}.
\end{equation} 
Recall that
$$
T_{0,b,c}f(z)=\int_{\D}\frac{(1-|w|^2)^b}{(1-z\overline{w})^c}f(w)\,dA(w).
$$
Since
$$
\frac{1}{(1-z\overline{w})^c}
=
\sum_{m=0}^{\infty}\frac{\Gamma(m+c)}{\Gamma(c)m!}z^m\overline{w}^{\,m},
$$
we have
$$
T_{0,b,c}f(z)
=
\sum_{m=0}^{\infty}
\frac{\Gamma(m+c)}{\Gamma(c)m!}
\left(\int_{\D}(1-|w|^2)^bf(w)\overline{w}^{\,m}\,dA(w)\right)z^m.
$$
Set
$$
A_m:=
\frac{\Gamma(m+c)}{\Gamma(c)m!}
\int_{\D}(1-|w|^2)^bf(w)\overline{w}^{\,m}\,dA(w).
$$
Take $g\in H^{\frac{1}{c-1-b}}(\D)$ with
$$
g(z)=\sum_{m=0}^{\infty}B_mz^m.
$$
Then
\begin{equation} \label{20260602eq02}
\langle T_{0,b,c}f,g\rangle_{\T}=
\sum_{m=0}^{\infty}A_m\overline{B_m} =
\int_{\D}(1-|w|^2)^bf(w)
\overline{\left(\sum_{m=0}^{\infty}
\frac{\Gamma(m+c)}{\Gamma(c)m!}B_mw^m\right)}\,dA(w).
\end{equation}

Recall the fractional radial derivative $R^{\alpha,t}$ on the space $\mathbb H(\D)$; see \cite[p.~18]{Zhu2005}. More precisely, if
$$
h(z)=\sum_{m=0}^{\infty}a_mz^m
$$
is the Taylor expansion of $h$, then for real parameters $\alpha$ and $t$ such that neither $1+\alpha$ nor $1+\alpha+t$ is a negative integer, $R^{\alpha,t}$ is defined by
$$
R^{\alpha,t}h(z)
=
\sum_{m=0}^{\infty}
\frac{\Gamma(2+\alpha)\Gamma(2+m+\alpha+t)}
{\Gamma(2+\alpha+t)\Gamma(2+m+\alpha)}
a_mz^m.
$$
In particular, taking $\alpha=-1$ and $t=c-1$, and using the assumption $c>1$, we obtain
$$
R^{-1,c-1}h(z)
=\sum_{m=0}^{\infty}
\frac{\Gamma(m+c)}{\Gamma(c)m!}a_mz^m.
$$
Therefore, \eqref{20260602eq02} can be written as
$$
\langle T_{0,b,c}f,g\rangle_{\T}
=
\int_{\D}(1-|w|^2)^bf(w)\overline{R^{-1,c-1}g(w)}\,dA(w).
$$
By H\"older's inequality with exponents $\frac{1}{2-c+b}$ and $\frac{1}{c-1-b}$, we obtain
\begin{align*}
\left|\langle T_{0,b,c}f,g\rangle_{\T}\right|
&\leq
\|f\|_{L^{\frac{1}{2-c+b}}(\D)}
\left(\int_{\D}|R^{-1,c-1}g(w)|^{\frac{1}{c-1-b}}
(1-|w|^2)^{\frac{b}{c-1-b}}\,dA(w)\right)^{c-1-b}.
\end{align*}
Since $c>1$, we have
$$
\frac{b}{c-1-b}>-1.
$$
This, together with the fact that $1/(c-1-b)\ge  2$, allows us to apply \cite[Theorem~4.41]{Zhu2005}, yielding
$$
\|R^{-1,c-1}g\|_{A^{\frac{1}{c-1-b}}_{\frac{b}{c-1-b}}(\D)}
\lesssim
\|g\|_{H^{\frac{1}{c-1-b}}(\D)}.
$$
Consequently,
$$
\left|\langle T_{0,b,c}f,g\rangle_{\T}\right|
\lesssim \|f\|_{L^{\frac{1}{2-c+b}}(\D)}
\|g\|_{H^{\frac{1}{c-1-b}}(\D)}.
$$
Taking the supremum over all $g\in H^{\frac{1}{c-1-b}}(\D)$ with
$\|g\|_{H^{\frac{1}{c-1-b}}}=1$, we obtain
$$
\|T_{0,b,c}f\|_{H^{\frac{1}{2-c+b}}(\D)}
\lesssim \|f\|_{L^{\frac{1}{2-c+b}}(\D)}.
$$
This finishes the proof of \textbf{\textsl{Step 1}}. 

\medskip 

\noindent\textbf{$\diamond$    \textsl{Step 2.}}
Next, we prove
$$
\left\|(1-|\cdot|^2)^aT_{0,b,c}f\right\|_{L^{-\frac{1}{a},\infty}(\D)}
\lesssim
\left\|T_{0,b,c}f\right\|_{H^{\frac{1}{2-c+b}}(\D)}.
$$
Since $c-2-a-b>0$, we have
$$
-a>2-c+b,
$$
and therefore
$$
-\frac1a<\frac{1}{2-c+b}.
$$
The \emph{key} point is the following weak-type estimate for Hardy functions.

\begin{lem} \label{weakHardy}
Let $0<q<p<\infty$. For each $h\in H^p(\D)$, define
$$
M_qh(z):=(1-|z|^2)^{-\frac1q}h(z).
$$
Then one has
$$
\|M_qh\|_{L^{q,\infty}(\D)}
\lesssim \|h\|_{H^p(\D)}.
$$
\end{lem}

\begin{proof}
For $\lambda>0$, set
$$
E_\lambda:=\{z\in\D: |M_qh(z)|>\lambda\}
=
\left\{z\in\D: |h(z)|>\lambda(1-|z|^2)^{\frac1q}\right\}.
$$
If $h\equiv 0$, there is nothing to prove. Otherwise, put
$$
A:=\frac{\|h\|_{H^p(\D)}}{\lambda}.
$$
For $0<r<1$, Chebyshev's inequality gives
\begin{align*}
\left|\left\{e^{i\theta}\in\T: |h(re^{i\theta})|>\lambda(1-r^2)^{\frac1q}\right\}\right|
& \leq \min\left\{1, \; \frac{\int_0^{2\pi} |h(re^{i\theta})|^p \frac{d\theta}{2\pi}}{\lambda^p (1-r^2)^{\frac{p}{q}}} \right\} \\
& \le \min \left\{1, \; \frac{\|h\|_{H^p}^p}{\lambda^p (1-r^2)^{\frac{p}{q}}} \right\} \\
& =\min\left\{1, \; A^p(1-r^2)^{-\frac{p}{q}}\right\}.
\end{align*}
Hence, by polar coordinates,
$$
|E_\lambda|
\lesssim 
\int_0^1 \min\left\{1,A^p(1-r^2)^{-\frac{p}{q}}\right\}r\,dr.
$$
If $A\geq 1$, then
$$
\lambda^q|E_\lambda|\lesssim \lambda^q\leq \|h\|_{H^p(\D)}^q.
$$
If $0<A<1$, we split the integral into two parts and obtain
$$
|E_\lambda|\lesssim 
A^p\int_0^{\sqrt{1-A^q}}(1-r^2)^{-\frac{p}{q}}r\,dr+\int_{\sqrt{1-A^q}}^1 r\,dr \le \frac{p}{2(p-q)}A^q.
$$
Therefore, in both cases,
$$
\lambda^q|E_\lambda|\lesssim \|h\|_{H^p(\D)}^q,
$$
and the desired claim follows by taking the supremum over $\lambda>0$.
\end{proof}

To this end, applying the lemma with
$$
p=\frac{1}{2-c+b},
\qquad
q=-\frac1a,
$$
we get
$$
\|(1-| \cdot |^2)^aT_{0,b,c}f\|_{L^{-\frac{1}{a},\infty}(\D)}
\lesssim \|T_{0,b,c}f\|_{H^{\frac{1}{2-c+b}}(\D)}.
$$
This completes the proof of \textbf{    \textsl{Step 2}}. 

\medskip 

Combining \textbf{ \textsl{Steps 1}} and \textbf{\textsl{2}}, we obtain Theorem~\ref{Forelli--Rudin threshold}, (iii).

\bigskip 

\section{Analysis of the Main Case: Restricted Weak-type Bounds at the Critical--Horizontal Intersection $\calP_{\textnormal{C-H}}$ away from the vertical cut-off line} \label{Sec06}

In Section~\ref{Sec05}, we studied the weak--type behavior of $T_{a,b,c}$ and $S_{a,b,c}$ at the critical--horizontal intersection $\calP_{\textnormal{C-H}}$, which lies strictly to the left of the vertical line \eqref{verticalcut}. A natural next step is to study the restricted weak--type estimates at the same endpoint. In this setting, the Forelli--Rudin phenomenon disappears: the operators $T_{a,b,c}$ and $S_{a,b,c}$ have the same restricted weak--type behavior at $\calP_{\textnormal{C-H}}$. Nevertheless, the borderline case $2-c+b=0$ remains singular, as it separates the bounded and unbounded restricted weak--type regimes. The main result of this section is the following.

\begin{thm} \label{restricted-weak Forelli--Rudin at CH}
Let $a,b,c\in\R$ satisfy \eqref{conditionP1}. Then the following statements hold.

\begin{enumerate}
    \item[(i)] If $2-c+b=0$, then the restricted weak--type estimates fail for both Forelli--Rudin operators $T_{a,b,c}$ and $S_{a,b,c}$ at the intersection $\calP_{\textnormal{C-H}}$.

    \item[(ii)] If $0<2-c+b<1$, then the Forelli--Rudin operators $T_{a,b,c}$ and $S_{a,b,c}$ are bounded from
    $L^{\frac{1}{2-c+b},1}(\D)$ to $L^{-\frac{1}{a},\infty}(\D)$.
\end{enumerate}
\end{thm}

\subsection{Proof of Theorem \ref{restricted-weak Forelli--Rudin at CH}, (i): Failure of the restricted weak-type estimates for $T_{a, b, c}$ and $S_{a, b, c}$ at the intersection $\calP_{\textnormal{C-H}}$ when $2-c+b=0$} \label{20260603sec02}

It is clear that it suffices to prove Theorem~\ref{restricted-weak Forelli--Rudin at CH}, (i), for $T_{a,b,c}$. We argue by contradiction. Assume that the desired conclusion fails for $T_{a,b,c}$, namely,
\begin{equation}\label{20260603eq02}
\sup_{E \subseteq \D}
\left\|T_{a,b,c}\one_E\right\|_{L^{-\frac1a,\infty}(\D)}
<+\infty.
\end{equation}
The argument here is indeed a consequence of the probabilistic construction used in the proof of Theorem~\ref{Forelli--Rudin threshold}, (ii).
Recall that, in Section~\ref{20260603sec01}, we constructed a sequence of test functions
$$
f_{N,\delta_N,\varepsilon,0}
:=
\sum_{I\in \calI_N}\delta_N(I)\one_{Q_{I,\varepsilon}},
\qquad \delta_N(I)\in\{-1,1\},
$$
(see \eqref{20260603eq20}). Since $c-2-b=0$, $\eta=(c-2-b)(1+\tau)=0$. Therefore, by the disjointness of the boxes $Q_{I,\varepsilon}$ in the construction, we may write
\begin{equation} \label{20260603eq30}
f_{N,\delta_N,\varepsilon,0}=
\one_{F_{N,\varepsilon}^+}-\one_{F_{N,\varepsilon}^-},
\end{equation} 
where
$$
F_{N,\varepsilon}^+:=\bigcup_{I\in \calI_N, \; \delta_N(I)=1} Q_{I,\varepsilon},
\qquad
F_{N,\varepsilon}^-:=\bigcup_{I\in \calI_N, \; \delta_N(I)=-1} Q_{I,\varepsilon}.
$$
Hence, using \eqref{20260528eq01} and the fact that $0<-a<1$, one has
$$
\begin{aligned}
\left\|T_{a,b,c} f_{N,\delta_N,\varepsilon,0} \right\|_{L^{-\frac{1}{a}, \infty}(\D)}
&\lesssim 
\left\|T_{a,b,c}\one_{F_{N,\varepsilon}^+}\right\|_{L^{-\frac{1}{a},\infty}(\D)}
+
\left\|T_{a,b,c}\one_{F_{N,\varepsilon}^-}\right\|_{L^{-\frac{1}{a},\infty}(\D)} \lesssim 1,
\end{aligned}
$$
where the last estimate follows from the contradiction assumption \eqref{20260603eq02}. This contradicts \eqref{20260601eq13} with $\eta=0$. The proof of Theorem~\ref{restricted-weak Forelli--Rudin at CH}, (i), is complete.

\begin{rem} 
We point out that the argument above applies only to the borderline case $2-c+b=0$. When $0<2-c+b<1/2$, each $Q_{I,\varepsilon}$ in the construction is associated with the weight 
$$ (1-|z|^2)^{c-2-b} \left(1+\log \frac{1}{1-|z|^2}\right)^{\eta}, $$ 
where $\eta>-1/2$ is defined as in \eqref{20260601eq01}. 

Consequently, the test functions $f_{N,\delta_N,\varepsilon,\eta}$ can no longer be decomposed as a difference of two characteristic functions, as in \eqref{20260603eq30}. This is not merely a technical obstacle: when $0<2-c+b<1$, we shall prove that the desired restricted weak--type estimate holds for both $T_{a,b,c}$ and $S_{a,b,c}$. \end{rem}

\subsection{Proof of Theorem \ref{restricted-weak Forelli--Rudin at CH}, (ii): Restricted weak-type estimates for $T_{a, b, c}$ and $S_{a, b, c}$ at the intersection $\calP_{\textnormal{C-H}}$ when $0<2-c+b<1$}

This section is devoted to the proof of Theorem~\ref{restricted-weak Forelli--Rudin at CH}, (ii). Clearly, we only need to prove the desired restricted weak--type estimate for $S_{a,b,c}$. The proof relies on two key ingredients:
\begin{enumerate}
    \item a refined sparse domination principle, tailored to handle the factor $(1-|z|^2)^a$ in the definition of $S_{a,b,c}$;
    \item a dyadic Carleson embedding argument over Carleson tents, inspired by the seminal work of Nazarov, Treil, and Volberg~\cite{NTV2003} on the $Tb$ theorem for non--homogeneous spaces.
\end{enumerate}

\vspace{0.1cm}

Now we turn to the details. 

\medskip 

\noindent\textbf{$\diamond$ \textsl{Step 1.}} We first reduce the estimate to the ``diagonal'' case by an argument similar to that in Proposition~\ref{20260527prop02}. Set $\eta:=c-2-a-b>0$ and
$W_\eta(z):=(1-|z|^2)^{-\eta}$.
Recall that $W_\eta\in L^{\frac{1}{\eta},\infty}(\D)$.
Since $a+\eta=c-2-b$, we may write
$$
	S_{a,b,c}f(z)
	=
	W_\eta(z)S_{a+\eta,b,c}f(z)
	=
	W_\eta(z)S_{c-2-b,b,c}f(z),
$$
which, by Lorentz--H\"older's inequality, gives
$$
	\|S_{a,b,c}f\|_{L^{-\frac{1}{a},\infty}(\D)}
	\lesssim
	\|W_\eta\|_{L^{\frac{1}{\eta},\infty}(\D)}
	\|S_{c-2-b,b,c}f\|_{L^{\frac{1}{2-c+b},\infty}(\D)}.
$$
Therefore, it suffices to show that
$$
\|S_{c-2-b,b,c}f\|_{L^{\frac{1}{2-c+b},\infty}(\D)} \lesssim \left\|f \right\|_{L^{\frac{1}{2-c+b}, 1}(\D)},
$$ 
which is equivalent to proving that 
\begin{equation} \label{simplified case of restricted weak type A}
\left\|S_{c-2-b, b, c} \one_F \right\|_{L^{\frac{1}{2-c+b}, \infty}(\D)} \lesssim |F|^{2-c+b}
\end{equation}
for every measurable set $F \subseteq \D$. 

\medskip 

\noindent\textbf{$\diamond$    \textsl{Step 2.}} Here, we prove the following slightly stronger result.

\begin{lem}\label{simple function}
Let $E,F\subseteq \D$ be measurable sets. Then
\begin{equation} \label{20260605eq01}
	\int_E |S_{c-2-b,b,c}\one_F(z)|\,dA(z)
	\lesssim |E|^{c-1-b}|F|^{2-c+b}.
\end{equation} 
\end{lem}

We temporarily assume Lemma~\ref{simple function}. For $\lambda>0$, set
$$
	E_\lambda
	:=
	\{z\in\D:\ |S_{c-2-b,b,c}\one_F(z)|>\lambda\}.
$$
Then, by Lemma~\ref{simple function},
$$
	\lambda |E_\lambda|
	\leq
	\int_{E_\lambda} |S_{c-2-b,b,c}\one_F(z)|\,dA(z)
	\lesssim
	|E_\lambda|^{c-1-b}|F|^{2-c+b}.
$$
Hence
$$
	\lambda |E_\lambda|^{2-c+b}
	\lesssim
	|F|^{2-c+b},
$$
which gives \eqref{simplified case of restricted weak type A}. Therefore, it remains to prove Lemma~\ref{simple function}.

\medskip 

\noindent\textbf{$\diamond$    \textsl{Step 3.}} We begin by estimating the left-hand side of \eqref{20260605eq01} through a refined sparse domination argument (over Carleson tents). More precisely, let $\calD^{(1)}$ and $\calD^{(2)}$ be a pair of adjacent dyadic systems\footnote{For instance, one may take $\calD^{(1)}$ to be the standard dyadic system on $\T$, and $\calD^{(2)}$ to be the $1/3$-shifted dyadic system. This reduction goes back to the work of Garnett, Jones, and Mei on BMO and its dyadic analog (see \cite{GJ1982, TM2003}). } on $\T$. Then, by a standard argument (see, e.g., \cite{RTW2017}), for any $z,w\in\D$, there exist $i\in\{1,2\}$ and $J\in\calD^{(i)}$ such that
\begin{enumerate}
    \item $z,w\in Q_J$;
    \item $|1-z\overline{w}|^2 \simeq |Q_J|$;
    \item there is a unique dyadic descendant $I\subseteq J$ such that $z\in Q_I^{\textrm{up}}$. Without loss of generality, we may assume that $I\in\calD_k(J)$ for some $k\geq 0$, that is, $|I|=2^{-k}|J|$.
\end{enumerate}

Then for any $z, w \in \D$, one has
\begin{align*}
\frac{1}{(1-|z|^2)^{2-c+b}} \cdot \frac{1}{|1-z\overline{w}|^c} 
&\lesssim  \sum_{i \in \{1, 2\}} \sum_{J \in \calD^{(i)}} \sum_{k=0}^\infty \sum_{I \in \calD_k(J)} \frac{\one_{Q_I^{\textrm{up}}}(z) \one_{Q_J}(w)}{(1-|z|^2)^{2-c+b} |1-z\overline{w}|^c } \\
&\simeq  \sum_{i \in \{1, 2\}} \sum_{J \in \calD^{(i)}} \frac{\one_{Q_J}(w)}{|J|^c} \left[ \sum_{k=0}^\infty \sum_{I \in \calD_k(J)}  \frac{\one_{Q_I^{\textnormal{up}}}(z) }{|I|^{2-c+b}}\right]  \\
& \simeq \sum_{i \in \{1, 2\}} \sum_{J \in \calD^{(i)}} \frac{\one_{Q_J}(w)}{|J|^{2+b}} \left[ \sum_{k=0}^\infty 2^{k(2-c+b)} \left(\sum_{I \in \calD_k(J)} \one_{Q_I^{\textnormal{up}}}(z)  \right) \right]  \\
& \simeq \sum_{i \in \{1, 2\}} \sum_{J \in \calD^{(i)}} \frac{\one_{Q_J}(w)}{|Q_J|^{1+\frac{b}{2}}} \left[ \sum_{k=0}^\infty 2^{k(2-c+b)} \left(\sum_{I \in \calD_k(J)} \one_{Q_I^{\textnormal{up}}}(z)  \right) \right] 
\end{align*}
Hence, 
\begin{align*}
\int_{E} |S_{c-2-b,b,c} \one_{F}(z)| dA(z) 
&=\int_E \left[\frac{1}{(1-|z|^2)^{2-c+b}} \int_F \frac{(1-|w|^2)^b}{|1-z\overline{w}|^c} dA(w) \right] dA(z) \\
&\lesssim \sum_{i \in \{1, 2\}} \sum_{J\in\calD^{(i)}} \frac{|Q_J \cap F|_b}{|Q_J|^{1+\frac{b}{2}}} \left[ \sum_{k=0}^{\infty} 2^{k(2-c+b)} \left(\sum_{I\in\calD_k(J)} |Q_I^{\textnormal{up}} \cap E| \right) \right].
\end{align*}
Here, $|Q_J \cap F|_b:=(b+1)\int_{Q_J \cap F} (1-|z|^2)^b dA(z)$.
Therefore, to prove Lemma~\ref{simple function}, it suffices to show that for any dyadic system $\calD$ on $\T$, one has 
\begin{equation}\label{Goal_1}
	\sum_{J\in\calD} \frac{|Q_J \cap F|_b}{|Q_J|^{1+\frac{b}{2}}} \left[ \sum_{k=0}^{\infty} 2^{k(2-c+b)} \left(\sum_{I\in\calD_k(J)} |Q_I^{\textnormal{up}} \cap E| \right) \right]\lesssim  |E|^{c-1-b} |F|^{2-c+b},
\end{equation}
where the implicit constant in the above estimate is independent of the choice of $E,F$ and $\calD$.

\medskip 

\noindent\textbf{$\diamond$ \textsl{Step 4.}} 
To estimate \eqref{Goal_1}, we notice that, by the pairwise disjointness of the sets $\{Q_I^{\textnormal{up}}\}_{I\in\calD_k(J)}$ and the fact that $|I|=2^{-k}|J|$, one has
$$
	\sum_{I\in\calD_k(J)} |Q_I^{\textnormal{up}}\cap E|
	\leq
	\min\{|Q_J\cap E|,\,2^{-k}|Q_J|\}.
$$
Therefore, \eqref{Goal_1} will follow once we prove
\begin{equation} \tag{\ref{Goal_1}$'$} \label{Goal_1a}
	\sum_{J\in\calD}
	\frac{|Q_J \cap F|_b}{|Q_J|^{1+\frac{b}{2}}}
	\left[
		\sum_{k=0}^{\infty}
		2^{k(2-c+b)}
		\min\{|Q_J\cap E|,\,2^{-k}|Q_J|\}
	\right]
	\lesssim
	|E|^{c-1-b}|F|^{2-c+b}.
\end{equation}

To prove \eqref{Goal_1a}, we first estimate the inner sum appearing there. We have the following

\begin{lem} \label{20260613lem01}
For any $J \in \calD$ and $E \subseteq \D$ measurable, one has
\begin{equation} \label{20260605eq11}
	\sum_{k=0}^{\infty} 2^{k(2-c+b)}
	\min\{|Q_J\cap E|,2^{-k}|Q_J|\}
	\lesssim
	|Q_J\cap E|^{c-1-b}|Q_J|^{2-c+b}.
\end{equation}
\end{lem} 

\begin{proof}
If $|Q_J\cap E|=0$, there is nothing to prove. We may therefore assume that
$|Q_J\cap E|>0$. Choose $k_0\in\N$ such that $
2^{-k_0-1}|Q_J| \leq |Q_J\cap E| \leq 2^{-k_0}|Q_J|$, which implies 
$$
\frac{|Q_J|}{|Q_J\cap E|} \simeq 2^{k_0}.
$$
Write
\begin{equation}\label{20260605eq12}
	\textnormal{LHS of \eqref{20260605eq11}}=\calA_1+\calA_2,
\end{equation}
where
\begin{align}\label{0-k_0}
\calA_1
&:=\sum_{k=0}^{k_0} 2^{k(2-c+b)} \min\{|Q_J\cap E|,2^{-k}|Q_J|\} \nonumber\\
&\leq |Q_J\cap E| \sum_{k=0}^{k_0}2^{k(2-c+b)} \lesssim |Q_J\cap E|\,2^{k_0(2-c+b)} \nonumber\\
&\lesssim |Q_J\cap E| \left(\frac{|Q_J|}{|Q_J\cap E|}\right)^{2-c+b}=
|Q_J\cap E|^{c-1-b}|Q_J|^{2-c+b}.
\end{align}
Here, we used $2-c+b>0$. On the other hand, since $c-1-b>0$, one has
\begin{align}\label{k_0-infty}
\calA_2
&:=\sum_{k=k_0+1}^{\infty} 2^{k(2-c+b)}
\min\{|Q_J\cap E|,2^{-k}|Q_J|\} \nonumber\\
&\leq |Q_J| \sum_{k=k_0+1}^{\infty}2^{-k(c-1-b)}
\lesssim |Q_J|\,2^{-k_0(c-1-b)} \nonumber\\
&\lesssim |Q_J| \left(\frac{|Q_J|}{|Q_J\cap E|}\right)^{-(c-1-b)} =|Q_J\cap E|^{c-1-b}|Q_J|^{2-c+b}.
\end{align}
Finally, the desired estimate \eqref{20260605eq11} follows by combining \eqref{0-k_0}, \eqref{k_0-infty}, and \eqref{20260605eq12}.
\end{proof}

Next, we estimate the coefficient appearing in the outer sum in \eqref{Goal_1a}.

\begin{lem}\label{Q_J cap F}
For any $J \in \calD$, and $F \subseteq \D$ measurable, 
\[
	\frac{|Q_J \cap F|_b}{|Q_J|^{1+\frac{b}{2}}} \lesssim \left( \frac{|Q_J \cap F|}{|Q_J|} \right)^{\min\{1,b+1\}}.
\]
\end{lem}
\begin{proof}
If $|Q_J\cap F|=0$, there is nothing to prove. Hence assume $|Q_J\cap F|>0$. We consider two different cases.

\vspace{0.1cm}

\noindent $\bullet$ If $b \geq 0$, then $\min\{1, b+1\}=1$, and for each $z \in Q_J$, we have 
\[
	(1-|z|^2)^b \lesssim (1-|z|)^b \lesssim |J|^b.
\]
Hence,
\begin{align*}
	|Q_J \cap F|_b
	&=(b+1)\int_{Q_J \cap F} (1-|z|^2)^b dA(z)\\
	&\lesssim |J|^b |Q_J \cap F|\\
	&\simeq |Q_J|^{\frac{b}{2}} |Q_J \cap F|,
\end{align*}
which gives 
\begin{equation}\label{b>=0}
	\frac{|Q_J \cap F|_b}{|Q_J|^{1+\frac{b}{2}}} \lesssim \frac{|Q_J \cap F|}{|Q_J|}.
\end{equation}

\vspace{0.1cm} 

\noindent $\bullet$ If $-1<b<0$, then $\min\{1, b+1\}=b+1$. Set $W_{-b}(z):=(1-|z|^2)^b$. A direct computation yields 
$$
\|W_{-b}\one_{Q_J}\|_{L^{-\frac{1}{b},\infty}(\D)}\simeq |J|^{-b}.
$$
Therefore, by Lorentz-H\"older's inequality, 
\begin{align*}
	|Q_J \cap F|_b
	& \lesssim  \|W_{-b}\one_{Q_J}\|_{L^{-\frac{1}{b},\infty}(\D)} \|\one_{Q_J \cap F}\|_{L^{\frac{1}{b+1},1}(\D)}\\
	& \simeq |J|^{-b} |Q_J \cap F|^{b+1}\\
	& \simeq |Q_J|^{-\frac{b}{2}} |Q_J \cap F|^{b+1}.
\end{align*}
Hence, 
\begin{equation}\label{-1<b<0}
	\frac{|Q_J \cap F|_b}{|Q_J|^{1+\frac{b}{2}}} \lesssim \left( \frac{|Q_J \cap F|}{|Q_J|} \right)^{b+1}.
\end{equation}
Combining \eqref{b>=0} and \eqref{-1<b<0} yields the desired estimate. 
\end{proof}

Since $0<2-c+b<\min\{1,b+1\}$, we may choose $\sigma>0$ such that
$2-c+b<\sigma<\min\{1,b+1\}$. Using this choice of $\sigma$, together with
$|Q_J\cap F|\le |Q_J|$ and Lemma~\ref{Q_J cap F}, we obtain
\[
	\frac{|Q_J \cap F|_b}{|Q_J|^{1+\frac{b}{2}}} 
	\lesssim \left( \frac{|Q_J \cap F|}{|Q_J|} \right)^{\min\{1,b+1\}} 
	\le \left( \frac{|Q_J \cap F|}{|Q_J|} \right)^\sigma.
\]
Therefore, combining the above estimate with \eqref{20260605eq11}, we have
\begin{align*}
	\mbox{LHS of \eqref{Goal_1a} }
	& \lesssim \sum_{J\in\calD} \left( \frac{|Q_J \cap F|}{|Q_J|} \right)^\sigma \cdot |Q_J\cap E|^{c-1-b}|Q_J|^{2-c+b} \\
	& =\sum_{J\in\calD} \left( \frac{|Q_J \cap E|}{|Q_J|} \right)^{c-1-b} \left( \frac{|Q_J \cap F|}{|Q_J|} \right)^\sigma |Q_J|,
\end{align*}
Therefore, it remains to prove the following variant of dyadic Carleson embedding estimate:
\begin{equation}\tag{\ref{Goal_1a}$'$}\label{Goal_2}
	\sum_{J\in\calD} \left( \frac{|Q_J \cap E|}{|Q_J|} \right)^{c-1-b} \left( \frac{|Q_J \cap F|}{|Q_J|} \right)^\sigma |Q_J| \lesssim |E|^{c-1-b} |F|^{2-c+b},
\end{equation}
where implicit constant above is independent of $E,F$ and $\calD$.

\vspace{0.1cm}

\noindent\textbf{$\diamond$ \textsl{Step 5.}} In the last step, we prove \eqref{Goal_2}. Denote 
\[
	e_J:=\frac{|Q_J \cap E|}{|Q_J|}, \qquad \textrm{and} \qquad f_J:=\frac{|Q_J \cap F|}{|Q_J|}.
\]
Then 
\[
	e_J^{c-1-b}=(c-1-b)\int_0^{e_J} s^{c-2-b} ds=(c-1-b)\int_0^1 s^{c-2-b} \one_{\{s<e_J\}} ds,
\]
and 
\[
	f_J^{\sigma}=\sigma\int_0^{f_J} t^{\sigma-1} dt=\sigma\int_0^1 t^{\sigma-1} \one_{\{t<f_J\}} dt.
\]
Recall that $0<c-1-b, \sigma<1$. Therefore, 
\begin{align} \label{20290609eq01}
	\mbox{LHS of \eqref{Goal_2} }
	&=\sum_{J\in\calD} e_J^{c-1-b} f_J^{\sigma} |Q_J| \nonumber \\
	& \simeq \sum_{J\in\calD}  \left(\int_0^1 s^{c-2-b} \one_{\{s<e_J\}} ds \right) \cdot \left( \int_0^1 t^{\sigma-1} \one_{\{t<f_J\}} dt \right) |Q_J| \nonumber  \\
	&= \int_0^1 \int_0^1 s^{c-2-b} t^{\sigma-1} \left(\sum_{J\in\calD} |Q_J| \one_{\{s<e_J\}}\one_{\{t<f_J\}} \right) dtds \nonumber  \\
    & =\int_0^1 \int_0^1 s^{c-2-b} t^{\sigma-1} \left(\sum_{J\in\calD} |Q_J| \one_{\{s<e_J, \; t<f_J\}} \right) dtds \nonumber  \\
	&= \int_0^1 \int_0^1 s^{c-2-b} t^{\sigma-1} \left(\sum_{J\in\mathcal{G}_{s,t}} |Q_J| \right) dtds.
\end{align}
where in the last equation above, we write, for any $0<s, t<1$, 
$$
\mathcal{G}_{s,t}:=\{J\in\calD: e_J>s, \; f_J>t\}.
$$
The following lemma estimates the inner sum in \eqref{20290609eq01}.
 
\begin{lem}
Assume that $E,F$ are measurable subsets of $\D$, and let $e_J$ and $f_J$ be defined as above. Then, for every $0<s,t<1$,
\[
	\sum_{J\in\mathcal{G}_{s,t}} |Q_J|
	\lesssim
	\min\left\{\frac{|E|}{s},\frac{|F|}{t}\right\},
\]
where the implicit constant is independent of $E,F,s,t$, and the dyadic system $\mathcal D$.
\end{lem}

\begin{proof}
Let $\mathcal{G}_{s,t}^{\max}$ be a maximal collection of $\mathcal{G}_{s,t}$.
Note that for any  $I \in\calD$, we have
\[
	\sum_{J\in\calD(I)} |Q_J| \lesssim  |Q_I|.
\]
Hence, 
\[
	\sum_{J\in\mathcal{G}_{s,t}} |Q_J| 
	\lesssim \sum_{J\in\mathcal{G}_{s,t}^{\max}} |Q_J|.
\]
By the definitions of $e_J,f_J$ and $\mathcal{G}_{s,t}$, we have for each $J\in\mathcal{G}_{s,t}^{\max}$, 
\[
	|Q_J|<\frac{|Q_J \cap E|}{s}, \quad \textrm{and} \quad  |Q_J|<\frac{|Q_J \cap F|}{t}.
\]
Note that the intervals in $\mathcal{G}_{s,t}^{\max}$ are pairwise disjoint, and hence
the corresponding Carleson tents $\left\{Q_J \right\}_{J \in \mathcal G_{s, t}^{\max}}$ are also pairwise disjoint. Therefore, 
\[
	\sum_{J\in\mathcal{G}_{s,t}^{\max}} |Q_J| \lesssim  \sum_{J\in\mathcal{G}_{s,t}^{\max}} \frac{|Q_J \cap E|}{s} \leq \frac{|E|}{s},
\]
and 
\[
	\sum_{J\in\mathcal{G}_{s,t}^{\max}} |Q_J| \lesssim \sum_{J\in\mathcal{G}_{s,t}^{\max}} \frac{|Q_J \cap F|}{t} \leq \frac{|F|}{t}.
\]
Hence,
\[
	\sum_{J\in\mathcal{G}_{s,t}} |Q_J| 
	\lesssim \sum_{J\in\mathcal{G}_{s,t}^{\max}} |Q_J|
	\lesssim \min\left\{ \frac{|E|}{s},\frac{|F|}{t} \right\}.
\]
\end{proof}

Therefore, we have 
$$
	\mbox{LHS of \eqref{Goal_2} } \lesssim \int_0^1 \int_0^1 s^{c-2-b} t^{\sigma-1} \min\left\{ \frac{|E|}{s},\frac{|F|}{t} \right\} ds dt.
$$ 
Finally, it suffices to prove the following 
\begin{lem}
Assume $E,F$ are measurable subsets of $\D$ and $0<2-c+b<\sigma<\min\{1, b+1\}$.
Then 
\begin{equation}\label{Goal_3}
	\int_0^1 \int_0^1 s^{c-2-b} t^{\sigma-1} \min\left\{ \frac{|E|}{s},\frac{|F|}{t} \right\} dt ds \lesssim |E|^{c-1-b}|F|^{2-c+b},
\end{equation}
where the implicit constant above is independent of the choice of $E$ and $F$. 
\end{lem}
\begin{proof}
Without loss of generality, we may assume $|E|>0$ and $|F|>0$. We consider two different cases.

\vspace{0.1cm}

\noindent $\bullet$ If $|E| \geq |F|$, then
\begin{align*}
	\mbox{LHS of \eqref{Goal_3} }
	=&|E|\int_0^1 s^{c-3-b} \left( \int_0^{\frac{|F|}{|E|}s} t^{\sigma-1} dt  \right) ds + |F|\int_0^1 s^{c-2-b} \left(\int_{\frac{|F|}{|E|}s}^1 t^{\sigma-2} dt \right) ds\\
	:=&I_1+I_2.
\end{align*}

\vspace{0.1cm}

\noindent \underline{\textit{Estimate of $I_1$.}} Since $c-2-b+\sigma>0$, we have
$$
I_1 \simeq|E|^{1-\sigma}|F|^{\sigma}\int_0^1 s^{c-3-b+\sigma} ds \simeq |E|^{1-\sigma}|F|^{\sigma}.
$$

\vspace{0.1cm}

\noindent \underline{\textit{Estimate of $I_2$.}} Since $\sigma<1$, we have
\begin{align*}
	I_2
	& \simeq |F| \int_0^1 s^{c-2-b}\left[ \left( \frac{|F|}{|E|}s \right)^{\sigma-1}-1 \right] ds\\
	& \lesssim |E|^{1-\sigma}|F|^{\sigma}\int_0^1 s^{c-3-b+\sigma} ds\\
	& \simeq |E|^{1-\sigma}|F|^{\sigma}.
\end{align*}
Thus, 
\[
	\mbox{LHS of \eqref{Goal_3} } 
	\lesssim |E|^{1-\sigma}|F|^{\sigma}
	 = |E|^{c-1-b}|F|^{2-c+b}\left( \frac{|F|}{|E|} \right)^{c-2-b+\sigma}.
\]
Since $|E| \geq |F|$ and $c-2-b+\sigma>0$, we deduce that 
\[
	\mbox{LHS of \eqref{Goal_3} } \lesssim |E|^{c-1-b}|F|^{2-c+b}.
\]

\medskip 

\noindent $\bullet$ If $|E|<|F|$, then 
\begin{align*}
	\mbox{LHS of \eqref{Goal_3} }
	&=|F|\int_0^1 t^{\sigma-2} \left(\int_0^{\frac{|E|}{|F|}t} s^{c-2-b} ds \right) dt + |E|\int_0^1 t^{\sigma-1} \left(\int_{\frac{|E|}{|F|}t}^1  s^{c-3-b} ds \right) dt\\
	&:=I_3+I_4.
\end{align*}
\vspace{0.1cm}

\noindent \underline{\textit{Estimate of $I_3$.}} 
Since $c-1-b>0$ and $c-2-b+\sigma>0$, we have
\begin{align*}
	I_3
	& \simeq |E|^{c-1-b}|F|^{2-c+b} \int_0^1 t^{c-3-b+\sigma} dt\\
	& \simeq |E|^{c-1-b}|F|^{2-c+b}.
\end{align*}

\vspace{0.1cm}

\noindent \underline{\textit{Estimate of $I_4$.}} Since $c-2-b<0$ and $c-2-b+\sigma>0$, we have
\begin{align*}
	I_4
	& \simeq |E|\int_0^1 t^{\sigma-1}\left[ \left( \frac{|E|}{|F|}t \right)^{c-2-b}-1 \right] dt\\
	& \lesssim|E|^{c-1-b}|F|^{2-c+b}\int_0^1 t^{c-2-b+\sigma-1} dt\\
	& \simeq |E|^{c-1-b}|F|^{2-c+b}.
\end{align*}

Therefore, in both cases above, one has
\[
	\mbox{LHS of \eqref{Goal_3} } 
	\lesssim |E|^{c-1-b}|F|^{2-c+b}.
\]

\medskip 

The proof of Theorem \ref{restricted-weak Forelli--Rudin at CH}, (ii) is complete.
\end{proof}

\bigskip 

\section{Analysis of the Main Case: Restricted Weak-type Bounds at the Critical--Vertical Intersection $\calP_{\textnormal{C-V}}$ away from the horizontal cut-off line} \label{Sec07}

Our next goal is to analyze the behavior at the critical--vertical intersection
$$
\calP_{\textnormal{C-V}}=(b+1,\;c-a-1),
$$
where the critical line \eqref{criticalline} meets the vertical cut-off line
\eqref{verticalcut}. We focus on the case where this point lies strictly above
the vertical--horizontal intersection $\calP_{\textnormal{V-H}}$. Equivalently, we impose the
following condition:
\begin{equation}  \tag{${\bf P_2}$} \label{conditionP2}
\begin{cases}
a,b,c\in\R \quad \textrm{satisfy \eqref{maincase}}; \\
c-a-1 \le 1; \\
c>1.
\end{cases}
\end{equation}
Indeed, the second condition, together with \eqref{maincase}, guarantees $\calP_{\textnormal{C-V}}$ lies in the region $0 <1/p<1$ and $0<1/q \le 1$; while the third condition ensures that $\calP_{\textnormal{C-V}}$ lies strictly above the horizontal cut-off line \eqref{horizontalcut}. This is the typical configuration in which the vertical--horizontal intersection $\calP_{\textnormal{V--H}}$ lies in Regions II and IV, where $c-a-1<1$. The same situation can also arise in Regions I and III when $c-a-1=1$ (see Figure \ref{Fig13}).

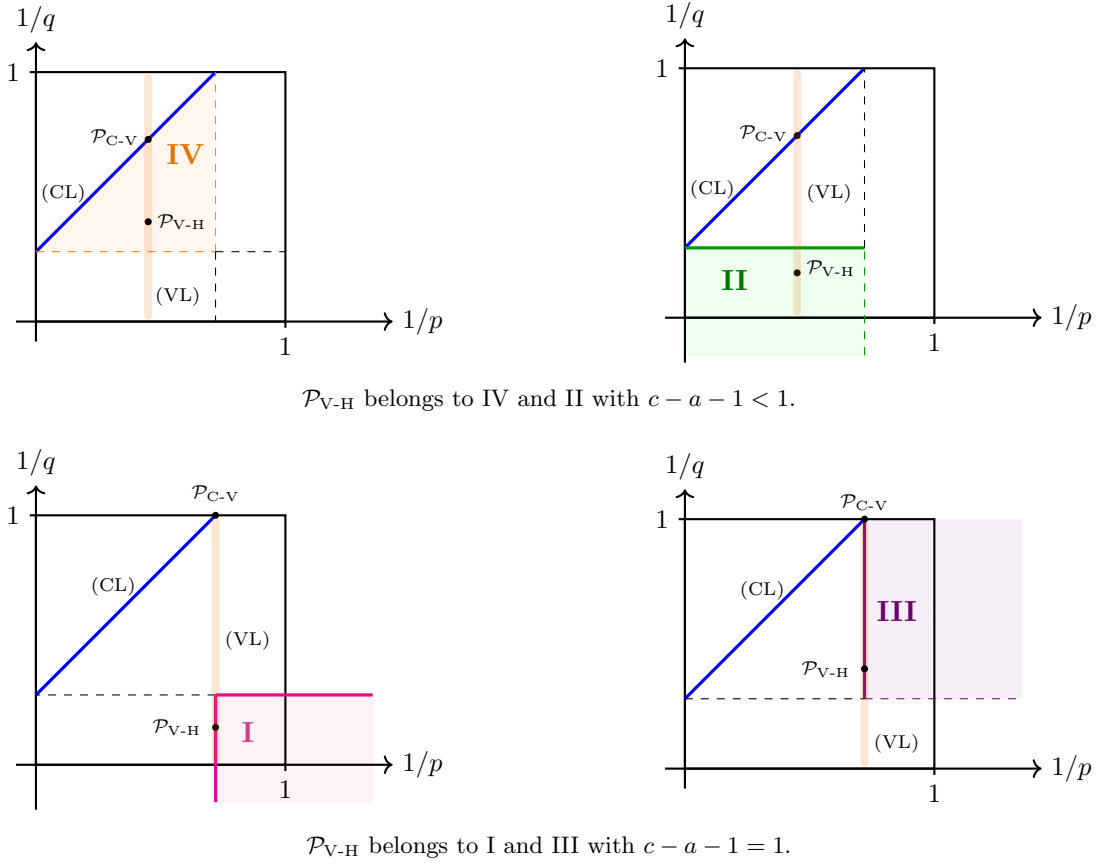
\begin{figure}[ht]
\centering

\begin{minipage}[t]{0.48\textwidth}
\centering
\begin{tikzpicture}[scale=3.3]
\def\yh{0.28}       
\def\xc{0.72}        
\def\xright{1.35}
\def\ybottom{-0.05}

\fill[orange!20, opacity=0.3] (0,\yh) -- (\xc,\yh) -- (\xc,1) -- cycle;    

\draw[->, thick] (-0.08,0) -- (\xright+0.08,0) node[right] {$1/p$};
\draw[->, thick] (0,\ybottom-0.03) -- (0,1.12) node[above] {$1/q$};

\draw[thick] (0,0) rectangle (1,1);
\draw[thick] (1,0) -- (1,-0.025) node[below] {$1$};
\draw[thick] (0,1) -- (-0.025,1) node[left] {$1$};

\draw[blue, very thick] (0,\yh) -- (\xc,1);
\node[above] at (\xc/2-0.25, \yh/2+0.3) {\scriptsize \eqref{criticalline}};

\draw[orange, dashed] (\xc,\yh) -- (\xc,1);    
\draw[orange, dashed] (0,\yh) -- (\xc,\yh);  
\draw[dashed] (\xc, \yh) -- (1, \yh);
\draw[dashed] (\xc, \yh) -- (\xc, 0);

\node[font=\bfseries\large, orange!90!black] at (0.6,0.67) {IV};

\draw[orange!90!black, line width=3pt, opacity=0.18, line cap=round,
      shorten <=2.2pt, shorten >=2.2pt]
    (0.45, 0) -- (0.45, 1);

\fill[black] (0.45, 0.73) circle (.4pt);
\node[left] at (0.45, 0.74) {\scriptsize $\calP_{\textnormal{C-V}}$};

\fill[black] (0.45, 0.4) circle (.4pt);
\node[right] at (0.45, 0.4) {\scriptsize $\calP_{\textnormal{V-H}}$};

\node[right] at (0.44, 0.1) {\scriptsize \eqref{verticalcut}};
\end{tikzpicture}
\end{minipage}
\hfill
\begin{minipage}[t]{0.48\textwidth}
\centering
\begin{tikzpicture}[scale=3.3]
\def\yh{0.28}       
\def\xc{0.72}        
\def\xright{1.35}
\def\ybottom{-0.15}

\fill[green!20, opacity=0.3] (0,\ybottom) rectangle (\xc,\yh);           

\draw[->, thick] (-0.08,0) -- (\xright+0.08,0) node[right] {$1/p$};
\draw[->, thick] (0,\ybottom-0.03) -- (0,1.12) node[above] {$1/q$};

\draw[thick] (0,0) rectangle (1,1);
\draw[thick] (1,0) -- (1,-0.025) node[below] {$1$};
\draw[thick] (0,1) -- (-0.025,1) node[left] {$1$};

\draw[blue, very thick] (0,\yh) -- (\xc,1);
\node[above] at (\xc/2-0.25, \yh/2+0.3) {\scriptsize \eqref{criticalline}};

\draw[green!60!black, dashed] (\xc,\ybottom) -- (\xc,\yh);   
\draw[black, dashed] (\xc,\yh) -- (\xc,1);     
\draw[green!60!black, very thick] (0,\yh) -- (\xc,\yh);  

\node[font=\bfseries\large, green!45!black] at (0.2,0.15) {II};

\fill[black] (0.45,\yh-0.1) circle (.4pt);
\node[right] at (0.45,\yh-0.08) {\scriptsize  $\calP_{\textnormal{V-H}}$};

\fill[black] (0.45, 0.73) circle (.4pt);
\node[left] at (0.45, 0.74) {\scriptsize $\calP_{\textnormal{C-V}}$};

\draw[orange!90!black, line width=3pt, opacity=0.18, line cap=round,
      shorten <=2.2pt, shorten >=2.2pt]
    (0.45, 0) -- (0.45, 1);

\node[right] at (0.45, 0.5) {\scriptsize  \eqref{verticalcut}};
\end{tikzpicture}
\end{minipage}

\vspace{0.4em}

{\small $\calP_{\textnormal{V-H}}$ belongs to IV and II with $c-a-1<1$.}

\vspace{1em}

\begin{minipage}[t]{0.48\textwidth}
\centering
\begin{tikzpicture}[scale=3.3]
\def\yh{0.28}       
\def\xc{0.72}        
\def\xright{1.35}
\def\ybottom{-0.15}    

\fill[magenta!20, opacity=0.3] (\xc,\ybottom) rectangle (\xright,\yh);        

\draw[->, thick] (-0.08,0) -- (\xright+0.08,0) node[right] {$1/p$};
\draw[->, thick] (0,\ybottom-0.03) -- (0,1.12) node[above] {$1/q$};

\draw[thick] (0,0) rectangle (1,1);
\draw[thick] (1,0) -- (1,-0.025) node[below] {$1$};
\draw[thick] (0,1) -- (-0.025,1) node[left] {$1$};

\draw[blue, very thick] (0,\yh) -- (\xc,1);
\node[above] at (\xc/2-0.05, \yh/2+0.5) {\scriptsize \eqref{criticalline}};

\draw[black, dashed] (\xc,\yh) -- (0, \yh);     
\draw[magenta, very thick] (\xc,\ybottom) -- (\xc,\yh);   
\draw[magenta, very thick] (\xc,\yh) -- (\xright,\yh);     

\node[font=\bfseries\large, magenta!85!black] at (0.85,0.13) {I};

\fill[black] (\xc, .15) circle (.4pt);
\node[left] at (\xc-0.02,.15) {\scriptsize $\calP_{\textnormal{V-H}}$};

\fill[black] (\xc, 1) circle (.4pt);
\node[above] at (\xc, 1.02) {\scriptsize $\calP_{\textnormal{C-V}}$};

\draw[orange!90!black, line width=3pt, opacity=0.18, line cap=round,
      shorten <=2.2pt, shorten >=2.2pt]
    (\xc, 0) -- (\xc, 1);
\node[right] at (\xc, 0.5) {\scriptsize  \eqref{verticalcut}};

\end{tikzpicture}
\end{minipage}
\hfill
\begin{minipage}[t]{0.48\textwidth}
\centering
\begin{tikzpicture}[scale=3.3]
\def\yh{0.28}       
\def\xc{0.72}        
\def\xright{1.35}
\def\ybottom{-0.05}          

\fill[violet!20, opacity=0.3] (\xc,\yh) rectangle (1,1);   
\fill[violet!20, opacity=0.3] (1,\yh) rectangle (\xright,1);

\draw[->, thick] (-0.08,0) -- (\xright+0.08,0) node[right] {$1/p$};
\draw[->, thick] (0,\ybottom-0.03) -- (0,1.12) node[above] {$1/q$};

\draw[thick] (0,0) rectangle (1,1);
\draw[thick] (1,0) -- (1,-0.025) node[below] {$1$};
\draw[thick] (0,1) -- (-0.025,1) node[left] {$1$};

\draw[blue, very thick] (0,\yh) -- (\xc,1);
\node[above] at (\xc/2-0.05, \yh/2+0.5) {\scriptsize \eqref{criticalline}};

\draw[violet, very thick] (\xc,\yh) -- (\xc,1);        
\draw[violet, dashed]    (\xc,\yh) -- (\xright,\yh);   
\draw[black, dashed] (0,\yh) -- (\xc,\yh);  

\node[font=\bfseries\large, violet!85!black] at (0.85,0.63) {III};

\draw[orange!90!black, line width=3pt, opacity=0.18, line cap=round,
      shorten <=2.2pt, shorten >=2.2pt]
    (\xc, 0) -- (\xc, 1);

\fill[black] (\xc, 0.4) circle (.4pt);
\node[left] at (\xc-0.02, 0.4) {\scriptsize $\calP_{\textnormal{V-H}}$};

\fill[black] (\xc, 1) circle (.4pt);
\node[above] at (\xc, 1) {\scriptsize $\calP_{\textnormal{C-V}}$};

\node[right] at (\xc, 0.1) {\scriptsize \eqref{verticalcut}};
\end{tikzpicture}
\end{minipage}

\vspace{0.4em}

{\small $\calP_{\textnormal{V-H}}$ belongs to I and III with $c-a-1=1$.}

\caption{\small Four cases when $\calP_{\textnormal{C-V}}$ lies strictly above $\calP_{\textnormal{V-H}}$.}
\label{Fig13}
\end{figure}

First, recall from Section~\ref{Subsec4.1} that the weak--type estimates for both
$T_{a,b,c}$ and $S_{a,b,c}$ fail along the vertical cut-off line
\eqref{verticalcut}. Therefore, under the assumption \eqref{conditionP2}, we shall
focus instead on the restricted weak--type behavior at $\mathcal P_{\textnormal{C-V}}$. Here is the main result in this section. 

\begin{thm}\label{restricted-weak Forelli--Rudin at CV}
Let $a, b, c \in \R$ satisfy \eqref{conditionP2}. Then the Forelli--Rudin operators $T_{a, b, c}$ and $S_{a, b, c}$ are bounded from $L^{\frac{1}{b+1}, 1}(\D)$ to $L^{\frac{1}{c-1-a}, \infty}(\D)$. 
\end{thm} 

\subsection{Proof of Theorem \ref{restricted-weak Forelli--Rudin at CV}: Restricted weak-type estimates for $T_{a, b, c}$ and $S_{a, b, c}$ at the intersection $\calP_{\textnormal{C-V}}$}

The proof of Theorem~\ref{restricted-weak Forelli--Rudin at CV} is similar to that of Theorem~\ref{restricted-weak Forelli--Rudin at CH}, (ii). We therefore include only a \emph{sketch}, indicating the necessary modifications.

\medskip 

\noindent\textbf{$\diamond$ \textsl{Step 1.}} We first reduce the problem to the diagonal case via the Forelli--Rudin method. Set $\eta: =c-2-a-b>0$ and $W_\eta(z)=(1-|z|^2)^{-\eta}$.
Then, as usual, $W_\eta \in L^{\frac{1}{\eta}, \infty}(\D)$ and we write 
\[
	S_{a,b,c}f(z)=W_\eta(z) S_{a+\eta,b,c}f(z)=W_\eta(z)S_{c-2-b ,b, c}f(z). 
\]
Since $c-1-a=\eta+(b+1)$,  Lorentz-H\"older's inequality gives
$$
\| S_{a,b,c}f \|_{L^{\frac{1}{c-1-a},\infty}(\D)} \lesssim \|W_\eta\|_{L^{\frac{1}{\eta},\infty}(\D)} \| S_{c-2-b, b, c}f \|_{L^{\frac{1}{b+1},\infty}(\D)}.
$$
Then it suffices to show
$$
\| S_{c-2-b, b, c}f \|_{L^{\frac{1}{b+1},\infty}(\D)} \lesssim \left\|f \right\|_{L^{\frac{1}{b+1}, 1}(\D)}, 
$$
which is equivalent to showing
\begin{equation} \tag{\ref{simplified case of restricted weak type A}'} \label{20260613eq01}
\| S_{c-2-b, b, c} \one_F \|_{L^{\frac{1}{b+1},\infty}(\D)} \lesssim |F|^{b+1}
\end{equation} 
for every measurable set $F \subseteq \D$. 

\medskip

\noindent\textbf{$\diamond$ \textsl{Step 2.}} 
As in \textbf{\textsl{Step 2}} of the proof in Theorem~\ref{restricted-weak Forelli--Rudin at CH}, (ii), we reduce the proof of \eqref{20260613eq01} to the following stronger estimate. 

\begin{lem}\label{simple function A}
Let $E, F \subseteq \D$ be measurable sets. Then
\[
	\int_{E} |S_{c-2-b,b,c} \one_{F}(z)| dA(z)
	\leq C |E|^{-b} |F|^{b+1}.
\]
\end{lem}

\medskip 

\noindent\textbf{$\diamond$ \textsl{Step 3.}}
Now, by interchanging the roles of $c-2-b$ and $b$ in \textbf{\textsl{Step 3}} of the proof of Theorem~\ref{restricted-weak Forelli--Rudin at CH}, (ii), we deduce that for any $z, w \in \D$, 
\[
	\frac{(1-|w|^2)^b}{|1-z\overline{w}|^{c}} \lesssim  \sum_{i \in \{1, 2\}}  \sum_{J\in\calD^{(i)}} \frac{\one_{Q_J}(z)}{|Q_J|^{1+\frac{c-2-b}{2}}} \left[ \sum_{k=0}^{\infty} 2^{-kb} \left(\sum_{I\in\calD_k(J)} \one_{Q_I^{\textrm{up}}}(w) \right) \right] .
\]
and hence
\begin{align*}
	\int_{E} |S_{c-2-b,b,c} \one_{F}(z)| dA(z) 
	&=\int_E \left[(1-|z|^2)^{c-2-b} \int_F \frac{(1-|w|^2)^b}{|1-z\overline{w}|^c} dA(w) \right] dA(z) \\
	& \lesssim  \sum_{i \in\{1, 2\}} \sum_{J\in\calD^{(i)}} \frac{|Q_J \cap E|_{c-2-b}}{|Q_J|^{1+\frac{c-2-b}{2}}} \left[ \sum_{k=0}^{\infty} 2^{-kb} \left(\sum_{I\in\calD_k(J)} |Q_I^{\textrm{up}} \cap F| \right) \right].
\end{align*}
Here $|Q_J \cap E|_{c-2-b}:=(c-1-b)\int_{Q_J \cap E} (1-|z|^2)^{c-2-b} dA(z)$, and note also that $c-2-b>a>-1$. Therefore, it suffices to show that for any dyadic system $\calD$ on $\T$
\begin{equation} \tag{\ref{Goal_1}a} \label{Goal_1A}
	\sum_{J\in\calD} \frac{|Q_J \cap E|_{c-2-b}}{|Q_J|^{1+\frac{c-2-b}{2}}} \left[ \sum_{k=0}^{\infty} 2^{-kb} \left(\sum_{I\in\calD_k(J)} |Q_I^{\textrm{up}} \cap F|  \right) \right] \lesssim  |E|^{-b} |F|^{b+1},
\end{equation}
where the implicit constant in the above estimate is independent of $E, F$ and $\calD$.

\medskip

\noindent\textbf{$\diamond$ \textsl{Step 4.}}
Note that $-1<b<c-a-2 \le 0$, and 
\[
	\sum_{I\in\calD_k(J)} |Q_I^{\textrm{up}} \cap F| \leq \min\{|Q_J \cap F|, 2^{-k}|Q_J|\}. 
\]
This reduces the proof of \eqref{Goal_1A} to showing that
\begin{equation} \tag{\ref{Goal_1A}'} \label{Goal_1AA}
	\sum_{J\in\calD}
	\frac{|Q_J\cap E|_{c-2-b}}{|Q_J|^{1+\frac{c-2-b}{2}}}
	\left[
		\sum_{k=0}^{\infty} 2^{-kb}
		\min\{|Q_J\cap F|,2^{-k}|Q_J|\}
	\right]
	\lesssim |E|^{-b}|F|^{b+1}.
\end{equation}

To prove \eqref{Goal_1AA}, we estimate the inner sum and the coefficient term in the preceding estimate. The proofs are similar to those of Lemma~\ref{20260613lem01} and Lemma~\ref{Q_J cap F}, respectively, and hence are omitted.

\begin{lem} \label{20260618lem01}
For any $J \in \calD$ and $F \subseteq \D$ measurable, one has
\[
	\sum_{k=0}^{\infty}  2^{-kb} \min\{|Q_J \cap F|, 2^{-k}|Q_J|\} \lesssim |Q_J \cap F|^{1+b} |Q_J|^{-b}.
\]

\end{lem}

\begin{lem}\label{Q_J cap E}
For any $J \in \calD$ and $E \subseteq \D$ measurable, 
\[
	\frac{|Q_J \cap E|_{c-2-b}}{|Q_J|^{1+\frac{c-2-b}{2}}} \lesssim \left( \frac{|Q_J \cap E|}{|Q_J|} \right)^{\min\{1,c-1-b\}}.
\]
\end{lem}

Since $c-1>0$ and $-1<b<0$, there exists $\sigma>0$ such that $-b<\sigma<\min\{1,c-1-b\}$.
It follows from $|Q_J \cap E| \le |Q_J|$ and Lemma \ref{Q_J cap E} that 
\begin{align*}
	\frac{|Q_J \cap E|_{c-2-b}}{|Q_J|^{1+\frac{c-2-b}{2}}} 
	& \lesssim \left( \frac{|Q_J \cap E|}{|Q_J|} \right)^{\min\{1,c-1-b\}} \\
	& \lesssim \left( \frac{|Q_J \cap E|}{|Q_J|} \right)^\sigma.
\end{align*}
Using the above estimate with Lemma \ref{20260618lem01}, we deduce that 
\begin{align*}
	\mbox{LHS of \eqref{Goal_1AA} }
	& \lesssim \sum_{J\in\calD} \frac{|Q_J \cap E|_{c-2-b} }{|Q_J|^{1+\frac{c-2-b}{2}}} \cdot \frac{|Q_J \cap F|^{1+b}}{|Q_J|^b}\\
	&\lesssim \sum_{J\in\calD} \left( \frac{|Q_J \cap E|}{|Q_J|} \right)^{\sigma} \left( \frac{|Q_J \cap F|}{|Q_J|} \right)^{1+b} |Q_J|.
\end{align*}
Therefore, it remains to prove
\begin{equation} \tag{\ref{Goal_1AA}$'$} \label{Goal_2A}
	\sum_{J\in\calD} \left( \frac{|Q_J \cap E|}{|Q_J|} \right)^{\sigma}
	\left( \frac{|Q_J \cap F|}{|Q_J|} \right)^{1+b} |Q_J|
	\lesssim |E|^{-b} |F|^{b+1},
\end{equation}
where the implicit constant is independent of $E$, $F$, and $\calD$. The proof of
\eqref{Goal_2A} is parallel to that of \eqref{Goal_2}, presented in
{\bf \textsl{Step 5}} of the proof of Theorem~\ref{restricted-weak Forelli--Rudin at CH}, (ii). We therefore omit the details and leave them to the interested reader.

\bigskip 

\section{Analysis of the Main Case: Restricted Weak-type Bounds at the Vertical--Horizontal Intersection $\calP_{\textnormal{V-H}}$ away from the critical line} \label{Sec08}

We next consider the vertical--horizontal intersection
$$
\calP_{\textnormal{V-H}}=(b+1,-a),
$$
where the vertical cut-off line \eqref{verticalcut} meets the horizontal cut-off line \eqref{horizontalcut}. This endpoint is relevant only in the case where $\calP_{\textnormal{V-H}}$ lies strictly above the critical line \eqref{criticalline}. Equivalently, this corresponds to the following parameter condition:
\begin{equation} \tag{${\bf P_3}$} \label{conditionP3}
\begin{cases}
a,b,c\in\R \quad \textnormal{satisfy \eqref{maincase}}; \\
c<1.
\end{cases}
\end{equation}
Here the condition $c<1$ is precisely what ensures that $\calP_{\textnormal{V-H}}$ lies strictly above the critical line \eqref{criticalline}. This is the typical situation when $\calP_{\textnormal{V-H}}$ lies in Region V (see Figure \ref{Fig14}). 

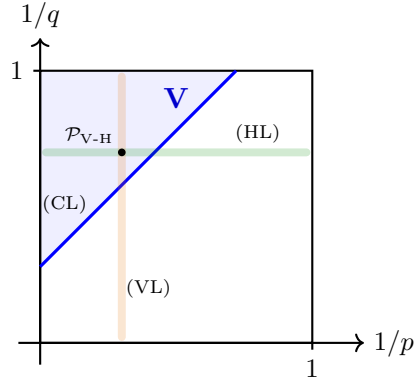
\begin{figure}[ht]
\centering
\begin{tikzpicture}[scale=3.6]
\def\yh{0.28}       
\def\xc{0.72}        
\def\xright{1.15}
\def\ybottom{-0.05}          

\fill[blue!20, opacity=0.3]   (0,\yh) -- (0,1) -- (\xc,1) -- cycle;  

\draw[->, thick] (-0.08,0) -- (\xright+0.04,0) node[right] {$1/p$};
\draw[->, thick] (0,\ybottom-0.03) -- (0,1.12) node[above] {$1/q$};

\draw[thick] (0,0) rectangle (1,1);
\draw[thick] (1,0) -- (1,-0.025) node[below] {$1$};
\draw[thick] (0,1) -- (-0.025,1) node[left] {$1$};

\draw[blue, very thick] (0,\yh) -- (\xc,1);
\node[above] at (\xc/2-0.27, \yh/2+0.3) {\scriptsize \eqref{criticalline}};

\node[font=\bfseries\large, blue!80!black] at (0.5,0.9) {V};

\draw[green!55!black, line width=3pt, opacity=0.18, line cap=round,
      shorten <=2.2pt, shorten >=2.2pt]
    (0,0.7) -- (1,0.7);

\draw[orange!90!black, line width=3pt, opacity=0.18, line cap=round,
      shorten <=2.2pt, shorten >=2.2pt]
    (0.3, 0) -- (0.3, 1);

\fill[black] (0.3,  0.7) circle (.4pt);
\node[above] at (0.18, 0.7) {\scriptsize $\calP_{\textnormal{V-H}}$};

\node[right] at (0.28, 0.2) {\scriptsize \eqref{verticalcut}};

\node[above] at (0.8, 0.7) {\scriptsize \eqref{horizontalcut}};
\end{tikzpicture}
\caption{\small The case when $\calP_{\textnormal{V-H}}$ lies strictly above \eqref{criticalline}.}
\label{Fig14}
\end{figure}
Since $\calP_{\textnormal{V-H}}$ lies on the vertical cut-off line \eqref{verticalcut}, the weak--type estimates for both $T_{a,b,c}$ and $S_{a,b,c}$ fail at $\calP_{\textnormal{V-H}}$, as shown in Section~\ref{Subsec4.1}. Thus, the remaining meaningful endpoint question is the restricted weak--type behavior at $\calP_{\textnormal{V-H}}$. We have the following result.

\begin{thm}\label{restricted-weak Forelli--Rudin at VH}
Let $a, b, c \in \R$ satisfy \eqref{conditionP3}. Then the Forelli--Rudin operators $T_{a, b, c}$ and $S_{a, b, c}$ are bounded from $L^{\frac{1}{b+1}, 1}(\D)$ to $L^{-\frac{1}{a}, \infty}(\D)$. 
\end{thm}

\subsection{Proof of Theorem \ref{restricted-weak Forelli--Rudin at VH}: Restricted weak-type estimates for $T_{a, b, c}$ and $S_{a, b, c}$ at the intersection $\calP_{\textnormal{V-H}}$}

The proof of Theorem~\ref{restricted-weak Forelli--Rudin at VH} is again parallel to that of Theorem~\ref{restricted-weak Forelli--Rudin at CH}, (ii). We therefore give only a \emph{sketch}, emphasizing the necessary modifications.

\medskip 

\noindent\textbf{$\diamond$ \textsl{Step 1.}} Again, it suffices to prove that, for any measurable $F \subseteq \D$, one has
$$
\|S_{a,b,c} \one_{F}\|_{L^{-\frac{1}{a},\infty}(\D)} \lesssim |F|^{1+b}.
$$
Following the argument in \textbf{\textsl{Step 2}} of the proof of Theorem~\ref{restricted-weak Forelli--Rudin at CH}, (ii), we reduce the above estimate to proving the following restricted estimate: for any measurable sets $E,F \subseteq \D$, one has
$$
\int_E \left|S_{a, b, c} \one_F(z) \right| dA(z) \lesssim |E|^{1+a}|F|^{1+b}. 
$$

\medskip 

\noindent\textbf{$\diamond$ \textsl{Step 2.}} In this case, we use a refined pointwise sparse domination argument for the whole kernel
$$
\frac{(1-|z|^2)^a(1-|w|^2)^b}{|1-z \overline{w}|^c}.
$$
This follows from a modification of the argument in \textbf{\textsl{Step 3}} of the proof of Theorem~\ref{restricted-weak Forelli--Rudin at CH}, (ii). More precisely, let $\calD^{(1)}$ and $\calD^{(2)}$ be a pair of adjacent dyadic systems on $\T$. Then, for any $z,w\in\D$, there exist $i\in \{1,2\}$ and $J\in\calD^{(i)}$ such that
\begin{enumerate}
\item $z,w\in Q_J$;
\item $|1-z\overline{w}|^2 \simeq |Q_J|$;
\item there is a unique dyadic descendant $I_1\subseteq J$ such that $z \in Q_{I_1}^{\textrm{up}}$. Say $I_1\in\calD_{k_1}(J)$ for some $k_1\geq 0$, that is, $|I_1|=2^{-k_1}|J|$;
\item there is a unique dyadic descendant $I_2 \subseteq J$ such that $w \in Q_{I_2}^{\textrm{up}}$. Say $I_2\in\calD_{k_2}(J)$ for some $k_2\geq 0$, that is, $|I_2|=2^{-k_2}|J|$.
\end{enumerate}
As a consequence, one has 
$$
\begin{aligned}
&\frac{(1-|z|^2)^a (1-|w|^2)^b}{|1-z\overline{w}|^{c}}\\
& \lesssim  \sum_{i \in \{1, 2\}} \sum_{J\in\calD^{(i)}}
\frac{1}{|Q_J|^{\frac{c-a-b}{2}}}
\left( \sum_{k_1=0}^{\infty} 2^{-k_1a}
\sum_{I_1\in\calD_{k_1}(J)}
\one_{Q_{I_1}^{\textnormal{up}}}(z)\right)
\left( \sum_{k_2=0}^{\infty} 2^{-k_2b}
\sum_{I_2\in\calD_{k_2}(J)}
\one_{Q_{I_2}^{\textnormal{up}}}(w)\right).
\end{aligned}
$$
Hence, 
\begin{align*} 
&\int_{E} |S_{a,b,c} \one_{F}(z)|\, dA(z)=\int_E\int_F \frac{(1-|z|^2)^a (1-|w|^2)^b}{|1-z\bar{w}|^{c}}\, dA(w)dA(z) \nonumber \\
&\lesssim \sum_{i \in \{1, 2\}} \sum_{J\in\calD^{(i)}}
	\frac{1}{|Q_J|^{\frac{c-a-b}{2}}}
	\left(\sum_{k_1=0}^{\infty} 2^{-k_1a}
	\sum_{I_1\in\calD_{k_1}(J)} |Q_{I_1}^{\textrm{up}} \cap E|\right) 
	\left(\sum_{k_2=0}^{\infty} 2^{-k_2b}
	\sum_{I_2\in\calD_{k_2}(J)} |Q_{I_2}^{\textrm{up}} \cap F|\right).
\end{align*}
Therefore, it suffices to show that for any dyadic system $\calD$ on $\T$, one has 
\begin{equation} \label{20260620eq01}
\sum_{J\in \calD}
	\frac{1}{|Q_J|^{\frac{c-a-b}{2}}}
	\left(\sum_{k_1=0}^{\infty} 2^{-k_1a}
	\sum_{I_1\in\calD_{k_1}(J)} |Q_{I_1}^{\textrm{up}} \cap E|\right) 
	\left(\sum_{k_2=0}^{\infty} 2^{-k_2b}
	\sum_{I_2\in\calD_{k_2}(J)} |Q_{I_2}^{\textrm{up}} \cap F|\right) \lesssim |E|^{1+a}|F|^{1+b}.
\end{equation} 

\medskip 

\noindent\textbf{$\diamond$ \textsl{Step 3.}} Observe that, under condition \eqref{conditionP3}, the roles of $a$ and $b$ are symmetric. Moreover, in this case, one has $-1<a<c-b-2<-1-b<0$, and hence the analogous estimate $-1<b<0$ also holds. Therefore, for the two inner sums in \eqref{20260620eq01}, it is enough to estimate the first one. To begin with, observe that
$$
	\sum_{I_1\in \calD_{k_1}(J)} \left| Q^{\textrm{up}}_{I_1} \cap E \right|
	\leq \min\{|Q_J \cap E|, \; 2^{-k_1}|Q_J|\}.
$$
Following the argument in Lemma~\ref{20260613lem01} and using the above estimate, we obtain
\begin{align} \label{20260620eq10}
  \sum_{k_1=0}^{\infty} 2^{-k_1a}
	\sum_{I_1\in\calD_{k_1}(J)} |Q_{I_1}^{\textrm{up}} \cap E|
& \lesssim \sum_{k_1=0}^{\infty} 2^{-k_1a} \min\{|Q_J \cap E|, \; 2^{-k_1}|Q_J|\} \nonumber \\
& \lesssim |Q_J \cap E|^{1+a}|Q_J|^{-a}.
\end{align}
Similarly, one has
\begin{equation} \label{20260620eq11}
\sum_{k_2=0}^{\infty} 2^{-k_2b}
	\sum_{I_2\in\calD_{k_2}(J)} |Q_{I_2}^{\textrm{up}} \cap F|
\lesssim |Q_J \cap F|^{1+b}|Q_J|^{-b}.
\end{equation}
Plugging \eqref{20260620eq10} and \eqref{20260620eq11} back into \eqref{20260620eq01}, we obtain
\begin{align} \label{20260621eq01}
\textrm{LHS of \eqref{20260620eq01}}
&\lesssim \sum_{J\in \calD} \frac{|Q_J \cap E|^{1+a} |Q_J \cap F|^{1+b}}{|Q_J|^{\frac{c+a+b}{2}}} \nonumber \\
&= \sum_{J\in \calD} \left( \frac{|Q_J \cap E|}{|Q_J|} \right)^{1+a}
\left( \frac{|Q_J \cap F|}{|Q_J|} \right)^{1+b}
|Q_J|^{1+\frac{2-c+a+b}{2}} \nonumber \\
&\leq
\left( \sum_{J\in \calD}\left( \frac{|Q_J \cap E|}{|Q_J|} \right)^{2+a+b}
|Q_J|^{1+\frac{2-c+a+b}{2}} \right)^{\frac{1+a}{2+a+b}} \nonumber \\
& \qquad \qquad \cdot 
\left(\sum_{J\in \calD} \left( \frac{|Q_J \cap F|}{|Q_J|} \right)^{2+a+b}
|Q_J|^{1+\frac{2-c+a+b}{2}} \right)^{\frac{1+b}{2+a+b}},
\end{align}
where, in the last inequality, we used H\"older's inequality and the fact that $-1<a,b<0$. Since the roles of $E$ and $F$ are symmetric in the above estimate, it remains to prove the following Carleson embedding estimate: for any dyadic system $\calD$ on $\T$,
\begin{equation} \label{20260620eq20}
 \sum_{J\in \calD} \left( \frac{|Q_J \cap E|}{|Q_J|} \right)^{2+a+b}
 |Q_J|^{1+\frac{2-c+a+b}{2}} \lesssim |E|^{2+a+b}.
\end{equation}
Indeed, temporarily assume \eqref{20260620eq20}. Then, by \eqref{20260621eq01}, we have
$$
\textrm{LHS of \eqref{20260620eq01}}
\lesssim
\left(|E|^{2+a+b}\right)^{\frac{1+a}{2+a+b}}
\left(|F|^{2+a+b}\right)^{\frac{1+b}{2+a+b}}
=|E|^{1+a}|F|^{1+b},
$$
which gives the desired estimate.

\medskip 

\noindent\textbf{$\diamond$ \textsl{Step 4.}} 
Finally, we prove \eqref{20260620eq20}. The proof is similar to that of \eqref{Goal_2}. First observe that, by $0<c-2-a-b<1$ and $c<1$,
$$
\frac{1}{2}<1+\frac{2-c+a+b}{2}<1
$$
and
$$
2+a+b<3-c+a+b=2\left(1+\frac{2-c+a+b}{2}\right)-1.
$$

\vspace{0.1cm}
Choose $\rho$ such that
$$
2+a+b<\rho<3-c+a+b.
$$
In particular, since $1+\frac{2-c+a+b}{2}<1$, we have $0<\rho<1$.

\vspace{0.1cm}

We first record the following observation.

\begin{obs} \label{20260621obs01}
For every $J\in\calD$,
$$
\sum_{I\in\calD(J)} |Q_I|^{1+\frac{2-c+a+b}{2}}
\lesssim |Q_J|^{1+\frac{2-c+a+b}{2}}.
$$
Here, $\calD(J)$ denotes the collection of dyadic descendants of $J$, and $\calD_k(J)$ denotes the collection of the $k$-th generation dyadic descendants of $J$.
\end{obs}

\begin{proof} 
Indeed, using $|Q_I|\simeq |I|^2$, we have
\begin{align*}
\sum_{I\in\calD(J)} |Q_I|^{1+\frac{2-c+a+b}{2}}
&\lesssim \sum_{k=0}^{\infty}\sum_{I\in\calD_k(J)}
|I|^{2\left(1+\frac{2-c+a+b}{2}\right)} \\
&= \sum_{k=0}^{\infty} 2^k \left(2^{-k}|J|\right)^{2\left(1+\frac{2-c+a+b}{2}\right)} \\
&= |J|^{2\left(1+\frac{2-c+a+b}{2}\right)}
\sum_{k=0}^{\infty}2^{-k(3-c+a+b)} \\
&\lesssim |Q_J|^{1+\frac{2-c+a+b}{2}},
\end{align*}
where the last estimate follows from $3-c+a+b>0$.
\end{proof}

For $J\in\calD$, denote
$$
e_J:=\frac{|Q_J\cap E|}{|Q_J|}.
$$
Then
$$
e_J^{2+a+b}
=(2+a+b)\int_0^{e_J} s^{1+a+b}\,ds
=(2+a+b)\int_0^1 s^{1+a+b}\one_{\{s<e_J\}}\,ds.
$$
For $0<s<1$, set
$$
\calG_s:=\{J\in\calD:\ s<e_J\}
=\left\{J\in\calD:\frac{|Q_J\cap E|}{|Q_J|}>s\right\}.
$$
Therefore,
\begin{align} \label{20260620eq21}
\textnormal{LHS of \eqref{20260620eq20}}
=&\sum_{J\in\calD} e_J^{2+a+b}|Q_J|^{1+\frac{2-c+a+b}{2}} \nonumber \\
\simeq& \sum_{J\in\calD}
\left(\int_0^1 s^{1+a+b}\one_{\{s<e_J\}}\,ds\right)
|Q_J|^{1+\frac{2-c+a+b}{2}} \nonumber \\
=& \int_0^1 s^{1+a+b}
\left(\sum_{J\in\calD}
|Q_J|^{1+\frac{2-c+a+b}{2}}\one_{\{s<e_J\}}\right)\,ds \nonumber \\
=& \int_0^1 s^{1+a+b}
\left(\sum_{J\in\calG_s}
|Q_J|^{1+\frac{2-c+a+b}{2}}\right)\,ds.
\end{align}
We next estimate the inner sum in \eqref{20260620eq21}. 
Let $\calG_s^{\max}$ be the collection of maximal elements of $\calG_s$. By Observation~\ref{20260621obs01},
\begin{equation} \label{20260621eq19}
\sum_{J \in\calG_s}|Q_J|^{1+\frac{2-c+a+b}{2}}
\lesssim
\sum_{J\in\calG_s^{\max}} |Q_J|^{1+\frac{2-c+a+b}{2}}.
\end{equation}
For each $J\in\calG_s^{\max}$, one has $|Q_J\cap E|>s|Q_J|$. Hence
\begin{align*}
|Q_J|^{1+\frac{2-c+a+b}{2}}
&=|Q_J|^\rho |Q_J|^{1+\frac{2-c+a+b}{2}-\rho} \\
&\leq s^{-\rho}|Q_J\cap E|^\rho
|Q_J|^{1+\frac{2-c+a+b}{2}-\rho} \\
&\lesssim s^{-\rho}|Q_J\cap E|^\rho
|J|^{2\left(1+\frac{2-c+a+b}{2}-\rho\right)}.
\end{align*}
Therefore,
\begin{equation} \label{20260621eq20}
\sum_{J\in\calG_s^{\max}} |Q_J|^{1+\frac{2-c+a+b}{2}}
\lesssim s^{-\rho}
\sum_{J\in\calG_s^{\max}}
|Q_J\cap E|^\rho
|J|^{2\left(1+\frac{2-c+a+b}{2}-\rho\right)}.
\end{equation}
By H\"older and using the fact that $0<\rho<1$,
\begin{align} \label{20260621eq21}
&\sum_{J\in\calG_s^{\max}}
|Q_J\cap E|^\rho
|J|^{2\left(1+\frac{2-c+a+b}{2}-\rho\right)} \nonumber \\
& \leq \left(\sum_{J\in\calG_s^{\max}} |Q_J\cap E|\right)^\rho
\left(
\sum_{J\in\calG_s^{\max}}
|J|^{\frac{2\left(1+\frac{2-c+a+b}{2}-\rho\right)}{1-\rho}}
\right)^{1-\rho} \nonumber \\
&\leq |E|^\rho
\left(
\sum_{J\in\calG_s^{\max}}
|J|^{\frac{2\left(1+\frac{2-c+a+b}{2}-\rho\right)}{1-\rho}}
\right)^{1-\rho}.
\end{align}
Since $\rho<3-c+a+b$, we have
$$
\frac{2\left(1+\frac{2-c+a+b}{2}-\rho\right)}{1-\rho}>1.
$$
Moreover, since the intervals in $\calG_s^{\max}$ are pairwise disjoint, we see that 
\begin{equation} \label{20260621eq22}
\sum_{J\in\calG_s^{\max}}
|J|^{\frac{2\left(1+\frac{2-c+a+b}{2}-\rho\right)}{1-\rho}}
\lesssim 1.
\end{equation} 
Combining \eqref{20260621eq19}, \eqref{20260621eq20}, \eqref{20260621eq21}, and \eqref{20260621eq22}, we deduce that 
$$
\sum_{I\in\calG_s}|Q_I|^{1+\frac{2-c+a+b}{2}}
\lesssim \left(\frac{|E|}{s}\right)^\rho.
$$
On the other hand, since $1+\frac{2-c+a+b}{2}>\frac12$, we also have
\begin{align*}
\sum_{J\in\calD}|Q_J|^{1+\frac{2-c+a+b}{2}}
&\lesssim \sum_{k=0}^{\infty}\sum_{J\in\calD_k}
|J|^{2\left(1+\frac{2-c+a+b}{2}\right)} \\
&= \sum_{k=0}^{\infty}2^k2^{-2k\left(1+\frac{2-c+a+b}{2}\right)}
\lesssim 1.
\end{align*}
Thus
$$
\sum_{J\in\calG_s}|Q_J|^{1+\frac{2-c+a+b}{2}}
\lesssim
\min\left\{1,\left(\frac{|E|}{s}\right)^\rho\right\}.
$$
Using the assumption that $\rho>2+a+b$, we obtain
\begin{align*}
\textnormal{RHS of \eqref{20260620eq21}}
&\lesssim \int_0^1 s^{1+a+b}
\min\left\{1,\left(\frac{|E|}{s}\right)^\rho\right\}\,ds \\
&\leq \int_0^{|E|} s^{1+a+b}\,ds
+ |E|^\rho \int_{|E|}^1 s^{1+a+b-\rho}\,ds \\
&\lesssim |E|^{2+a+b}.
\end{align*}
This proves \eqref{20260620eq20}, and the proof for Theorem \ref{restricted-weak Forelli--Rudin at VH} is complete. 

\bigskip 

\section{Analysis of the Main Case: Restricted Weak-type Bounds at the Critical--Horizontal--Vertical Intersection $\calP_{\textnormal{C-H-V}}$ and the Forelli--Rudin Phenomenon at the level of restricted weak-type estimates} \label{Sec09}

Finally, we treat the last endpoint configuration in our analysis under the Main case~\eqref{maincase}, namely the case where the critical line, the horizontal cut-off line, and the vertical cut-off line meet at a common point:
$$
\calP_{\textnormal{C-H-V}}=(b+1,-a).
$$
This is equivalent to the condition
$$
\begin{cases}
a,b,c \in \R \quad \textrm{satisfy \eqref{maincase}}; \\
c=1,
\end{cases}
$$
or, equivalently,
\begin{equation}  \tag{${\bf P}_4$}
\label{conditionP4}
\begin{cases}
a+b<-1, \quad a,b>-1; \\
c=1.
\end{cases}
\end{equation}
Here, the condition $c=1$ guarantees that the three lines intersect at the same point. This situation can happen when $\calP_{\textnormal{V-H}}$ lies in Region V (see Figure \ref{Fig15}). 

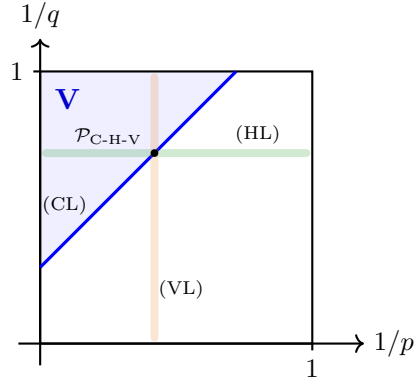
\begin{figure}[ht]
\centering
\begin{tikzpicture}[scale=3.6]
\def\yh{0.28}       
\def\xc{0.72}        
\def\xright{1.15}
\def\ybottom{-0.05}          

\fill[blue!20, opacity=0.3]   (0,\yh) -- (0,1) -- (\xc,1) -- cycle;  

\draw[->, thick] (-0.08,0) -- (\xright+0.04,0) node[right] {$1/p$};
\draw[->, thick] (0,\ybottom-0.03) -- (0,1.12) node[above] {$1/q$};

\draw[thick] (0,0) rectangle (1,1);
\draw[thick] (1,0) -- (1,-0.025) node[below] {$1$};
\draw[thick] (0,1) -- (-0.025,1) node[left] {$1$};

\draw[blue, very thick] (0,\yh) -- (\xc,1);
\node[above] at (\xc/2-0.27, \yh/2+0.3) {\scriptsize \eqref{criticalline}};

\node[font=\bfseries\large, blue!80!black] at (0.1,0.9) {V};

\draw[green!55!black, line width=3pt, opacity=0.18, line cap=round,
      shorten <=2.2pt, shorten >=2.2pt]
    (0,0.7) -- (1,0.7);

\draw[orange!90!black, line width=3pt, opacity=0.18, line cap=round,
      shorten <=2.2pt, shorten >=2.2pt]
    (0.42, 0) -- (0.42, 1);

\fill[black] (0.42,  0.7) circle (.4pt);
\node[above] at (0.25, 0.7) {\scriptsize $\calP_{\textnormal{C-H-V}}$};

\node[right] at (0.4, 0.2) {\scriptsize \eqref{verticalcut}};

\node[above] at (0.8, 0.7) {\scriptsize \eqref{horizontalcut}};
\end{tikzpicture}
\caption{\small The case when the three lines meet at a common point}
\label{Fig15}
\end{figure}

Again, by the reduction in Section~\ref{Subsec4.1}, it remains only to consider the restricted weak--type estimates at this point. The behavior of $T_{a,b,c}$ and $S_{a,b,c}$ in the present case is quite different from that in the previous situations, where the relevant intersection point was formed by only two of the three lines and stayed away from the third. More precisely, in the present situation the Forelli--Rudin phenomenon occurs at the level of restricted weak--type estimates. The precise result is as follows: 

\begin{thm} \label{Forelli--Rudin Threshold restricted weak-type}
Let $a, b, c \in \R$ satisfy \eqref{conditionP4}. Then the following statement holds.
\begin{enumerate}
    \item[(i)] The Forelli--Rudin type operator $S_{a, b, 1}$ is unbounded from $L^{\frac{1}{b+1}, 1}(\D)$ to $L^{-\frac{1}{a}, \infty}(\D)$.
    \item[(ii)] The Forelli--Rudin type operator $T_{a, b, 1}$ is bounded from $L^{\frac{1}{b+1}, 1}(\D)$ to $L^{-\frac{1}{a}, \infty}(\D)$.
\end{enumerate}
\end{thm}

Note that Theorem \ref{Forelli--Rudin Threshold restricted weak-type} asserts that the \emph{Forelli--Rudin phenomenon occurs at the level of restricted weak-type estimates at the critical--horizontal--vertical intersection $\calP_{\textnormal{C-H-V}}$}. In other words, $\calP_{\textnormal{C-H-V}}$ is a Forelli--Rudin pair at the level of restricted weak-type estimates.

\vspace{0.1cm}

We now turn to the proof of Theorem \ref{Forelli--Rudin Threshold restricted weak-type}.

\subsection{Proof of Theorem \ref{Forelli--Rudin Threshold restricted weak-type}, (i): Failure of the restricted weak-type estimate for $S_{a, b, 1}$ at the intersection $\calP_{\textnormal{C-H-V}}$ }

For sufficiently small $\varepsilon>0$, denote 
$$
	F_{\varepsilon}:=\{z\in\D: \varepsilon<1-|z|^2<2\varepsilon\}.
$$
Hence, $|F_{\varepsilon}| \simeq \varepsilon$.
For each $z\in F_{\varepsilon}$, we define
$$
	E_{\varepsilon}(z)=\left\{w \in F_\varepsilon:  0<\arg z - \arg w<\frac{1}{10}\right\}.
$$
Then for any $z\in F_{\varepsilon}$, we have
\begin{align*}
	S_{a,b,1} \one_{F_{\varepsilon}}(z)
	=&(1-|z|^2)^a \int_{\D} \frac{(1-|w|^2)^b}{|1-z\overline{w}|} \one_{F_{\varepsilon}}(w) dA(w)\\
	\gtrsim & \varepsilon^{a+b} \int_{E_{\varepsilon}(z)} \frac{1}{|1-z\overline{w}|} dA(w)\\
	\gtrsim & \varepsilon^{a+b} \int_{\sqrt{1-2\varepsilon}}^{\sqrt{1-\varepsilon}}\int_{0}^{\frac{1}{10}} \frac{\rho}{\left|1-|z|\rho e^{-i\theta}\right|} d\theta d\rho.
\end{align*}
Note that
\begin{align*}
	\left|1-|z|\rho e^{-i\theta}\right|
	=&\left|1-|z|\rho+|z|\rho-|z|\rho e^{-i\theta}\right|\\
	\leq&(1-|z|\rho)+|z|\rho\left|1-e^{-i\theta}\right|\\
	\leq&(1-|z|)+|z|(1-\rho)+\left|1-e^{-i\theta}\right|\\
	\leq&(1-|z|^2)+(1-\rho^2)+\theta\\
	\leq&4\varepsilon+\theta.
\end{align*}
Hence, 
\begin{align*}
	S_{a,b,1} \one_{F_{\varepsilon}}(z)
	\gtrsim & \varepsilon^{a+b} \int_{\sqrt{1-2\varepsilon}}^{\sqrt{1-\varepsilon}}\int_{0}^{\frac{1}{10}} \frac{\rho}{4\varepsilon+\theta} d\theta d\rho\\
    \simeq & \varepsilon^{a+b+1}\int_{0}^{\frac{1}{10}} \frac{1}{4\varepsilon+\theta} d\theta \\
    \gtrsim &\varepsilon^{a+b+1}\log\frac{1}{\varepsilon}.
\end{align*}
Therefore, 
$$
	F_{\varepsilon}\subset\left\{z\in\D: S_{a,b,1} \one_{F_{\varepsilon}}(z) \gtrsim \varepsilon^{a+b+1}\log\frac{1}{\varepsilon}\right\}.
$$
It follows from $|F_{\varepsilon}|\simeq \varepsilon$ that 
$$
	\|S_{a,b,1} \one_{F_{\varepsilon}}\|_{L^{-\frac{1}{a},\infty}(\D)}
	\gtrsim \left(\varepsilon^{a+b+1}\log\frac{1}{\varepsilon}\right) |F_{\varepsilon}|^{-a}
	\simeq \varepsilon^{b+1}\log\frac{1}{\varepsilon},
$$
which implies 
$$
	\frac{\|S_{a,b,1} \one_{F_{\varepsilon}}\|_{L^{-\frac{1}{a},\infty}(\D)}}{|F_\varepsilon|^{b+1}} 
	\gtrsim \log\frac{1}{\varepsilon} 
	\to \infty \qquad \textrm{as} \quad \varepsilon \to 0.
$$
This completes the proof of Theorem \ref{Forelli--Rudin Threshold restricted weak-type}, (i). 

\subsection{Proof of Theorem \ref{Forelli--Rudin Threshold restricted weak-type}, (ii): Restricted weak-type estimate for $T_{a, b, 1}$ at the intersection $\calP_{\textnormal{C-H-V}}$ } \label{20260831subsec54}

The key idea to prove this positive result is to make use of the weak-type estimates for Hardy functions; see Lemma \ref{weakHardy}. Since $-1<a, b<0$, and $a+b<-1$, we have
$$
	1<-\frac{1}{a}<\frac{1}{b+1}<\infty.
$$
Hence we can take 
$$
	-\frac{1}{a}<r<\frac{1}{b+1}.
$$
Moreover, write 
\begin{equation} \label{20260701eq04}
	T_{a,b,1} f(z)=(1-|z|^2)^a T_{0,b,1} f(z).
\end{equation} 
We first show that 
\begin{equation} \label{202606701eq03}
	\|T_{0,b,1} f\|_{H^{r}(\D)} \lesssim \|f\|_{L^{\frac{1}{b+1},1}(\D)}.
\end{equation} 
By Taylor expansion, 
\begin{align*}
T_{0,b,1} f (z) 
&= \int_{\D} \frac {(1-|w|^{2}) ^{b}}{1 - z\bar{w}} f(w) dA(w) \\
&= \sum_{m=0}^\infty \left(\int_{\D} (1-|w|^{2}) ^{b} f(w) \overline{w}^m dA(w) \right) z^m.
\end{align*}
For $g\in H^{r'}(\D)$, assume
$$
	g(z)=\sum_{m=0}^\infty a_m z^m.
$$
Therefore, 
\begin{align*}
	\langle T_{0,b,1} f, g \rangle_{\T}
	=&\sum_{m=0}^\infty \left(\int_{\D} (1-|w|^{2}) ^{b} f(w) \overline{w}^m dA(w) \right)  \overline{a_m}\\
	=&\int_{\D} (1-|w|^{2}) ^{b} f(w)\cdot \left( \;  \overline{\sum_{m=0}^\infty a_m w^m} \; \right) dA(w)\\
	=&\int_{\D} (1-|w|^{2}) ^{b} f(w) \overline{g(w)} dA(w)\\
	=&\int_{\D} f(w) \overline{(1-|w|^{2}) ^{b} g(w)} dA(w) \\
    =&\int_{\D} f(w) \overline{M_{-\frac{1}{b}} g(w)} dA(w) 
\end{align*}
where we recall that $M_{-\frac{1}{b}} g(w):=(1-|w|^2)^bg(w)$ and $\langle \cdot, \cdot \rangle_{\T}$ refers to the standard duality between Hardy spaces; see \eqref{Hardyduality}. 
Therefore, by Lorentz-H\"older's inequality
\begin{equation} \label{20260701eq01}
	|\langle T_{0,b,1} f, g \rangle_{\T}| 
	\lesssim  \|f\|_{L^{\frac{1}{b+1},1}(\D)} \|M_{-\frac{1}{b}}g\|_{L^{-\frac{1}{b},\infty}(\D)}.
\end{equation} 
Since $r<\frac{1}{b+1}$, we have
$$
	r'>-\frac{1}{b}.
$$
Hence, by Lemma \ref{weakHardy}, we have
\begin{equation} \label{20260701eq02}
	\|M_{-\frac{1}{b}}g\|_{L^{-\frac{1}{b},\infty}(\D)} \lesssim \|g\|_{H^{r'}(\D)}.
\end{equation} 
Combining \eqref{20260701eq01} and \eqref{20260701eq02}, we deduce that
$$
|\langle T_{0,b,1} f, g \rangle_{\T}| \lesssim \left\|f \right\|_{L^{\frac{1}{b+1}, 1}(\D)} \left\|g \right\|_{H^{r'}(\D)}, 
$$
which clearly gives the desired estimate \eqref{202606701eq03}.

To this end, since $-\frac{1}{a}<r$, another application of Lemma \ref{weakHardy}, together with \eqref{20260701eq04} and then \eqref{202606701eq03}, gives
\begin{align*}
	\|T_{a,b,1} f\|_{L^{-\frac{1}{a},\infty}(\D)}
	&=\|(1-|\cdot|^2)^a T_{0,b,1} f\|_{L^{-\frac{1}{a},\infty}(\D)} \\
    &=\left\|M_{-\frac{1}{a}} \left(T_{0, b, 1}f \right) \right\|_{L^{-\frac{1}{a},\infty}(\D)} \\
	&\lesssim \|T_{0,b,1} f\|_{H^{r}(\D)} \\
	&\lesssim \|f\|_{L^{\frac{1}{b+1},1}(\D)}.
\end{align*}
The proof of Theorem \ref{Forelli--Rudin Threshold restricted weak-type}, (ii) is complete. 

\bigskip 

\section{Analysis of the Main Case: Putting all pieces together} \label{Sec10}

In this section, we put together the building blocks developed in Sections \ref{Sec05}--\ref{Sec09} to obtain a full picture of the $L^p$ theory for the Forelli--Rudin type operators $T_{a,b,c}$ and $S_{a,b,c}$ in the \emph{Main Case} \eqref{maincase}. 
In this Case, conditions \eqref{eq:weak} and \eqref{eq:restricted-weak} are equivalent to 
\[
	\begin{cases}
		a,b,c \in \R \quad \textrm{satisfy \eqref{maincase}}, \\
		1/2 \le 2-c+b <1, \quad a<0, \\
		c>1,
	\end{cases}
	\qquad
	\begin{cases}
		a,b,c \in \R \quad \textrm{satisfy \eqref{maincase}}, \\
		c=1,
	\end{cases}
\]
respectively.
In the corresponding cases, the weak and restricted weak Forelli--Rudin phenomena occur at $\mathcal{P}_{\mathrm{C-H}}$ and $\mathcal{P}_{\mathrm{C-H-V}}$, respectively.

On the other hand, at each of the two cases, the operators have identical boundedness behavior away from the corresponding Forelli--Rudin pairs.
Recall that the strong-type region $\Omega(a,b,c)$ defined by \eqref{eq:strong} is nonempty and is uniquely determined by     
\[
	\eqref{criticalline}:\ 
	\frac{1}{q}=\frac{1}{p}+c-a-b-2,
	\qquad
	\eqref{verticalcut}:\ 
	\frac{1}{p}=b+1,
	\qquad
	\eqref{horizontalcut}:\ 
	\frac{1}{q}=-a.
\]
Denote
\[
	\partial\Omega_{bd}=\overline{\Omega(a,b,c)}\setminus \Omega(a,b,c).
\]
Section~\ref{Sec10.1} shows that the restricted weak-type region cannot extend beyond $\overline{\Omega(a,b,c)}$, and hence neither can the weak-type region.
Therefore, in what follows, we restrict our attention to the boundary $\partial\Omega_{\mathrm{bd}}$.
Recall from the beginning of Section~\ref{Sec04} that the boundedness behavior of the Forelli--Rudin operators at the intersection points $\calP_{\mathrm{C-H}}$, $\calP_{\mathrm{C-V}}$, and $\calP_{\mathrm{V-H}}$ plays a crucial role. 
However, such a point is relevant only when it lies on $\partial\Omega_{\mathrm{bd}}$. 
In this case, we call the corresponding intersection point active.
See Figure \ref{20260906Fig01} for the active intersection points in each case and the boundedness behavior of the Forelli--Rudin operators at corresponding points.

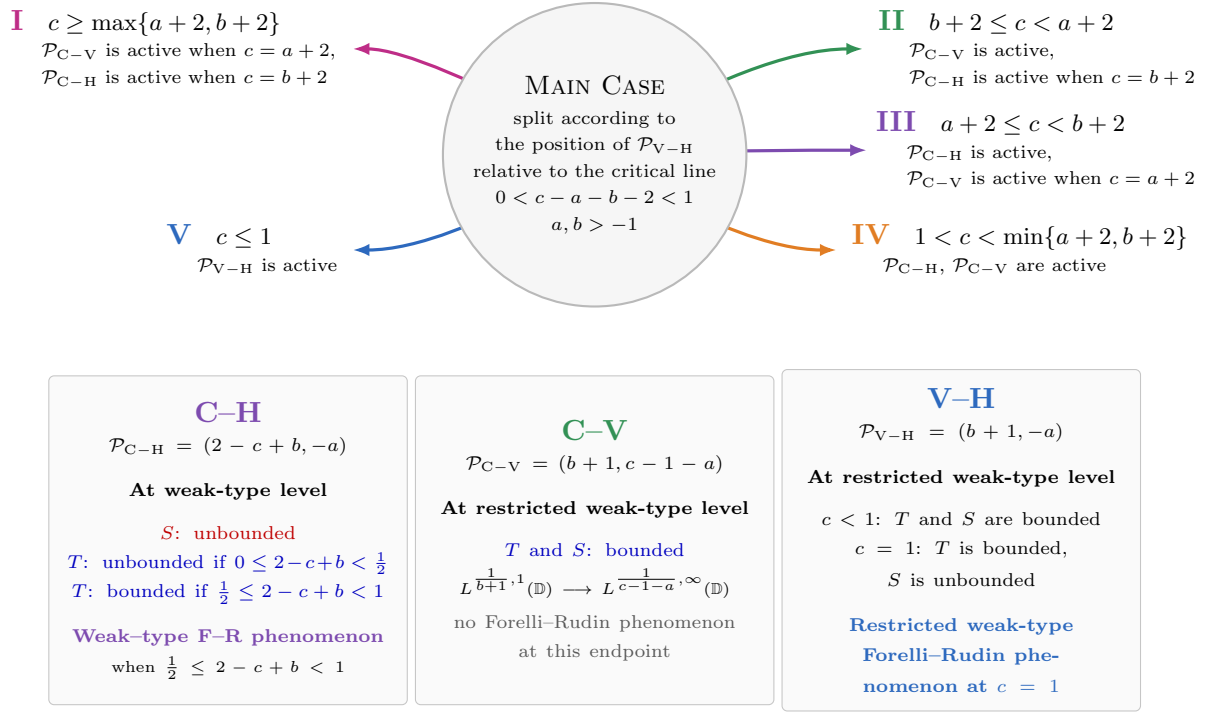
\begin{figure}[ht]
\centering
\begin{tikzpicture}[
    >=Latex,
    every path/.style={line cap=round, line join=round},
    branch/.style={very thick, -{Latex[length=2.2mm]}},
    case label/.style={font=\small, align=left},
    point label/.style={font=\scriptsize, align=center, text=gray!75!black}
]

\definecolor{caseI}{RGB}{190,45,135}
\definecolor{caseII}{RGB}{50,145,85}
\definecolor{caseIII}{RGB}{125,75,175}
\definecolor{caseIV}{RGB}{225,125,35}
\definecolor{caseV}{RGB}{45,105,190}

\draw[branch,caseI]
    (-0.52,1.88) .. controls (-1.45,2.42) and (-2.55,2.90) .. (-3.20,2.90);

\draw[branch,caseII]
    (0.52,1.88) .. controls (1.45,2.42) and (2.55,2.90) .. (3.55,2.90);

\draw[branch,caseIII]
    (0.90,1.52) .. controls (1.75,1.56) and (2.75,1.56) .. (3.6,1.56);

\draw[branch,caseIV]
    (0.52,1.12) .. controls (1.45,0.62) and (2.35,0.24) .. (3.20,0.24);

\draw[branch,caseV]
    (-0.52,1.12) .. controls (-1.45,0.62) and (-2.35,0.24) .. (-3.20,0.24);

\node[case label, anchor=east] at (-3.3,2.90)
    {\textcolor{caseI}{\large\bfseries I}\quad
    $c\ge \max\{a+2,b+2\}$\\[-1pt]
    \hspace*{1.25em}{\scriptsize $\calP_{\mathrm{C-V}}$ is active when $c=a+2$, } \\[-1pt] 
    \hspace*{1.25em}{\scriptsize $\calP_{\mathrm{C-H}}$ is active when $c=b+2$}};

\node[case label, anchor=west] at (3.63,2.90)
    {\textcolor{caseII}{\large\bfseries II}\quad
    $b+2\le c<a+2$\\[-1pt]
    \hspace*{1.25em}{\scriptsize $\calP_{\mathrm{C-V}}$ is active,}\\[-1pt]
    \hspace*{1.25em}{\scriptsize $\calP_{\mathrm{C-H}}$ is active when $c=b+2$ }};

\node[case label, anchor=west] at (3.6,1.56)
    {\textcolor{caseIII}{\large\bfseries III}\quad
    $a+2\le c<b+2$\\[-1pt]
    \hspace*{1.25em}{\scriptsize $\calP_{\mathrm{C-H}}$ is active,}\\[-1pt]
    \hspace*{1.25em}{\scriptsize $\calP_{\mathrm{C-V}}$ is active when $c=a+2$}};

\node[case label, anchor=west] at (3.28,0.24)
    {\textcolor{caseIV}{\large\bfseries IV}\quad
    $1<c<\min\{a+2,b+2\}$\\[-1pt]
    \hspace*{1.25em}{\scriptsize
    $\calP_{\mathrm{C-H}}$, $\calP_{\mathrm{C-V}}$ are active}};

\node[case label, anchor=east] at (-3.28,0.24)
    {\textcolor{caseV}{\large\bfseries V}\quad
    $c\le 1$\\[-1pt]
    \hspace*{1.25em}{\scriptsize
    $\calP_{\mathrm{V-H}}$ is active}};

\node[circle, draw=gray!55, fill=gray!7, line width=.8pt,
      minimum size=2.65cm, inner sep=2pt, align=center] at (0,1.50)
      {\textsc{Main Case}\\[-1pt]
       {\scriptsize split according to}\\[-2pt]
       {\scriptsize the position of $\calP_{\mathrm{V-H}}$}\\[-2pt]
       {\scriptsize relative to the critical line}\\[-2pt]
       {\scriptsize $0<c-a-b-2<1$}\\[-2pt]
       {\scriptsize $a,b>-1$}};

\tikzset{
    endpoint card/.style={
        draw=gray!45,
        rounded corners=2pt,
        fill=gray!4,
        text width=4.25cm,
        minimum height=4.35cm,
        inner sep=7pt,
        align=center
    }
}

\node[endpoint card] at (-4.85,-3.60) {
    \textcolor{caseIII}{\large\bfseries C--H}\\[-1pt]
    {\scriptsize $\calP_{\mathrm{C-H}}=(2-c+b,-a)$}\\[5pt]
    {\scriptsize\bfseries At weak-type level}\\[4pt]
    \textcolor{red!75!black}{\scriptsize $S$: unbounded}\\[-1pt]
    \textcolor{blue!75!black}{\scriptsize $T$: unbounded if $0\le 2-c+b<\frac12$}\\[-1pt]
    \textcolor{blue!75!black}{\scriptsize $T$: bounded if $\frac12\le 2-c+b<1$}\\[5pt]
    \textcolor{caseIII}{\scriptsize\bfseries Weak--type F--R phenomenon}\\[-1pt]
    {\tiny when $\frac12\le 2-c+b<1$}
};

\node[endpoint card] at (0,-3.60) {
    \textcolor{caseII}{\large\bfseries C--V}\\[-1pt]
    {\scriptsize $\calP_{\mathrm{C-V}}=(b+1,c-1-a)$}\\[5pt]
    {\scriptsize\bfseries At restricted weak-type level}\\[4pt]
    \textcolor{blue!75!black}{\scriptsize $T$ and $S$: bounded}\\[2pt]
    {\tiny $L^{\frac1{b+1},1}(\D)\longrightarrow
    L^{\frac1{c-1-a},\infty}(\D)$}\\[5pt]
    \textcolor{gray!70!black}{\scriptsize no Forelli--Rudin phenomenon\\[-1pt]
    at this endpoint}
};

\node[endpoint card] at (4.85,-3.60) {
    \textcolor{caseV}{\large\bfseries V--H}\\[-1pt]
    {\scriptsize $\calP_{\mathrm{V-H}}=(b+1,-a)$}\\[5pt]
    {\scriptsize\bfseries At restricted weak-type level}\\[4pt]
    {\scriptsize $c<1$: $T$ and $S$ are bounded}\\[-1pt]
    {\scriptsize $c=1$: $T$ is bounded, $S$ is unbounded}\\[5pt]
    \textcolor{caseV}{\scriptsize\bfseries Restricted weak-type}\\[-1pt]
    \textcolor{caseV}{\scriptsize\bfseries Forelli--Rudin phenomenon at $c=1$}
};

\end{tikzpicture}
\caption{Summary of the Main Cases}
\label{20260906Fig01}
\end{figure}

Before proceeding to the proofs, we summarize the conclusions.
Sections~\ref{Sec10.2}--\ref{Sec10.6} determine the weak-type and restricted weak-type regions, with the weak-type points for $T_{a,b,c}$ and $S_{a,b,c}$ coinciding, apart from the corresponding Forelli--Rudin pairs.
More precisely, the following statements hold on $\partial\Omega_{\mathrm{bd}}$ away from the corresponding Forelli--Rudin pairs:
\begin{itemize}
	\item Both $T_{a,b,c}$ and $S_{a,b,c}$ are of weak type at every point except the point $\calP_{\mathrm{C-H}}$ and the vertical cut-off line \eqref{verticalcut}, whenever they belong to the boundary.
	\item Both $T_{a,b,c}$ and $S_{a,b,c}$ are of restricted weak type at every point except for the point $\calP_{\textnormal{C-H}}$ when $2-c+b=0$.
\end{itemize}

\vspace{0.1cm}

\subsection{Sharpness of the Estimates in the Main Cases}\label{Sec10.1}
We first show that neither operator satisfies a restricted
weak-type estimate outside $\overline{\Omega(a,b,c)}$.

\begin{prop}
Let $a,b,c$ satisfy \eqref{maincase}.
If 
\[
	\left(\frac{1}{p}, \frac{1}{q}\right) \in ([0,1]\times[0,1]) \setminus \overline{\Omega(a,b,c)},
\]
then neither $T_{a,b,c}$ nor $S_{a,b,c}$ is of restricted weak type $(p,q)$.
In particular, neither operator is of weak type $(p,q)$.
\end{prop}

\begin{proof}
Assume 
\begin{equation}\label{Main: outside}
	\left(\frac{1}{p}, \frac{1}{q}\right) \in ([0,1]\times[0,1]) \setminus \overline{\Omega(a,b,c)}.
\end{equation}
We split the proof into three cases.

\vspace{0.1cm}

\noindent\textbf{Case I: $p\neq\infty,\ q\neq1$.}
We prove this case by contradiction.
Suppose $T_{a,b,c}$ or $S_{a,b,c}$ is bounded from $L^{p,1}(\D)$ to $L^{q,\infty}(\D)$.
By \eqref{Main: outside}, there exist $1<p_1,q_1<\infty$ such that $p<p_1<\infty,\ 1<q_1<q$ and 
\[
	\left(\frac{1}{p_1}, \frac{1}{q_1}\right) \in ([0,1]\times[0,1]) \setminus \overline{\Omega(a,b,c)}.
\]
Since
\[
	L^{p_1}(\D) \hookrightarrow L^{p,1}(\D),
	\quad
	L^{q,\infty}(\D) \hookrightarrow L^{q_1}(\D),
\]
by assumed boundedness, we have $T_{a,b,c}$ or $S_{a,b,c}$ is bounded from $L^{p_1}(\D)$ to $L^{q_1}(\D)$.
But $\left(\frac{1}{p_1}, \frac{1}{q_1}\right)$ lies outside of $\Omega(a,b,c)$, which leads to a contradiction.
Hence, for $p,q$ satisfying \eqref{Main: outside} with $p\neq\infty,\ q\neq1$, we have
\[
	T_{a,b,c},\ S_{a,b,c}: L^{p,1}(\D) \not\to L^{q,\infty}(\D).
\]

\vspace{0.1cm}

\noindent\textbf{Case II: $p=\infty$.}
We prove this case by contradiction.
Suppose $T_{a,b,c}$ or $S_{a,b,c}$ is of restricted weak type $(\infty,q)$.

In this case, the condition $\left(\frac{1}{p}, \frac{1}{q}\right) \in \Omega(a,b,c)$ can be simplified to
\[
	\frac{1}{q}>\max\{-a,c-a-b-2\},
\]
and \eqref{Main: outside} can be simplified to
\[
	\frac{1}{q}<\max\{-a,c-a-b-2\}.
\]
Since $\Omega(a,b,c)$ is not empty, choose a pair $\left(\frac{1}{p_0}, \frac{1}{q_0}\right)\in\operatorname{int}\Omega(a,b,c)$.
Then $1<p_0<\infty$, $1<q_0<q$ and 
\[
	T_{a,b,c},\ S_{a,b,c}: L^{p_0}(\D) \to L^{q_0}(\D)
\]
are bounded.
Interpolating this estimate with restricted weak type $(\infty,q)$ boundedness of $T_{a,b,c}$ or $S_{a,b,c}$ by the off-diagonal Marcinkiewicz interpolation theorem \cite[Theorem~1.4.19]{Grafakos2014}, there exists some $1<p_1, q_1<\infty$ such that
\begin{enumerate}
	\item $\left(\frac{1}{p_1}, \frac{1}{q_1}\right) \in ([0,1]\times[0,1]) \setminus \overline{\Omega(a,b,c)}$;
	\item $T_{a, b, c}$ or $S_{a,b,c}$ is bounded from $L^{p_1,1}(\D)$ to $L^{q_1, \infty}(\D)$.
\end{enumerate}
This contradicts Case I.

\vspace{0.1cm}

\noindent\textbf{Case III: $p\neq\infty, q=1$.}
The same argument as in Case II applies.
In this case, the condition $\left(\frac{1}{p}, \frac{1}{q}\right) \in \Omega(a,b,c)$ can be simplified to
\[
	\frac{1}{p}<\min\{b+1,3+a+b-c\},
\]
and \eqref{Main: outside} can be simplified to
\[
	\frac{1}{p}>\min\{b+1,3+a+b-c\}.
\]
Since $\Omega(a,b,c)$ is not empty, choose a pair $\left(\frac{1}{p_0}, \frac{1}{q_0}\right)\in\operatorname{int}\Omega(a,b,c)$.
Then $p<p_0<\infty$, $1<q_0<\infty$ and
\[
	T_{a,b,c},\ S_{a,b,c}: L^{p_0}(\D) \to L^{q_0}(\D)
\]
are bounded.
Interpolating this estimate with restricted weak type $(p,1)$ boundedness of $T_{a,b,c}$ or $S_{a,b,c}$ by the off-diagonal Marcinkiewicz interpolation theorem \cite[Theorem~1.4.19]{Grafakos2014}, there exists some $1<p_1, q_1<\infty$ such that
\begin{enumerate}
	\item $\left(\frac{1}{p_1}, \frac{1}{q_1}\right) \in ([0,1]\times[0,1]) \setminus \overline{\Omega(a,b,c)}$;
	\item $T_{a, b, c}$ or $S_{a,b,c}$ is bounded from  $L^{p_1,1}(\D)$ to $L^{q_1, \infty}(\D)$.
\end{enumerate}
This contradicts Case I.
\end{proof}

Consequently, the only remaining problem is to determine the weak- and restricted weak-type behavior on $\partial\Omega_{bd}$.
We prove it by considering Main Cases~I--V; see Figure~\ref{Fig7}.

\vspace{0.1cm}

\subsection{Treatment of Main Case I}\label{Sec10.2}

Recall that \emph{Main Case I} refers to the case where the vertical--horizontal intersection $\calP_{\textnormal{V--H}}$ lies in the region
$$
\left\{ (p,q): \frac{1}{p} \ge 3+a+b-c, \; \frac{1}{q} \le c-a-b-2 \right\};
$$
see Figure \ref{Fig7}. By \cite[Theorem~1.1]{ZZ2022}, the strong-type bounds in this case are as follows: the operators $T_{a,b,c}$ and $S_{a,b,c}$ map $L^p(\D)$ boundedly into $L^q(\D)$ if and only if
$$
(p,q) \in \left\{ 1 \le p,q \le \infty: \; \frac{1}{q}>\frac{1}{p}+c-a-b-2 \right\}.
$$

\vspace{0.1cm}

Our main result is the following. 

\begin{thm} \label{mainsubcaseI}
Let $a, b, c \in \R$ satisfy \eqref{maincase} with $c \ge \max\{a+2, \; b+2\}$. Then the following statements hold. 
\begin{enumerate}
    \item [(a)] {\bf (Weak--type bounds)} If $\calP_{\textnormal{V-H}}$ stays away from both lines
    \begin{equation} \tag{${\bf L_1}$} \label{20260710line01}
    \frac{1}{p}=3+a+b-c, \quad 
    \end{equation} 
    and 
     \begin{equation} \tag{${\bf L_2}$} \label{20260710line02}
    \frac{1}{q}=c-a-b-2, 
    \end{equation} 
    then $T_{a, b, c}$ and $S_{a, b, c}$ map $L^p(\D)$ boundedly into $L^{q, \infty}(\D)$ if
        $$
        (p, q) \in \left\{ 1 \le p, q \le \infty: \; \frac{1}{q}=\frac{1}{p}+c-a-b-2 \right\}.
        $$

    \item [(b)] If $\calP_{\textnormal{V-H}}$ lies on the line \eqref{20260710line02}
    but stays away from the line 
    \eqref{20260710line01}, 
    then 
    \begin{enumerate}
        \item [(1)] {\bf (Weak--type bounds)} $T_{a, b, c}$ and $S_{a, b, c}$ map $L^p(\D)$ boundedly into $L^{q, \infty}(\D)$ if
        $$
        (p, q) \in \left\{ 1 \le p, q \le \infty: \; \frac{1}{q}=\frac{1}{p}+c-a-b-2,  \;  p \neq \infty \right\}, 
        $$
        while at the point $(p, q)=\left(\infty, \frac{1}{c-a-b-2} \right)$, $T_{a, b, c}$ and $S_{a, b, c}$ map $L^\infty(\D)$ unboundedly into $L^{\frac{1}{c-a-b-2}, \infty}(\D)$. 

        \vspace{0.1cm}
        
        \item [(2)] {\bf (Restricted weak-type bounds)} Further, the restricted weak--type estimates at the point $\left(p, q\right)=\left(\infty, \frac{1}{c-a-b-2} \right)$ also fail for both $T_{a, b, c}$ and $S_{a, b, c}$. 
    \end{enumerate}
    \item [(c)] If $\calP_{\textnormal{V-H}}$ lies on the line \eqref{20260710line01}
    but stays away from the line \eqref{20260710line02}, 
    then \begin{enumerate}
        \item [(1)] {\bf (Weak--type bounds)} $T_{a, b, c}$ and $S_{a, b, c}$ map $L^p(\D)$ boundedly into $L^{q, \infty}(\D)$ if
        $$
        (p, q) \in \left\{ 1 \le p, q \le \infty: \; \frac{1}{q}=\frac{1}{p}+c-a-b-2,  \; q \neq 1\right\}, 
        $$
        while at the point $(p, q)=\left(\frac{1}{3+a+b-c}, 1\right)$, $T_{a, b, c}$ and $S_{a, b, c}$ map $L^{\frac{1}{3+a+b-c}}(\D)$ unboundedly into $L^{1, \infty}(\D)$. 

        \vspace{0.1cm}
        
        \item [(2)] {\bf (Restricted weak-type bounds)} $T_{a, b, c}$ and $S_{a, b, c}$ map $L^{\frac{1}{3+a+b-c}, 1}(\D)$ boundedly into $L^{1, \infty}(\D)$. 
    \end{enumerate}

\vspace{0.1cm}

\item [(d)] If $\calP_{\textnormal{V-H}}$ lies on the intersection of \eqref{20260710line01} and \eqref{20260710line02}, then
\begin{enumerate}
        \item [(1)] {\bf (Weak--type bounds)} $T_{a, b, c}$ and $S_{a, b, c}$ map $L^p(\D)$ boundedly into $L^{q, \infty}(\D)$ if
        $$
        (p, q) \in \left\{ 1 \le p, q \le \infty: \; \frac{1}{q}=\frac{1}{p}+c-a-b-2,  \; p \neq \infty, \; q \neq 1\right\}, 
        $$
        while at the points $(p, q)= \left(\infty, \frac{1}{c-a-b-2} \right)$ and $\left(\frac{1}{3+a+b-c}, 1 \right)$, $T_{a, b, c}$ and $S_{a, b, c}$ fail to be of weak-type there. 

        \vspace{0.1cm}
        
        \item [(2)] {\bf (Restricted weak-type bounds)} $T_{a, b, c}$ and $S_{a, b, c}$ map $L^{\frac{1}{3+a+b-c}, 1}(\D)$ boundedly into $L^{1, \infty}(\D)$, while they fail to have a restricted weak-type estimate at the point $(p, q)=\left(\infty, \frac{1}{c-a-b-2} \right)$. 
    \end{enumerate}

\end{enumerate}

\end{thm}

We refer the reader to Figure~\ref{Fig16} below for an illustration of Main Case~I.

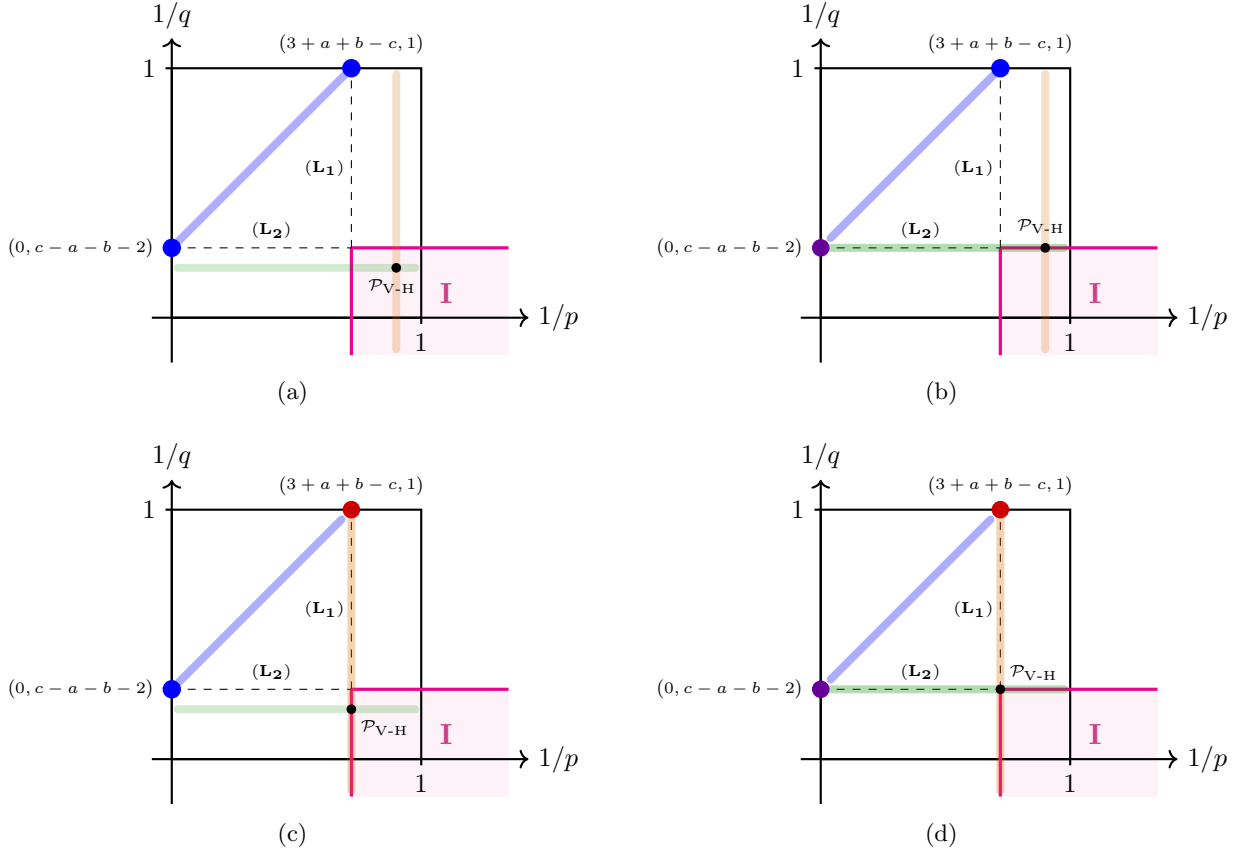
\begin{figure}[ht]
\centering

\begin{minipage}[t]{0.48\textwidth}
\centering
\begin{tikzpicture}[scale=3.3]
  \def\xA{0.72}
    \def\yB{0.28}

    \def\xV{0.90}
    \def\yH{0.28}

    \def\xright{1.35}
    \def\ybottom{-0.15}

    \coordinate (B) at (0,\yB);
    \coordinate (A) at (\xA,1);
    \coordinate (PVH) at (\xV,\yH);
    \fill[magenta!20, opacity=0.30]
        (\xA,\ybottom) rectangle (\xright,\yB);

    \draw[->, thick]
        (-0.08,0) -- (\xright+0.08,0) node[right] {$1/p$};
    \draw[->, thick]
        (0,\ybottom-0.03) -- (0,1.12) node[above] {$1/q$};

    \draw[thick] (0,0) rectangle (1,1);
    \draw[thick] (1,0) -- (1,-0.025) node[below] {$1$};
    \draw[thick] (0,1) -- (-0.025,1) node[left] {$1$};

    \draw[magenta, very thick]
        (\xA,\ybottom) -- (\xA,\yB);
    \draw[magenta, very thick]
        (\xA,\yB) -- (\xright,\yB);
    \draw[black, dashed]
        (\xA,\yB) -- (\xA,1);
        \draw[black, dashed]
        (0,\yB) -- (\xA,\yB);

    \draw[green!55!black, line width=3pt, opacity=0.18,
          line cap=round, shorten <=2.2pt, shorten >=2.2pt]
        (0, 0.2) -- (1, 0.2);

    \draw[orange!85!black, line width=3pt, opacity=0.22,
          line cap=round, shorten <=2.2pt, shorten >=2.2pt]
        (\xV,\ybottom) -- (\xV,1);

    \draw[blue, line width=3pt, opacity=0.32,
          line cap=round, shorten <=2.2pt, shorten >=2.2pt]
        (B) -- (A);

    \fill[blue] (B) circle (1.05pt);
    \fill[blue] (A) circle (1.05pt);

    \fill[black] (0.9, 0.2) circle (0.55pt);
    \node[above right] at (0.75, 0.06)
        {\tiny $\calP_{\textnormal{V-H}}$};

    \node[font=\bfseries\large, magenta!85!black]
        at (1.1, 0.10) {I};

    \node[above] at (0.40,0.28)
        {\tiny \eqref{20260710line02}};
    \node[left] at (0.73,0.60)
        {\tiny \eqref{20260710line01}};

    \node[left=3pt] at (B)
        {\tiny  $\bigl(0,c-a-b-2\bigr)$};
    \node[above=2pt] at (A)
        {\tiny $\bigl(3+a+b-c,1\bigr)$};
\end{tikzpicture}

\vspace{0.05cm}
{\small (a)}
\end{minipage}
\hfill
\begin{minipage}[t]{0.48\textwidth}
\centering
\begin{tikzpicture}[scale=3.3]
    \def\xA{0.72}
    \def\yB{0.28}

    \def\xV{0.90}
    \def\yH{0.28}

    \def\xright{1.35}
    \def\ybottom{-0.15}

    \coordinate (B) at (0,\yB);
    \coordinate (A) at (\xA,1);
    \coordinate (PVH) at (\xV,\yH);

    \fill[magenta!20, opacity=0.30]
        (\xA,\ybottom) rectangle (\xright,\yB);

    \draw[->, thick]
        (-0.08,0) -- (\xright+0.08,0) node[right] {$1/p$};
    \draw[->, thick]
        (0,\ybottom-0.03) -- (0,1.12) node[above] {$1/q$};

    \draw[thick] (0,0) rectangle (1,1);
    \draw[thick] (1,0) -- (1,-0.025) node[below] {$1$};
    \draw[thick] (0,1) -- (-0.025,1) node[left] {$1$};

    \draw[magenta, very thick]
        (\xA,\ybottom) -- (\xA,\yB);
    \draw[magenta, very thick]
        (\xA,\yB) -- (\xright,\yB);
    \draw[black, dashed]
        (\xA,\yB) -- (\xA,1);
         \draw[black, dashed]
        (0,\yB) -- (\xA,\yB);

    \draw[green!55!black, line width=3pt, opacity=0.30,
          line cap=round, shorten <=2.2pt, shorten >=2.2pt]
        (0,\yH) -- (1,\yH);

    \draw[orange!85!black, line width=3pt, opacity=0.22,
          line cap=round, shorten <=2.2pt, shorten >=2.2pt]
        (\xV,\ybottom) -- (\xV,1);

    \draw[blue, line width=3pt, opacity=0.32,
          line cap=round, shorten <=5.2pt, shorten >=2.2pt]
        (B) -- (A);

    \fill[blue] (A) circle (1.05pt);

    \fill[violet!80!blue] (B) circle (1pt);

    \fill[black] (PVH) circle (0.55pt);
    \node[above right] at (0.75, 0.3)
        {\tiny $\calP_{\textnormal{V-H}}$};

    \node[font=\bfseries\large, magenta!85!black]
        at (1.1,0.10) {I};

    \node[above] at (0.40,0.28)
        {\tiny \eqref{20260710line02}};
    \node[left] at (0.73,0.60)
        {\tiny \eqref{20260710line01}};

    \node[left=3pt] at (B)
        {\tiny $\bigl(0,c-a-b-2\bigr)$};
    \node[above=2pt] at (A)
        {\tiny $\bigl(3+a+b-c,1\bigr)$};
\end{tikzpicture}

\vspace{0.05cm}
{\small (b)}
\end{minipage}

\vspace{0.35cm}

\begin{minipage}[t]{0.48\textwidth}
\centering
\begin{tikzpicture}[scale=3.3]
  \def\xA{0.72}
    \def\yB{0.28}

    \def\xV{0.90}
    \def\yH{0.28}

    \def\xright{1.35}
    \def\ybottom{-0.15}

    \coordinate (B) at (0,\yB);
    \coordinate (A) at (\xA,1);
    \coordinate (PVH) at (\xV,\yH);

    \fill[magenta!20, opacity=0.30]
        (\xA,\ybottom) rectangle (\xright,\yB);

    \draw[->, thick]
        (-0.08,0) -- (\xright+0.08,0) node[right] {$1/p$};
    \draw[->, thick]
        (0,\ybottom-0.03) -- (0,1.12) node[above] {$1/q$};

    \draw[thick] (0,0) rectangle (1,1);
    \draw[thick] (1,0) -- (1,-0.025) node[below] {$1$};
    \draw[thick] (0,1) -- (-0.025,1) node[left] {$1$};

    \draw[magenta, very thick]
        (\xA,\ybottom) -- (\xA,\yB);
    \draw[magenta, very thick]
        (\xA,\yB) -- (\xright,\yB);
    \draw[black, dashed]
        (\xA,\yB) -- (\xA,1);
                \draw[black, dashed]
        (0,\yB) -- (\xA,\yB);

    \draw[green!55!black, line width=3pt, opacity=0.18,
          line cap=round, shorten <=2.2pt, shorten >=2.2pt]
        (0,0.2) -- (1,0.2);

    \draw[orange!85!black, line width=3pt, opacity=0.30,
          line cap=round, shorten <=2.2pt, shorten >=2.2pt]
        (\xA,\ybottom) -- (\xA,1);

    \draw[blue, line width=3pt, opacity=0.32,
          line cap=round, shorten <=2.2pt, shorten >=5.2pt]
        (B) -- (A);

    \fill[blue] (B) circle (1.05pt);

    \fill[red!80!black] (A) circle (1pt);

    \fill[black] (\xA, 0.2) circle (0.55pt);
    \node[above right] at (\xA, 0.06)
        {\tiny $\calP_{\textnormal{V-H}}$};

    \node[font=\bfseries\large, magenta!85!black]
        at (1.1,0.10) {I};

   \node[above] at (0.40,0.28)
        {\tiny \eqref{20260710line02}};
    \node[left] at (0.73,0.60)
        {\tiny \eqref{20260710line01}};

    \node[left=3pt] at (B)
        {\tiny $\bigl(0,c-a-b-2\bigr)$};
    \node[above=2pt] at (A)
        {\tiny $\bigl(3+a+b-c,1\bigr)$};
\end{tikzpicture}

\vspace{0.05cm}
{\small (c)}
\end{minipage}
\hfill
\begin{minipage}[t]{0.48\textwidth}
\centering
\begin{tikzpicture}[scale=3.3]
    \def\xA{0.72}
    \def\yB{0.28}

    \def\xV{0.72}
    \def\yH{0.28}

    \def\xright{1.35}
    \def\ybottom{-0.15}

    \coordinate (B) at (0,\yB);
    \coordinate (A) at (\xA,1);
    \coordinate (PVH) at (\xV,\yH);

    \fill[magenta!20, opacity=0.30]
        (\xA,\ybottom) rectangle (\xright,\yB);

    \draw[->, thick]
        (-0.08,0) -- (\xright+0.08,0) node[right] {$1/p$};
    \draw[->, thick]
        (0,\ybottom-0.03) -- (0,1.12) node[above] {$1/q$};

    \draw[thick] (0,0) rectangle (1,1);
    \draw[thick] (1,0) -- (1,-0.025) node[below] {$1$};
    \draw[thick] (0,1) -- (-0.025,1) node[left] {$1$};

    \draw[magenta, very thick]
        (\xA,\ybottom) -- (\xA,\yB);
    \draw[magenta, very thick]
        (\xA,\yB) -- (\xright,\yB);
    \draw[black, dashed]
        (\xA,\yB) -- (\xA,1);
                \draw[black, dashed]
        (0,\yB) -- (\xA,\yB);

    \draw[green!55!black, line width=3pt, opacity=0.30,
          line cap=round, shorten <=2.2pt, shorten >=2.2pt]
        (0,\yH) -- (1,\yH);

    \draw[orange!85!black, line width=3pt, opacity=0.30,
          line cap=round, shorten <=2.2pt, shorten >=2.2pt]
        (\xV,\ybottom) -- (\xV,1);

    \draw[blue, line width=3pt, opacity=0.32,
          line cap=round, shorten <=5.2pt, shorten >=5.2pt]
        (B) -- (A);

    \fill[violet!80!blue] (B) circle (1pt);

    \fill[red!80!black] (A) circle (1pt);

    \fill[black] (PVH) circle (0.55pt);
    \node[above right] at (PVH)
        {\tiny $\calP_{\textnormal{V-H}}$};

    \node[font=\bfseries\large, magenta!85!black]
        at (1.1,0.10) {I};

   \node[above] at (0.40,0.28)
        {\tiny \eqref{20260710line02}};
    \node[left] at (0.73,0.60)
        {\tiny \eqref{20260710line01}};

    \node[left=3pt] at (B)
        {\tiny $\bigl(0,c-a-b-2\bigr)$};
    \node[above=2pt] at (A)
        {\tiny $\bigl(3+a+b-c,1\bigr)$};
\end{tikzpicture}

\vspace{0.05cm}
{\small (d)}
\end{minipage}

\caption{\small
Weak-type and restricted weak-type behavior in Main Case~I. The blue segments and points represent weak-type bounds, the red points represent restricted weak-type bounds only, and the violet points represent failure even of restricted weak-type bounds. As usual, the green and orange lines represent the horizontal cut-off line~\eqref{horizontalcut} and the vertical cut-off line~\eqref{verticalcut}, respectively.}
\label{Fig16}
\end{figure}

\begin{proof}
Each part follows by combining the endpoint estimates established earlier with the corresponding threshold and restricted weak-type results.

 $\bullet$ For part~(a), the desired weak-type bounds follow directly from Propositions~\ref{20260527prop01} and \ref{20260527prop02}.

$\bullet$  For part~(b), the weak-type statement in part~(b.1) is a consequence of Proposition~\ref{20260527prop02}, together with parts~(i) and~(ii) of Theorem~\ref{Forelli--Rudin threshold}. The failure of the restricted weak-type estimate asserted in part~(b.2) follows from part~(i) of Theorem~\ref{restricted-weak Forelli--Rudin at CH}.

 $\bullet$ For part~(c), combining Propositions~\ref{20260710prop01}, \ref{20260527prop01} and \ref{20260527prop02} yields part~(c.1), while part~(c.2) follows from Theorem~\ref{restricted-weak Forelli--Rudin at CV}.

$\bullet$ Finally, part~(d.1) follows from Propositions~\ref{20260527prop02} and~\ref{20260710prop01}, together with parts~(i) and~(ii) of Theorem~\ref{Forelli--Rudin threshold}. Part~(d.2) is then obtained by combining Theorems~\ref{restricted-weak Forelli--Rudin at CV} and~\ref{restricted-weak Forelli--Rudin at CH}.
\end{proof}

\vspace{0.1cm}

\subsection{Treatment of Main Case II}

Recall that \emph{Main Case II} refers to the case where the vertical--horizontal intersection $\calP_{\textnormal{V--H}}$ lies in the region
$$
\left\{ (p,q): \frac{1}{p} < 3+a+b-c, \; \frac{1}{q} \le c-a-b-2 \right\};
$$
see Figure \ref{Fig7}. By \cite[Theorem~1.1]{ZZ2022}, the strong-type bounds in this case are characterized as follows: the operators $T_{a,b,c}$ and $S_{a,b,c}$ map $L^p(\D)$ boundedly into $L^q(\D)$ if and only if
$$
(p,q)\in \left\{1\le p,q\le\infty:\ \frac{1}{q}>\frac{1}{p}+c-a-b-2,\quad \frac{1}{p}<b+1\right\}.
$$

\vspace{0.1cm}

We now state our main result.

\begin{thm} \label{mainsubcaseII}
Let $a, b, c \in \R$ satisfy \eqref{maincase} with $b+2 \le c < a+2$. Then the following statements hold. 
\begin{enumerate}
    \item [(a)] If $\calP_{\textnormal{V-H}}$ stays away from the line \eqref{20260710line02},
    then 
    \begin{enumerate}
        \item [(1)] {\bf (Weak--type bounds)} 
    $T_{a, b, c}$ and $S_{a, b, c}$ map $L^p(\D)$ boundedly into $L^{q, \infty}(\D)$ if
        $$
        (p, q) \in \left\{ 1 \le p, q \le \infty: \; \frac{1}{q}=\frac{1}{p}+c-a-b-2,\ \frac{1}{p}<b+1\right\}.
        $$
    $T_{a, b, c}$ and $S_{a, b, c}$ map $L^p(\D)$ unboundedly into $L^{q, \infty}(\D)$ if 
    	\[
    		p=\frac{1}{b+1},
    		\quad
    		c-1-a \le \frac{1}{q} \le 1.
    	\]
        
        \item [(2)] {\bf (Restricted weak-type bounds)}
        $T_{a, b, c}$ and $S_{a, b, c}$ map $L^{p,1}(\D)$ boundedly into $L^{q, \infty}(\D)$ if
    	\[
    		p=\frac{1}{b+1},
    		\quad
    		c-1-a \le \frac{1}{q} \le 1.
    	\]
    \end{enumerate}

    \item [(b)] If $\calP_{\textnormal{V-H}}$ lies on the line \eqref{20260710line02}, 
    then 
    \begin{enumerate}
        \item [(1)] {\bf (Weak--type bounds)} $T_{a, b, c}$ and $S_{a, b, c}$ map $L^p(\D)$ boundedly into $L^{q, \infty}(\D)$ if
        $$
        (p, q) \in \left\{ 1 \le p, q \le \infty: \; \frac{1}{q}=\frac{1}{p}+c-a-b-2,\ \frac{1}{p}<b+1 ,  \;  p \neq \infty \right\}.
        $$
       $T_{a, b, c}$ and $S_{a, b, c}$ map $L^p(\D)$ unboundedly into $L^{q, \infty}(\D)$ if
       \[
       	p=\infty,
       	\quad
       	q=\frac{1}{c-a-b-2},
       \]
       or
    	\[
    		p=\frac{1}{b+1},
    		\quad
    		c-1-a \le \frac{1}{q} \le 1.
    	\]

        \vspace{0.1cm}
        
        \item [(2)] {\bf (Restricted weak-type bounds)}
        $T_{a, b, c}$ and $S_{a, b, c}$ map $L^{p,1}(\D)$ boundedly into $L^{q, \infty}(\D)$ if
    	\[
    		p=\frac{1}{b+1},
    		\quad
    		c-1-a \le \frac{1}{q} \le 1.
    	\]
    	Further, the restricted weak--type estimates at the point $\left(p, q\right)=\left(\infty, \frac{1}{c-a-b-2} \right)$ also fail for both $T_{a, b, c}$ and $S_{a, b, c}$. 
    \end{enumerate}

\end{enumerate}

\end{thm}

We refer the reader to Figure~\ref{Fig17} below for an illustration of Main Case~II.

\begin{figure}[ht]
\centering

\begin{minipage}[t]{0.48\textwidth}
\centering
\begin{tikzpicture}[scale=3.3]
    \def\xA{0.72}
    \def\yB{0.28}

    \def\xV{0.5}
    \def\yH{0.2}

    \def\xright{1.35}
    \def\ybottom{-0.15}

    \coordinate (B) at (0,\yB);
    \coordinate (A) at (\xA,1);
    \coordinate (PVH) at (\xV,\yH);
    \fill[green!20, opacity=0.30]
        (0,\ybottom) rectangle (\xA,\yB);

    \draw[->, thick]
        (-0.08,0) -- (\xright+0.08,0) node[right] {$1/p$};
    \draw[->, thick]
        (0,\ybottom-0.03) -- (0,1.12) node[above] {$1/q$};

    \draw[thick] (0,0) rectangle (1,1);
    \draw[thick] (1,0) -- (1,-0.025) node[below] {$1$};
    \draw[thick] (0,1) -- (-0.025,1) node[left] {$1$};

    \draw[green, very thick]
        (0,\yB) -- (\xA,\yB);
    \draw[black, dashed]
        (\xA,\ybottom) -- (\xA,1);
   \draw[black, dashed]
        (\xA,\yB) -- (\xright,\yB);
   \draw[black, dashed]
        (\xV, \xV+\yB) -- (\xA,1);

    \draw[green!55!black, line width=3pt, opacity=0.18,
          line cap=round, shorten <=2.2pt, shorten >=2.2pt]
        (0, \yH) -- (1, \yH);

    \draw[orange!85!black, line width=3pt, opacity=0.22,
          line cap=round, shorten <=2.2pt, shorten >=2.2pt]
        (\xV,\ybottom) -- (\xV, \xV+\yB);

    \draw[blue, line width=3pt, opacity=0.32,
          line cap=round, shorten <=2.2pt, shorten >=2.2pt]
        (B) -- (\xV, \xV+\yB);
        
    \draw[red, line width=3pt, opacity=0.5,
          line cap=round, shorten <=2.2pt, shorten >=2.2pt]
        (\xV,1) -- (\xV, \xV+\yB);

    \fill[blue] (B) circle (1.05pt);
    \fill[red] (\xV, \xV+\yB) circle (1.05pt);

    \fill[black] (\xV, \yH) circle (0.55pt);
    \node[below] at (\xV, \yH)
        {\tiny $\calP_{\textnormal{V-H}}$};

    \node[font=\bfseries\large, green!45!black] at (0.3,0.1) {II};

    \node[above] at (0.35,0.28)
        {\tiny \eqref{20260710line02}};
    \node[right] at (0.73,0.60)
        {\tiny \eqref{20260710line01}};

    \node[left=3pt] at (B)
        {\tiny  $\bigl(0,c-a-b-2\bigr)$};
    \node[above=2pt] at (A)
        {\tiny $\bigl(3+a+b-c,1\bigr)$};
\end{tikzpicture}

\vspace{0.05cm}
{\small (a)}
\end{minipage}
\hfill
\begin{minipage}[t]{0.48\textwidth}
\centering
\begin{tikzpicture}[scale=3.3]
    \def\xA{0.72}
    \def\yB{0.28}

    \def\xV{0.5}
    \def\yH{0.28}

    \def\xright{1.35}
    \def\ybottom{-0.15}

    \coordinate (B) at (0,\yB);
    \coordinate (A) at (\xA,1);
    \coordinate (PVH) at (\xV,\yH);
    \fill[green!20, opacity=0.30]
        (0,\ybottom) rectangle (\xA,\yB);

    \draw[->, thick]
        (-0.08,0) -- (\xright+0.08,0) node[right] {$1/p$};
    \draw[->, thick]
        (0,\ybottom-0.03) -- (0,1.12) node[above] {$1/q$};

    \draw[thick] (0,0) rectangle (1,1);
    \draw[thick] (1,0) -- (1,-0.025) node[below] {$1$};
    \draw[thick] (0,1) -- (-0.025,1) node[left] {$1$};

    \draw[green, very thick]
        (0,\yB) -- (\xA,\yB);
    \draw[black, dashed]
        (\xA,\ybottom) -- (\xA,1);
    \draw[black, dashed]
        (\xA,\yB) -- (\xright,\yB);
   \draw[black, dashed]
        (\xV, \xV+\yB) -- (\xA,1);

    \draw[green!55!black, line width=3pt, opacity=0.18,
          line cap=round, shorten <=2.2pt, shorten >=2.2pt]
        (0, \yH) -- (1, \yH);

    \draw[orange!85!black, line width=3pt, opacity=0.22,
          line cap=round, shorten <=2.2pt, shorten >=2.2pt]
        (\xV,\ybottom) -- (\xV, \xV+\yB);

    \draw[blue, line width=3pt, opacity=0.32,
          line cap=round, shorten <=2.2pt, shorten >=2.2pt]
        (B) -- (\xV, \xV+\yB);
        
    \draw[red, line width=3pt, opacity=0.5,
          line cap=round, shorten <=2.2pt, shorten >=2.2pt]
        (\xV,1) -- (\xV, \xV+\yB);

    \fill[violet] (B) circle (1.05pt);
    \fill[red] (\xV, \xV+\yB) circle (1.05pt);

    \fill[black] (\xV, \yH) circle (0.55pt);
    \node[below] at (\xV, \yH)
        {\tiny $\calP_{\textnormal{V-H}}$};

    \node[font=\bfseries\large, green!45!black] at (0.3,0.1) {II};

    \node[above] at (0.35,0.28)
        {\tiny \eqref{20260710line02}};
    \node[right] at (0.73,0.60)
        {\tiny \eqref{20260710line01}};

    \node[left=3pt] at (B)
        {\tiny  $\bigl(0,c-a-b-2\bigr)$};
    \node[above=2pt] at (A)
        {\tiny $\bigl(3+a+b-c,1\bigr)$};
\end{tikzpicture}

\vspace{0.05cm}
{\small (b)}
\end{minipage}

\caption{\small
Weak-type and restricted weak-type behavior in Main Case~II. The blue segments and points represent weak-type bounds, the red segments and points represent restricted weak-type bounds only and the violet points represent failure even of restricted weak-type bounds. As usual, the green and orange lines represent the horizontal cut-off line~\eqref{horizontalcut} and the vertical cut-off line~\eqref{verticalcut}, respectively.}
\label{Fig17}
\end{figure}
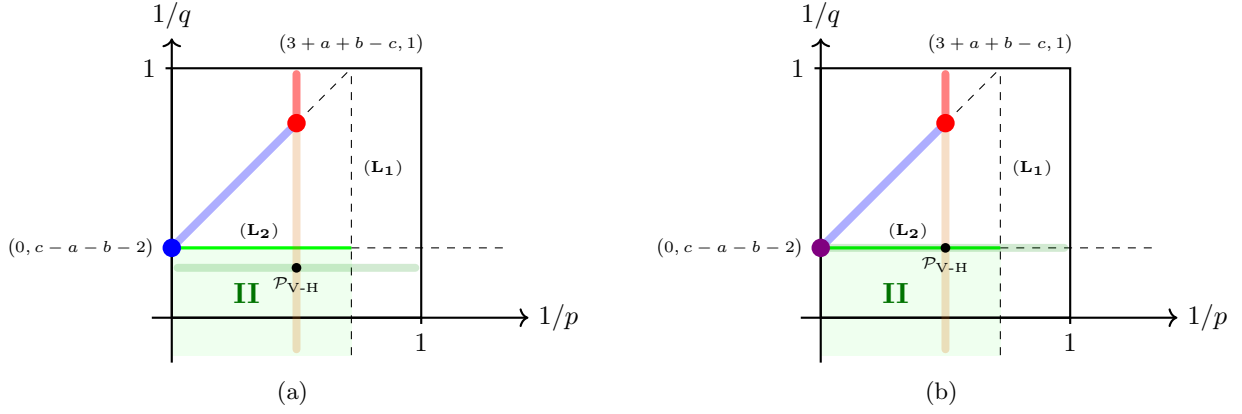

\begin{proof}
Each assertion follows from the endpoint estimates obtained earlier, together with the corresponding threshold and restricted weak-type results.

 $\bullet$ For part~(a), the desired weak-type bounds in (a.1) follow directly from Propositions~\ref{20260710prop01}, \ref{20260527prop01} and \ref{20260527prop02}.
Since $L^{t,\infty}(\D) \hookrightarrow L^{s,\infty}(\D)$ for $0<s<t<\infty$, part~(a.2) follows from the restricted weak-type estimate at $\left(\frac{1}{b+1}, \frac{1}{c-a-1}\right)$, which is obtained by Theorem~\ref{restricted-weak Forelli--Rudin at CV}. 

$\bullet$  For part~(b), combining Proposition~\ref{20260710prop01} and Propositions~\ref{20260527prop02} with (i) and (ii) of Theorem~\ref{Forelli--Rudin threshold} yields part~(b.1).
The restricted weak-type estimate in part~(b.2) is the same as the proof of (a.2), while the failure of the restricted weak-type estimate asserted follows from part~(i) of Theorem~\ref{restricted-weak Forelli--Rudin at CH}.
\end{proof}

\vspace{0.1cm}

\subsection{Treatment of Main Case III}

Recall that \emph{Main Case III} refers to the case where the vertical--horizontal intersection $\calP_{\textnormal{V--H}}$ lies in the region
$$
\left\{ (p,q): \frac{1}{p} \ge 3+a+b-c, \; \frac{1}{q} > c-a-b-2 \right\};
$$
see Figure \ref{Fig7}. According to \cite[Theorem~1.1]{ZZ2022}, the strong-type mapping properties in this case are completely determined by the following conditions:
$$
T_{a,b,c},\, S_{a,b,c}: L^p(\D)\longrightarrow L^q(\D)
$$
are bounded if and only if
$$ 
(p,q) \in \left\{ 1 \le p,q \le \infty: \; \frac{1}{q}>\frac{1}{p}+c-a-b-2, \frac{1}{q}>-a\right\}. 
$$

\vspace{0.1cm}

The main result in this case is as follows.

\begin{thm} \label{mainsubcaseIII}
Let $a, b, c \in \R$ satisfy \eqref{maincase} with $a+2 \le c < b+2$. Then the following statements hold. 
\begin{enumerate}
    \item [(a)] If $\calP_{\textnormal{V-H}}$ stays away from the line \eqref{20260710line01} and $0<2-c+b<1/2$,
    then 
    \begin{enumerate}
        \item [(1)] {\bf (Weak--type bounds)} 
    $T_{a, b, c}$ and $S_{a, b, c}$ map $L^p(\D)$ boundedly into $L^{q, \infty}(\D)$ if
        \[
        (p, q) \in \left\{ 1 \le p, q \le \infty: \; \frac{1}{q}=\frac{1}{p}+c-a-b-2,\ \frac{1}{q}>-a\right\}.
        \]
        or
        \[
        	0 \le \frac{1}{p} < 2-c+b,
        	\quad 
        	q=-\frac{1}{a}.
        \]
        At the point $(p, q)=\left(\frac{1}{2-c+b}, -\frac{1}{a}\right)$, $T_{a, b, c}$ and $S_{a, b, c}$ map $L^{\frac{1}{2-c+b}}(\D)$ unboundedly into $L^{-\frac{1}{a}, \infty}(\D)$.
        
        \item [(2)] {\bf (Restricted weak-type bounds)}
        $T_{a, b, c}$ and $S_{a, b, c}$ map $L^{\frac{1}{2-c+b},1}(\D)$ boundedly into $L^{-\frac{1}{a}, \infty}(\D)$.
    \end{enumerate}
    
    \medskip

    \item [(b)] If $\calP_{\textnormal{V-H}}$ lies on the line \eqref{20260710line01} and $0<2-c+b<1/2$, 
    then 
    \begin{enumerate}
        \item [(1)] {\bf (Weak--type bounds)} $T_{a, b, c}$ and $S_{a, b, c}$ map $L^p(\D)$ boundedly into $L^{q, \infty}(\D)$ if
        $$
        (p, q) \in \left\{ 1 \le p, q \le \infty: \; \frac{1}{q}=\frac{1}{p}+c-a-b-2,\ -a<\frac{1}{q}<1 \right\}, 
        $$
        or
        \[
        	0 \le \frac{1}{p} < 2-c+b,
        	\quad 
        	q=-\frac{1}{a}.
        \]
        While $T_{a, b, c}$ and $S_{a, b, c}$ map $L^{p}(\D)$ unboundedly into $L^{q, \infty}(\D)$ if
        \[
        	p=\frac{1}{2-c+b},
        	\quad
        	q=-\frac{1}{a},
        \]
        or
        \[
        	p=\frac{1}{3+a+b-c},
        	\quad
        	q=1.
        \]

        \vspace{0.1cm}
        
        \item [(2)] {\bf (Restricted weak-type bounds)}
        $T_{a, b, c}$ and $S_{a, b, c}$ map $L^{p,1}(\D)$ boundedly into $L^{q, \infty}(\D)$ if
    	\[
    		p=\frac{1}{2-c+b},
    		\quad
    		q=-\frac{1}{a},
    	\]
    	or 
    	\[
    		p=\frac{1}{3+a+b-c},
    		\quad
    		q=1.
    	\] 
    \end{enumerate}
    
    \medskip
    
    \item [(c)] If $\calP_{\textnormal{V-H}}$ stays away from the line \eqref{20260710line01} and $1/2 \le 2-c+b<1$,
    then 
    \begin{enumerate}
        \item [(1)] {\bf (Weak--type bounds)} 
    $T_{a, b, c}$ and $S_{a, b, c}$ map $L^p(\D)$ boundedly into $L^{q, \infty}(\D)$ if
        \[
        (p, q) \in \left\{ 1 \le p, q \le \infty: \; \frac{1}{q}=\frac{1}{p}+c-a-b-2,\ \frac{1}{q}>-a\right\}.
        \]
        or
        \[
        	0 \le \frac{1}{p} < 2-c+b,
        	\quad 
        	q=-\frac{1}{a}.
        \]
        
        \item [(2)] {\bf (Forelli--Rudin phenomenon for weak--type estimates)}
        At the point $(p, q)=\left(\frac{1}{2-c+b}, -\frac{1}{a}\right)$, $T_{a, b, c}$ maps $L^{\frac{1}{2-c+b}}(\D)$ boundedly into $L^{-\frac{1}{a}, \infty}(\D)$ while $S_{a, b, c}$ fails.
        
        \item [(3)] {\bf (Restricted weak-type bounds)}
        $T_{a, b, c}$ and $S_{a, b, c}$ map $L^{\frac{1}{2-c+b},1}(\D)$ boundedly into $L^{-\frac{1}{a}, \infty}(\D)$.
    \end{enumerate}
    
    \medskip

    \item [(d)] If $\calP_{\textnormal{V-H}}$ lies on the line \eqref{20260710line01} and $1/2 \le 2-c+b<1$,
    then 
    \begin{enumerate}
        \item [(1)] {\bf (Weak--type bounds)} $T_{a, b, c}$ and $S_{a, b, c}$ map $L^p(\D)$ boundedly into $L^{q, \infty}(\D)$ if
        $$
        (p, q) \in \left\{ 1 \le p, q \le \infty: \; \frac{1}{q}=\frac{1}{p}+c-a-b-2,\ -a<\frac{1}{q}<1 \right\}, 
        $$
        or
        \[
        	0 \le \frac{1}{p} < 2-c+b,
        	\quad 
        	q=-\frac{1}{a},
        \]
        while at the point $(p, q)=\left(\frac{1}{3+a+b-c}, 1\right)$, $T_{a, b, c}$ and $S_{a, b, c}$ map $L^{\frac{1}{3+a+b-c}}(\D)$ unboundedly into $L^{1, \infty}(\D)$. 

        \item [(2)] {\bf (Forelli--Rudin phenomenon for weak-type estimates)}
        At the point $(p, q)=\left(\frac{1}{2-c+b}, -\frac{1}{a}\right)$, $T_{a, b, c}$ maps $L^{\frac{1}{2-c+b}}(\D)$ boundedly into $L^{-\frac{1}{a}, \infty}(\D)$ while $S_{a, b, c}$ fails.
        
        \item [(3)] {\bf (Restricted weak-type bounds)}
        $T_{a, b, c}$ and $S_{a, b, c}$ map $L^{p,1}(\D)$ boundedly into $L^{q, \infty}(\D)$ if
    	\[
    		p=\frac{1}{2-c+b},
    		\quad
    		q=-\frac{1}{a},
    	\]
    	or 
    	\[
    		p=\frac{1}{3+a+b-c},
    		\quad
    		q=1.
    	\] 
    \end{enumerate}

\end{enumerate}

\end{thm}

We refer the reader to Figure~\ref{Fig18} below for an illustration of Main Case~III.

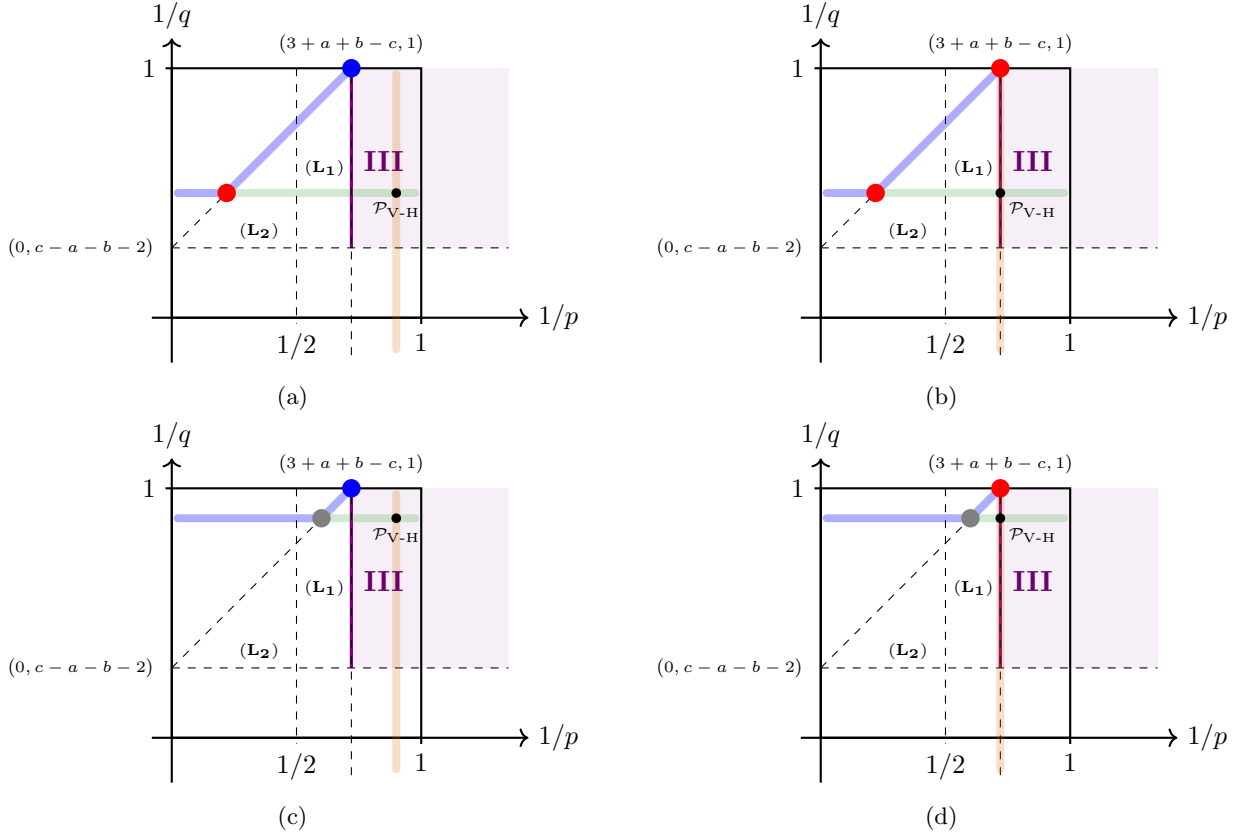
\begin{figure}[ht]
\centering

\begin{minipage}[t]{0.48\textwidth}
\centering
\begin{tikzpicture}[scale=3.3]
    \def\xA{0.72}
    \def\yB{0.28}

    \def\xV{0.9}
    \def\yH{0.5}

    \def\xright{1.35}
    \def\ybottom{-0.15}

    \coordinate (B) at (0,\yB);
    \coordinate (A) at (\xA,1);
    \coordinate (PVH) at (\xV,\yH);
	\fill[violet!20, opacity=0.3] (\xA,\yB) rectangle (\xright,1);

    \draw[->, thick]
        (-0.08,0) -- (\xright+0.08,0) node[right] {$1/p$};
    \draw[->, thick]
        (0,\ybottom-0.03) -- (0,1.12) node[above] {$1/q$};

    \draw[thick] (0,0) rectangle (1,1);
    \draw[thick] (1,0) -- (1,-0.025) node[below] {$1$};
    \draw[dashed] (0.5,1) -- (0.5,-0.025) node[below] {$1/2$};
    \draw[thick] (0,1) -- (-0.025,1) node[left] {$1$};

    \draw[violet, very thick]
        (\xA,1) -- (\xA,\yB);
    \draw[black, dashed]
        (\xA,\ybottom) -- (\xA,1);
    \draw[black, dashed]
        (0,\yB) -- (\xright,\yB);
    \draw[black, dashed]
        (0,\yB) -- (\yH-\yB,\yH);

    \draw[green!55!black, line width=3pt, opacity=0.18,
          line cap=round, shorten <=2.2pt, shorten >=2.2pt]
        (\yH-\yB,\yH) -- (1, \yH);

    \draw[orange!85!black, line width=3pt, opacity=0.22,
          line cap=round, shorten <=2.2pt, shorten >=2.2pt]
        (\xV,\ybottom) -- (\xV, 1);

    \draw[blue, line width=3pt, opacity=0.32,
          line cap=round, shorten <=2.2pt, shorten >=2.2pt]
        (\yH-\yB,\yH) -- (A);
        
    \draw[blue, line width=3pt, opacity=0.32,
          line cap=round, shorten <=2.2pt, shorten >=2.2pt]
        (0, \yH) -- (\yH-\yB, \yH);

    \fill[blue] (A) circle (1.05pt);
    \fill[red] (\yH-\yB, \yH) circle (1.05pt);

    \fill[black] (\xV, \yH) circle (0.55pt);
    \node[below] at (\xV, \yH)
        {\tiny $\calP_{\textnormal{V-H}}$};

	\node[font=\bfseries\large, violet!85!black] at (0.85,0.63) {III};

    \node[above] at (0.35,0.28)
        {\tiny \eqref{20260710line02}};
    \node[left] at (0.73,0.60)
        {\tiny \eqref{20260710line01}};

    \node[left=3pt] at (B)
        {\tiny  $\bigl(0,c-a-b-2\bigr)$};
    \node[above=2pt] at (A)
        {\tiny $\bigl(3+a+b-c,1\bigr)$};
\end{tikzpicture}

\vspace{0.05cm}
{\small (a)}
\end{minipage}
\hfill
\begin{minipage}[t]{0.48\textwidth}
\centering
\begin{tikzpicture}[scale=3.3]
    \def\xA{0.72}
    \def\yB{0.28}

    \def\xV{0.72}
    \def\yH{0.5}

    \def\xright{1.35}
    \def\ybottom{-0.15}

    \coordinate (B) at (0,\yB);
    \coordinate (A) at (\xA,1);
    \coordinate (PVH) at (\xV,\yH);
	\fill[violet!20, opacity=0.3] (\xA,\yB) rectangle (\xright,1);

    \draw[->, thick]
        (-0.08,0) -- (\xright+0.08,0) node[right] {$1/p$};
    \draw[->, thick]
        (0,\ybottom-0.03) -- (0,1.12) node[above] {$1/q$};

    \draw[thick] (0,0) rectangle (1,1);
    \draw[thick] (1,0) -- (1,-0.025) node[below] {$1$};
    \draw[dashed] (0.5,1) -- (0.5,-0.025) node[below] {$1/2$};
    \draw[thick] (0,1) -- (-0.025,1) node[left] {$1$};

    \draw[violet, very thick]
        (\xA,1) -- (\xA,\yB);
    \draw[black, dashed]
        (\xA,\ybottom) -- (\xA,1);
        \draw[black, dashed]
        (0,\yB) -- (\xright,\yB);
    \draw[black, dashed]
        (0,\yB) -- (\yH-\yB,\yH);

    \draw[green!55!black, line width=3pt, opacity=0.18,
          line cap=round, shorten <=2.2pt, shorten >=2.2pt]
        (\yH-\yB,\yH) -- (1, \yH);

    \draw[orange!85!black, line width=3pt, opacity=0.22,
          line cap=round, shorten <=2.2pt, shorten >=2.2pt]
        (\xV,\ybottom) -- (\xV, 1);

    \draw[blue, line width=3pt, opacity=0.32,
          line cap=round, shorten <=2.2pt, shorten >=2.2pt]
        (\yH-\yB,\yH) -- (A);
        
    \draw[blue, line width=3pt, opacity=0.32,
          line cap=round, shorten <=2.2pt, shorten >=2.2pt]
        (0, \yH) -- (\yH-\yB, \yH);

    \fill[red] (A) circle (1.05pt);
    \fill[red] (\yH-\yB, \yH) circle (1.05pt);

    \fill[black] (\xV, \yH) circle (0.55pt);
    \node[below] at (\xV+0.13, \yH)
        {\tiny $\calP_{\textnormal{V-H}}$};

	\node[font=\bfseries\large, violet!85!black] at (0.85,0.63) {III};

    \node[above] at (0.35,0.28)
        {\tiny \eqref{20260710line02}};
    \node[left] at (0.73,0.60)
        {\tiny \eqref{20260710line01}};

    \node[left=3pt] at (B)
        {\tiny  $\bigl(0,c-a-b-2\bigr)$};
    \node[above=2pt] at (A)
        {\tiny $\bigl(3+a+b-c,1\bigr)$};
\end{tikzpicture}

\vspace{0.05cm}
{\small (b)}
\end{minipage}
\hfill
\begin{minipage}[t]{0.48\textwidth}
\centering
\begin{tikzpicture}[scale=3.3]
    \def\xA{0.72}
    \def\yB{0.28}

    \def\xV{0.9}
    \def\yH{0.88}

    \def\xright{1.35}
    \def\ybottom{-0.15}

    \coordinate (B) at (0,\yB);
    \coordinate (A) at (\xA,1);
    \coordinate (PVH) at (\xV,\yH);
	\fill[violet!20, opacity=0.3] (\xA,\yB) rectangle (\xright,1);

    \draw[->, thick]
        (-0.08,0) -- (\xright+0.08,0) node[right] {$1/p$};
    \draw[->, thick]
        (0,\ybottom-0.03) -- (0,1.12) node[above] {$1/q$};

    \draw[thick] (0,0) rectangle (1,1);
    \draw[thick] (1,0) -- (1,-0.025) node[below] {$1$};
    \draw[dashed] (0.5,1) -- (0.5,-0.025) node[below] {$1/2$};
    \draw[thick] (0,1) -- (-0.025,1) node[left] {$1$};

    \draw[violet, very thick]
        (\xA,1) -- (\xA,\yB);
    \draw[black, dashed]
        (\xA,\ybottom) -- (\xA,1);
        \draw[black, dashed]
        (0,\yB) -- (\xright,\yB);
    \draw[black, dashed]
        (0,\yB) -- (\yH-\yB,\yH);

    \draw[green!55!black, line width=3pt, opacity=0.18,
          line cap=round, shorten <=2.2pt, shorten >=2.2pt]
        (\yH-\yB,\yH) -- (1, \yH);

    \draw[orange!85!black, line width=3pt, opacity=0.22,
          line cap=round, shorten <=2.2pt, shorten >=2.2pt]
        (\xV,\ybottom) -- (\xV, 1);

    \draw[blue, line width=3pt, opacity=0.32,
          line cap=round, shorten <=2.2pt, shorten >=2.2pt]
        (\yH-\yB,\yH) -- (A);
        
    \draw[blue, line width=3pt, opacity=0.32,
          line cap=round, shorten <=2.2pt, shorten >=2.2pt]
        (0, \yH) -- (\yH-\yB, \yH);

    \fill[blue] (A) circle (1.05pt);
    \fill[gray] (\yH-\yB, \yH) circle (1.05pt);

    \fill[black] (\xV, \yH) circle (0.55pt);
    \node[below] at (\xV, \yH)
        {\tiny $\calP_{\textnormal{V-H}}$};

	\node[font=\bfseries\large, violet!85!black] at (0.85,0.63) {III};

    \node[above] at (0.35,0.28)
        {\tiny \eqref{20260710line02}};
    \node[left] at (0.73,0.60)
        {\tiny \eqref{20260710line01}};

    \node[left=3pt] at (B)
        {\tiny  $\bigl(0,c-a-b-2\bigr)$};
    \node[above=2pt] at (A)
        {\tiny $\bigl(3+a+b-c,1\bigr)$};
\end{tikzpicture}

\vspace{0.05cm}
{\small (c)}
\end{minipage}
\hfill
\begin{minipage}[t]{0.48\textwidth}
\centering
\begin{tikzpicture}[scale=3.3]
    \def\xA{0.72}
    \def\yB{0.28}

    \def\xV{0.72}
    \def\yH{0.88}

    \def\xright{1.35}
    \def\ybottom{-0.15}

    \coordinate (B) at (0,\yB);
    \coordinate (A) at (\xA,1);
    \coordinate (PVH) at (\xV,\yH);
	\fill[violet!20, opacity=0.3] (\xA,\yB) rectangle (\xright,1);

    \draw[->, thick]
        (-0.08,0) -- (\xright+0.08,0) node[right] {$1/p$};
    \draw[->, thick]
        (0,\ybottom-0.03) -- (0,1.12) node[above] {$1/q$};

    \draw[thick] (0,0) rectangle (1,1);
    \draw[thick] (1,0) -- (1,-0.025) node[below] {$1$};
    \draw[dashed] (0.5,1) -- (0.5,-0.025) node[below] {$1/2$};
    \draw[thick] (0,1) -- (-0.025,1) node[left] {$1$};

    \draw[violet, very thick]
        (\xA,1) -- (\xA,\yB);
    \draw[black, dashed]
        (\xA,\ybottom) -- (\xA,1);
        \draw[black, dashed]
        (0,\yB) -- (\xright,\yB);
    \draw[black, dashed]
        (0,\yB) -- (\yH-\yB,\yH);

    \draw[green!55!black, line width=3pt, opacity=0.18,
          line cap=round, shorten <=2.2pt, shorten >=2.2pt]
        (\yH-\yB,\yH) -- (1, \yH);

    \draw[orange!85!black, line width=3pt, opacity=0.22,
          line cap=round, shorten <=2.2pt, shorten >=2.2pt]
        (\xV,\ybottom) -- (\xV, 1);

    \draw[blue, line width=3pt, opacity=0.32,
          line cap=round, shorten <=2.2pt, shorten >=2.2pt]
        (\yH-\yB,\yH) -- (A);
        
    \draw[blue, line width=3pt, opacity=0.32,
          line cap=round, shorten <=2.2pt, shorten >=2.2pt]
        (0, \yH) -- (\yH-\yB, \yH);

    \fill[red] (A) circle (1.05pt);
    \fill[gray] (\yH-\yB, \yH) circle (1.05pt);

    \fill[black] (\xV, \yH) circle (0.55pt);
    \node[below] at (\xV+0.13, \yH)
        {\tiny $\calP_{\textnormal{V-H}}$};

	\node[font=\bfseries\large, violet!85!black] at (0.85,0.63) {III};

    \node[above] at (0.35,0.28)
        {\tiny \eqref{20260710line02}};
    \node[left] at (0.73,0.60)
        {\tiny \eqref{20260710line01}};

    \node[left=3pt] at (B)
        {\tiny  $\bigl(0,c-a-b-2\bigr)$};
    \node[above=2pt] at (A)
        {\tiny $\bigl(3+a+b-c,1\bigr)$};
\end{tikzpicture}

\vspace{0.05cm}
{\small (d)}
\end{minipage}

\caption{\small
Weak-type and restricted weak-type behavior in Main Case~III. The blue segments and points represent weak-type bounds, the red points represent restricted weak-type bounds only. The gray points represent the Forelli--Rudin phenomenon for weak-type estimates. As usual, the green and orange lines represent the horizontal cut-off line~\eqref{horizontalcut} and the vertical cut-off line~\eqref{verticalcut}, respectively.}
\label{Fig18}
\end{figure}

\begin{proof}
The conclusion in each case follows by combining the preceding endpoint estimates with the relevant threshold and restricted weak-type bounds.

$\bullet$ We first establish assertions (a.1), (b.1), (c.1), and (d.1). The desired weak-type bounds along the critical line follow directly from Proposition~\ref{20260527prop02}. 

Since $L^{t}(\D) \hookrightarrow L^{s,1}(\D)$ whenever $0<s<t\leq\infty$, the weak-type bounds along the horizontal cut-off line~\eqref{horizontalcut} follow from the restricted weak-type estimate at
$\left(\frac{1}{2-c+b}, \frac{1}{-a}\right)$, 
which is established in Theorem~\ref{restricted-weak Forelli--Rudin at CH}~(ii).

The unboundedness at the point $(p,q)=\left(\frac{1}{2-c+b},-\frac{1}{a}\right)$ in parts (a.1) and (b.1) follows from Theorem~\ref{Forelli--Rudin threshold}~(i) and (ii).

Similarly, the unboundedness at the point
$(p,q)=\left(\frac{1}{3+a+b-c},1\right)$
in parts (b.1) and (d.1) follows from  Proposition~\ref{20260710prop01}.

$\bullet$ Next, the Forelli--Rudin phenomena for weak-type estimates asserted in parts (c.2) and (d.2) follow directly from Theorem~\ref{Forelli--Rudin threshold}~(i) and (iii).

$\bullet$ Finally, combining Theorem~\ref{restricted-weak Forelli--Rudin at CH}~(ii) with Theorem~\ref{restricted-weak Forelli--Rudin at CV} yields parts (a.2), (b.2), (c.3), and (d.3).
\end{proof}

\vspace{0.1cm}

\subsection{Treatment of Main Case IV}

Recall that \emph{Main Case IV} refers to the case where the vertical--horizontal intersection $\calP_{\textnormal{V--H}}$ lies in the region
$$
\left\{ (p,q):  \frac{1}{p} < 3+a+b-c, \; \frac{1}{q} > c-a-b-2, \; \frac{1}{q}<\frac{1}{p}+c-a-b-2 \right\};
$$
see Figure \ref{Fig7}. Recall that as a consequence of \cite[Theorem~1.1]{ZZ2022},  the strong-type range in this case is given by
$$
(p,q) \in \left\{ 1\le p,q\le\infty: 
\frac{1}{q}>\frac{1}{p}+c-a-b-2,\;
\frac{1}{p}<b+1,\;
\frac{1}{q}>-a
\right\}.
$$
\vspace{0.1cm}

We are ready to state our main result.

\begin{thm} \label{mainsubcaseIV}
Let $a, b, c \in \R$ satisfy \eqref{maincase} with $1<c<\min\{a+2,b+2\}$. Then the following statements hold. 
\begin{enumerate}
    \item [(a)] If $0<2-c+b<1/2$,
    then 
    \begin{enumerate}
        \item [(1)] {\bf (Weak--type bounds)} 
    $T_{a, b, c}$ and $S_{a, b, c}$ map $L^p(\D)$ boundedly into $L^{q, \infty}(\D)$ if
        \[
        (p, q) \in \left\{ 1 \le p, q \le \infty: \; \frac{1}{q}=\frac{1}{p}+c-a-b-2, \; \frac{1}{p}<b+1, \; \frac{1}{q}>-a\right\}.
        \]
        or
        \[
        	0 \le \frac{1}{p} < 2-c+b,
        	\quad 
        	q=-\frac{1}{a}.
        \]
        While $T_{a, b, c}$ and $S_{a, b, c}$ map $L^{p}(\D)$ unboundedly into $L^{q, \infty}(\D)$ if
    	\[
    		p=\frac{1}{2-c+b},
    		\quad
    		q=-\frac{1}{a},
    	\]
    	or 
    	\[
    		p=\frac{1}{b+1},
    		\quad
    		c-1-a \le \frac{1}{q} \le 1.
    	\]
        
        \item [(2)] {\bf (Restricted weak-type bounds)}
        $T_{a, b, c}$ and $S_{a, b, c}$ map $L^{p,1}(\D)$ boundedly into $L^{q, \infty}(\D)$ if
    	\[
    		p=\frac{1}{2-c+b},
    		\quad
    		q=-\frac{1}{a},
    	\]
    	or 
    	\[
    		p=\frac{1}{b+1},
    		\quad
    		c-1-a \le \frac{1}{q} \le 1.
    	\]
    \end{enumerate}
    
    \medskip
    
    \item [(b)] If $1/2 \le 2-c+b<1$,
    then 
    \begin{enumerate}
        \item [(1)] {\bf (Weak--type bounds)} 
    $T_{a, b, c}$ and $S_{a, b, c}$ map $L^p(\D)$ boundedly into $L^{q, \infty}(\D)$ if
        \[
        (p, q) \in \left\{ 1 \le p, q \le \infty: \; \frac{1}{q}=\frac{1}{p}+c-a-b-2, \; \frac{1}{p}<b+1, \; \frac{1}{q}>-a\right\}.
        \]
        or
        \[
        	0 \le \frac{1}{p} < 2-c+b,
        	\quad 
        	q=-\frac{1}{a}.
        \]
        $T_{a, b, c}$ and $S_{a, b, c}$ map $L^{p}(\D)$ unboundedly into $L^{q, \infty}(\D)$ if
    	\[
    		p=\frac{1}{b+1},
    		\quad
    		c-1-a \le \frac{1}{q} \le 1.
    	\]
        
        \item [(2)] {\bf (Weak-type Forelli--Rudin phenomenon)}
        At the point $(p, q)=\left(\frac{1}{2-c+b}, -\frac{1}{a}\right)$, $T_{a, b, c}$ maps $L^{\frac{1}{2-c+b}}(\D)$ boundedly into $L^{-\frac{1}{a}, \infty}(\D)$ while $S_{a, b, c}$ fails.
        
        \item [(3)] {\bf (Restricted weak-type bounds)}
        $T_{a, b, c}$ and $S_{a, b, c}$ map $L^{p,1}(\D)$ boundedly into $L^{q, \infty}(\D)$ if
    	\[
    		p=\frac{1}{2-c+b},
    		\quad
    		q=-\frac{1}{a}
    	\]
    	or 
    	\[
    		p=\frac{1}{b+1},
    		\quad
    		c-1-a \le \frac{1}{q} \le 1.
    	\]
    \end{enumerate}

\end{enumerate}

\end{thm}

We refer the reader to Figure~\ref{Fig19} below for an illustration of Main Case~IV.

\begin{figure}[ht]
\centering

\begin{minipage}[t]{0.48\textwidth}
\centering
\begin{tikzpicture}[scale=3.3]
    \def\xA{0.85}
    \def\yB{0.15}

    \def\xV{0.65}
    \def\yH{0.5}

    \def\xright{1.35}
    \def\ybottom{-0.15}

    \coordinate (B) at (0,\yB);
    \coordinate (A) at (\xA,1);
    \coordinate (PVH) at (\xV,\yH);
	 \fill[orange!20, opacity=0.3] (0,\yB) -- (\xA,\yB) -- (\xA,1) -- cycle; 

    \draw[->, thick]
        (-0.08,0) -- (\xright+0.08,0) node[right] {$1/p$};
    \draw[->, thick]
        (0,\ybottom-0.03) -- (0,1.12) node[above] {$1/q$};

    \draw[thick] (0,0) rectangle (1,1);
    \draw[thick] (1,0) -- (1,-0.025) node[below] {$1$};
    \draw[dashed] (0.5,1) -- (0.5,-0.025) node[below] {$1/2$};
    \draw[thick] (0,1) -- (-0.025,1) node[left] {$1$};

    \draw[black, dashed]
        (\xA,\ybottom) -- (A);
    \draw[black, dashed]
        (B) -- (\xright,\yB);
    \draw[black, dashed]
        (B) -- (\yH-\yB,\yH);
    \draw[black, dashed]
        (A) -- (\xV, \xV+\yB);

    \draw[green!55!black, line width=3pt, opacity=0.18,
          line cap=round, shorten <=2.2pt, shorten >=2.2pt]
        (\yH-\yB,\yH) -- (1, \yH);

    \draw[orange!85!black, line width=3pt, opacity=0.22,
          line cap=round, shorten <=2.2pt, shorten >=2.2pt]
        (\xV,\ybottom) -- (\xV, \xV+\yB);

    \draw[blue, line width=3pt, opacity=0.32,
          line cap=round, shorten <=2.2pt, shorten >=2.2pt]
        (\yH-\yB,\yH) -- (\xV, \xV+\yB);
        
    \draw[blue, line width=3pt, opacity=0.32,
          line cap=round, shorten <=2.2pt, shorten >=2.2pt]
        (0, \yH) -- (\yH-\yB, \yH);
        
    \draw[red, line width=3pt, opacity=0.5,
          line cap=round, shorten <=2.2pt, shorten >=2.2pt]
        (\xV,1) -- (\xV, \xV+\yB);

    \fill[red] (\xV, \xV+\yB) circle (1.05pt);
    \fill[red] (\yH-\yB, \yH) circle (1.05pt);

    \fill[black] (\xV, \yH) circle (0.55pt);
    \node[below] at (\xV, \yH)
        {\tiny $\calP_{\textnormal{V-H}}$};

	\node[font=\bfseries\large, orange!90!black] at (0.75,0.25) {IV};

    \node[below] at (0.35,0.28)
        {\tiny \eqref{20260710line02}};
    \node[right] at (0.82,0.60)
        {\tiny \eqref{20260710line01}};

    \node[left=3pt] at (B)
        {\tiny  $\bigl(0,c-a-b-2\bigr)$};
    \node[above=2pt] at (A)
        {\tiny $\bigl(3+a+b-c,1\bigr)$};
\end{tikzpicture}

\vspace{0.05cm}
{\small (a)}
\end{minipage}
\begin{minipage}[t]{0.48\textwidth}
\centering
\begin{tikzpicture}[scale=3.3]
    \def\xA{0.85}
    \def\yB{0.15}

    \def\xV{0.75}
    \def\yH{0.75}

    \def\xright{1.35}
    \def\ybottom{-0.15}

    \coordinate (B) at (0,\yB);
    \coordinate (A) at (\xA,1);
    \coordinate (PVH) at (\xV,\yH);
	 \fill[orange!20, opacity=0.3] (0,\yB) -- (\xA,\yB) -- (\xA,1) -- cycle; 

    \draw[->, thick]
        (-0.08,0) -- (\xright+0.08,0) node[right] {$1/p$};
    \draw[->, thick]
        (0,\ybottom-0.03) -- (0,1.12) node[above] {$1/q$};

    \draw[thick] (0,0) rectangle (1,1);
    \draw[thick] (1,0) -- (1,-0.025) node[below] {$1$};
    \draw[dashed] (0.5,1) -- (0.5,-0.025) node[below] {$1/2$};
    \draw[thick] (0,1) -- (-0.025,1) node[left] {$1$};

    \draw[black, dashed]
        (\xA,\ybottom) -- (A);
        \draw[black, dashed]
        (B) -- (\xright,\yB);
    \draw[black, dashed]
        (B) -- (\yH-\yB,\yH);
    \draw[black, dashed]
        (A) --(\xV, \xV+\yB);

    \draw[green!55!black, line width=3pt, opacity=0.18,
          line cap=round, shorten <=2.2pt, shorten >=2.2pt]
        (\yH-\yB,\yH) -- (1, \yH);

    \draw[orange!85!black, line width=3pt, opacity=0.22,
          line cap=round, shorten <=2.2pt, shorten >=2.2pt]
        (\xV,\ybottom) -- (\xV, \xV+\yB);

    \draw[blue, line width=3pt, opacity=0.32,
          line cap=round, shorten <=2.2pt, shorten >=2.2pt]
        (\yH-\yB,\yH) -- (\xV, \xV+\yB);
        
    \draw[blue, line width=3pt, opacity=0.32,
          line cap=round, shorten <=2.2pt, shorten >=2.2pt]
        (0, \yH) -- (\yH-\yB, \yH);
        
    \draw[red, line width=3pt, opacity=0.5,
          line cap=round, shorten <=2.2pt, shorten >=2.2pt]
        (\xV,1) -- (\xV, \xV+\yB);

    \fill[red] (\xV, \xV+\yB) circle (1.05pt);
    \fill[gray] (\yH-\yB, \yH) circle (1.05pt);

    \fill[black] (\xV, \yH) circle (0.55pt);
    \node[below] at (\xV, \yH)
        {\tiny $\calP_{\textnormal{V-H}}$};

	\node[font=\bfseries\large, orange!90!black] at (0.5,0.5) {IV};

    \node[below] at (0.35,0.28)
        {\tiny \eqref{20260710line02}};
    \node[right] at (0.82,0.60)
        {\tiny \eqref{20260710line01}};

    \node[left=3pt] at (B)
        {\tiny  $\bigl(0,c-a-b-2\bigr)$};
    \node[above=2pt] at (A)
        {\tiny $\bigl(3+a+b-c,1\bigr)$};
\end{tikzpicture}

\vspace{0.05cm}
{\small (b)}
\end{minipage}

\caption{\small
Weak-type and restricted weak-type behavior in Main Case~IV. The blue segments and points represent weak-type bounds, the red segments and points represent restricted weak-type bounds only. The gray points represent the Forelli--Rudin phenomenon for weak--type estimates. As usual, the green and orange lines represent the horizontal cut-off line~\eqref{horizontalcut} and the vertical cut-off line~\eqref{verticalcut}, respectively.}
\label{Fig19}
\end{figure}
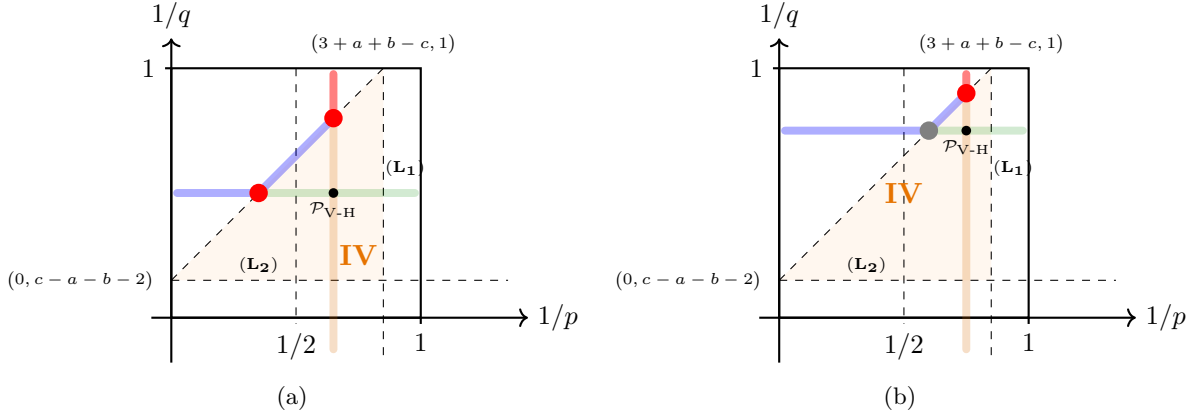

\begin{proof}
Each case is obtained by combining the endpoint estimates proved above with the corresponding threshold and restricted weak-type estimates.

$\bullet$ The weak-type boundedness assertions in parts~(a.1) and~(b.1) follow by the same argument as in the proof of Theorem~\ref{mainsubcaseIII}~(a.1).
The failure of weak-type bounds along the vertical cut-off line in both parts follows from Proposition~\ref{20260710prop01}.
The additional failure at $(p,q)=\left(\frac{1}{2-c+b},-\frac{1}{a}\right)$ in part~(a.1) follows from parts~(i) and~(ii) of Theorem~\ref{Forelli--Rudin threshold}.

$\bullet$ Next, the weak-type Forelli--Rudin phenomenon asserted in part (b.2) follows directly from Theorem~\ref{Forelli--Rudin threshold}~(i) and (iii).

$\bullet$ Finally, combining Theorem~\ref{restricted-weak Forelli--Rudin at CH}~(ii) with Theorem~\ref{restricted-weak Forelli--Rudin at CV} yields parts (a.2) and (b.3).
\end{proof}

\vspace{0.1cm}

\subsection{Treatment of Main Case V}\label{Sec10.6}

Recall that \emph{Main Case V} refers to the case where the vertical--horizontal intersection $\calP_{\textnormal{V--H}}$ lies in the region
$$
\left\{ (p,q): \frac{1}{p} < 3+a+b-c, \; \frac{1}{q} > c-a-b-2, \; \frac{1}{q} \ge \frac{1}{p}+c-a-b-2 \right\};
$$
see Figure \ref{Fig7}. By \cite[Theorem~1.1]{ZZ2022}, the strong-type range in this case is characterized by
$$
(p,q) \in \left\{ 1 \le p,q \le \infty: \; \frac{1}{p}<b+1, \; \frac{1}{q}>-a\right\}.
$$

\vspace{0.1cm}

Our main result in this subsection is stated as follows.  

\begin{thm} \label{mainsubcaseV}
Let $a, b, c \in \R$ satisfy \eqref{maincase} with $c \le 1$. Then the following statements hold. 
\begin{enumerate}
    \item [(a)] If $c=1$,
    then 
    \begin{enumerate}
        \item [(1)] {\bf (Weak--type bounds)} 
    $T_{a, b, c}$ and $S_{a, b, c}$ map $L^p(\D)$ boundedly into $L^{q, \infty}(\D)$ if
        \[
        	0 \le \frac{1}{p} < b+1,
        	\quad 
        	q=-\frac{1}{a}.
        \]
        While $T_{a, b, c}$ and $S_{a, b, c}$ map $L^{p}(\D)$ unboundedly into $L^{q, \infty}(\D)$ if
    	\[
    		p=\frac{1}{b+1},
    		\quad
    		-a \le \frac{1}{q} \le 1.
    	\]
        
        \item [(2)] {\bf (Restricted weak-type bounds)}
        $T_{a, b, c}$ and $S_{a, b, c}$ map $L^{p,1}(\D)$ boundedly into $L^{q, \infty}(\D)$ if 
    	\[
    		p=\frac{1}{b+1},
    		\quad
    		-a < \frac{1}{q} \le 1.
    	\]
    	
    	\item [(3)] {\bf (Restricted weak-type Forelli--Rudin phenomenon)}
        At the point $(p, q)=\left(\frac{1}{b+1}, -\frac{1}{a}\right)$, $T_{a, b, c}$ maps $L^{\frac{1}{b+1},1}(\D)$ boundedly into $L^{-\frac{1}{a}, \infty}(\D)$ while $S_{a, b, c}$ fails.
    \end{enumerate}
    
    \medskip
    
    \item [(b)] If $c<1$,
    then 
    \begin{enumerate}
        \item [(1)] {\bf (Weak--type bounds)} 
    $T_{a, b, c}$ and $S_{a, b, c}$ map $L^p(\D)$ boundedly into $L^{q, \infty}(\D)$ if
        \[
        	0 \le \frac{1}{p} < b+1,
        	\quad 
        	q=-\frac{1}{a}.
        \]
        
        \item [(2)] {\bf (Restricted weak-type bounds)}
        $T_{a, b, c}$ and $S_{a, b, c}$ map $L^{p,1}(\D)$ boundedly into $L^{q, \infty}(\D)$ if
    	\[
    		p=\frac{1}{b+1},
    		\quad
    		-a \le \frac{1}{q} \le 1.
    	\]
    \end{enumerate}

\end{enumerate}

\end{thm}

We refer the reader to Figure~\ref{Fig20} below for an illustration of Main Case~V.

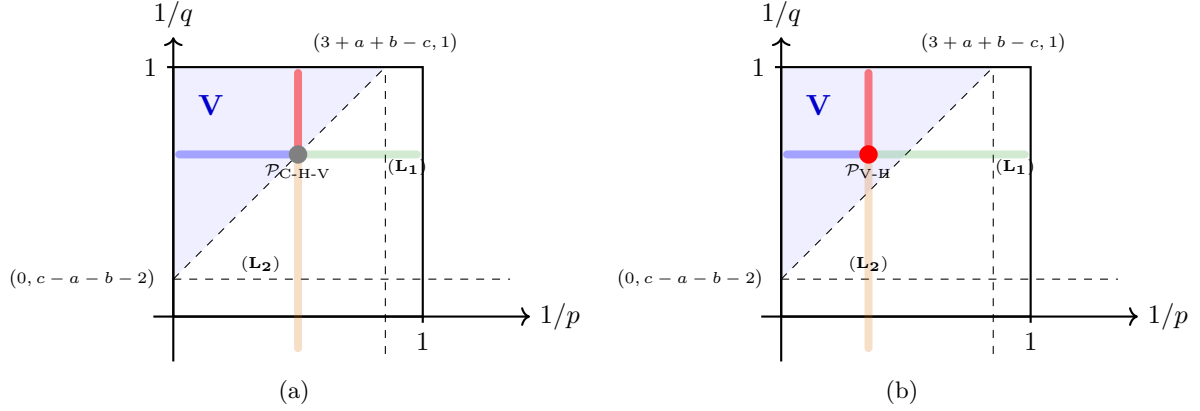
\begin{figure}[ht]
\centering

\begin{minipage}[t]{0.48\textwidth}
\centering
\begin{tikzpicture}[scale=3.3]
    \def\xA{0.85}
    \def\yB{0.15}

    \def\xV{0.5}
    \def\yH{0.65}

    \def\xright{1.35}
    \def\ybottom{-0.15}

    \coordinate (B) at (0,\yB);
    \coordinate (A) at (\xA,1);
    \coordinate (PVH) at (\xV,\yH);
	 \fill[blue!20, opacity=0.3]   (0,\yB) -- (0,1) -- (\xA,1) -- cycle;  

    \draw[->, thick]
        (-0.08,0) -- (\xright+0.08,0) node[right] {$1/p$};
    \draw[->, thick]
        (0,\ybottom-0.03) -- (0,1.12) node[above] {$1/q$};

    \draw[thick] (0,0) rectangle (1,1);
    \draw[thick] (1,0) -- (1,-0.025) node[below] {$1$};
    \draw[thick] (0,1) -- (-0.025,1) node[left] {$1$};

    \draw[black, dashed]
        (\xA,\ybottom) -- (A);
    \draw[black, dashed]
        (B) -- (\xright,\yB);
    \draw[black, dashed]
        (B) -- (A);

    \draw[green!55!black, line width=3pt, opacity=0.18,
          line cap=round, shorten <=2.2pt, shorten >=2.2pt]
        (\xV, \yH) -- (1, \yH);

    \draw[orange!85!black, line width=3pt, opacity=0.22,
          line cap=round, shorten <=2.2pt, shorten >=2.2pt]
        (\xV,\ybottom) -- (\xV, \yH);
        
    \draw[blue, line width=3pt, opacity=0.32,
          line cap=round, shorten <=2.2pt, shorten >=2.2pt]
        (0, \yH) -- (\xV, \yH);
        
    \draw[red, line width=3pt, opacity=0.5,
          line cap=round, shorten <=2.2pt, shorten >=2.2pt]
        (\xV,1) -- (\xV, \yH);

    \fill[gray] (\xV, \yH) circle (1.05pt);

    \node[below] at (\xV, \yH)
        {\tiny $\calP_{\textnormal{C-H-V}}$};

	\node[font=\bfseries\large, blue!80!black] at (0.15,0.85) {V};

    \node[below] at (0.35,0.28)
        {\tiny \eqref{20260710line02}};
    \node[right] at (0.82,0.60)
        {\tiny \eqref{20260710line01}};

    \node[left=3pt] at (B)
        {\tiny  $\bigl(0,c-a-b-2\bigr)$};
    \node[above=2pt] at (A)
        {\tiny $\bigl(3+a+b-c,1\bigr)$};
\end{tikzpicture}

\vspace{0.05cm}
{\small (a)}
\end{minipage}
\begin{minipage}[t]{0.48\textwidth}
\centering
\begin{tikzpicture}[scale=3.3]
    \def\xA{0.85}
    \def\yB{0.15}

    \def\xV{0.35}
    \def\yH{0.65}

    \def\xright{1.35}
    \def\ybottom{-0.15}

    \coordinate (B) at (0,\yB);
    \coordinate (A) at (\xA,1);
    \coordinate (PVH) at (\xV,\yH);
	 \fill[blue!20, opacity=0.3]   (0,\yB) -- (0,1) -- (\xA,1) -- cycle;  

    \draw[->, thick]
        (-0.08,0) -- (\xright+0.08,0) node[right] {$1/p$};
    \draw[->, thick]
        (0,\ybottom-0.03) -- (0,1.12) node[above] {$1/q$};

    \draw[thick] (0,0) rectangle (1,1);
    \draw[thick] (1,0) -- (1,-0.025) node[below] {$1$};
    \draw[thick] (0,1) -- (-0.025,1) node[left] {$1$};

    \draw[black, dashed]
        (\xA,\ybottom) -- (A);
        \draw[black, dashed]
        (B) -- (\xright,\yB);
    \draw[black, dashed]
        (B) -- (A);

    \draw[green!55!black, line width=3pt, opacity=0.18,
          line cap=round, shorten <=2.2pt, shorten >=2.2pt]
        (\xV, \yH) -- (1, \yH);

    \draw[orange!85!black, line width=3pt, opacity=0.22,
          line cap=round, shorten <=2.2pt, shorten >=2.2pt]
        (\xV,\ybottom) -- (\xV, \yH);
        
    \draw[blue, line width=3pt, opacity=0.32,
          line cap=round, shorten <=2.2pt, shorten >=2.2pt]
        (0, \yH) -- (\xV, \yH);
        
    \draw[red, line width=3pt, opacity=0.5,
          line cap=round, shorten <=2.2pt, shorten >=2.2pt]
        (\xV,1) -- (\xV, \yH);

    \fill[red] (\xV, \yH) circle (1.05pt);

    \node[below] at (\xV, \yH)
        {\tiny $\calP_{\textnormal{V-H}}$};

	\node[font=\bfseries\large, blue!80!black] at (0.15,0.85) {V};

    \node[below] at (0.35,0.28)
        {\tiny \eqref{20260710line02}};
    \node[right] at (0.82,0.60)
        {\tiny \eqref{20260710line01}};

    \node[left=3pt] at (B)
        {\tiny  $\bigl(0,c-a-b-2\bigr)$};
    \node[above=2pt] at (A)
        {\tiny $\bigl(3+a+b-c,1\bigr)$};
\end{tikzpicture}

\vspace{0.05cm}
{\small (b)}
\end{minipage}

\caption{\small
Weak-type and restricted weak-type behavior in Main Case~V. The blue segments and points represent weak-type bounds, the red segments and points represent restricted weak-type bounds only. The gray points represent the restricted weak-type Forelli--Rudin phenomenon. As usual, the green and orange lines represent the horizontal cut-off line~\eqref{horizontalcut} and the vertical cut-off line~\eqref{verticalcut}, respectively.}
\label{Fig20}
\end{figure}

\begin{proof}
Each statement follows from the endpoint estimates established earlier, together with the corresponding threshold and restricted weak-type results.

$\bullet$ For the weak-type bounds in (a.1), it suffices to show that
\[
	\|S_{a,b,1}f\|_{L^{-\frac{1}{a},\infty}(\D)} \lesssim \|f\|_{L^{p}(\D)}
\]
for $0\leq1/p<b+1$ and every nonnegative $f$.
For such $p$, choose $\varepsilon>0$ such that
\[
	0\leq\frac{1}{p}<b+1-\varepsilon<b+1.
\]
Since
\[
	|1-z\overline{w}|
	\geq 1-|z||w|
	\geq 1-|w|
	\gtrsim 1-|w|^2,
\]
we have
\[
	\frac{1-|w|^2}{|1-z\overline{w}|}\lesssim 1.
\]
Hence, for any positive function $f$,
\begin{align}\label{remain case1}
	S_{a,b,1}f(z)
	=&(1-|z|^2)^a \int_{\D} \frac{(1-|w|^2)^b}{|1-z\overline{w}|} f(w)\,dA(w)\nonumber\\
	\lesssim&(1-|z|^2)^a \int_{\D} \frac{(1-|w|^2)^{b-\varepsilon}}{|1-z\overline{w}|^{1-\varepsilon}} f(w)\,dA(w)\nonumber\\
	=&S_{a,b-\varepsilon,1-\varepsilon}f(z).
\end{align}
Note that the triple $(a,b-\varepsilon,1-\varepsilon)$ satisfies condition~\eqref{conditionP3}.
Hence, by Theorem~\ref{restricted-weak Forelli--Rudin at VH}, we obtain
\begin{equation}\label{remain case2}
	\|S_{a,b-\varepsilon,1-\varepsilon}f\|_{L^{-\frac{1}{a},\infty}(\D)}
	\lesssim \|f\|_{L^{\frac{1}{b+1-\varepsilon},1}(\D)}
	\lesssim \|f\|_{L^{p}(\D)}.
\end{equation}
Combining \eqref{remain case1} and \eqref{remain case2} yields the weak-type boundedness in (a.1). On the other hand, the failure of weak-type boundedness in (a.1) follows from Proposition~\ref{20260710prop01}.

$\bullet$ Next, for part (a.2), it suffices to show that, for each $-a<1/q\leq 1$ and each positive function $f$,
\[
	\|S_{a,b,1}f\|_{L^{q,\infty}(\D)}
	\lesssim \|f\|_{L^{\frac{1}{b+1},1}(\D)}.
\]
For such $q$, choose $\varepsilon>0$ such that
\[
	-a<-a+\varepsilon<\frac{1}{q}\leq 1.
\]
Then, for each positive function $f$,
\begin{equation}\label{remain case3}
	S_{a,b,1}f(z) \lesssim S_{a-\varepsilon,b,1-\varepsilon}f(z).
\end{equation}
Moreover, the triple $(a-\varepsilon,b,1-\varepsilon)$ satisfies condition~\eqref{conditionP3}.
Hence, by Theorem~\ref{restricted-weak Forelli--Rudin at VH}, we obtain
\begin{equation}\label{remain case4}
	\|S_{a-\varepsilon,b,1-\varepsilon}f\|_{L^{q,\infty}(\D)}
	\lesssim \|S_{a-\varepsilon,b,1-\varepsilon}f\|_{L^{-\frac{1}{a-\varepsilon},\infty}(\D)}
	\lesssim \|f\|_{L^{\frac{1}{b+1},1}(\D)}.
\end{equation}
Combining \eqref{remain case3} and \eqref{remain case4} completes the proof of (a.2).

$\bullet$ Third, the restricted weak-type Forelli--Rudin phenomenon in part (a.3) follows directly from Theorem~\ref{Forelli--Rudin Threshold restricted weak-type}.

$\bullet$ Finally, since
\[
	L^{t,\infty}(\D) \hookrightarrow L^{s,\infty}(\D),
	\qquad
	L^{t}(\D) \hookrightarrow L^{s,1}(\D)
\]
for $0<s<t\leq\infty$, the weak-type estimate on the horizontal cut-off line~\eqref{horizontalcut} in (b.1) and the restricted weak-type estimate on the vertical cut-off line~\eqref{verticalcut} in (b.2) follow from the restricted weak-type estimate at $\left(\frac{1}{b+1}, -\frac{1}{a}\right)$, which is obtained in Theorem~\ref{restricted-weak Forelli--Rudin at VH}.
\end{proof}

\bigskip 

\section{Analysis of the Minor Case}\label{Sec11}

We next discuss the minor case
\begin{equation} \tag{{\bf Minor}} \label{minorcase-final}
	a=-1,\quad b>-1, \quad \textrm{and} \quad 0\leq2-c+b<1.
\end{equation} 
Note that for fixed $b, c \in \R$ satisfying the condition \eqref{minorcase}, by Reduction I in Section \ref{20260830subsec01}, we have
\[
	T_{-1, b, c}1(\cdot) \simeq (1-|\cdot|^2)^{-1} \notin L^{q, \infty}(\D)
\]
for $q>1$.
Hence, it suffices to characterize for which $1 \leq p \leq \infty$ we have
\[
	T_{-1, b, c} \quad \textrm{or} \quad S_{-1, b, c}: L^p(\D) \to L^{1, \infty}(\D)
\]
is bounded.

Here, we divide the minor case into two subcases according to the relative position of the intersection points $\calP_{\textnormal{C-H}}$ and $\calP_{\textnormal{V-H}}$.
\begin{enumerate}
    \item[$\bullet$] \textit{Minor Case I:} 
    \[
	    c>1;
    \]
    \item[$\bullet$] \textit{Minor Case II:} 
    \[
	    c\leq1;
    \]
\end{enumerate}
We refer the reader to Figure~\ref{Minor: Fig1} below for a visualization of these cases.

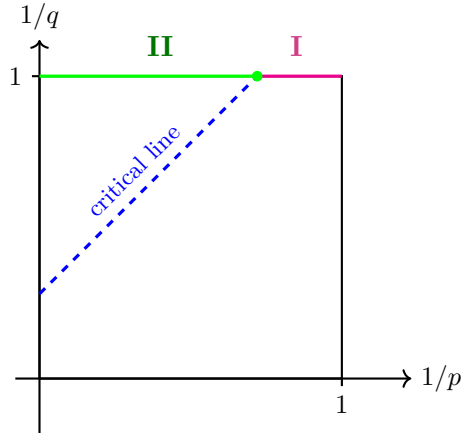
\begin{figure}[H]
\centering
\begin{tikzpicture}[scale=4]
\def\yh{0.28}
\def\xc{0.72}
\def\xright{1.15}
\def\ybottom{-0.15}
  
\draw[->, thick] (-0.08,0) -- (\xright+0.08,0) node[right] {$1/p$};
\draw[->, thick] (0,\ybottom-0.03) -- (0,1.12) node[above] {$1/q$};
\draw[thick] (0,0) rectangle (1,1);
\draw[thick] (1,0) -- (1,-0.025) node[below] {$1$};
\draw[thick] (0,1) -- (-0.025,1) node[left] {$1$};
\draw[blue, dashed, very thick] 
(0,\yh) -- (\xc,1)
node[midway, above, sloped, yshift=1pt] {\small critical line};
\draw[green, very thick] (0,1) -- (\xc,1);
\draw[magenta, very thick] (\xc,1) -- (1,1);
\fill[green] (\xc,1) circle (.5pt);

\node[font=\bfseries\large, magenta!85!black] at (0.85,1.1) {I};
\node[font=\bfseries\large, green!45!black] at (0.4,1.1) {II};
\end{tikzpicture}
\caption{\small The two subcases of the \textit{Minor Case}.}
\label{Minor: Fig1}
\end{figure}

\subsection{Reduction I for Minor Case: Failure of the restricted weak-type estimates to the right of $\calP_{\textnormal{C-H}}$ or $\calP_{\textnormal{V-H}}$}

Denote 
\begin{equation}\label{Minor: p*}
	\frac{1}{p_*}=\min\{2-c+b,b+1\}.
\end{equation}
In this subsection, we show that the operators $T_{-1,b,c}$ and $S_{-1,b,c}$ map $L^{p, 1}(\D)$ unboundedly into $L^{1,\infty}(\D)$ for every $1\le p<p_*$.

\begin{figure}[H]
\centering
\begin{tikzpicture}[scale=2.75]

    \tikzset{
        axis/.style={->, thick},
        box/.style={thick},
        crit/.style={blue, very thick},
        vcut/.style={orange!55!black, very thick},
    }

    \begin{scope}[shift={(0,0)}]
        \def\xA{0.72}
        \def\yB{0.28}
        \def\yh{0.85}

        \draw[axis] (-0.06,0) -- (1.18,0) node[right] {$1/p$};
        \draw[axis] (0,-0.06) -- (0,1.18) node[above] {$1/q$};
        \draw[box] (0,0) rectangle (1,1);
        \draw (1,0) -- (1,-0.022) node[below] {$1$};
        \draw (0,1) -- (-0.022,1) node[left] {$1$};

        \coordinate (B) at (0,\yB);
        \coordinate (A) at (\xA,1);
        \draw[crit] (B) -- (A);
        \draw[vcut] (\yh,0) -- (\yh,1);
        \draw[red, very thick] (A) -- (1,1);
        \fill[black] (A) node[above] {$1/p_*$} circle (.5pt);

        \node[font=\scriptsize] at (0.2, 0.65){\eqref{criticalline}};
        \node[font=\scriptsize] at (0.73,0.5){\eqref{verticalcut}};
        \node[align=center] at (0.5,-0.32)
    {\scriptsize The case when $c>1$};
    \end{scope}

    \begin{scope}[shift={(1.95,0)}]
        \def\xA{0.72}
        \def\yB{0.28}
        \def\yh{0.72}

        \draw[axis] (-0.06,0) -- (1.18,0) node[right] {$1/p$};
        \draw[axis] (0,-0.06) -- (0,1.18) node[above] {$1/q$};
        \draw[box] (0,0) rectangle (1,1);
        \draw (1,0) -- (1,-0.022) node[below] {$1$};
        \draw (0,1) -- (-0.022,1) node[left] {$1$};

        \coordinate (B) at (0,\yB);
        \coordinate (A) at (\xA,1);
        \draw[crit] (B) -- (A);
        \draw[vcut] (\yh,0) -- (\yh,1);
        \draw[red, very thick] (A) -- (1,1);
        \fill[black] (A) node[above] {$1/p_*$} circle (.5pt);

        \node[font=\scriptsize] at (0.2, 0.65) {\eqref{criticalline}};
        \node[font=\scriptsize] at (0.85,0.5) {\eqref{verticalcut}};
		  \node[align=center] at (0.5,-0.32)
        {\scriptsize The case when $c=1$};
    \end{scope}

    \begin{scope}[shift={(4,0)}]
        \def\xA{0.72}
        \def\yB{0.28}
        \def\yh{0.55}

        \draw[axis] (-0.06,0) -- (1.18,0) node[right] {$1/p$};
        \draw[axis] (0,-0.06) -- (0,1.18) node[above] {$1/q$};
        \draw[box] (0,0) rectangle (1,1);
        \draw (1,0) -- (1,-0.022) node[below] {$1$};
        \draw (0,1) -- (-0.022,1) node[left] {$1$};

        \coordinate (B) at (0,\yB);
        \coordinate (A) at (\xA,1);
        \draw[crit] (B) -- (A);
        \draw[vcut] (\yh,0) -- (\yh,1);
        \draw[red, very thick] (\yh,1) -- (1,1);
        \fill[black] (\yh,1) node[above] {$1/p_*$} circle (.5pt);

        \node[font=\scriptsize] at (0.2, 0.65) {\eqref{criticalline}};
        \node[font=\scriptsize] at (0.7,0.5) {\eqref{verticalcut}};
       \node[align=center] at (0.5,-0.32)
        {\scriptsize The case when $c<1$};
    \end{scope}
\end{tikzpicture}
\caption{\small Relative positions of the critical line and the vertical cut-off line along the horizontal line $1/q=1$: the blue line represents the critical line, the orange line represents the vertical cut-off line, the red line represents failure of the restricted weak-type estimates.}
\label{Minor: Fig2}
\end{figure}
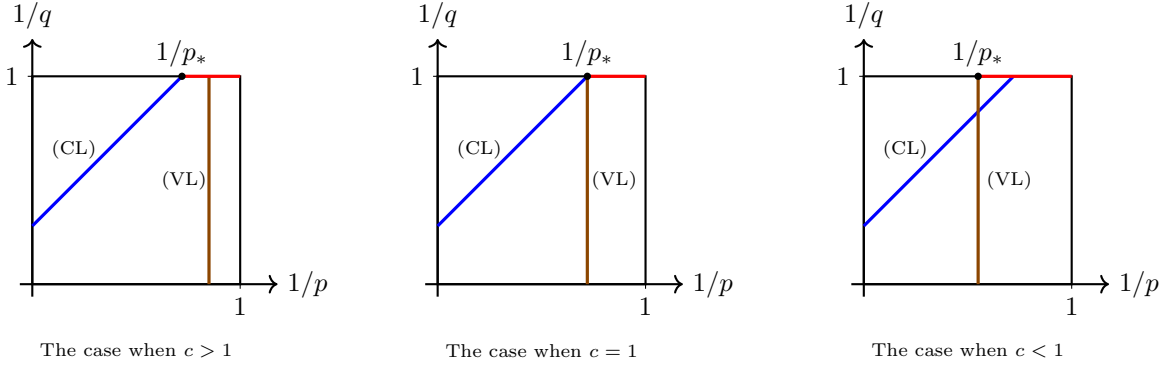

\begin{prop}\label{Minor: Reduction I}
Let $b, c \in \R$ satisfy the condition \eqref{minorcase}.
Then the following statements hold.
\begin{enumerate}
	\item[(i)] Assume $2-c+b<1/p \le 1$.
		Then we have $T_{-1, b, c}$ and $S_{-1, b, c}$ are unbounded from $L^{p,1}(\D)$ to $L^{1, \infty}(\D)$.
	\item[(ii)] Assume $b+1<1/p \le 1$.
		Then we have $T_{-1, b, c}$ and $S_{-1, b, c}$ are unbounded from $L^{p,1}(\D)$ to $L^{1, \infty}(\D)$.
\end{enumerate}
\end{prop}
\begin{proof}[Proof of Proposition~\ref{Minor: Reduction I}, $(i)$]
For $N\in\N$, let
\[
	f_N(w)=w^N.
\]
Then
\[
	|\{w\in\D: |f_N|>t\}|
	=\left|\left\{w\in\D: |w|^N>t\right\}\right|
	\simeq 1-t^{\frac{2}{N}}.
\]
Hence, 
\begin{align*}
	\|f_N\|_{L^{p,1}}
	\simeq&\int_0^\infty |\{w\in\D: |f_N|>t\}|^{\frac{1}{p}} dt
	\simeq\int_0^1 (1-t^{\frac{2}{N}})^{\frac{1}{p}} dt\\
	=&\frac{N}{2}\int_0^1 s^{\frac{N}{2}-1}(1-s)^{\frac{1}{p}} ds
	\simeq NB\left(\frac{N}{2},1+\frac{1}{p}\right),
\end{align*}
where $B(\cdot, \cdot)$ is the Beta function. By Stirling's formula, we have
\begin{equation}\label{Minor Reduction II: fN}
	\|f_N\|_{L^{p,1}(\D)} \simeq N \cdot N^{-1-\frac{1}{p}}=N^{-\frac{1}{p}}.
\end{equation}

Since,
\[
	\frac{1}{(1-z\overline{w})^c} = \sum_{m=0}^{\infty}\frac{\Gamma(m+c)}{\Gamma(c)m!}z^m\overline{w}^{\,m},
\]
we have
\begin{equation}\label{Minor Reduction II: T}
\begin{aligned}
	T_{-1, b, c} f_N (z) 
	=&(1 - |z|^2) ^{-1} \int_{\D} \frac {(1-|w|^{2}) ^{b}}{(1 - z\overline{w}) ^{c}} w^N dA(w)\\
	=&(1 - |z|^2) ^{-1} \sum_{m=0}^{\infty}\frac{\Gamma(m+c)}{\Gamma(c)m!}z^m \int_{\D} (1-|w|^{2}) ^{b} w^N \overline{w}^m dA(w)\\
	=&C_N\frac{z^N}{1-|z|^2},
\end{aligned}
\end{equation}
where 
\[
	 C_N=\frac{\Gamma(N+c)}{\Gamma(c)N!} \int_{\D} (1-|w|^{2}) ^{b} |w|^{2N} dA(w)=\frac{\Gamma(N+c)}{\Gamma(c)N!}B(N+1,b+1).
\]
By Stirling's formula, we have
\begin{equation}\label{Minor Reduction II: CN}
	C_N \simeq N^{c-1}N^{-b-1}=N^{c-2-b}.
\end{equation}
Moreover, we claim that 
\begin{equation}\label{Minor Reduction II: func}
	\left\|\frac{z^N}{1-|z|^2}\right\|_{L^{1,\infty}(\D)} \gtrsim 1.
\end{equation}
Indeed, let 
\[
	E_N=\left\{z\in\D: |z|^2>1-\frac{1}{N}\right\}.
\]
Then we have
\[
	|E_N| \simeq \frac{1}{N}.
\]
Moreover, for each $z\in E_N$, we have
\[
	|z|^N \geq \left(1-\frac{1}{N}\right)^{\frac{N}{2}} \gtrsim 1, \qquad 1-|z|^2 \leq \frac{1}{N}.
\]
Hence, we obtain 
\[
	\frac{|z|^N}{1-|z|^2} \gtrsim N
\]
for each $z\in E_N$.
Then
\[
	\left\|\frac{z^N}{1-|z|^2}\right\|_{L^{1,\infty}(\D)}
	\gtrsim N|E_N| 
	\gtrsim 1.
\]
Combining \eqref{Minor Reduction II: T}, \eqref{Minor Reduction II: CN} and \eqref{Minor Reduction II: func} we have
\[
	\|T_{-1, b, c} f_N\|_{L^{1,\infty}(\D)} \gtrsim N^{c-2-b}.
\]
Since $1/p>2-c+b$, by \eqref{Minor Reduction II: fN}, we obtain
\[
	\frac{\|T_{-1, b, c} f_N\|_{L^{1,\infty}(\D)}}{\|f_N\|_{L^{p,1}(\D)}} 
	\gtrsim N^{c-2-b+\frac{1}{p}}
	\to \infty,
	\qquad
	N\to\infty.
\]
This completes the proof.
\end{proof}

\begin{proof}[Proof of Proposition~\ref{Minor: Reduction I}, $(ii)$]
It suffices to prove that $T_{-1,b,c}$ is not of restricted weak-type $(p,1)$ for $1/p>1+b$.
We claim that 
\[
	f(w)=(1-|w|^2)^{-(b+1)}
\]
belongs to $L^{p,1}(\D)$.
Indeed, since $b>-1$, we have $|\{w\in\D: |f(w)|>t\}|=1$ for $0<t\leq1$ and
\[
	|\{w\in\D: |f(w)|>t\}|
	=\left|\left\{w\in\D: 1-|w|^2<t^{-\frac{1}{b+1}}\right\}\right|
	\simeq t^{-\frac{1}{b+1}}
\]
for $t>1$.
It follows from $1/p>b+1$ that
\[
	\|f\|_{L^{p,1}}
	\simeq \int_0^\infty |\{w\in\D: |f|>t\}|^{\frac{1}{p}} dt
	\simeq 1 + \int_1^\infty t^{-\frac{1}{p(b+1)}} dt
	<\infty.
\]

On the other hand, let $r>0$ be sufficiently small (only depending on $c$) such that for any $z \in \D_r:=\left\{z \in \D: |z| \le r\right\}$, one has
\[
	\left|\arg(1-z\overline{w}) \right| \le \frac{\pi}{10(|c|+1)}, \qquad \textrm{for any} \quad w \in \D. 
\]
Then, for all $z \in \D_r$, one has 
\[
	|T_{-1,b,c} f (z)|
	\gtrsim (1-|z|^2)^{-1}\int_{\D}\frac{(1-|w|^2)^{-1}}{|1 - z\overline{w}| ^{c}} dA(w)
	\gtrsim_r \int_{\D}(1-|w|^2)^{-1}dA(w)=+\infty.
\]
In particular, the above computation shows that  $T_{-1,b,c} f \notin L^{1,\infty}({\D})$, which completes the proof. 
\end{proof}

\subsection{Reduction II for Minor Case: Failure of the weak-type estimates at $\calP_{\textnormal{V-H}}$}

\begin{prop}\label{Minor: Reduction II}
Let $b, c \in \R$ satisfy the condition \eqref{minorcase}.
Then
\[
	T_{-1, b, c}, \; S_{-1, b, c}: L^{\frac{1}{b+1}}(\D) \not\to L^{1, \infty}(\D).
\]
\end{prop}
\begin{proof}
Without loss of generality, we may restrict ourselves to the case $b\le 0$. Otherwise, \eqref{verticalcut} lies outside the hypersingular regime $\Omega_{\mathcal H}$.

Since $0\leq2-c+b<1$, by Proposition \ref{Minor: Reduction I}, it suffices to prove that $T_{-1,b,c}$ is not of weak-type $(1/(b+1),1)$ for $b<0$.
Hence, take $\del>0$ such that $(b+1)(1+\del)<1$.
Consider the function
\[
	f(w):=(1-|w|^2)^{-(b+1)}\left(1+\log\frac{1}{1-|w|^2}\right)^{-(b+1)(1+\del)}.
\]
By \eqref{Reduction I: b<0,f}, $f \in L^{\frac{1}{b+1}}(\D)$.
Moreover, $T_{-1, b, c} f$ is not well-defined on $\D_r:=\left\{z \in \D: |z| \le r\right\}$ for $r>0$ sufficiently small.
Hence, 
\[
	T_{-1, b, c}: L^{\frac{1}{b+1}}(\D) \not\to L^{1, \infty}(\D).
\]
This completes the proof.
\end{proof}

\subsection{Reduction III for Minor Case: Restricted weak-type estimates on $\calP_{\textnormal{C-H}}$ and $\calP_{\textnormal{V-H}}$ for $c\neq1$ and $0<2-c+b<1$}
In this subsection, we show that the operators $T_{-1,b,c}$ and $S_{-1,b,c}$ are of restricted weak-type $(p_*,1)$ for $c\neq1$ and $0<2-c+b<1$.
The endpoint case $c=1$ is precisely where the restricted weak-type Forelli--Rudin phenomenon occurs.
We refer the reader to the first and third figures in Figure \ref{Minor: Fig2}.
\begin{prop}\label{Minor: Reduction III}
Let $b, c \in \R$ satisfy the condition \eqref{minorcase} with $c\neq1$ and $0<2-c+b<1$.
Define $p_*$ as in \eqref{Minor: p*}.
Then we have
\[
	T_{-1, b, c}, \; S_{-1, b, c}: L^{p_*,1}(\D) \to L^{1, \infty}(\D).
\]
\end{prop}
\begin{proof}
It suffices to show $S_{-1,b,c}$ is of restricted weak-type $(p_*,1)$ for $c\neq1$.
Since $0<2-c+b<1$, we have
\[
	0<c-1-b<1.
\]
Take $0<\varepsilon<c-1-b$ and define
\[
	a^*=-1+\varepsilon>-1.
\]
Then we have
\[
	S_{-1, b, c}f(z)=W_\varepsilon(z) S_{a^*,b,c}f(z),
\]
where $W_\varepsilon(z)=(1-|z|^2)^{-\varepsilon}$.
By Lorentz--H\"older inequality, we have
\[
	\|S_{-1, b, c}f\|_{L^{1,\infty}(\D)} 
	\lesssim \|W_\varepsilon\|_{L^{\frac{1}{\varepsilon},\infty}(\D)} \|S_{a^*,b,c}f\|_{L^{-\frac{1}{a^*},\infty}(\D)}
	\lesssim \|S_{a^*,b,c}f\|_{L^{-\frac{1}{a^*},\infty}(\D)}.
\]
Note that the triple $(a^*,b,c)$ satisfies either condition \eqref{conditionP1} or \eqref{conditionP3}.
Due to Theorem~\ref{restricted-weak Forelli--Rudin at CH} (ii) and Theorem~\ref{restricted-weak Forelli--Rudin at VH}, $S_{a^*,b,c}$ is of restricted weak-type $(p_*,-1/a^*)$ with $c\neq1$.
This completes the proof.
\end{proof}

\subsection{Reduction IV for Minor Case: Weak-type estimates on horizontal cut-off line to the left of $\calP_{\textnormal{C-H}}$ and $\calP_{\textnormal{V-H}}$}
In this subsection, we show that the operators $T_{-1,b,c}$ and $S_{-1,b,c}$ map $L^p(\D)$ boundedly into $L^{1,\infty}(\D)$ for every $p_*<p\leq\infty$.

\begin{prop}\label{Minor: Reduction IV}
Let $b, c \in \R$ satisfy the condition \eqref{minorcase}.
Then for each 
\[
	0 \le \frac{1}{p}<\min\{2-c+b,b+1\},
\]
we have
\[
	T_{-1, b, c}, \; S_{-1, b, c}: L^{p}(\D) \to L^{1, \infty}(\D).
\]
\end{prop}
\begin{proof}
The same argument as in Proposition \ref{Minor: Reduction III} can apply.
Combining the Lorentz--Hölder inequality with Theorem~\ref{mainsubcaseIII} (a.1), (c.1) and Theorem~\ref{mainsubcaseV} (a.1), (b.1) proves this proposition.
\end{proof}

\subsection{Analysis of the Minor Case: Weak-type Bounds at the Critical--Horizontal Intersection $\calP_{\textnormal{C-H}}$ on the left of vertical cut-off line and the Forelli--Rudin Phenomenon at the level of weak-type estimates}

The goal of this subsection is to study the weak-type bounds at the critical--horizontal intersection 
\[ \calP_{\textnormal{C-H}}=(2-c+b,1), \]
where the critical line \eqref{criticalline} meets the horizontal cut-off line \eqref{horizontalcut}, and which lies strictly to the left of the vertical-horizontal intersection $\calP_{\textnormal{V-H}}$. 
Equivalently, we have the following condition:
\begin{equation} \tag{${\bf P_1}'$} \label{Minor: conditionP1}
\begin{cases}
	b,c \in \R \quad \textrm{satisfy \eqref{minorcase}}; \\
	c>1.
\end{cases}
\end{equation}
We refer the reader to Figure~\ref{Minor: Fig3} below for a visualization of this case.

\begin{figure}[H]
\centering
\begin{tikzpicture}[scale=4]
\def\yh{0.28}
\def\xc{0.72}
\def\xright{1.15}
\def\ybottom{-0.15}
  
\draw[->, thick] (-0.08,0) -- (\xright+0.08,0) node[right] {$1/p$};
\draw[->, thick] (0,\ybottom-0.03) -- (0,1.12) node[above] {$1/q$};
\draw[thick] (0,0) rectangle (1,1);
\draw[thick] (1,0) -- (1,-0.025) node[below] {$1$};
\draw[thick] (0,1) -- (-0.025,1) node[left] {$1$};
\draw[blue, dashed, very thick] 
(0,\yh) -- (\xc,1)
node[midway, above, sloped, yshift=1pt] {\small critical line};
\draw[thick, green, very thick] (0,1) -- (\xc,1);
\draw[thick, magenta, very thick] (\xc,1) -- (1,1);
\draw[orange!90!black, line width=3pt, opacity=0.18, line cap=round, shorten <=2.2pt, shorten >=2.2pt] (0.85, 0) -- (0.85, 1);

\fill[black] (\xc, 1) circle (.4pt);
\node[above] at (0.68, 1) {\scriptsize $\calP_{\textnormal{C-H}}$};

\fill[black] (0.85, 1) circle (.4pt);
\node[below] at (.85, 1) {\scriptsize $\calP_{\textnormal{V-H}}$};

\node[font=\bfseries\large, magenta!85!black] at (0.85,1.1) {I};
\node[font=\bfseries\large, green!45!black] at (0.4,1.1) {II};
\end{tikzpicture}
\caption{\small $\calP_{\textnormal{V-H}}$ belongs to Minor Case I with $c>1$.}
\label{Minor: Fig3}
\end{figure}
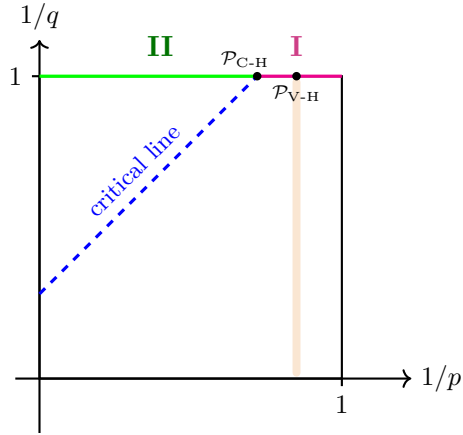

\begin{thm}\label{Minor: weak Forelli--Rudin}
Let $b, c \in \R$ satisfy \eqref{Minor: conditionP1}. 
Then the following statements hold.
\begin{enumerate}
	\item[(i)] The Forelli--Rudin type operator $S_{-1,b,c}$ is unbounded from $L^{\frac{1}{2-c+b}}(\D)$ to $L^{1,\infty}(\D)$.
	\item[(ii)] If $0\leq2-c+b<1/2$, the Forelli--Rudin type operator $T_{-1,b,c}$ is unbounded from $L^{\frac{1}{2-c+b}}(\D)$ to $L^{1,\infty}(\D)$.
	\item[(iii)] If $1/2 \le 2-c+b<1$, the Forelli--Rudin type operator $T_{-1,b,c}$ is bounded from $L^{\frac{1}{2-c+b}}(\D)$ to $L^{1,\infty}(\D)$.
\end{enumerate}
Here, in the borderline case $2-c+b=0$, the space $L^{\frac{1}{2-c+b}}(\D)$ is understood as $L^\infty(\D)$.
\end{thm}
\begin{proof}[Proof of Theorem \ref{Minor: weak Forelli--Rudin} $(i)$]
We first consider the case $0<2-c+b<1$. For sufficiently large integer $N\ge 1$, denote 
\[
	f_N(w):=\frac{\one_{\{1/2<|w|^2<1-2^{-N}\}}(w)}{(1-|w|^2)^{2-c+b}}, \qquad w \in \D.
\]
By \eqref{20260529eq00}, we have
\[
	\|f_N\|_{L^{\frac{1}{2-c+b}}(\D)}
	\simeq N^{2-c+b}.
\]
Set
\[
	A_N:=\{z\in\D:\ 1-2^{-N}<|z|^2<1-2^{-N-1}\}, 
\]
and for each fixed $z\in A_N$, define
\[
	E_z:=\left\{w\in\D:\ 2(1-|z|^2)<1-|w|^2<2^{-1},\ |\arg z-\arg w|<1-|w|^2\right\}.
\]
By \eqref{20260530eq01}, we have
for any $z \in A_N$,
\begin{align*}
    S_{-1,b,c} f_N(z)
    &=(1 - |z|^2) ^{-1} \int_{\D} \frac {(1-|w|^{2}) ^{b}}{|1 - z\bar{w}|^{c}} f_N(w) dA(w)  \\
    & \ge (1 - |z|^2) ^{-1} \int_{E_z} \frac {(1-|w|^{2}) ^{b}}{|1 - z\bar{w}|^{c}} \cdot \frac{1}{(1-|w|^2)^{2-c+b}} dA(w)  \\
    &\gtrsim (1 - |z|^2) ^{-1} \int_{E_z} \frac{1}{(1-|w|^2)^2} dA(w).
\end{align*}
Moreover,
\[
    \int_{E_z} \frac{1}{(1-|w|^2)^2} dA(w)
    \simeq \int_{2^{-1/2}}^{\sqrt{1-2(1-|z|^2)}} \frac{r}{(1-r^2)^2}
    \left(\int_{\arg z-(1-r^2)}^{\arg z+(1-r^2)} d\theta  \right) dr \gtrsim \log\left(\frac{1}{1-|z|^2}\right).
\]
Hence,
\[
	S_{-1,b,c} f_N(z) 
	\gtrsim  (1 - |z|^2) ^{-1} \log\left(\frac{1}{1-|z|^2}\right) 
	\simeq  N2^{N}, \qquad z \in A_N.
\]
Therefore,
\begin{align*}
    \|S_{-1,b,c} f_N\|_{L^{1,\infty}(\D)}
    &\gtrsim  N2^{N} \left| \left\{z\in\D: |S_{-1,b,c} f_N(z)| \gtrsim N2^{N}\right\} \right| \\
    &\gtrsim  N2^{N}|A_N| 
    \simeq N2^{N} \cdot 2^{-N} =N.
\end{align*}
It follows from $0<2-c+b<1$ that
\[
    \frac{\|S_{-1,b,c} f_N\|_{L^{1,\infty}(\D)}}{\|f_N\|_{L^{\frac{1}{2-c+b}}(\D)}}
    \gtrsim \frac{N}{N^{2-c+b}}=N^{c-b-1}  \to \infty
\]
as $N\to\infty$.

\medskip

Finally, we consider the case $2-c+b=0$. The same argument above applies, with the only modification that we take
\[
f_N(w):=\one_{\{1/2<|w|^2<1-2^{-N}\}}(w).
\]
Then $\|f_N\|_{L^\infty(\D)}=1$, while the preceding lower bound becomes 
\[
\|S_{-1,b,c}f_N\|_{1,\infty}\gtrsim N.
\]
Therefore, $S_{-1,b,c}:L^\infty(\D)\to L^{1,\infty}(\D)$ fails to be bounded.
\end{proof}

\begin{proof}[Proof of Theorem \ref{Minor: weak Forelli--Rudin} $(ii)$]
The proof of this statement is almost the same as the proof of Theorem~\ref{Forelli--Rudin threshold} (ii).
We point out the differences as follows.

For the case $0<2-c+b<1/2$, define $Q_{I,\varepsilon}, \delta, \eta, \calI_N$ as in Section \ref{20260831subsec01} and \ref{20260603sec01}.
Set 
\[
	f_{N,\delta,\varepsilon,\eta}(z)
	:=\sum_{I\in\calI_N}
	\delta(I)\,\one_{Q_{I,\varepsilon}}(z)
	(1-|z|^2)^{c-2-b}
	\left(1+\log \frac{1}{1-|z|^2}\right)^{\eta}
\]
as in \eqref{20260603eq20}.
Then
\[
	\sup_{N \ge 1} \sup_{\delta \in \Delta_N}\|f_{N,\delta, \varepsilon, \eta}\|_{L^{\frac{1}{2-c+b}}(\D)}<\infty.
\]
Note that Lemma~\ref{20260531eq30} also holds for $a,b,c\in\R$ satisfying \eqref{minorcase-final}.
Hence, the same argument as Section 5.3 tells us there exists $\del_N\in\Delta_N$ such that 
\[
	\left|(T_{-1,b,c}f_{N,\del_N,\varepsilon,\eta})(z)\right|
	\gtrsim
	H_{N,\varepsilon,\eta}(z)^{1/2}
	\gtrsim
	2^{N}N^{\eta+\frac{1}{2}},
\]
for every $z\in E_{N,\del_N,\varepsilon,\eta}$.
Moreover, 
\[
	\left|E_{N,\del_N,\varepsilon,\eta}\right|\gtrsim 2^{-N}.
\]
Combining the above estimates, we obtain
\begin{align*}
	\left\|T_{-1,b,c}f_{N,\del_N,\varepsilon,\eta}\right\|_{L^{1,\infty}(\D)}
	&\gtrsim 2^{N}N^{\eta+\frac{1}{2}} \left|\left\{z\in\D:\left|\left(T_{-1,b,c}f_{N,\del_N,\varepsilon,\eta}\right)(z)\right| \gtrsim 2^{N}N^{\eta+\frac{1}{2}}\right\}\right|  \\
	&\gtrsim 2^{N}N^{\eta+\frac{1}{2}}
	\left|E_{N,\del_N,\varepsilon,\eta}\right| \\
	&\gtrsim N^{\eta+\frac{1}{2}}.
\end{align*}
Since $\eta>-1/2$, we have
\[
	\lim_{N\to\infty}
	\left\|T_{-1,b,c}f_{N,\del_N,\varepsilon,\eta}\right\|_{L^{1,\infty}(\D)}
	=+\infty.
\]
This proves the desired unboundedness of
$T_{-1,b,c}:L^{\frac{1}{2-c+b}}(\D)\to L^{1,\infty}(\D)$ under the assumption $0<2-c+b<1/2$.

\medskip

The proof of the case $2-c+b=0$ is a simple modification of the preceding argument. Indeed, one only needs to replace the space $L^{\frac{1}{2-c+b}}(\D)$ by $L^\infty(\D)$ and take $\eta=0$ in the above construction. The remaining details are left to the interested reader.
\end{proof}

\begin{proof}[Proof of Theorem \ref{Minor: weak Forelli--Rudin} $(iii)$]
The same argument as in Proposition \ref{Minor: Reduction III} applies.
Combining the Lorentz--Hölder inequality and Theorem~\ref{Forelli--Rudin threshold} (iii) proves this Theorem.
\end{proof}

\subsection{Analysis of the Minor Case: Restricted Weak-type Bounds at the Critical--Horizontal--Vertical Intersection $\calP_{\textnormal{C-H-V}}$ and the Forelli--Rudin Phenomenon at the level of restricted weak-type estimates}

The goal of this subsection is to study the restricted weak-type bounds at the common point where the critical line, the horizontal cut-off line, and the vertical cut-off line meet.
Equivalently, we have the following condition:
\begin{equation} \tag{${\bf P_2}'$} \label{Minor: conditionP2}
\begin{cases}
	b,c \in \R \quad \textrm{satisfy \eqref{minorcase}}; \\
	c=1.
\end{cases}
\end{equation}
We refer the reader to Figure~\ref{Minor: Fig4} below for a visualization of this case.

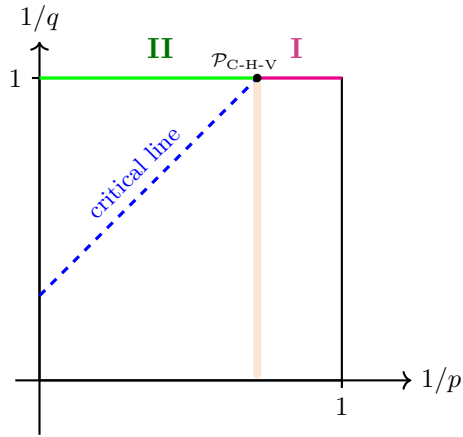
\begin{figure}[H]
\centering
\begin{tikzpicture}[scale=4]
\def\yh{0.28}
\def\xc{0.72}
\def\xright{1.15}
\def\ybottom{-0.15}
  
\draw[->, thick] (-0.08,0) -- (\xright+0.08,0) node[right] {$1/p$};
\draw[->, thick] (0,\ybottom-0.03) -- (0,1.12) node[above] {$1/q$};
\draw[thick] (0,0) rectangle (1,1);
\draw[thick] (1,0) -- (1,-0.025) node[below] {$1$};
\draw[thick] (0,1) -- (-0.025,1) node[left] {$1$};
\draw[blue, dashed, very thick] 
(0,\yh) -- (\xc,1)
node[midway, above, sloped, yshift=1pt] {\small critical line};
\draw[thick, green, very thick] (0,1) -- (\xc,1);
\draw[thick, magenta, very thick] (\xc,1) -- (1,1);\draw[orange!90!black, line width=3pt, opacity=0.18, line cap=round, shorten <=2.2pt, shorten >=2.2pt] (\xc, 0) -- (\xc, 1);

\fill[black] (\xc, 1) circle (.4pt);
\node[above] at (0.68, 1) {\scriptsize $\calP_{\textnormal{C-H-V}}$};

\node[font=\bfseries\large, magenta!85!black] at (0.85,1.1) {I};
\node[font=\bfseries\large, green!45!black] at (0.4,1.1) {II};
\end{tikzpicture}
\caption{\small $\calP_{\textnormal{V-H}}$ belongs to Minor Case II with $c=1$.}
\label{Minor: Fig4}
\end{figure}
\begin{thm}\label{Minor: restricted weak Forelli--Rudin}
Let $b, c \in \R$ satisfy \eqref{Minor: conditionP2}. 
Then the following statements hold.
\begin{itemize}
	\item[(i)] The Forelli--Rudin type operator $S_{-1,b,c}$ is unbounded from $L^{\frac{1}{b+1},1}(\D)$ to $L^{1,\infty}(\D)$.
	\item[(ii)] The Forelli--Rudin type operator $T_{-1,b,c}$ is bounded from $L^{\frac{1}{b+1},1}(\D)$ to $L^{1,\infty}(\D)$.
\end{itemize}
\end{thm}
\begin{proof}[Proof of Theorem \ref{Minor: restricted weak Forelli--Rudin} $(i)$]
Since $c=1$, it suffices to find a sequence of sets $\{F_\varepsilon\}_{\varepsilon>0}$ such that
\[
	\lim_{\varepsilon\to0}\frac{\|S_{-1,b,1} \chi_{F_{\varepsilon}}\|_{L^{1,\infty}(\D)}}{|F_\varepsilon|^{b+1}} = \infty.
\]
For $\varepsilon>0$, let 
\[
	F_{\varepsilon}=\{z\in\D: \varepsilon<1-|z|^2<2\varepsilon\}.
\]
Hence, $|F_{\varepsilon}|=\varepsilon$.
For each fixed $z\in F_{\varepsilon}$, define
\[
	E_{\varepsilon}(z)=\left\{w\in\D: \varepsilon<1-|w|^2<2\varepsilon, 0<\arg z - \arg w<\frac{1}{10}\right\}.
\]
Then $E_{\varepsilon}(z) \subset F_{\varepsilon}$.
Thus for any $z\in F_{\varepsilon}$, we have
\begin{align*}
	S_{-1,b,1} \chi_{F_{\varepsilon}}(z)
	=&(1-|z|^2)^{-1} \int_{\D} \frac{(1-|w|^2)^b}{|1-z\overline{w}|} \chi_{F_{\varepsilon}}(w) dA(w)\\
	\gtrsim& \varepsilon^{b-1} \int_{E_{\varepsilon}(z)} \frac{1}{|1-z\overline{w}|} dA(w)\\
	\gtrsim& \varepsilon^{b-1} \int_{\sqrt{1-2\varepsilon}}^{\sqrt{1-\varepsilon}}\int_{0}^{\frac{1}{10}} \frac{\rho}{\left|1-|z|\rho e^{-i\theta}\right|} d\theta d\rho.
\end{align*}
Note that
\begin{align*}
	\left|1-|z|\rho e^{-i\theta}\right|
	=&\left|1-|z|\rho+|z|\rho-|z|\rho e^{-i\theta}\right|\\
	\leq&(1-|z|\rho)+|z|\rho\left|1-e^{-i\theta}\right|\\
	\leq&(1-|z|)+|z|(1-\rho)+\left|1-e^{-i\theta}\right|\\
	\leq&(1-|z|^2)+(1-\rho^2)+\theta\\
	\leq&4\varepsilon+\theta.
\end{align*}
Hence, 
\begin{align*}
	S_{-1,b,1} \chi_{F_{\varepsilon}}(z)
	\gtrsim& \varepsilon^{b-1} \int_{\sqrt{1-2\varepsilon}}^{\sqrt{1-\varepsilon}}\int_{0}^{\frac{1}{10}} \frac{\rho}{4\varepsilon+\theta} d\theta d\rho\\
	=& \varepsilon^{b}\int_{0}^{\frac{1}{10}} \frac{1}{4\varepsilon+\theta} d\theta.
\end{align*}
For sufficiently small $\varepsilon>0$, we have
\[
	S_{-1,b,1} \chi_{F_{\varepsilon}}(z) \gtrsim \varepsilon^{b}\log\frac{1}{\varepsilon}.
\]
Therefore, 
\[
	\|S_{-1,b,1} \chi_{F_{\varepsilon}}\|_{L^{1,\infty}(\D)}
	\gtrsim \left(\varepsilon^{b}\log\frac{1}{\varepsilon}\right) |F_{\varepsilon}|
	=\varepsilon^{b+1}\log\frac{1}{\varepsilon}.
\]
Furthermore, we have
\[
	\frac{\|S_{-1,b,1} \chi_{F_{\varepsilon}}\|_{L^{1,\infty}(\D)}}{|F_\varepsilon|^{b+1}} 
	\gtrsim \log\frac{1}{\varepsilon} 
	\to \infty
\]
as $\varepsilon\to0$.
This completes the proof.
\end{proof}

\begin{proof}[Proof of Theorem \ref{Minor: restricted weak Forelli--Rudin} $(ii)$]
The same argument as in Proposition \ref{Minor: Reduction III} applies.
Combining the Lorentz--Hölder inequality and Theorem~\ref{mainsubcaseV} (a.3) proves this Theorem.
\end{proof}

\bigskip 

\section{Analysis of the Minor Case: Putting all pieces together} \label{Sec12}

In this section, we put together the building blocks developed in the
preceding section to obtain a complete description of the weak-type and
restricted weak-type behavior of the Forelli--Rudin type operators
$T_{a,b,c}$ and $S_{a,b,c}$ in the \emph{Minor Case}
\eqref{minorcase-final}.
In this case, conditions \eqref{eq:weak} and
\eqref{eq:restricted-weak} are equivalent to
\[
\begin{cases}
a,b,c\in\R \quad \textrm{satisfy \eqref{minorcase-final}},\\
1/2 \le 2-c+b<1,\qquad c>1,
\end{cases}
\qquad
\begin{cases}
a,b,c\in\R \quad \textrm{satisfy \eqref{minorcase-final}},\\
c=1,
\end{cases}
\]
respectively.
Thus, the weak-type and restricted weak-type Forelli--Rudin phenomena
occur at $\calP_{\mathrm{C-H}}$ and
$\calP_{\mathrm{C-H-V}}$, respectively.

Similarly to the argument in the \emph{Main case} \eqref{maincase}, apart from the corresponding Forelli--Rudin pairs, the operators
$T_{a,b,c}$ and $S_{a,b,c}$ have identical boundedness behavior.
The sharpness result in Section \ref{Sec12.1} shows that neither operator can be of
restricted weak type $(p,q)$ when $q>1$.
Consequently, it remains only to analyze the line $q=1$.

In the Minor Case, the strong-type region is empty.
Therefore, we can no longer determine whether an intersection point is active by checking whether it lies on $\partial\Omega_{\mathrm{bd}}$. 
In fact, the only intersection points relevant in the Minor Case are
\[
\calP_{\mathrm{C-H}}=(2-c+b,1)
\qquad\text{and}\qquad
\calP_{\mathrm{V-H}}=(b+1,1),
\]
and which intersection point is active is determined by the relative positions of these two intersection points.
More precisely, $\calP_{\mathrm{C-H}}$ is active when $c>1$,
the two intersections coincide at
$\calP_{\mathrm{C-H-V}}$ when $c=1$, and
$\calP_{\mathrm{V-H}}$ is active when $c<1$.
The corresponding endpoint behavior is summarized in
Figure~\ref{MinorSummaryFig}.

\begin{figure}[ht]
\centering
\begin{tikzpicture}[
    >=Latex,
    every path/.style={line cap=round, line join=round},
    branch/.style={very thick, -{Latex[length=2.2mm]}},
    case label/.style={font=\small, align=left},
    endpoint card/.style={
        draw=gray!45,
        fill=gray!4,
        rounded corners=2pt,
        line width=.65pt,
        text width=4.20cm,
        minimum height=5.55cm,
        inner sep=7pt,
        align=center
    }
]

\definecolor{minorI}{RGB}{116,76,170}
\definecolor{minorII}{RGB}{36,139,122}
\definecolor{triple}{RGB}{205,116,35}
\definecolor{vhblue}{RGB}{45,105,190}

\draw[branch,minorI]
    (-1.30,2.95)
    .. controls (-1.80,3.32) and (-2.55,3.32) ..
    (-3.05,2.95);

\node[case label,anchor=east] at (-3.22,2.95)
    {\textcolor{minorI}{\large\bfseries I}\quad $c>1$\\[-1pt]
     \hspace*{1.25em}
     {\scriptsize $\calP_{\mathrm{C-H}}$ is active}};

\draw[branch,minorII]
    (1.30,2.95)
    .. controls (1.80,3.32) and (2.55,3.32) ..
    (3.05,2.95);

\node[case label,anchor=west] at (3.22,2.95)
    {\textcolor{minorII}{\large\bfseries II}\quad $c\le1$\\[-1pt]
     \hspace*{1.25em}
     {\scriptsize $c=1$:
     $\calP_{\mathrm{C-H-V}}$ is active}\\[-1pt]
     \hspace*{1.25em}
     {\scriptsize $c<1$:
     $\calP_{\mathrm{V-H}}$ is active}};

\node[
    circle,
    draw=gray!55,
    fill=gray!7,
    line width=.8pt,
    minimum size=2.70cm,
    inner sep=2pt,
    align=center
] at (0,2.95)
    {\textsc{Minor Case}\\[-1pt]
     {\scriptsize only the line $q=1$ remains,}\\[-2pt]
     {\scriptsize and split according to the}\\[-2pt]
     {\scriptsize relative position of $\calP_{\mathrm{C-H}}$}\\[-2pt]
     {\scriptsize and $\calP_{\mathrm{V-H}}$}};

\node[endpoint card] at (-4.80,-2.35)
{
    \textcolor{minorI}{\large\bfseries C--H}\\[-1pt]
    {\scriptsize
    $\calP_{\mathrm{C-H}}=(2-c+b,1)$}\\[-1pt]
    {\scriptsize $c>1$}\\[3pt]

    {\color{gray!32}\rule{3.35cm}{.35pt}}\\[2pt]

    {\scriptsize\bfseries At weak-type level}\\[-1pt]
    \textcolor{red!75!black}
        {\scriptsize $S$: unbounded}\\[-1pt]
    {\scriptsize
        $T$: unbounded if
        $0\leq2-c+b<\frac12$}\\[-1pt]
    \textcolor{blue!75!black}
        {\scriptsize
        $T$: bounded if
        $\frac12\le2-c+b<1$}\\[3pt]

    {\scriptsize\bfseries At restricted weak-type level}\\[-1pt]
    \textcolor{blue!75!black}
        {\scriptsize
        $T,S$: bounded if $0<2-c+b<1$}\\[-1pt]
    \textcolor{red!75!black}
        {\scriptsize
        $T,S$: unbounded if $2-c+b=0$}\\[3pt]

    \textcolor{minorI}
        {\scriptsize\bfseries Weak-type F--R phenomenon}\\[-1pt]
    {\tiny when $1/2 \le 2-c+b<1$}
};

\node[endpoint card] at (0,-2.35)
{
    \textcolor{triple}{\large\bfseries C--H--V}\\[-1pt]
    {\scriptsize
    $\calP_{\mathrm{C-H-V}}=(b+1,1)$}\\[-1pt]
    {\scriptsize $c=1$}\\[3pt]

    {\color{gray!32}\rule{3.35cm}{.35pt}}\\[2pt]

    {\scriptsize\bfseries At weak-type level}\\[-1pt]
    \textcolor{red!75!black}
        {\scriptsize $T$ and $S$: unbounded}\\[5pt]

    {\scriptsize\bfseries At restricted weak-type level}\\[-1pt]
    \textcolor{blue!75!black}
        {\scriptsize $T$: bounded}\\[-1pt]
    \textcolor{red!75!black}
        {\scriptsize $S$: unbounded}\\[7pt]

    \textcolor{triple}
        {\scriptsize\bfseries Restricted weak-type}\\[-1pt]
    \textcolor{triple}
        {\scriptsize\bfseries F--R phenomenon}
};

\node[endpoint card] at (4.80,-2.35)
{
    \textcolor{vhblue}{\large\bfseries V--H}\\[-1pt]
    {\scriptsize
    $\calP_{\mathrm{V-H}}=(b+1,1)$}\\[-1pt]
    {\scriptsize $c<1$}\\[3pt]

    {\color{gray!32}\rule{3.35cm}{.35pt}}\\[2pt]

    {\scriptsize\bfseries At weak-type level}\\[-1pt]
    \textcolor{red!75!black}
        {\scriptsize $T$ and $S$: unbounded}\\[5pt]

    {\scriptsize\bfseries At restricted weak-type level}\\[-1pt]
    \textcolor{blue!75!black}
        {\scriptsize $T$ and $S$: bounded}\\[7pt]

    \textcolor{gray!70!black}
        {\scriptsize no Forelli--Rudin phenomenon\\[-1pt]
        at this endpoint}
};

\end{tikzpicture}

\caption{Summary of the Minor Cases and the corresponding active
intersection points.}
\label{MinorSummaryFig}
\end{figure}
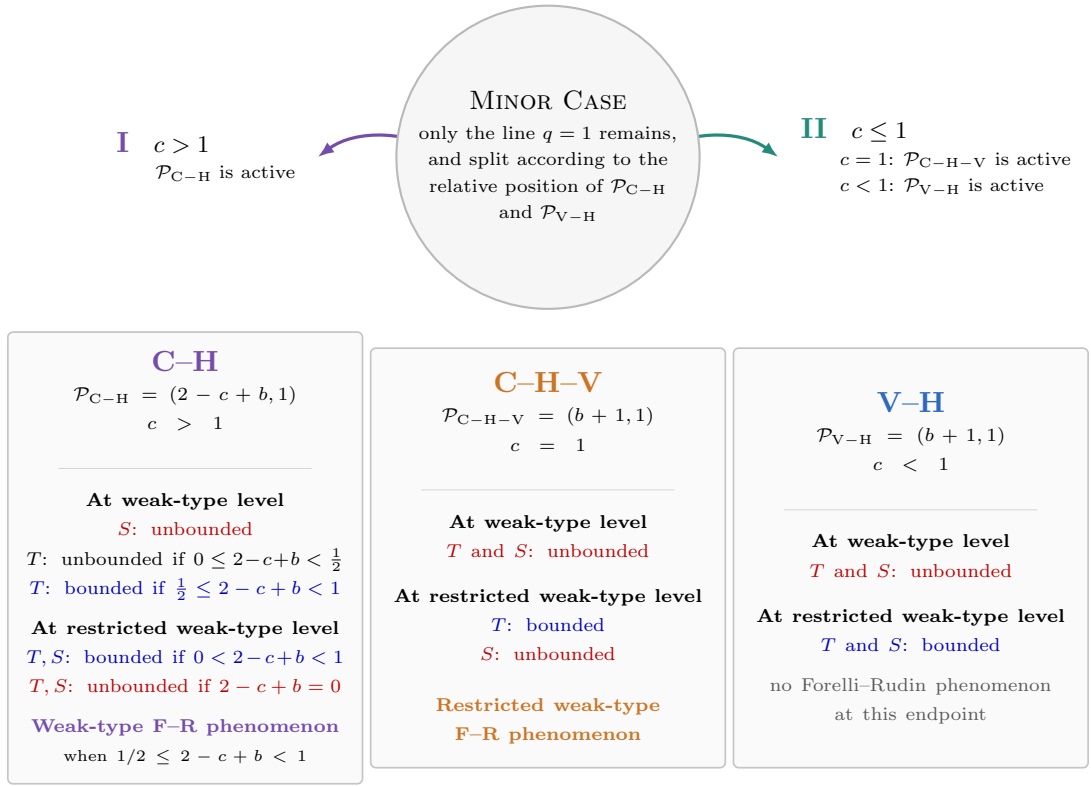

Before proceeding to the proofs, we summarize the conclusions.
The following statements give the complete picture in the Minor Case.

\begin{itemize}
    \item If $q>1$, neither $T_{a,b,c}$ nor $S_{a,b,c}$ is
    of restricted weak type $(p,q)$ for any $1\le p\le\infty$.
    Hence neither operator is of weak type $(p,q)$.

    \item On the line $q=1$, both operators are of weak type
    $(p,1)$ whenever
    \[
        0\le\frac1p<
        \min\{2-c+b,b+1\}.
    \]
    At the active endpoint, both operators fail to be of weak type,
    except when
    \[
        c>1
        \qquad\text{and}\qquad
        \frac12\le2-c+b<1,
    \]
    in which case $T_{a,b,c}$ is of weak type while
    $S_{a,b,c}$ fails.

    \item If $2-c+b=0$, neither operator is of restricted weak type
    $(p,1)$ for any $1\le p\le\infty$.
    Otherwise, both operators are of restricted weak type at the
    active endpoint when $c\ne1$.
    When $c=1$, $T_{a,b,c}$ is of restricted weak type at
    $\calP_{\mathrm{C-H-V}}$, whereas $S_{a,b,c}$ fails.
    To the right of the active endpoint, both operators fail to be
    of restricted weak type.
\end{itemize}

\subsection{Sharpness of the Estimates in the Minor Cases}\label{Sec12.1}

By Proposition \ref{Minor: Reduction I}, the sharpness in the minor cases follows from the following proposition.
\begin{prop}
Let $a,b,c$ satisfy \eqref{minorcase}.
If $1\leq p \leq \infty$ and $1<q\leq\infty$, then neither $T_{a,b,c}$ nor $S_{a,b,c}$ is of restricted weak type $(p,q)$.
In particular, neither operator is of weak type $(p,q)$.
\end{prop}
\begin{proof}
It suffices to test the restricted weak-type estimate on $\one_{\D}$.
Indeed, an easy computation yields
\[
	T_{-1,b,c}\one_{\D}(z) = \frac{1}{1-|z|^2} \int_{\D}(1-|w|^2)^b\,dA(w)
	= \frac{1}{(b+1)(1-|z|^2)}.
\]
This function does not belong to $L^{q,\infty}(\D)$ for any $1<q\leq\infty$.
Hence $T_{-1,b,c}$ fails to be of restricted weak type $(p,q)$ for every $1\leq p\leq\infty$.
The same conclusion holds for $S_{-1,b,c}$, since $
\bigl|T_{-1,b,c}\one_{\D}\bigr| \leq S_{-1,b,c}\one_{\D}$. 
\end{proof}

It therefore remains to prove the classifications on the line
$q=1$ stated above.

\subsection{Treatment of Minor Case I}

Recall that \emph{Minor Case I} refers to the case where the vertical--horizontal intersection $\calP_{\textnormal{V--H}}$ lies in the region
$$
\left\{ (p,q): \frac{1}{p} > 2-c+b, \; q=1 \right\};
$$
see Figure \ref{Minor: Fig1}. 

\vspace{0.1cm}

Our main result is the following. 

\begin{thm} \label{minorsubcaseI}
Let $a, b, c \in \R$ satisfy \eqref{minorcase-final} with $c>1$. Then the following statements hold. 
\begin{enumerate}
	 \item [(a)] If $2-c+b=0$, then neither $T_{a,b,c}$ nor $S_{a,b,c}$ is of weak type or restricted weak type $(p,1)$ for any $1\le p\le\infty$.
    \item [(b)] If $0<2-c+b<1/2$,
    then 
    \begin{enumerate}
        \item [(1)] {\bf (Weak--type bounds)} 
       $T_{a, b, c}$ and $S_{a, b, c}$ map $L^p(\D)$ boundedly into $L^{1, \infty}(\D)$ if
        \[
        0 \le \frac{1}{p} < 2-c+b,
        \]
        while at point $\left(\frac{1}{2-c+b}, 1\right)$, $T_{a, b, c}$ and $S_{a, b, c}$ map $L^{\frac{1}{2-c+b}}(\D)$ unboundedly into $L^{1, \infty}(\D)$.
        
        \item [(2)] {\bf (Restricted weak-type bounds)}
        $T_{a, b, c}$ and $S_{a, b, c}$ map $L^{\frac{1}{2-c+b},1}(\D)$ boundedly into $L^{1, \infty}(\D)$, while $T_{a, b, c}$ and $S_{a, b, c}$ map $L^{p,1}(\D)$ unboundedly into $L^{1, \infty}(\D)$ if
        \[
        2-c+b < \frac{1}{p} \le 1.
        \]
    \end{enumerate}
    
    \medskip
    
    \item [(c)] If $1/2 \le 2-c+b<1$,
    then 
    \begin{enumerate}
        \item [(1)] {\bf (Weak--type bounds)} 
       $T_{a, b, c}$ and $S_{a, b, c}$ map $L^p(\D)$ boundedly into $L^{1, \infty}(\D)$ if
        \[
        0 \le \frac{1}{p} < 2-c+b.
        \]
        
        \item [(2)] {\bf (Weak-type Forelli--Rudin phenomenon)}
        At the point $(p, q)=\left(\frac{1}{2-c+b}, 1\right)$, $T_{a, b, c}$ maps $L^{\frac{1}{2-c+b}}(\D)$ boundedly into $L^{1, \infty}(\D)$ while $S_{a, b, c}$ fails.
        
        \item [(3)] {\bf (Restricted weak-type bounds)}
        $T_{a, b, c}$ and $S_{a, b, c}$ map $L^{\frac{1}{2-c+b},1}(\D)$ boundedly into $L^{1, \infty}(\D)$, while $T_{a, b, c}$ and $S_{a, b, c}$ map $L^{p,1}(\D)$ unboundedly into $L^{1, \infty}(\D)$ if
        \[
        2-c+b < \frac{1}{p} \le 1.
        \]
    \end{enumerate}
\end{enumerate}
\end{thm}

We refer the reader to Figure~\ref{Fig25} below for an illustration of Minor Case~I.

\begin{figure}[ht]
\centering

\begin{minipage}[t]{0.48\textwidth}
\centering
\begin{tikzpicture}[scale=3.3]
    \def\xA{0.4}
    \def\yB{0.6}

    \def\xV{0.9}
    \def\yH{1}

    \def\xright{1.15}
    \def\ybottom{-0.15}

    \coordinate (B) at (0,\yB);
    \coordinate (A) at (\xA,1);
    \coordinate (PVH) at (\xV,\yH);

    \draw[->, thick]
        (-0.08,0) -- (\xright+0.08,0) node[right] {$1/p$};
    \draw[->, thick]
        (0,\ybottom-0.03) -- (0,1.12) node[above] {$1/q$};

    \draw[thick] (0,0) rectangle (1,1);
    \draw[thick] (1,0) -- (1,-0.025) node[below] {$1$};
    \draw[dashed] (0.5,1) -- (0.5,-0.025) node[below] {$1/2$};
    \draw[thick] (0,1) -- (-0.025,1) node[left] {$1$};

   \draw[blue, very thick]
        (0,1) -- (A);
   \draw[black, dashed]
        (B) -- (A);
   \draw[violet, very thick]
        (1,1) -- (A);

    \draw[orange!85!black, line width=3pt, opacity=0.22, line cap=round, shorten <=2.2pt, shorten >=2.2pt] (\xV,\ybottom) -- (\xV, 1);
    
    \draw[violet!50, line width=5pt, opacity=0.18, line cap=round, shorten <=2.2pt, shorten >=2.2pt] (1,1) -- (A);

    \draw[violet!60, line width=3pt, opacity=0.28, line cap=round, shorten <=2.2pt, shorten >=2.2pt] (1,1) -- (A);

    \fill[red] (A) circle (1.05pt);

    \fill[black] (\xV, \yH) circle (0.55pt);
    \node[below] at (\xV, \yH)
        {\tiny $\calP_{\textnormal{V-H}}$};

    \node[above=2pt] at (A)
        {\tiny $\bigl(2-c+b,1\bigr)$};
\end{tikzpicture}

\vspace{0.05cm}
{\small (b)}
\end{minipage}
\hfill
\begin{minipage}[t]{0.48\textwidth}
\centering
\begin{tikzpicture}[scale=3.3]
    \def\xA{0.6}
    \def\yB{0.4}

    \def\xV{0.9}
    \def\yH{1}

    \def\xright{1.15}
    \def\ybottom{-0.15}

    \coordinate (B) at (0,\yB);
    \coordinate (A) at (\xA,1);
    \coordinate (PVH) at (\xV,\yH);

    \draw[->, thick]
        (-0.08,0) -- (\xright+0.08,0) node[right] {$1/p$};
    \draw[->, thick]
        (0,\ybottom-0.03) -- (0,1.12) node[above] {$1/q$};

    \draw[thick] (0,0) rectangle (1,1);
    \draw[thick] (1,0) -- (1,-0.025) node[below] {$1$};
    \draw[dashed] (0.5,1) -- (0.5,-0.025) node[below] {$1/2$};
    \draw[thick] (0,1) -- (-0.025,1) node[left] {$1$};

   \draw[blue, very thick]
        (0,1) -- (A);
   \draw[black, dashed]
        (B) -- (A);
   \draw[violet, very thick]
        (1,1) -- (A);

    \draw[orange!85!black, line width=3pt, opacity=0.22,
          line cap=round, shorten <=2.2pt, shorten >=2.2pt]
        (\xV,\ybottom) -- (\xV, 1);

    \draw[violet!50, line width=5pt, opacity=0.18, line cap=round, shorten <=2.2pt, shorten >=2.2pt] (1,1) -- (A);

    \draw[violet!60, line width=3pt, opacity=0.28, line cap=round, shorten <=2.2pt, shorten >=2.2pt] (1,1) -- (A);

    \fill[gray] (A) circle (1.05pt);

    \fill[black] (\xV, \yH) circle (0.55pt);
    \node[below] at (\xV, \yH)
        {\tiny $\calP_{\textnormal{V-H}}$};

    \node[above=2pt] at (A)
        {\tiny $\bigl(2-c+b,1\bigr)$};
\end{tikzpicture}

\vspace{0.05cm}
{\small (c)}
\end{minipage}

\caption{\small
Weak-type and restricted weak-type behavior in Minor Case~I. The blue segments and points represent weak-type bounds, the red points represent restricted weak-type bounds only. The violet segments represent failure even of restricted weak-type bounds and the gray points represent the weak-type Forelli--Rudin phenomenon. As usual, the orange lines represent the vertical cut-off line~\eqref{verticalcut}.}
\label{Fig25}
\end{figure}
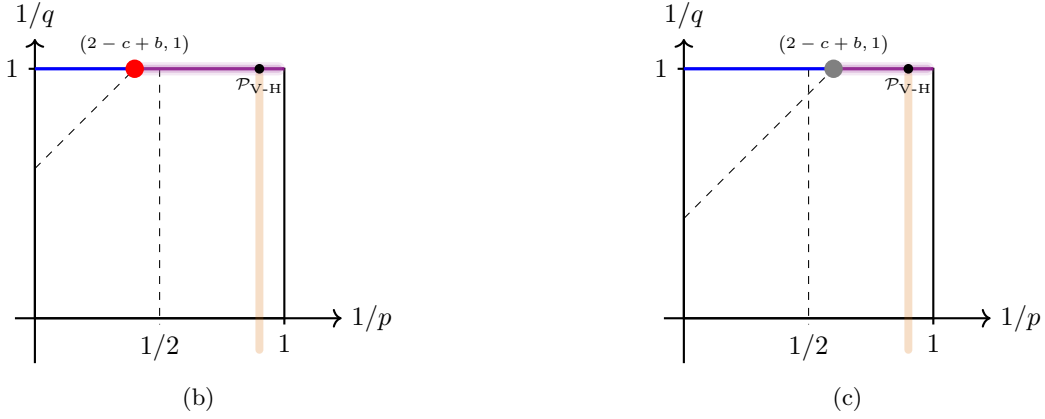

\begin{proof}
$\bullet$ By Proposition~\ref{reductionII}(i), it remains only to consider the endpoint $(p,q)=(\infty,1)$. Applying the argument from the proof of Theorem~\ref{restricted-weak Forelli--Rudin at CH}(i) to the test functions constructed in the proof of Theorem~\ref{Minor: weak Forelli--Rudin}(ii) with $\eta=0$, we conclude that neither $T_{a,b,c}$ nor $S_{a,b,c}$ is of restricted weak type $(\infty,1)$. Hence, part~(a) follows.

$\bullet$ Part~(b.1) is a consequence of Proposition~\ref{Minor: Reduction IV}, together with parts~(i) and (ii) of Theorem~\ref{Minor: weak Forelli--Rudin}. 
The restricted weak-type estimate asserted in part~(b.2) follows from Proposition~\ref{Minor: Reduction III}, while the failure of restricted weak-type estimate follows from part~(i) of Proposition~\ref{Minor: Reduction I}.

$\bullet$ Part~(c.1) follows from Propositions~\ref{Minor: Reduction IV}.
Part~(c.2) is a direct consequence of parts~(i) and (iii) of Theorem~\ref{Minor: weak Forelli--Rudin}.
Combining Proposition~\ref{Minor: Reduction III} and part~(i) of Proposition~\ref{Minor: Reduction I} yields part~(c.3).
\end{proof}

\subsection{Treatment of Minor Case II}

Recall that \emph{Minor Case II} refers to the case where the vertical--horizontal intersection $\calP_{\textnormal{V--H}}$ lies in the region
$$
\left\{ (p,q): \frac{1}{p} \le 2-c+b, \; q=1 \right\};
$$
see Figure \ref{Minor: Fig1}. 

\vspace{0.1cm}

Our main result is the following. 

\begin{thm} \label{minorsubcaseII}
Let $a, b, c \in \R$ satisfy \eqref{minorcase-final} with $c \le 1$. Then the following statements hold. 
\begin{enumerate}
    \item [(a)] If $c=1$,
    then 
    \begin{enumerate}
        \item [(1)] {\bf (Weak--type bounds)} 
       $T_{a, b, c}$ and $S_{a, b, c}$ map $L^p(\D)$ boundedly into $L^{1, \infty}(\D)$ if
        \[
        0 \le \frac{1}{p} < b+1,
        \]
        while at point $\left(\frac{1}{b+1}, 1\right)$, $T_{a, b, c}$ and $S_{a, b, c}$ map $L^{\frac{1}{b+1}}(\D)$ unboundedly into $L^{1, \infty}(\D)$.
        
        \item [(2)] {\bf (Restricted weak-type Forelli--Rudin phenomenon)}
        At the point $(p, q)=\left(\frac{1}{b+1}, 1\right)$, $T_{a, b, c}$ maps $L^{\frac{1}{b+1},1}(\D)$ boundedly into $L^{1, \infty}(\D)$ while $S_{a, b, c}$ fails.
        
		  \item [(3)] {\bf (Restricted weak-type bounds)} 
		  $T_{a, b, c}$ and $S_{a, b, c}$ map $L^{p,1}(\D)$ unboundedly into $L^{1, \infty}(\D)$ if
        \[
        b+1 < \frac{1}{p} \le 1.
        \]
    \end{enumerate}
    
    \medskip
    
    \item [(b)] If $c<1$,
    then 
    \begin{enumerate}
        \item [(1)] {\bf (Weak--type bounds)} 
       $T_{a, b, c}$ and $S_{a, b, c}$ map $L^p(\D)$ boundedly into $L^{1, \infty}(\D)$ if
        \[
        0 \le \frac{1}{p} < b+1,
        \]
		  while at point $\left(\frac{1}{b+1}, 1\right)$, $T_{a, b, c}$ and $S_{a, b, c}$ map $L^{\frac{1}{b+1}}(\D)$ unboundedly into $L^{1, \infty}(\D)$.
        
        \item [(2)] {\bf (Restricted weak-type bounds)}
        $T_{a, b, c}$ and $S_{a, b, c}$ map $L^{\frac{1}{b+1},1}(\D)$ boundedly into $L^{1, \infty}(\D)$, while $T_{a, b, c}$ and $S_{a, b, c}$ map $L^{p,1}(\D)$ unboundedly into $L^{1, \infty}(\D)$ if
        \[
        b+1 < \frac{1}{p} \le 1.
        \]
    \end{enumerate}
\end{enumerate}
\end{thm}

We refer the reader to Figure~\ref{Fig26} below for an illustration of Minor Case~II.

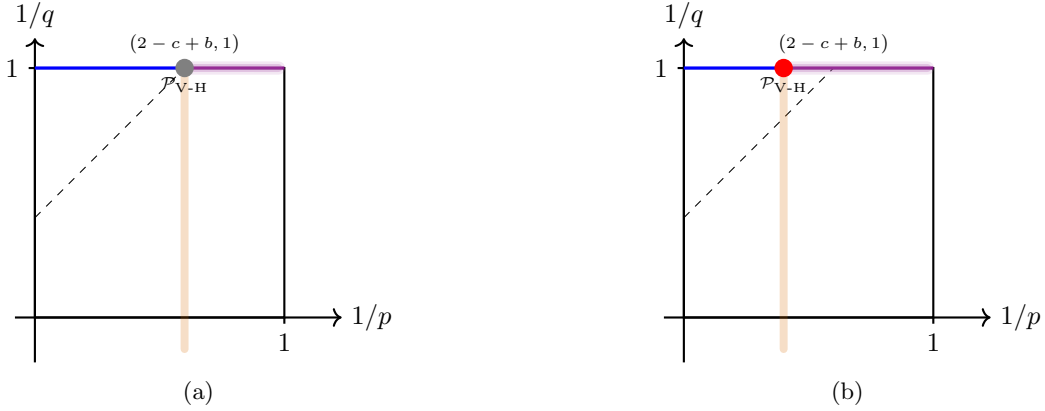
\begin{figure}[ht]
\centering

\begin{minipage}[t]{0.48\textwidth}
\centering
\begin{tikzpicture}[scale=3.3]
    \def\xA{0.6}
    \def\yB{0.4}

    \def\xV{0.6}
    \def\yH{1}

    \def\xright{1.15}
    \def\ybottom{-0.15}

    \coordinate (B) at (0,\yB);
    \coordinate (A) at (\xA,1);
    \coordinate (PVH) at (\xV,\yH);

    \draw[->, thick]
        (-0.08,0) -- (\xright+0.08,0) node[right] {$1/p$};
    \draw[->, thick]
        (0,\ybottom-0.03) -- (0,1.12) node[above] {$1/q$};

    \draw[thick] (0,0) rectangle (1,1);
    \draw[thick] (1,0) -- (1,-0.025) node[below] {$1$};
    \draw[thick] (0,1) -- (-0.025,1) node[left] {$1$};

   \draw[blue, very thick]
        (0,1) -- (A);
   \draw[black, dashed]
        (B) -- (A);
   \draw[violet, very thick]
        (1,1) -- (A);

    \draw[orange!85!black, line width=3pt, opacity=0.22,
          line cap=round, shorten <=2.2pt, shorten >=2.2pt]
        (\xV,\ybottom) -- (\xV, 1);

    \draw[violet!50, line width=5pt, opacity=0.18, line cap=round, shorten <=2.2pt, shorten >=2.2pt] (1,1) -- (A);

    \draw[violet!60, line width=3pt, opacity=0.28, line cap=round, shorten <=2.2pt, shorten >=2.2pt] (1,1) -- (A);

    \fill[gray] (A) circle (1.05pt);

    \node[below] at (\xV, \yH)
        {\tiny $\calP_{\textnormal{V-H}}$};

    \node[above=2pt] at (A)
        {\tiny $\bigl(2-c+b,1\bigr)$};
\end{tikzpicture}

\vspace{0.05cm}
{\small (a)}
\end{minipage}
\hfill
\begin{minipage}[t]{0.48\textwidth}
\centering
\begin{tikzpicture}[scale=3.3]
    \def\xA{0.6}
    \def\yB{0.4}

    \def\xV{0.4}
    \def\yH{1}

    \def\xright{1.15}
    \def\ybottom{-0.15}

    \coordinate (B) at (0,\yB);
    \coordinate (A) at (\xA,1);
    \coordinate (PVH) at (\xV,\yH);

    \draw[->, thick]
        (-0.08,0) -- (\xright+0.08,0) node[right] {$1/p$};
    \draw[->, thick]
        (0,\ybottom-0.03) -- (0,1.12) node[above] {$1/q$};

    \draw[thick] (0,0) rectangle (1,1);
    \draw[thick] (1,0) -- (1,-0.025) node[below] {$1$};
    \draw[thick] (0,1) -- (-0.025,1) node[left] {$1$};

   \draw[blue, very thick]
        (0,1) -- (\xV,1);
   \draw[black, dashed]
        (B) -- (A);
   \draw[violet, very thick]
        (1,1) -- (\xV,1);

    \draw[orange!85!black, line width=3pt, opacity=0.22,
          line cap=round, shorten <=2.2pt, shorten >=2.2pt]
        (\xV,\ybottom) -- (\xV, 1);

    \draw[violet!50, line width=5pt, opacity=0.18, line cap=round, shorten <=2.2pt, shorten >=2.2pt] (1,1) -- (\xV,1);

    \draw[violet!60, line width=3pt, opacity=0.28, line cap=round, shorten <=2.2pt, shorten >=2.2pt] (1,1) -- (\xV,1);

    \fill[red] (\xV,1) circle (1.05pt);

    \node[below] at (\xV, \yH)
        {\tiny $\calP_{\textnormal{V-H}}$};

    \node[above=2pt] at (A)
        {\tiny $\bigl(2-c+b,1\bigr)$};
\end{tikzpicture}

\vspace{0.05cm}
{\small (b)}
\end{minipage}

\caption{\small
Weak-type and restricted weak-type behavior in Minor Case~II. The blue segments and points represent weak-type bounds, while the red point represents restricted weak-type bounds only. The violet segments represent failure even of restricted weak-type bounds, and the gray points represent the restricted weak-type Forelli--Rudin phenomenon. As usual, the orange lines represent the vertical cut-off line~\eqref{verticalcut}.}
\label{Fig26}
\end{figure}

\begin{proof}
$\bullet$ Part~(a.1) is a consequence of Proposition~\ref{Minor: Reduction IV}, together with Proposition~\ref{Minor: Reduction II}. 
The restricted weak-type Forelli--Rudin phenomenon asserted in part~(a.2) follows from Theorem~\ref{Minor: restricted weak Forelli--Rudin}, while the failure of restricted weak-type estimate in part~(a.3) follows from part~(ii) of Proposition~\ref{Minor: Reduction I}.

$\bullet$ Part~(b.1) follows from Propositions~\ref{Minor: Reduction IV} and Proposition~\ref{Minor: Reduction II}.
Part~(b.2) is a consequence of Proposition~\ref{Minor: Reduction III} and part~(ii) of Proposition~\ref{Minor: Reduction I}.
\end{proof}

\bigskip 

\section{Further remark: extension to higher dimensions}\label{Sec14}

In the final section of the paper, we explain how the approach developed here extends to Forelli--Rudin operators on the unit ball in $\mathbb{C}^n$. 

Let $\B_n$ denote the unit ball in $\C^n$. In this setting, the Forelli--Rudin operators $T_{a,b,c}^{(n)}$ and $S_{a,b,c}^{(n)}$ are defined, respectively, by
$$
T_{a,b,c}^{(n)}f(z):=(1-|z|^2)^a\int_{\B_n}\frac{(1-|w|^2)^b}{(1-\langle z,w\rangle)^c}f(w)\,dV(w),
$$
and
$$
S_{a,b,c}^{(n)}f(z):=(1-|z|^2)^a\int_{\B_n}\frac{(1-|w|^2)^b}{|1-\langle z,w\rangle|^c}f(w)\,dV(w),
$$
where $dV$ denotes the normalized volume measure on $\B_n$. By \cite[Theorem~1.1]{ZZ2022}, the critical line in this setting is given by
\begin{equation} \tag{CL$_n$} \label{higher-dimensional-critical-line}
	\frac{1}{q}=\frac{1}{p}+c-a-b-(n+1),
\end{equation}
whereas the vertical and horizontal cut-off lines remain unchanged and are given by \eqref{verticalcut} and \eqref{horizontalcut}, respectively. 

Since the arguments developed in this paper are largely dyadic in nature, the passage from $\D$ to $\B_n$ requires only natural adjustments to the underlying geometry, while the main mechanisms remain unchanged. In view of the many cases and endpoint configurations considered above, it is more convenient to formulate their higher-dimensional analogues through a concise dictionary of the relevant dimension-dependent quantities, followed by a discussion of the corresponding modifications to the proofs. The resulting correspondence between the dimension-dependent quantities in the two settings is summarized in Table~\ref{higher-dimensional-dictionary} below.

\begin{table}[ht]
\centering
\renewcommand{\arraystretch}{1.35}
\begin{tabular}{c|c}
\hline
Unit-disc setting & Unit-ball setting \\ \hline
$c-a-b-2$ & $c-a-b-(n+1)$ \\
$2-c+b$ & $n+1-c+b$ \\
$c-a-1$ & $c-a-n$ \\
$c-2-b$ & $c-(n+1)-b$ \\
$c-2$ & $c-(n+1)$ \\
$3+a+b-c$ & $n+2+a+b-c$ \\
$c>1$, $c=1$, or $c<1$ & $c>n$, $c=n$, or $c<n$, respectively \\ \hline
\end{tabular}
\caption{Correspondence between the dimension-dependent quantities in the unit-disc and unit-ball settings.}
\label{higher-dimensional-dictionary}
\end{table}

Accordingly, the three distinguished intersection points are given by
\begin{align*}
\calP_{\textnormal{C-H}}^{(n)}&:=\bigl(n+1-c+b,-a\bigr),\\
\calP_{\textnormal{C-V}}^{(n)}&:=\bigl(b+1,c-a-n\bigr),\\
\calP_{\textnormal{V-H}}^{(n)}&:=\bigl(b+1,-a\bigr).
\end{align*}
In particular, the critical, vertical, and horizontal lines meet at a common point if and only if $c=n$.

With these modifications, the higher-dimensional analogues of the results established above follow from the same arguments, subject to the modifications summarized in Table~\ref{higher-dimensional-dictionary}. We therefore omit a case-by-case repetition and leave the routine details to the interested reader. In particular, the Forelli--Rudin phenomenon persists in higher dimensions at both the weak-type and restricted weak-type levels. Moreover, the weak-type Forelli--Rudin threshold remains
$$
\frac{1}{p}=\frac{1}{2},
\qquad\textnormal{or equivalently,}\qquad
p=2,
$$
independently of the dimension $n$. This threshold is independent of the dimension because the exponent $1/2$ arises from Khintchine's inequality in the construction of the counterexample and from the endpoint $p'=2$ in \cite[Theorem~4.41]{Zhu2005} used in the argument. Finally, for completeness, we conclude this section by recording several important modifications in the higher-dimensional proofs that require some additional care.

\subsection{Dyadic structure in $\B_n$}

Recall that the probabilistic construction of the counterexample used in the proof of Theorem~\ref{Forelli--Rudin threshold}, (ii), as developed in Sections~\ref{20260831subsec01} and~\ref{20260603sec01}, relies on a Whitney-type decomposition of $\D$ associated with a fixed dyadic system $\calD$ on $\T$. The higher-dimensional analogue of this decomposition is standard; see, e.g., \cite{RTW2017}. For the reader's convenience, we recall the construction below in a form adapted to the present setting.

\begin{defn} 
A collection of sets  
\[ 
	\mathcal{D}=\{\mathfrak B_{k,i}\subset\partial\B_n: k,i\in\N ,1\leq i\leq M(k)\} 
\] 
is called a \emph{dyadic system} on $\partial\B_n$ if it satisfies the following conditions: 
\begin{itemize} 
	\item[(i)] for each integer $k\geq 1$, $\{\mathfrak B_{k,i}\}_{i=1}^{M(k)}$ is a disjoint covering of $\partial\B_n$, i.e., 
	\begin{itemize} 
		\item[(a)] $\partial\B_n=\bigcup\limits_{i=1}^{M(k)}\mathfrak B_{k,i}$ and  
		\item[(b)] for all $1\leq i\neq j\leq M(k)$, $\mathfrak B_{k,i}\cap\mathfrak B_{k,j}=\emptyset$. 
	\end{itemize} 
	\item[(ii)] if integers $k\geq l\geq1$, then either $\mathfrak B_{k,i}\subset\mathfrak B_{l,j}$ or $\mathfrak B_{k,i}\cap\mathfrak B_{l,j}=\emptyset$. 
\end{itemize} 
Note that such a system on $\partial\B_n$ exists since $(\partial\B_n,\rho,\sigma)$ is a space of homogeneous type and one can use the Hyt\"onen--Kairema decomposition \cite{HK2012} to guarantee the existence of such a dyadic system on $\partial\B_n$. Here, $\rho(x,y):=\sqrt{|1-\langle x,y\rangle|}$, $x,y\in\partial\B_n$, refers to the Cauchy metric on $\partial\B_n$, and $\sigma$ denotes the normalized surface measure on $\partial\B_n$. 

Finally, for each $k\in\N$, denote the $k$-th layer of $\mathcal{D}$ by  
\[
	\mathcal{D}_k=\{\mathfrak B_{k,i}:1\leq i\leq M(k)\}.
\] 
\end{defn} 
 
For any $x\in\partial\B_n$ and $0<r<1$, denote  
\[ 
	B_\rho(x,r)=\{y\in\partial\B_n:\rho(x,y)<r\}. 
\] 
It is a well-known fact that 
\begin{equation}\label{ballarea} 
	\sigma(B_\rho(x,r))\simeq r^{2n}. 
\end{equation}  
 
\begin{defn} 
A dyadic system $\mathcal{D}$ on $\partial\B_n$ is said to be \emph{associated with a triple of positive constants $(\gamma,\kappa_0,\kappa_1)$} if there exists a collection of points 
$$ 
\mathcal{P}:=\{x_{k,i}\in\partial\B_n:k\in\N,\ 1\leq i\leq M(k)\} 
$$ 
such that 
\begin{equation}\label{eq:dyadic-to-ball} 
	B_\rho(x_{k,i},\kappa_0\gamma^{k/2}) 
	\subseteq\mathfrak B_{k,i} 
	\subseteq B_\rho(x_{k,i},\kappa_1\gamma^{k/2}) 
\end{equation} 
for every $k\in\N$ and $1\leq i\leq M(k)$. 
 
The elements of $\mathcal{D}$ are referred to as \emph{dyadic cubes}, and $x_{k,i}$ is called the \emph{center} of the dyadic cube $\mathfrak B_{k,i}$. 
\end{defn} 
 
The following result on $\partial\B_n$ is a consequence of \cite[Theorem~4.1]{HK2012}. 
 
\begin{prop} 
There are constants $\gamma,\kappa_0,\kappa_1>0$ such that 
\begin{itemize} 
	\item[(i)] There is a dyadic system $\mathcal{D}$ on $\partial\B_n$ associated with the triple $(\gamma,\kappa_0,\kappa_1)$ on $\partial\B_n$. 
	\item[(ii)] There is a finite collection of dyadic systems $\{\mathcal{D}_t\}_{t=1}^K$ on $\partial\B_n$ associated with the triple $(\gamma,\kappa_0,\kappa_1)$ on $\partial\B_n$. In addition, for every ball $B\subset\partial\B_n$, there exist $1 \leq t \leq K$ and $\mathfrak B\in\mathcal{D}_t$ such that  
		\[ 
			B\subset\mathfrak B, 
			\qquad\textrm{and}\qquad  
			\operatorname{diam}_\rho{\mathfrak B}\lesssim\operatorname{diam}_\rho{B}, 
		\] 
		where the implicit constant depends only on $\gamma,\kappa_0,\kappa_1$. 
\end{itemize} 
\end{prop} 
 
Now we define Carleson-type boxes in $\B_n$. Assume $\mathcal{D}$ is a dyadic system associated with a triple $(\gamma,\kappa_0,\kappa_1)$. For $k\in\N$ and a dyadic cube $\mathfrak B\in\mathcal{D}_k$, the Carleson box associated with $\mathfrak B$ is defined by 
\[ 
	Q_{\mathfrak B}:=\left\{z\in\B_n:\frac{z}{|z|}\in\mathfrak B,\;1-\gamma^k<|z|<1\right\}. 
\] 
Assume $x$ is the center of the dyadic cube $\mathfrak B$. For $0<\varepsilon<1$, we have
$$
B_\rho(x,\varepsilon\kappa_0\gamma^{k/2})\subset\mathfrak B.
$$
Denote 
\[ 
	Q_{\mathfrak B,\varepsilon}:=\left\{z\in\B_n:\frac{z}{|z|}\in B_\rho\left(x,\varepsilon\kappa_0\gamma^{k/2}\right),\;1-\gamma^k<|z|<1-(1-\varepsilon)\gamma^k\right\}. 
\] 
As a consequence of \eqref{ballarea}, we have  
\[ 
	|Q_{\mathfrak B,\varepsilon}|\simeq\varepsilon^{2n+1}\gamma^{(n+1)k}. 
\] 
For each $N\in\N$, let 
\[ 
	A_N=\{z\in\B_n:1-\gamma^N\leq|z|<1-\gamma^{N+1}\}. 
\] 

We also record the following important geometric lemma, which is the higher-dimensional analogue of Lemma~\ref{20260530lem01}.

\begin{lem}
Let $N\geq1$ be sufficiently large. Furthermore, let $1\leq k\leq N$ and $\mathfrak B\in\calD_k$. Then, for every $\delta>0$, there exists $\varepsilon=\varepsilon(\delta)>0$, independent of the choices of $\mathfrak B$, $N$, and $k$, such that
$$
\sup_{\substack{z\in A_N\cap Q_{\mathfrak B}\\ w,w'\in Q_{\mathfrak B,\varepsilon}}}
\left|
\arg(1-\langle z,w\rangle)-\arg(1-\langle z,w'\rangle)
\right|
\leq\delta.
$$
\end{lem}

\begin{proof}
The proof of this result is similar to that of Lemma~\ref{20260530lem01}. Let $1\leq k\leq N$ and $\mathfrak B\in\mathcal{D}_k$. For $z\in A_N\cap Q_{\mathfrak B}$ and $w,w'\in Q_{\mathfrak B,\varepsilon}$, a direct computation gives 
\begin{itemize}
	\item[(i)] $\gamma^{N+1}\leq1-|z|^2\leq2\gamma^N$;
	\item[(ii)] $(1-\varepsilon)\gamma^k\leq1-|w|^2\leq2\gamma^k$;
	\item[(iii)] $(1-\varepsilon)\gamma^k\leq|1-\langle z,w\rangle|\lesssim\gamma^k$;
	\item[(iv)] $|\langle z,w-w'\rangle|\lesssim\varepsilon\gamma^k$.
\end{itemize}
Consequently,
$$
\left|
\frac{1-\langle z,w\rangle}{1-\langle z,w'\rangle}-1
\right|
=\frac{|\langle z,w-w'\rangle|}{|1-\langle z,w'\rangle|}
\lesssim\frac{\varepsilon}{1-\varepsilon}.
$$
By choosing $\varepsilon=\varepsilon(\delta)>0$ sufficiently small, we obtain
$$
\left|
\arg(1-\langle z,w\rangle)
-\arg(1-\langle z,w'\rangle)
\right|
\leq\delta,
$$
which gives the desired claim.
\end{proof}

Using the dyadic structure described above, the corresponding argument on $\B_n$ follows in the same way as in the case of $\D$, after making the parameter substitutions listed in Table~\ref{higher-dimensional-dictionary}.

\subsection{Weak Hardy estimates in higher dimensions}

The second important ingredient that requires some care is the weak-type estimate for Hardy functions. This estimate plays an essential role in establishing the Forelli--Rudin phenomenon at both the weak-type and restricted weak-type levels. Since this part of the argument is complex-analytic rather than dyadic, its extension from $\D$ to $\B_n$ requires some additional explanation. We first illustrate the necessary modifications by presenting the higher-dimensional counterpart of the argument in Section~\ref{20260831subsec10}, which proves the weak-type estimate in Theorem~\ref{Forelli--Rudin threshold}, (iii).

Let $d\sigma$ denote the normalized surface measure on $\partial\B_n$. For $0<p<\infty$, the Hardy space $H^p(\B_n)$ consists of all holomorphic functions $h$ on $\B_n$ such that
$$
\|h\|_{H^p(\B_n)}:=\sup_{0<r<1}\left(\int_{\partial\B_n}|h(r\zeta)|^p\,d\sigma(\zeta)\right)^{1/p}<+\infty.
$$
For $1<p<\infty$, we shall use the standard duality
$$
\|h\|_{H^p(\B_n)}\simeq\sup_{\|g\|_{H^{p'}(\B_n)}=1}\left|\langle h,g\rangle_{\partial\B_n}\right|,
\qquad \frac1p+\frac1{p'}=1.
$$
Moreover, for $\alpha>-1$ and $0<p<\infty$, the weighted Bergman space $A^p_\alpha(\B_n)$ consists of all holomorphic functions $h$ on $\B_n$ such that
$$
\|h\|_{A^p_\alpha(\B_n)}:=\left(\int_{\B_n}|h(z)|^p(1-|z|^2)^\alpha\,dV(z)\right)^{1/p}<+\infty.
$$

We now consider the higher-dimensional analogue of Theorem~\ref{Forelli--Rudin threshold}, (iii). Under the corresponding assumptions by Table \ref{higher-dimensional-dictionary}, we have
$$
1/2 \leq n+1-c+b<-a<1.
$$
Moreover, since $\calP_{\textnormal{C-H}}^{(n)}$ lies strictly to the left of the vertical cut-off line, we have $c>n$. Consequently,
$$
1<\frac{1}{n+1-c+b}\leq2,
\qquad
2\leq\frac{1}{c-n-b}<+\infty.
$$
As in Subsection~\ref{20260831subsec10}, write
$$
T_{a,b,c}^{(n)}f(z)=(1-|z|^2)^aT_{0,b,c}^{(n)}f(z).
$$
We first establish the higher-dimensional analogue of \eqref{20260602eq01}:
\begin{equation} \label{higher-dimensional-Hardy-bound}
\left\|T_{0,b,c}^{(n)}f\right\|_{H^{\frac{1}{n+1-c+b}}(\B_n)}
\lesssim \|f\|_{L^{\frac{1}{n+1-c+b}}(\B_n)}.
\end{equation}

Recall that every $g\in\mathbb H(\B_n)$ has a homogeneous expansion
$$
g(z)=\sum_{k=0}^{\infty}g_k(z),
$$
where $g_k$ is a homogeneous polynomial of degree $k$. The fractional radial derivative on $\B_n$ is defined by
$$
R^{\alpha,t}g(z):=\sum_{k=0}^{\infty}
\frac{\Gamma(n+1+\alpha)\Gamma(n+1+k+\alpha+t)}
{\Gamma(n+1+\alpha+t)\Gamma(n+1+k+\alpha)}g_k(z);
$$
see \cite[p.~18]{Zhu2005}. Taking $\alpha=-1$ and $t=c-n>0$, we obtain
\begin{equation} \label{higher-dimensional-radial-derivative}
R^{-1,c-n}g(z)=\sum_{k=0}^{\infty}
\frac{\Gamma(n)\Gamma(k+c)}
{\Gamma(c)\Gamma(n+k)}g_k(z).
\end{equation}

For $0<r<1$, set $g_r(z):=g(rz)$ and
$$
\bigl(T_{0,b,c}^{(n)}f\bigr)_r(z):=T_{0,b,c}^{(n)}f(rz).
$$
The kernel expansion
$$
\frac{1}{(1-\langle z,w\rangle)^c}
=\sum_{\alpha\in\mathbb N_0^n}
\frac{\Gamma(c+|\alpha|)}{\Gamma(c)\alpha!}z^\alpha\overline{w}^{\,\alpha},
$$
together with the orthogonality identity
$$
\int_{\partial\B_n}\zeta^\alpha\overline{\zeta}^{\,\beta}\,d\sigma(\zeta)
=\begin{cases}
\displaystyle\frac{\Gamma(n)\alpha!}{\Gamma(n+|\alpha|)},&\alpha=\beta;\\[0.3cm]
0,&\alpha\neq\beta,
\end{cases}
$$
gives
\begin{equation} \label{higher-dimensional-pairing}
\left\langle\bigl(T_{0,b,c}^{(n)}f\bigr)_r,g\right\rangle_{\partial\B_n}
= \int_{\B_n}(1-|w|^2)^bf(w)
\overline{R^{-1,c-n}g_r(w)}\,dV(w).
\end{equation}
Thus, by H\"older's inequality,
\begin{align*}
\left|\left\langle\bigl(T_{0,b,c}^{(n)}f\bigr)_r,g\right\rangle_{\partial\B_n}\right|
&\leq \|f\|_{L^{\frac{1}{n+1-c+b}}(\B_n)}
\left\|R^{-1,c-n}g_r\right\|_
{A^{\frac{1}{c-n-b}}_{\frac{b}{c-n-b}}(\B_n)}.
\end{align*}
Since
$$
\frac{1}{c-n-b}\geq2
\qquad\textnormal{and}\qquad
\frac{c-n}{c-n-b}-1=\frac{b}{c-n-b}>-1,
$$
by \cite[Theorem~4.41]{Zhu2005} with $\alpha=-1$ and $t=c-n$ there, we deduce that
$$
\left\|R^{-1,c-n}g_r\right\|_
{A^{\frac{1}{c-n-b}}_{\frac{b}{c-n-b}}(\B_n)}
\lesssim \|g_r\|_{H^{\frac{1}{c-n-b}}(\B_n)}
\leq \|g\|_{H^{\frac{1}{c-n-b}}(\B_n)}.
$$
It follows that
$$
\left|\left\langle\bigl(T_{0,b,c}^{(n)}f\bigr)_r,g\right\rangle_{\partial\B_n}\right|
\lesssim \|f\|_{L^{\frac{1}{n+1-c+b}}(\B_n)}
\|g\|_{H^{\frac{1}{c-n-b}}(\B_n)},
$$
uniformly in $0<r<1$. Taking the supremum over all $g$ with
$$
\|g\|_{H^{\frac{1}{c-n-b}}(\B_n)}=1,
$$
and then over $0<r<1$, proves \eqref{higher-dimensional-Hardy-bound}.

It remains to record the higher-dimensional version of the weak Hardy estimate.

\begin{lem} \label{weakHardy-higher-dimensional}
Let $0<q<p<\infty$. For each $h\in H^p(\B_n)$, define
$$
M_qh(z):=(1-|z|^2)^{-\frac1q}h(z).
$$
Then
$$
\|M_qh\|_{L^{q,\infty}(\B_n)}
\lesssim \|h\|_{H^p(\B_n)}.
$$
\end{lem}

\begin{proof}
The proof is parallel. For $\lambda>0$, set
$$
E_\lambda:=\{z\in\B_n:|M_qh(z)|>\lambda\}
=\left\{z\in\B_n:|h(z)|>\lambda(1-|z|^2)^{\frac1q}\right\}.
$$
If $h\equiv0$, there is nothing to prove. Otherwise, put
$$
A:=\frac{\|h\|_{H^p(\B_n)}}{\lambda}.
$$
For each $0<r<1$, by Chebyshev's inequality, we have
\begin{align*}
\sigma\left(\left\{\zeta\in\partial\B_n:
|h(r\zeta)|>\lambda(1-r^2)^{\frac1q}\right\}\right)
&\leq \min\left\{1,
\frac{\|h\|_{H^p(\B_n)}^p}
{\lambda^p(1-r^2)^{\frac pq}}\right\}\\
&= \min\left\{1,A^p(1-r^2)^{-\frac pq}\right\}.
\end{align*}
Using polar coordinates with real dimension $2n$, we therefore obtain
$$
|E_\lambda| \leq 2n\int_0^1
\min\left\{1,A^p(1-r^2)^{-\frac pq}\right\}
r^{2n-1}\,dr.
$$
If $A\geq1$, then $|E_\lambda|\leq1\leq A^q$. Suppose next that $0<A<1$ and set
$$
r_A:=\sqrt{1-A^q}, 
$$
and we estimate 
$$
|E_\lambda| \lesssim_n A^p\int_0^{r_A}(1-r^2)^{-\frac pq}r\,dr
+\int_{r_A}^1r\,dr \lesssim_{n,p,q}A^q.
$$
Thus, in either case,
$$
\lambda^q|E_\lambda| \lesssim \|h\|_{H^p(\B_n)}^q.
$$
Taking the supremum over $\lambda>0$ proves the desired estimate.
\end{proof}

Applying Lemma~\ref{weakHardy-higher-dimensional} with
$$
p=\frac{1}{n+1-c+b}, \qquad q=-\frac{1}{a},
$$
and using $n+1-c+b<-a$, we conclude that
\begin{align*}
\left\|T_{a,b,c}^{(n)}f\right\|_{L^{-\frac1a,\infty}(\B_n)}
&=\left\|(1-|\cdot|^2)^aT_{0,b,c}^{(n)}f\right\|_{L^{-\frac1a,\infty}(\B_n)}\\
&\lesssim \left\|T_{0,b,c}^{(n)}f\right\|_{H^{\frac{1}{n+1-c+b}}(\B_n)}\\
&\lesssim \|f\|_{L^{\frac{1}{n+1-c+b}}(\B_n)}.
\end{align*}
This gives the required higher-dimensional modification of the proof of Theorem~\ref{Forelli--Rudin threshold}, (iii).

\medskip

Finally, at the common intersection $\calP_{\textnormal{C-H-V}}^{(n)}$, we have $c=n$ and hence $R^{-1,c-n}=R^{-1,0}=I$. Equivalently, the kernel $(1-\langle z,w\rangle)^{-n}$ has the form of the Cauchy--Szeg\H{o} kernel on $\B_n$, and \eqref{higher-dimensional-pairing} reduces to
$$
\left\langle T_{0,b,n}^{(n)}f,g\right\rangle_{\partial\B_n}
=\int_{\B_n}(1-|w|^2)^bf(w)\overline{g(w)}\,dV(w).
$$
The higher-dimensional counterpart of the argument in Section~\ref{20260831subsec54} then follows from Lorentz--H\"older's inequality and Lemma~\ref{weakHardy-higher-dimensional} in exactly the same way. We leave the details to the interested reader.

\end{document}